\documentclass[10pt]{amsart}
\usepackage{amsfonts}
\usepackage{amsmath,amsthm,extarrows} 
\usepackage{hyperref} 
\usepackage{latexsym}
\usepackage{array}
\usepackage{amssymb}
\usepackage{enumerate}
\usepackage[english]{babel}

\theoremstyle{plain} 
\newtheorem{theorem}{Theorem}[section]
\newtheorem{lemma}[theorem]{Lemma}
\newtheorem{proposition}[theorem]{Proposition}
\newtheorem{corollary}[theorem]{Corollary} 
\newtheorem{definition}[theorem]{Definition} 

\theoremstyle{remark}
\newtheorem{remark}[theorem]{Remark}
\newtheorem{example}[theorem]{Example}

\newtheorem*{theorem*}{Theorem}
\newtheorem*{lem*}{Lemma}
\newtheorem*{sublem*}{Sublemma}
\newtheorem*{remark*}{Remark}
\newtheorem*{NB*}{NB}

\usepackage{color}

\definecolor{antiquebrass}{rgb}{0.8, 0.58, 0.46}
\definecolor{armygreen}{rgb}{0.29, 0.33, 0.13}
\definecolor{bole}{rgb}{0.47, 0.27, 0.23}

\newcommand{\C}{  \mathbb{C}   }
\newcommand{\Z}{  \mathbb{Z}   }
\newcommand{\N}{  \mathbb{N}   }
\newcommand{\R}{  \mathbb{R}}
\newcommand{\PP}{  \mathbb{P}\,}
\newcommand{\bJ}{ \,\mathbb J  }
\newcommand{\Q}{  \mathbb{Q}   }
\newcommand{\T}{  \mathbb{T}   }

\newcommand{\bA}{  \mathbb{A}   }

\newcommand{\A}{  \mathcal{A}   }

\newcommand{\M}{  \mathcal{M}   }
\newcommand{\F}{  \mathcal{F}   }
\newcommand{\Ca}{  \mathcal{C}   }
\newcommand{\J}{  \mathcal{J}   }
\newcommand{\E}{  \mathcal{E}   }
\newcommand{\NF}{  \mathcal{NF}   }
\renewcommand{\H}{  \mathcal{H}   }
\renewcommand{\O}{  \mathcal{O}   }
\newcommand{\B}{  \mathcal{B}   }
\newcommand{{\Tc}}{  \mathcal{T}   }
\newcommand{\Rc}{  \mathcal{R}   }
\newcommand{\Cc}{  \mathcal{C}   }

\newcommand{\D}{  \mathcal{D}}
\renewcommand{\P}{  \mathcal{P}   }
\renewcommand{\L}{  \mathcal{L}   }

\newcommand{\Omad}{  \Omega_{\text{adm}}}
\newcommand{\Omsad}{  \Omega_{\text{s\_adm}}}

\newcommand{\cA}{  \mathcal{A}   }
\newcommand{\cB}{  \mathcal{B}   }
\newcommand{\cC}{  \mathcal{C}   }

\newcommand{\cM}{  \mathcal{M}   }

\newcommand{\cT}{  \mathcal{T}   }

\newcommand{\cZ}{  \mathcal{Z}   }

\newcommand{\Da}{  \mathcal{D}_{c_1}}
\newcommand{\Db}{  \mathcal{D}_{c_2}}

\newcommand{\bB}{  \mathbf{B}   }
\newcommand{\bS}{  \mathbf{S}   }

\newcommand{\Tg}{  {\mathbf T}_\rho  }

\newcommand{\fJ}{  \mathfrak J  }
\newcommand{\zz}{\mathfrak z}

\newcommand{\eps}{\varepsilon}
\newcommand{\om}{  \omega   }
\newcommand{\Om}{  \Omega}

\newcommand{\tl}{  \theta_\ell }

\newcommand{\ga}{\gamma   }
\newcommand{\s}{  \sigma   }
\newcommand{\ka}{  \kappa   }
\newcommand{\ls}{  \lambda_s   }

\newcommand{\Am}{  \Lambda}
\newcommand{\la}{  \lambda_a   }
\newcommand{\lb}{  \lambda_b   }
\newcommand{\La}{  \Lambda_a   }
\newcommand{\Lb}{  \Lambda_b   }
\newcommand{\li}{  \lambda_i   }
\newcommand{\lj}{  \lambda_j   }
\newcommand{\lk}{  \lambda_k   }
\newcommand{\lel}{  \lambda_\ell   }

\newcommand{\yy}{\rho}
\renewcommand{\r}{  \rho   }
\newcommand{\vark}{  \varkappa   }
\newcommand{\vs}{  \varsigma    }

\newcommand{\de}{  \delta   }
\newcommand{\al}{  \alpha   }

\renewcommand{\phi}{  \varphi  }

\newcommand{\p}{\partial}
\newcommand{\bb}{\beta_{\#}}
\renewcommand{\(}{  \big(   }
\renewcommand{\)}{  \big)   }    
\newcommand{\sa}{  strongly admissible}

\newcommand{\no}{ \;\text{not}\;   }
\newcommand{\ap}{ a^{\prime\prime}   }
\newcommand{\aaa}{ a, a', a^{\prime\prime}  }
\newcommand{\an}{ \, \angle\,}
\newcommand{\ann}{ \, \angle\!\angle\,}
\newcommand{\ov}{  \overline  }

\newcommand{\DD}{  \tilde{Q}_{l}  }
\newcommand{\dd}{  \text{d}   }

\newcommand{\tkd}{  { K}^d   }
\newcommand{\tkn}{   K ^{n/d}  }

\newcommand{\cc}{\frac1{2\sqrt2}}

\def\der#1#2{\frac{d^{#1}\omega_{#2}}{dm^{#1}}}

\newcommand{\lsim}{  \lesssim   }
\newcommand{\gsim}{  \gtrsim   }

\def\ab#1{\left|#1\right|}
\def\aa#1{\left\Vert#1\right\Vert}

\newcommand{\diag}{\operatorname{diag}}
\newcommand{\meas}{\operatorname{meas}}

\newcommand{\dist}{\operatorname{dist}}

\newcommand{\Leb}{\operatorname{meas}}

\newcommand{\cAd}{\operatorname{ad}}

\newcommand{\Id}{\operatorname{Id}}

\newcommand{\cte}{ {\operatorname{ct.}  } }
\newcommand{\Cte}{ {\operatorname{Ct.}  } }
\newcommand{\Haus}{ {\operatorname{Hausdorff}  } }

\newcommand{\be}{\begin{equation}}
\newcommand{\ee}{\end{equation}}
\newcommand{\ben}{\begin{equation*}}
\newcommand{\een}{\end{equation*}}
\newcommand{\ban}{\begin{align*}}
\newcommand{\ean}{\end{align*}}

\numberwithin{equation}{section}

\author{L. Hakan   Eliasson}
\address{
Univ. Paris Diderot, Sorbonne Paris Cit\'e\\
Institut de Math\'emathiques de Jussieu-Paris rive gauche, UMR 7586\\
CNRS\\
Sorbonne Universit\'es, UPMC Univ. Paris 06\\
F-75013, Paris, France
} 
\email{hakan.eliasson@imj-prg.fr}

 \author{ Beno\^it Gr\'ebert}
\address{Laboratoire de Math\'ematiques Jean Leray, Universit\'e de Nantes, UMR CNRS 6629\\
2, rue de la Houssini\`ere \\
44322 Nantes Cedex 03, France}
\email{benoit.grebert@univ-nantes.fr}
\author{ Sergei B. Kuksin }
\address{
CNRS\\
Institut de Math\'emathiques de Jussieu-Paris rive gauche, UMR 7586\\
Univ. Paris Diderot, Sorbonne Paris Cit\'e\\
Sorbonne Universit\'es, UPMC Univ. Paris 06\\
F-75013, Paris, France
}
\email{sergei.kuksin@imj-prg.fr}

\title[KAM for the nonlinear beam equation]{KAM for the nonlinear beam equation.
}

\begin{document}

\begin{abstract}
In this paper we  prove a KAM theorem for small-amplitude solutions of 
 the non linear beam equation on the d-dimensional torus
$$u_{tt}+\Delta^2 u+m u + \partial_u G(x,u)=0\ ,\quad  t\in {  \mathbb{R}}
, \; x\in \ {  \mathbb{T}}^d, \qquad \qquad
(*)
$$
 where $G(x,u)=u^4+ O(u^5)$.  Namely, we show  that,
for generic $m$, many of the small amplitude invariant finite dimensional  tori of the linear equation $(*)_{G=0}$,
written as the system
$$
u_t=-v,\quad v_t=\Delta^2 u+mu,
$$
persist as invariant tori  of the nonlinear equation  $(*)$, re-written similarly. 
The persisted tori are filled in with time-quasiperiodic solutions of  $(*)$. 
If $d\ge2$, then not all the persisted tori are  linearly stable,
and  we construct explicit examples  of partially hyperbolic  invariant tori.
The unstable invariant tori, situated in the vicinity of the origin, create around 
them   some local instabilities, in agreement with the popular belief  in the nonlinear 
physics that 
small-amplitude solutions of  space-multidimensional Hamiltonian
PDEs behave in a chaotic way. 

 \end{abstract}

%\thanks{  }

\maketitle
\tableofcontents

\section{Introduction}
\subsection{The beam equation  and  KAM for PDE's} \label{s_1.1}
The paper deals with small-amplitude solutions of 
 the multi-dimensional nonlinear 
 beam equation on the torus:
\be \label{beam}u_{tt}+\Delta^2 u+m  u =   -  g(x,u)\,,\quad u=u(t,x), \ 
  t\in \R, \ x\in \T^d=\R^d/(2\pi\Z)^d,
\ee
 where  $g$ is a real analytic function of $x\in\T^d$ and of $u$ in the vicinity of the origin in 
 $\R$. We shall consider  functions $g$ of the form
 \begin{equation}  \label{g}
 g=\partial_u G,\quad G(x,u)=u^4+O(u^5).
 \end{equation}
 The polynomial $u^4$ is the {\it main part} of $G$ and $O(u^5)$ is its {\it higher order  part}.  $m$ is the mass parameter and we assume that $m\in[1,2]$. 
 
 This equation is interesting by itself. Besides, it is a good model for the Klein--Gordon equation 
 \be\label{KG}
 u_{tt} - \Delta u+mu=-\partial_u G(x,u),\qquad x\in\T^d, 
 \ee
 which is among the most important equations of mathematical physics. We feel confident that the ideas and methods 
 of our work apply -- with additional technical efforts -- to  eq.~\eqref{KG} (but the situation with the nonlinear wave 
 equation \eqref{KG}${}_{m=0}$, as well as with the zero-mass beam equation, may be quite different).
 
 Our goal is to develop a general KAM-theory for small-amplitude solutions of \eqref{beam}.  To do this we compare them 
     with time-quasi-periodic solution of the linearised at zero equation 
  \be\label{linear}
  u_{tt} +\Delta^2 u+ mu=0\,.
  \ee
 Decomposing real functions $u(x)$ on $\T^d$ to Fourier series 
 $$
 u(x)= \sum_{a\in\Z^d} u_a e^{{\bf i}\langle a,  x\rangle }\ +\text{c.c.}
 $$
 (here c.c. stands for  ``complex conjugated"), we write time-quasiperiodic solutions for \eqref{linear}, corresponding 
 to a finite set of excited wave-vectors $\A \subset \Z^d $,  as
 \be\label{sol}
 u(t,x) = \sum_{a\in\A} (\xi_a e^{{\bf i}\lambda_a t}+ \eta_ae^{-{\bf i}\lambda_a t}) e^{{\bf i}\langle a,  x\rangle }
 + \text{c.c.},
 \ee
 where $\lambda_a = \sqrt{|a|^4+m}\,$.   We examine these solutions and their perturbations in eq.~\eqref{beam}
 under the assumption that the  action-vector 
 $
I= \{\tfrac12( |\xi_a|^2 +|\eta_a|^2),\ a\in\A\}\ 
 $
 is small. 
  In our work this goal is achieved  provided that 
 
 \noindent
 - the  finite set $\A$ is typical in a probabilistic  sense;
  
 \noindent
 - the mass parameter $m$ does not belong to a certain set of zero measure.

  The linear stability of the obtained solutions for  \eqref{beam} is under control. If $d\ge2$, and $|\A|\ge2$, 
  then some of them are  linearly unstable. 
 \smallskip

 The specific choice of a Hamiltonian PDE with the mass parameter which we work with   -- the beam equation \eqref{beam} -- 
 is sufficiently arbitrary. This is simply the easiest non-linear space-multidimensional equation from mathematical
 physics for which we can perform our programme of the KAM-study of small-amplitude solutions in space-multidimensional 
 Hamiltonian PDEs, and obtain for them the results, outlines above. 
 %We are certain that out picture of the KAM-behaviour of  small solutions, as well as the method, developed to prove it, are sufficiently, general. In particular,  we believe that out method applies to the Klein-Gordon equation \eqref{KG}. 
 \smallskip
 
 Before to give exact statement of the result, we discuss the state of affairs in the KAM for PDE theory. The theory 
 started in late 1980's and originally applied to 1d Hamiltonian PDEs, see in \cite{K87, K93, Cr}. The first works 
 on this theory  treated 
 
 a) perturbations of linear Hamiltonian PDE, depending on a vector-parameter of the dimension, equal to 
  the number  of frequencies of the unperturbed quasiperiodic solution of the linear system (for solutions \eqref{sol} this is
 $|\A|$). Next the theory was applied to 
 
 b) perturbations of integrable Hamiltonian PDE, e.g.  of the KdV or Sine-Gordon equations, see \cite{K00}. In paper \cite{BoK} 
 
 c) small-amplitude solutions of the 1d Klein-Gordon equation \eqref{KG} with $G(x,u)=-u^4+O(u^4)$ 
  were treated as perturbed  solutions of the Sine-Gordon equation,  and a singular version of the KAM-theory b) was developed to study them. 
 (Notice that for suitable $a$ and $b$ we   have $mu-u^3+O(u^4) = a\sin bu+O(u^4)$. So the 1d equation \eqref{KG} is the Sine-Gordon equation,   perturbed by a small term $O(u^4)$.)

  It was proved in \cite{BoK}   that for a.a. values of $m$ and for 
  any finite set $\A$  most of the  small-amplitude solutions \eqref{sol} for the 
  linear Klein-Gordon equation (with $\lambda_a=\sqrt{|a|^2+m}$) persist as linearly stable time-quasipe\-rio\-dic 
  solutions for \eqref{KG}. In \cite{KP} it was realised that it can be fruitful in 1d
  equations like   \eqref{KG}, just as it is in finite-dimensional Hamiltonian systems  (see for example \cite{E88}), to study small solutions not as perturbations of solutions for an integrable PDE, but rather as perturbations of solutions 
  for a Birkhoff--integrable system, after the equation is normalised by a Birkhoff transformation. The paper \cite{KP} deals not with 
  1d Klein-Gordon
   equation \eqref{KG}, but with 1d NLS equation, which is similar to \eqref{KG} 
   for the problem under discussion; in \cite{P} the method of \cite{KP} was applied to the 1d equation \eqref{KG}.
    The approach of \cite{KP}   turned out to be very efficient and later was used for many other 1d Hamiltonian PDEs.  
     In \cite{GY06b} it was applied to the $d$-dimensional beam equation \eqref{beam} with an $x$-independent 
     nonlinearity $g$    and allowed to treat perturbations of some special solutions \eqref{sol}.
     \smallskip
    
  Space-multidimensional KAM for PDE theory started  10 years later with the paper \cite{B1} and, next, publications 
  \cite{B2}
  and \cite{EK10, EK09}. The just mentioned works deal with perturbations of 
  parameter-depending linear equations (cf. a)\,). The approach of 
  \cite{EK10, EK09} is different from that of \cite{B1, B2} and allows to analyse the linear stability of the obtained KAM-solutions. 
  Also see \cite{BB12, BB13}.  Since integrable space-multidimensional PDE (practically) do not exist, then no 
  multi-dimensional analogy of the 1d theory b) is available. 
  
  Efforts to create space-multidimensional analogies of the KAM-theory c) were made in \cite{WM} and \cite{PP1, PP2}, using the
  KAM-techniques of \cite{B1, B2} and \cite{EK10}, respectively. Both works deal with the NLS equation. Their main 
  disadvantage compare to the 1d theory c) is severe restrictions on the finite set $\A$ (i.e. on the class of unperturbed solutions 
  which the methods allow to perturb).
  The result of \cite{WM} gives examples 
  of some sets $\A$ for which the KAM-persistence of the corresponding small-amplitude solutions \eqref{sol} holds,
  while the result of \cite{PP1, PP2} applies to solutions \eqref{sol}, where the set $\A$ is nondegenerate in certain very 
  non-explicit way. The corresponding 
   notion of non-degeneracy is so complicated that it is  not easy to give  examples of 
  non-degenerate sets $\A$. 
  
  Some KAM-theorems for small-amplitude solutions of multidimensional beam equations \eqref{beam}   
  with typical $m$   were obtained in
  \cite{GY06a, GY06b}. Both works treat equations  with a constant-coefficient nonlinearity 
  $g(x,u)=g(u)$, which  is significantly easier than the general case (cf. the linear theory, where constant-coefficient
  equations may be integrated by the Fourier method). Similar to \cite{WM, PP1, PP2}, the theorems of  \cite{GY06a, GY06b}
  only allow to perturb solutions \eqref{sol} with very special sets $\A$ (see also Appendix B). Solutions of \eqref{beam}, constructed in these works, 
  all are  linearly stable.

 \subsection{Beam equation in real and complex variables}\label{s_complex}
 Introducing $v=u_t\equiv\dot u$ we rewrite 
 \eqref{beam} as 
\be\label{beam'}
 \left\{\begin{array}{ll}
 \dot u &= - 
 v,\\
 \dot v &=\Lambda^2 u    +g(x,u)\,,
\end{array}\right.
\ee
where $\Lambda=(\Delta^2+m)^{1/2}$. 
Defining 
 $
 \psi(t,x) =\frac 1{\sqrt 2}(\Lambda^{1/2}u  +{\bf i}\Lambda^{-1/2}v) $
 we get for the complex function 
  $\psi(t,x)$ the equation
$$
\frac 1{\bf i}\dot \psi =\Lambda \psi+ \frac{1}{\sqrt 2}\Lambda^{-1/2}g\left(x,\Lambda^{-1/2}\left(\frac{\psi+
\bar\psi}{\sqrt 2}\right)\right)\,.
$$
Thus, if we endow the space   $L^2(\T^d, \C)$ with the standard  real symplectic structure,  given by the two-form
$\ 
-{\bf i}d\psi\wedge d\bar \psi,
$
%where $\psi =\frac1{\sqrt2} (\tilde u+i\tilde v)$, 
then equation 
 \eqref{beam} becomes a Hamiltonian system 
$$\dot \psi={\bf i} \,{\partial h} /{\partial \bar\psi}
$$
with the Hamiltonian function
$$
h(\psi,\bar\psi)=\int_{\T^d}(\Lambda \psi)\bar\psi \dd x +\int_{\T^d}G\left(x,\Lambda^{-1/2}\left(\frac{\psi+\bar\psi}{\sqrt 2}\right)\right)\dd x.
$$
The linear operator $\Lambda$ is diagonal in the complex Fourier basis  
$$
\{e_a(x)= {(2\pi)^{-d/2}}e^{{\bf i}\langle a,  x\rangle }, \ a\in\Z^d\}.
$$
Namely, 
$$
\Lambda e_a=\lambda_a e_a,\;\;\lambda_a= \sqrt{|a|^4+m},
\qquad  \forall\,a\in\Z^d\,.
$$

Let us decompose $\psi$ and $\bar\psi$  in the   basis $\{e_a\}$:
$$
\psi=\sum_{a\in\Z^d}\xi_a e_a,\quad \bar\psi=\sum_{a\in\Z^d}\eta_a e_{-a}\,.
$$
Let
\be\label{change}
\left\{\begin{array}{l} 
p_a=\frac1{\sqrt2}(\xi_a+\eta_a) \\
q_a=\frac{{\bf i}}{\sqrt2}(\xi_a-\eta_a)
\end{array}\right.\ee
and denote by $\zeta_a$ the pair $(p_a,q_a)$. 
\footnote{\  $\zeta_a$ will be considered as a line-vector or a colon-vector
according to the context.}

We fix any $m_*>d/2$ and define the Hilbert space
\be\label{YC}
Y = \{\zeta=(p,q)\in \ell^2(\Z^d,\C)\times\ell^2(\Z^d,\C) \mid 
\aa{\zeta}^2=\sum_a \langle a\rangle^{2m_*} |\zeta_a|^2  <\infty \}\,,
\ee
-- $\langle a\rangle =\max(1, |a|)$ -- corresponding to the decay of Fourier coefficients of complex functions 
$(\psi(x), \bar\psi(x))$ from the Sobolev space 
 $ H^{m_*}(\T^d, \C^2)$. A vector $\zeta\in Y$ is called {\it real} if all its components are real.
 
 Let us endow $Y$ with the 
symplectic structure 
\be\label{J}
\big(dp \wedge dq\big) (\zeta,\zeta')= \sum_a \langle J\zeta_a,\zeta'_a\rangle,\quad
J=\left(\begin{array}{cc} 0&1\\-1&0\end{array}\right)\,,
\ee
 and  consider there the Hamiltonian system
\be \label{beam2} 
\dot\zeta_a=J\frac{\partial h}{\partial \zeta_a},\quad a\in\Z^d\,,
%\left\{\begin{array}{ll}\dot \xi_a&=i\frac{\partial h}{\partial \eta_a}\\ \dot \eta_a&=-i\frac{\partial h}{\partial \xi_a}\end{array}\right. \quad a\in\Z^d\,,
\ee
where the Hamiltonian function $h$ equals the quadratic part
\be\label{H2} 
h_2=\frac12 \sum_{a\in\Z^d}\lambda_a (p^2_a+q^2_a)\ee
plus the higher order term
\be\label{H1}
h_{\ge4}=
\int_{\T^d}G\left(x,\sum_{a\in\Z^d}\frac{(p_a-{\bf i}q_a) e_a+(p_{-a}+{\bf i}q_{-a}) e_a}{2\sqrt{\lambda_a}}\right)\dd x.\ee
The  beam equation \eqref{beam'}, considered in the Sobolev space $\{(u,v) \mid(\psi, \bar\psi) \in H^{m_*}\}$, is 
 equivalent to the  Hamiltonian system \eqref{beam2}.

We will  write the Hamiltonian $h$ as 
\be\label{PPP}
h = h_2+h_{\ge4}
= h_2+h_4+h_{\ge5}\,,
\ee
where 
 \be\label{quatr}
 h_4=\int_{\T^d}u^4\dd x= \int_{\T^d}\left(\sum_{a\in\Z^d}\frac{(p_a-{\bf i}q_a) e_a+(p_{-a}+{\bf i}q_{-a}) e_a}{2\sqrt{ \lambda_a}}\right)^4\dd x,
 \ee
 $h_{\ge5} = O(u^5)$ comprise the remaining higher order terms and $ h_{\ge4} = h_4+h_{\ge5}$. 
  Note that $h_4$ 
 satisfies the {\it zero momentum condition}, i.e. 
 $$
 h_4=\sum_{a,b,c,d\in\Z^d}C(a,b,c,d) (\xi_a+\eta_{-a})
( \xi_b+\eta_{-b}) 
(\xi_c+\eta_{-c})
(\xi_d+\eta_{-d})\,,
%\;\; \text{ with } \;\; C(i,j,k,\ell)\neq0 \Rightarrow i+j=k+\ell.
 $$
 where $C(a,b,c,d)\ne0$ only if $a+b+c+d=0$.  This  condition turns out to  be  useful to restrict the set of small divisors that  have to be controlled.   If the function $G$ does not depend on $x$, then $h$ satisfies a similar property at any order.
 
 \subsection{Invariant tori and admissible sets}
 The quadratic Hamiltonian $h_2$ (which is $h$ when $G=0$ in \eqref{beam}) is integrable and its phase-space is foliated into
(Lagrangian or isotropic)  invariant tori. Indeed, take  a finite  subset  $\A\subset\Z^d$ and let
 $$
\L = \Z^d\setminus \A\,.$$
 For any subset $X$ of $\Z^d$, consider the projection
$$\pi_X:(\C^2)^{\Z^d}\to (\C^2)^{X}=\{\zeta\in (\C^2)^{\Z^d}: \zeta_a=0\ \forall a\notin X\}.$$
We can thus write $(\C^2)^{\Z^d}=(\C^2)^{X}\oplus (\C^2)^{\Z^d\setminus X}$,
$\zeta=(\zeta_X,\zeta_{\Z^d\setminus X})$,
and when $X$ is finite this gives an injection
$$
(\C^2)^{\#X}\hookrightarrow (\C^2)^{\Z^d}$$
whose image is $ (\C^2)^{X}$. \footnote{\ we shall frequently, without saying, identify $ (\C^2)^{X}$ and $(\C^2)^{\#X}$}

For any real vector  with positive components $I_\A=(I_a)_{a\in\A}$, the $|\A|$-dimensional torus
\be\label{ttorus}
T_{I_\A}=
 \left\{\begin{array}{lll}
 p_a^2+q_a^2 =2I_a& p_a, q_a\in \R,\; &a\in \A\\
p_a=q_a=0& & a\in \L \,,
\end{array}\right.
\ee
is invariant under the flow of $h_2$. $T_{I_\A}$ is the image of the torus 
\be\label{ttorusbis}
\T^\A=\{r_\A=0\}\times \{\theta_a\in\T: a\in\A\}\times\{\zeta_\L=0\}\ee
%\footnote{\ the torus $\T^\A$ is of course trivially identified with the $|\A|$-dimensional torus $\T^{|\A|}$} 
under the  embedding
\be\label{embedding}
U_{I_\A}:\theta_\A\mapsto 
 \left\{\begin{array}{ll}
p_a-{\bf i}q_a=\sqrt{2I_a}  \, e^{{\bf i} \theta_a} &a\in \A\\
p_a=q_a=0 & a\in \L  \,,
\end{array}\right.
\ee
and the pull-back, by $U_{I_\A}$, of the induced flow is simply the translation
\be\label{inducedflow}
\theta_\A\mapsto \theta_\A+t\om_\A,\ee
where we have denoted  the translation vector (the tangential frequencies) by $\om_\A$, i.e. $\lambda_a=\om_a$ for $a\in\A$.
The  parametrised  curve
$$
 t \mapsto U_{I_\A}(\theta +t\omega)$$
is thus a  quasi-periodic solution of the beam equation \eqref{beam2} when $G=0$.

When $G\not=0$ the higher order terms in $h$ give rise to a perturbation of $h_2$  --  a perturbation that gets smaller, the smaller
is $I$. Our goal is to prove the persistency of the invariant torus $T^{\A}_I$, or, more precisely, of the invariant embedding $ U_I^\A$,  for most values of $I$ when the higher order terms are taken into account. The problem doing this for this model is two-fold. First the integrable Hamiltonian $h_2$ is completely degenerate in the sense of KAM-theory: the frequencies $\om_\A$ do not depend on $I$. One can try to improve  this  by adding to $h_2$ an integrable part of the Birkhoff normal form. This will,  in ``generic'' situations, correct this default. However, and that's the second problem, our model is far from ``generic'' since the eigenvalues $\{\la: a\in \Z^d\}$ are very resonant. This has the effect that the  Birkhoff normal form is not integrable, and therefore is difficult to use.

An important part of our analysis will be to show that this program can be carried out if we exclude a zero-measure set of masses $m$ and restrict the choice of $\A$ to {\it admissible} or {\it strongly admissible} sets.

\medskip

Let $|\cdot|$ denote the euclidean norm in $\R^d$. For  vectors $a,b \in\Z^d$ we define 
\be\label{ddd}
a \an b \quad \text{iff} \quad \#\{x\in \Z^d \mid |x|=|a| \;\text{and}\;
 |x-b| = |a-b|\} \le 2\,.\ee
Relation $a\an b$ means that 
the integer sphere of radius $|b-a|$ with the centre at $b$ intersects the integer sphere
$\{x\in\Z^d\mid |x|=|a|\}$  
in at most two points. 

\begin{definition}\label{adm}
A finite set $\A\in\Z^d$  is called admissible iff
$$
 a,b\in\A, \ a\ne b
\Rightarrow |a|\neq |b|\,.$$

An admissible set $\A$ is called strongly admissible iff
$$
 a,b\in\A, \ a\ne b
\Rightarrow a\an a+ b\,.
$$

\end{definition}
Certainly if $|\A|\le1$, then $\A$  is admissible, but for  $|\A| >1$  this is not true.
For $d\le2$ every admissible set is strongly admissible, but in higher dimension this is no longer true:
see  for example the set \eqref{AAA} in Appendix B. 

However, strongly admissible, and hence admissible sets are typical: see Appendix E for a precise formulation and proof
of this statement.

%\be\label{admis}\text{for any $n$,   strongly admissible $n$-points random $R$-sets  with $R\gg1$ are typical,} \ee where an {\it $n$-points random $R$-set} is a set  with $n$ elements, chosen at random in the  integer ball $\{a\in\Z^d:  |a|\le R\}$.  

\medskip

We shall define a subset of $\L$, important for our construction:
\be\label{L+}
 {\L_f}=\{a\in\L \mid \exists\ b\in\A \text{ such that } |a|=|b|  \}.
\ee
Clearly $\L_f$ is a finite subset of $\L$. For example, if $d=1$ and $\A$ is admissible, then $\A\cap-\A\subset\{0\}$, so
if $d=1$,  then  ${\L_f} = -(\A\setminus\{0\})$.

 \subsection{The Birkhoff normal form}
 
In a neighbourhood of an invariant torus $T_{I_\A}$  
%in the real sub-space  $\{(\xi_a=\bar\eta_a: a\in\A)\}\subset \C^{\A}\times\C^{\A}$
 we introduce (partial) action-angle variables $(r_\A,\theta_\A,\xi_\L,\eta_\L)$ by the relation 
\be\label{actionangle}
 \frac1{\sqrt2}(p_a-{\bf i}q_a)=\sqrt{I_a+r_a}  \, e^{{\bf i} \theta_a},\quad a\in\A.
\ee

These variables define a diffeomorphism from a neighbourhood of $\T^\A$ in (the Hilbert manifold)
\be\label{YCmanifold}
\C^{\cA}\times(\C/2\pi \Z)^{\cA}\times \pi_{\L}Y\ee
to a neighbourhood of $T_{I_\A}$ in $Y$. It is {\it real} in the sense that it gives real values to real arguments.

The symplectic structure on $Y$ is pull-backed to
\be\label{Jbis}
dr_\A\wedge d\theta_\A +d\xi_\L\wedge d\eta_\L,\ee
which endows the space \eqref{YCmanifold} with a symplectic structure.

In these variables $h$ will depend on $I$, but its integrable part $h_2$ becomes, up to an additive constant, 
$$
\sum_{a\in\A}\om_a r_a+  \frac12\sum_{a\in\L}\lambda_a (p_a^2+q_a^2)$$
which does not depend in $I$. 
\footnote{\  both $h_2$ and the higher order terms of $h$ depend on the mass $m$.} 
%The induced Hamiltonian flow of $h_2$ on $\T^\A$ is simply \eqref{inducedflow}.
The Birkhoff normal form will provide us with an integrable part that does depend on $I$.
We shall prove

  \begin{theorem}\label{NFTl}
 There exists a zero-measure Borel set $\Cc\subset[1,2]$ such that for any $m\notin\Cc$, any
  admissible set $\A$,  any $c_*\in(0,1/2]$ and any analytic  nonlinearity  of the form \eqref{g}, 
 there exist $\nu_0>0$ and $\beta_0>0$ such that for any $0<\nu\le\nu_0$, $0<\bb\le\beta_0$ there exists an open set
  $Q \subset [\nu c_*,\nu]^\A$,
  $$\Leb([\nu c_*,\nu]^\A\setminus Q)\le  C\nu^{\#\A+\bb},$$
  and for every $I=I_\A\in Q$ there exists a real symplectic holomorphic 
 diffeomorphism $\Phi_{I} $, defined in a neighbourhood (that depends on $c_*$ and $\nu$) of  $\T^\A$
 %$$\Phi^*(-i d\xi\wedge d_\eta)=-dr_\A\wedge d\theta_\A -id\xi_\L\wedge d\eta_\L$$ and 
 such that
 \be\label{transff}
\begin{split}
& h\circ\Phi_I(r_\A, \theta_\A, p_\L, q_\L)=
\langle\Omega(I), r_\A\rangle +\frac12\sum_{a\in \L\setminus \L_f  } \Lambda_a(I) (p_a^2+q_a^2)+\\
&+  \frac12\sum_{b\in\L_f\setminus\F} \Lambda_b(I)    (p_b^2+q_b^2)
 +  \langle K(I)\zeta_\F,\zeta_\F\rangle + f_I(r_\A,\theta_\A, p_\L,q_\L)\,,
\end{split}
\ee
where $\F=\F_I$ is a (possibly empty) subset of $ \L_f$, has the following properties:

i) $\Omega(I)=\om_\A + M I$ and the matrix $M$ is invertible;  

ii) each  $\La(I)$, $a\in \L\setminus \L_f$,   is real and close to $\la$,
$$\ab{ \La(I)-\la }\le C\ab{I} \langle a\rangle^{-2};$$

iii) each  $\Lb(I)$, $b\in  \L_f\setminus \F$, is real and non-zero,
$$C^{-1}   \ab{I}^{1+ c\bb}\le |\Lb (I)  |\le  C\ab{I}^{1-  c\bb};$$

iv) the operator $K(I)$ is real symmetric and satisfies  $\| K(I)\|\le C \ab{I}^{1-c\bb}$.
The Hamiltonian operator $JK(I)$  is hyperbolic (unless $\F_I$ is empty), and the moduli of the real parts of its 
eigenvalues are bigger than $C^{-1} \ab{I}^{1+\bb}$.
%It may be complex-diagonalised by means of a smooth in $\yy$ complex transformation $U(\yy)$ such that  $\|U(\yy)\| + \|U(\yy)^{-1}\| \le C\nu^{-c_3\bb} $.

v)  The function $f_I$ is much smaller than the quadratic part.

Moreover, all objects depend $C^\infty$ on $I$. 
 \end{theorem}
 
  This result is proven in Part II.  For a more precise formulation, giving in particular the domain of definition of  $\Phi_I$, the smallness in $f_I$ and estimates of the derivatives with respect to $I$, see Theorem~\ref{NFT}. The matrix $M$ is explicitly defined in \eqref{Om}, and 
 the functions $\La$ are explicitly defined in \eqref{Lam}.  An interesting information is that the mapping $\Phi_I$ and the domain $Q$
 only depend on $h_2+h_4$, and that the set $\F_I$ is empty on some connected components of $Q$.

\subsection{The KAM theorem}

The Hamiltonian $h_I\circ\Phi_I$ \eqref{transff} is much better than $h_I$ since its integrable part depends on $I$ in a non-degenerate
way  because $M$ is invertible. Does the invariant torus \eqref{ttorusbis}
persist under the perturbation $f_I$? \dots and, if so, is the persisted torus reducible? 

In finite dimension the answer is yes under very general conditions  --  for the first proof in the purely elliptic case see \cite{E88}, and for a more general case see \cite{GY99}. These statements say that, under general conditions, the invariant torus persists and remains reducible under sufficiently small perturbations for a subset of parameters of large Lebesgue measure. 

In infinite dimension the situation is more delicate, and results can only be proven under quite severe restrictions
on  the normal frequencies $\Lambda_a$; see the discussion above in Section~\ref{s_1.1}. 
  A result for the beam equation (which is a simpler model than the 
Schr\"odinger  and  wave equations)
was first obtained in \cite{GY06a} and  \cite{GY06b}. Here we prove  a KAM-theorem which 
improves  on these results
 % this result 
  in at least  two respects:
\begin{itemize}
\item
We have imposed no ``conservation of momentum''  on the perturbation, which allows us to treat equations \eqref{beam} with 
 $x$-dependent nonlinearities $g$. This   has the effect that our normal form is not diagonal in the purely elliptic directions. In this respect it resembles the normal form obtained in \cite{EK10} for the non-linear 
Schr\"odinger equation, and where the block diagonal form is the same.
\item
We have a finite-dimensional, possibly hyperbolic, component, whose treatment requires higher smoothness in the parameters.
\end{itemize}

The proof has the structure of a classical KAM-theorem carried out in a complex infinite-dimensional situation. The main part is,
as usual, the solution of the homological equation with reasonable estimates. The fact that the block structure is not
diagonal complicates a lot:  see for example, \cite{EK10} where this difficulty was also encountered. The iteration combines a finite linear iteration with a ``super-quadratic'' infinite iteration. This has become quite common in KAM and was also used 
in \cite{EK10}.

 A technical difference, with respect to  \cite{EK10}, is that here we use a different matrix norm
which has much better multiplicative properties. This simplifies a lot the functional analysis which is described in Part I.

A special difficulty in our setting is  that we are facing a  {\it singular perturbation problem}. The perturbation $f_I$ becomes small only by taking  $I$ small, but when $I$ gets smaller the integrable part becomes more degenerate. 
This is seen for example in the lower bounds for
$\Lb(I)$ and for  the real parts of the eigenvalues of $JK(I)$. 
So there is a competition between the smallness condition on the
perturbation and the degeneracies of the integrable part which requires quite careful estimates.

A KAM-theorem which is adapted to our beam equation is proven in Part III and formulated in Theorem~\ref{main} 
and its Corollary~\ref{cMain-bis}.

 \subsection{Small amplitude solutions for the beam equation}
 Applying to the normal form of Part II, the KAM theorem of Part III, we  in Part IV obtain the 
 main results of this work. To state them we  recall that a Borel subset ${\mathfrak J}\subset \R^{\A}_+$ is said to have 
  a {\it positive density}  at the origin if
 \be\label{posdens}
 \liminf_{\nu\to0}\frac{\meas(\fJ\cap \{x\in\R^{\A}_+ \ab{x}<\nu\})}{\meas\{x\in\R^{\A}_+\ab{x}<\nu\}} >0\,.
 \ee
 The set $\fJ$ has the {\it density one} at the origin if the $ \liminf$ above equals one (so the ratio of the measures
 of the two sets converges to one as $\nu\to0$). 
  
 \begin{theorem}\label{t72} There exists a  zero-measure 
 Borel set $\Cc\subset[1,2]$ such that for any\sa\  set $\A\subset\Z^d $,  any $m\notin\Cc$ and any analytic 
 nonlinearity  \eqref{g},  there exist constants $\aleph_1\in (0,1/16], 
 \aleph_2>0$, only depending on $\cA$ and $m$, and
 a set $\fJ=\fJ_\A\subset ]0,1]^{\A}$,  having  density one  at the origin, with the following property: 
 
 \noindent 
 There exist a constant $C>0$,   a real continuous    mapping
 $\ 
U'=U'_\A: \T^{\A}\times \fJ   \to Y,
 $ 
 analytic in the first argument,  satisfying
 % \footnote{The distance below is the Hausdorff distance between subsets of $ Y^R $.}
 \be\label{dist1}
  \big|\big| U'(\theta,I)-U_{I}(\theta)
  \big|\big| 
  \le C |I|^{1 -\aleph_1}\
  \ee
 (see \eqref{embedding}) 
 for all $(\theta,I)\in\T^\A\times \fJ$,  and a continuous  mapping
 $ \Om'=\Om'_\A : \fJ \to \R^{\A}$,
 \be\label{dist11}
 |  \Om'(I) -\om_\A - MI| \leq C |I|^{1+ \aleph_2}\,,
\ee
where the matrix  $M$ is the same as in \eqref{transff},  such that:

i) for any $I\in \fJ$ and $\theta\in\T^{\A}$  the  parametrised  curve
\be\label{solution}
 t \mapsto U'(\theta +t \Om'(I),I)
\ee
is a  solution of the beam equation \eqref{beam2}-\eqref{H1}, and,
 accordingly, the analytic torus $U'(\T^{\A},I)$ is invariant for this equation; 
 
 ii) the set $ \fJ$ may be written as a countable  disjoint union of compact sets $ \fJ_j$, 
 such that the restrictions of the mappings $U'$  and $\Om'$ to the sets $\T^{\A}\times \fJ_j$  are $C^1$ Whitney -smooth;
  
iii) the  solution \eqref{solution}  is linearly stable if and only if in \eqref{transff} the operator $K(I)$ is 
 trivial (i.e. the set  $\F=\F_I$ is non-empty). 
 The set of $I\in \fJ$ such that  $K(I)$ is trivial is always of positive measure, and it equals $\fJ$ if 
 $d=1$ or $| \A |=1$, but for $d\ge2$ and for some choices of the  set $\A$  its complement 
 has positive measure. 
 
   \end{theorem}
   
 If  the set  $\A$ is admissible but not strongly admissible, 
 then a weaker version of the theorem above is true.

 \begin{theorem}\label{t73} 
 There exists a  zero-measure 
 Borel set $\Cc\subset[1,2]$ such that for any admissible set $\A\subset\Z^d $, any $m\notin\Cc$ and  any analytic 
 nonlinearity  \eqref{g},  there exist constants $\aleph_1\in (0,1/16],  \aleph_2>0$, only
  depending on $\cA$ and $m$, and
 a set $\fJ=\fJ_\A\subset ]0,1]^{\A}$,  having  positive density   at the origin, such that all assertions of Theorem~\ref{t72} are true.
  \end{theorem}

  \begin{remark} \label{r_1}
  1) The torus $U_{I}(\T^\A,I)$ (see \eqref{ttorus}), 
   invariant for the linear beam equation \eqref{beam2}${}_{G=0}$, 
  is of  size $\sim\sqrt I$. The constructed invariant torus $U'_\A(\T^{\A},I)$
  of the nonlinear beam equation is its  small perturbation
     since by \eqref{dist1} the   Hausdorff   distance between $U'_{\A}(\T^{\A},I)$ and 
   $U_{I}(\T^\A)$ is smaller than $C |I|^{1 -2\aleph_1}\le C |I|^{7/8}$.

 2) Denote by $\cT_\A$ the image of the mapping $U'_\A$. This set is invariant for the beam equation
 and is filled in with its time-quasiperiodic solutions. By the item~ii) of Theorem~\ref{t72} its Hausdorff 
 dimension equals $2|\A|$. Now consider $\cT = \cup \cT_\A$, where the onion is taken over all 
 strongly admissible sets $\A\subset\Z^d$. This invariant set has infinite Hausdorff dimension. Some 
 time-quasiperiodic solutions of \eqref{beam}, lying on $\cT$, are linearly stable, while, if $d\ge2$, then
  some others are unstable.

     3) Our result applies to eq. \eqref{beam} with any $d$. Notice that   for $d$ sufficiently large the global in time 
 well-posedness of this equation is unknown. 
 
  4) The construction of  solutions \eqref{solution} crucially depends on certain equivalence relation in 
  $\Z^d$, defined in terms of the set $\A$ (see \eqref{class}). 
   This equivalence is trivial if $d=1$ or $|\A|=1$ and is non-trivial  otherwise. 
   
 5) We discuss in 
  Appendix~B examples of sets $\A$ for which the operator $K(I)$ is non-trivial  for certain values of $I$.

    6) The solutions \eqref{solution} of eq. \eqref{beam2},  written in terms of the $u(x)$-variable as solutions $u(t,x)$ of eq.~\eqref{beam}, 
  are $H^{m_*+1}$-smooth as functions of $x$ and analytic as functions of $t$. Here $m_*$ is
   a parameter of the   construction for which we can take any real number $>d/2$ (see \eqref{YC}).   The set $\fJ $ depends on
  $m_*$, so the assertion of the theorem does not imply immediately that the  solutions $u(t,x)$ 
    are  $C^\infty$--smooth in $x$. 
  Still, since 
    $$
    -(\Delta^2+m) u = u_{tt}+\partial_u G(x,u),
    $$
    where $G$ is an analytic function, then the theorems 
    %second assertion of Amplification
     imply by induction that the   solutions $u(t,x)$      define analytic curves
    $\R\to H^m(\T^d)$, for any $m$. In particular, they are smooth functions. 
  \end{remark}

 \noindent 
{\bf  Structure of text}
 The paper consists of Introduction and four parts.  Part~I comprises general techniques needed to read the paper. The main Parts~II-III
 are independent of each other, and the final Part~IV, containing the proofs of Theorems~\ref{t72},~\ref{t73}, 
  uses only the main theorems of Parts~II-III, and the intermediate results 
 are not needed to understand it. 
 
 \medskip
  
 \noindent 
{\bf Some notation and agreements.} 
 We denote a cardinality of a set $X$ as $|X|$ or as $\,\# X$. For $a\in\Z^N$ we denote $\langle a\rangle =\max(1, |a|)$.
 
In any finite-dimensional  space $X$ we denote by
$|\cdot|$ the Euclidean norm.  For subsets $X$ and $Y$ of a Euclidean space we denote 
$$
\underline{\text{dist}}\,(X,Y) = \inf_{x\in X, y\in Y} |x-y|\,,\qquad \text{diam}\,(X) = \sup_{x,y\in X}|x-y|\,.$$
The distance on a torus  induced by the  Euclidean distance (on the tangent space) will be denoted 
$|\cdot - \cdot|$.

For any matrix $A$, finite or infinite, we denote by ${}^t\!A$ the transposed matrix. 
 $I$ stands for  the identity  matrix of any dimension. 

The space of bounded linear operators between Banach spaces $X$ and
$Y$ is denoted $\B(X,Y)$. Its operator norm will be usually denoted $\|\cdot\|$ without specification the 
spaces.  If $A$ is a finite matrix, then $\|A\|$ stands for its operator-norm.  

 We call analytic mappings between domains  in complex Banach 
spaces {\it holomorphic} to reserve the name {\it analytic} for mappings between domains in real Banach 
spaces. This definition extends from Banach spaces to Banach manifolds.
 \smallskip

\noindent{\it Pairings in $l^2$-spaces}. 
The scalar product on any complex Hilbert space is, by convention, complex anti-linear in the first variable and
complex linear in the second variable.
For any $l^2$-space $X$ of finite or infinite dimension, 
 the natural complex-bilinear pairing is denoted 
\be\label{pairing}\langle \zeta,\zeta'\rangle=\langle \bar \zeta,\zeta'\rangle_{l^2},\qquad \zeta,\zeta'\in X.\ee
This is a symmetric complex-bilinear mapping.
\smallskip

\noindent{\it Constants}. 
The numbers $d$ (the space-dimension) and  $\#\A$, as well as  $s_*,m_*$ and $\#\P, \#\F$ (that will occur in Part II) 
will be fixed in this paper. Constants depending only on the numbers and on the choice of finite-dimensional norms are regarded as absolute constants. 
An absolute constant only depending on $x$ is thus a constant that, 
besides these factors, only depends on $x$. Arbitrary constants will often be denoted by $Ct., ct.$ and, 
when they occur as an exponent, by $exp$. Their values may change from line to line. 
For example we allow ourselves to write $2Ct. \le Ct.$. 
\smallskip

%\noindent{\it Parameters}. Our functions depend on a parameter $\yy\in\D$, where $\D \subset \R^{\A}.$When $\D$ is a compact set, differentiability of functions on $\D$ is understood in the sense of   Whitney. That is, $f\in C^k(\D)$ if it  extends to a $C^k$-smooth function $\tilde f$ on $ \R^{\A}$, and $|f|_{ C^k(\D)}$ is the infimum of  $|\tilde f|_{ C^k(\R^{\A})}$, taken over all $C^k$-extensions $\tilde f$ of $f$. \medskip
 
\noindent 
{\bf Acknowledgments.} We acknowledge the support from Agence Nationale de la Recherche through the grant
  ANR-10-BLAN~0102. 
  The third author wishes to thank P.~Milman and V.~\v{S}ver\'ak for
  helpful discussions.

\bigskip
\bigskip
\begin{samepage}
\centerline{PART I. SOME FUNCTIONAL ANALYSIS}
%\nopagebrake
\section{Matrix algebras and function spaces.}
\end{samepage}

\subsection{The phase space}\label{sThePhaseSpace}
Let $\cA$ and $\F$ be two finite sets in $\Z^{d}$ and let $\L_\infty$ be an infinite subset of $\Z^{d}$. Let
$\L$ be the disjoint union $\F\sqcup \L_{\infty}$.
\footnote{\ this is a more general setting than in the introduction, where $\L$ and $\A$ were two disjoint subsets of $\Z^d$}
Let $\cZ$ be the disjoint union $\cA\sqcup \F\sqcup \L_{\infty}$
and consider $(\C^2)^{\cZ}$.

 For any subset $X$ of $\cZ$, consider the projection
$$\pi_X:(\C^2)^{\cZ}\to (\C^2)^{X}=\{\zeta\in (\C^2)^{\cZ}: \zeta_a=0\ \forall a\notin X\}.$$
We can thus write $(\C^2)^{\cZ}=(\C^2)^{X}\times (\C^2)^{\cZ\setminus X}$,
$\zeta=(\zeta_X,\zeta_{\L\setminus X})$,
and when $X$ is finite this gives an injection
$$
(\C^2)^{\#X}\hookrightarrow (\C^2)^{\cZ}$$
whose image is $ (\C^2)^{X}$.

In $\R^2$ we consider the partial ordering
$(\ga_1',\ga_2')\le (\ga_1,\ga_2)$ if, and only if $\ga_1'\le\ga_1$ and $\ga_1'\le \ga_2'$.

\medskip

Let $\ga=(\ga_1,\ga_2)\in\R^2$ and  let
$Y_\ga$  be the Hilbert  space of sequences $\zeta\in (\C^2)^{\cZ}  $ such that
\be\label{Ygamma}
||\zeta||_{\ga}^2= %\sqrt
{ \sum_{a\in\cZ} |\zeta_a|^2e^{2\ga_1|a|}\langle a\rangle^{2 \ga_2} }<\infty\,,
\ee
provided with the scalar product
 \footnote{\ complex linear in the second variable and complex anti-linear in the first}
$$ \langle \zeta, \zeta' \rangle_{\ga} =
 \sum_{a\in\cZ} \langle \zeta_a,\zeta_a'\rangle_{\C^2} e^{2\ga_1|a|}\langle a\rangle^{2 \ga_2}.$$
\noindent
If $\ga_1\ge0$ and $\ga_2>d/2$, then this space is an algebra with respect to the convolution.
If $\ga_1=0$, this is a classical property of Sobolev spaces. For the case  $\ga_1>0$ 
see \cite{EK08}, Lemma~1.1.  (The space $Y_{(0,m_*)}$ coincides with the space $Y$, defined in \eqref{YC}, while $Y_{(0,0)}$ is the $l^2$-space of complex sequences $(\C^2)^\cZ$.)

\begin{example}\label{analyt} Let $\A=\F=\emptyset$, $\L_\infty=\Z^d$ and $\varrho>0$.  Then any vector  $\hat f=(\hat f_a, a\in\Z^d)\in Y_\varrho$ defines  a holomorphic vector-function $f(y) = \sum \hat f_ae^{{\bf i}\langle a, y\rangle}$ on the $\varrho$-vicinity  $\T^n_\varrho$ of the torus $\T^n$, $\T^n_\varrho=\{y \in\C^n/2\pi\Z^n\mid |\Im y|<\sigma\}$, where  its norm is bounded by $C_d\|\hat f\|_\varrho$. Conversely, if  $f: \T^n_\varrho  \to\C^2$ is a bounded holomorphic function, then    its Fourier coefficients satisfy  $|\hat f_a|\le\,$Const$\, e^{-|a|\varrho}$, so  $\hat f\in Y_{\varrho'}$ for any $\varrho'<\varrho$. \end{example}

\medskip

Write $\zeta_a=(p_a,q_a)$ and let
$$\Om= \sum_{a\in\cZ} dp_a\wedge dq_a.$$
$\Om$ is  an anti-symmetric bi-linear form which is  continuous on
$$Y_\ga\times Y_{-\ga}\cup Y_{-\ga}\times Y_{\ga}\to \C$$ 
with norm $\aa{\Om}= 1$.  The subspaces $(\C^2)^{\{a\}}$  are symplectic subspaces of two (complex) dimensions carrying the canonical symplectic structure.

$\Om$  defines (by contraction on the second factor ) a bounded bijective operator
$$Y_\ga\ni \zeta\mapsto \Om(\cdot,\zeta) \in Y^*_{-\ga}$$
where $Y^*_{-\ga}$ denotes the Banach space dual of $Y_{-\ga}$.
(Notice that $\zeta'\mapsto \Om(\zeta', \zeta)$ is a well-defined bounded linear form on $Y_{-\ga}\,.$)
We shall denote its inverse by
$${J_\Om}:   Y^*_{-\ga}\to Y_\ga.$$ 

We shall also let $J_\Om$ act on operators
$$J_\Om:  \cB(X, Y^*_{-\ga})\to \cB(X,Y_\ga)$$
through $(J_\Om H)(x)=J_\Om (H(x))$ for any bounded operator $H:X\to Y^*_{-\ga}$.

\medskip
\begin{remark}
 The complex-bilinear pairing \eqref{pairing} on the $l^2$-space $Y_{(0,0)}$
 extends to a continuous mapping $Y_\ga \times Y_{-\ga} \to \C\,$
 which allows to identify $Y_{-\ga} $ with the dual space  $Y_{\ga}^* $. 
 %$$\|\zeta^2\|_{-\ga} = \sup \{ |\langle \zeta^1, \zeta^2\rangle| :\,    \|\zeta^1\|_\ga=1\}\,. $$
Then
\be\label{oega}
\Om(\zeta, \zeta') = \langle J \zeta, \zeta'\rangle\,,
\ee
where $J$ here stands for the linear operator $\zeta \mapsto  J\zeta$ defined by
$$
(J\zeta)_a = J\zeta_a\;\  \; \; \forall a\in\cZ,$$
where the $2\times2$-matrix $J$ (in the right hand side) is defined in \eqref{J}. 
\footnote{\ sorry for the abuse of notation}
Then we have
\be\label{equal}
J_\Om \zeta = J\zeta \qquad \forall\, \zeta\in Y^*_{-\ga}\,,
\ee
where $\zeta$ in the r.h.s. is regarded as a vector in $Y_\ga$, and we shall frequently denote the operator $J_\Om$ by $J$. (It will be clear 
from the context which of the two operators $J$ denotes.)
\end{remark}

\medskip

A  bijective bounded operator  $A:Y_\ga\to Y_\ga$, $\ga\ge(0,0)$,
is {\it symplectic} if, and only if,
$$
\Om(A\zeta, A\zeta') = \Om(\zeta, \zeta')\qquad \forall \ \zeta,\zeta'\in  Y_\ga\,. 
$$
Writing $\Om$ in the form \eqref{oega} we see that $A$ is symplectic if and only if ${}^t AJA=J$.
Here  ${}^t A$ stands for the operator, symmetric to $A$ with respect to the pairing  $\langle \cdot,\cdot\rangle$
(its matrix is transposed to that of $A$). 

\medskip

Let 
$$\bA^{\cA} =\C^{\cA}\times(\C/2\pi \Z)^{\cA}$$
and consider the Hilbert manifold
$\bA^{\cA} \times \pi_{\L} Y_\ga$ whose elements are denoted $x=(r,\theta= [z],w)$.
%\footnote{\ $[z]$ being the class of $z\in \C^{\cA}$}
We provide 
this manifold with the metric 
$$\aa{x-x'}_\ga=
\inf_{p\in\Z^{d}}  ||(r, z+2\pi p, w)-(r', z',w')||_\ga.
\quad\footnote{\  using this notation for the metric on the manifold will not confuse it with 
the norm on the tangent space, which is also denoted $||\cdot||_\ga$, we hope}
$$

We provide $\bA^{\cA} \times \pi_{\L} Y_\ga$ with the  symplectic structure $\Om$. To any 
$C^{1}$-function  $f(r,\theta,w)$ on (some open set in)  $\bA^{\cA}\times \pi_{\L} Y_{\ga}$ it associates
a  vector field $X_f= J(df)$  --  the Hamiltonian vector field of $f$
-- which in the coordinates $(r,\theta,w)$ takes the form
$$
\left(\begin{array}{c} \dot r_a \\ \dot \theta_a\end{array}\right)=J
\left(\begin{array}{c} \frac{\p}{\p r_a} f(r,\theta,w)\\
\frac{\p}{\p \theta_a} f(r,\theta,w)\end{array}\right)
\qquad
\left(\begin{array}{c} \dot p_a \\ \dot q_a\end{array}\right)=J
\left(\begin{array}{c} \frac{\p}{\p p_a} f(r,\theta,w)\\
 \frac{\p}{\p q_a} f(r,\theta,w)\end{array}\right).
 \quad\footnote{\ there is no agreement as to the sign of the Hamiltonian vectorfield --  we've used the choice of Arnold 
\cite{Arn}}
$$

\subsection{A matrix algebra}\label{ssMatrixAlgebra}
\ 
The mapping
\be\label{pdist}
(a,b)\mapsto [a-b]=\min (|a-b|,|a+b|)\ee
is a pseudo-metric on $\Z^{d}$, i.e. it verifies  all the relations of a metric with the only exception that
$[a-b]$ is $=0$ for some $a\not=b$.
This is most easily seen by observing that $[a-b]=\textrm{d}_{\Haus}(\{\pm a\},\{\pm b\})$.
We have $ [a-0]=|a|$.

Define, for any $\ga=(\ga_1,\ga_2)\ge(0,0)$ and $\vark\ge0$,
\be\label{weight}
e_{\ga,\vark}(a,b)=Ce^{\ga_1[a-b]}\max([a-b],1)^{\ga_2}\min(\langle a\rangle,\langle b\rangle)^\vark.\ee

\begin{lemma}\label{lWeights}
\ 
\begin{itemize}
\item[(i)] If $\ga_1, \ga_2- \vark\ge0$, then 
$$e_{\ga, \vark}(a,b)\le   e_{\ga,0}(a,c) e_{\ga,\vark}(c,b),\quad \forall a,b,c,$$
if $C$ is sufficiently large (bounded with  $\ga_2,\vark$).

\item[(ii )] If $-\ga\le\tilde \ga\le \ga$, then
$$e_{\tilde\ga, \vark}(a,0)\le   e_{\ga,\vark}(a,b) e_{\tilde\ga,\vark}(b,0),\quad \forall a,b$$
if $C$ is sufficiently large  (bounded with  $\ga_2,\vark$).
\end{itemize}
\end{lemma}

\begin{proof} 
(i). Since $[a-b]\le [a-c]+[c-b]$ it is sufficient to prove this for $\ga_1=0$. If $\ga_2=0$ then the statement holds for any
$C\ge1$, so it is sufficient to consider $\ga_2>0$. This reduces easily to $\ga_2=1$ and, hence, $\vark\le 1$. 
Then we want to prove
\begin{multline*}
\max([a-b],1)\min(\langle a\rangle,\langle b\rangle)^\vark 
\le C \max([a-c],1)\max([c-b],1)
\min(\langle c \rangle,\langle b\rangle)^\vark.
\end{multline*}
Now $\max([a-b],1)\le \max([a-c],1)+\max([c-b],1)$,
$$
\max([c-b],1)
\min(\langle c \rangle,\langle b\rangle)^\vark\gsim  \langle b\rangle^\vark,
$$
and
$$
\max([a-c],1)
\min(\langle c \rangle,\langle b\rangle)^\vark\gsim  \min(\langle a \rangle,\langle b\rangle)^\vark.$$
This gives the estimate.

(ii) Again it suffices to prove this for $\ga_1=0$ and $\ga_2=1$. Then we want to prove
$$
\max(\ab{a},1)^{\tilde \ga_2}  \le C \max([a-b],1)\min(\langle a \rangle,\langle b\rangle)^\vark \max(\ab{b},1)^{\tilde \ga_2}.$$
The inequality is fulfilled with $C\ge1$ if $a$ or $b$ equal $0$. Hence we need to prove
$$
\ab{a}^{\tilde \ga_2}  \le C \max([a-b],1)\min(\langle a \rangle,\langle b\rangle)^\vark \ab{b}^{\tilde \ga_2}.$$
Suppose $\tilde\ga_2\ge0$. If $\ab{a}\le 2 \ab{b}$ then this holds for any  $C\ge2$. If $\ab{a}\ge2\ab{b}$ then
$[a-b]\ge \tfrac12\ab{a}$ and the statement holds again  for any  $C\ge2$.

If instead $\tilde\ga_2<0$, then we get the same result with $a$ and $b$ interchanged.

\end{proof} 

\subsubsection {The space $\cM_{\ga,\vark}$}\label{sM2}

We shall consider matrices $A:\cZ\times\cZ\to gl(2,\C)$, formed by $2\times2$-blocs,
(each $A_a^{b}$ is a complex  $2\times2$-matrix). Define
\be\label{matrixnorm}
|A|_{\ga,\vark}=\max\left\{\begin{array}{l}
\sup_a\sum_{b} \aa{A_a^b} e_{\ga,\vark}(a,b)\\
\sup_b\sum_{a} \aa{A_a^b} e_{\ga,\vark}(a,b),
\end{array}\right.\ee
where the norm on $A_a^{b}$ is the matrix operator norm.

Let $\cM_{\ga,\vark}$ denote the space of all matrices $A$ such that
$\ab{A}_{\ga,\vark}<\infty$.  Clearly $\ab{\cdot}_{\ga,\vark}$ is a norm
on $\cM_{\ga,\vark}$  --  this is indeed true for all $(\ga_1,\ga_2,\vark)\in\R^3$. It follows by well-known results that $\cM_{\ga,\vark}$,
provided with this norm, is a Banach space.

Transposition --  $({}^tA)_a^b={}^t\!A_b^a$ --  and $\C$-conjugation --
$(\overline{A})_a^b={\overline{A_a^b}})$  -- do not change 
this norm.The identity matrix is in $\cM_{\ga,\vark}$ if, and only if, $\vark=0$, and then $|I|_{\ga,0}=C$.

\begin{remark*} The ``$l^1$-norm'' used here is a bit 
 more complicated than the ``sup-norm'' used in \cite{EK10}, but it has, as we shall see,
much better multiplicative properties.
\end{remark*}

\subsubsection{Matrix multiplication}

We define (formally) the {\it matrix product}
$$(AB)_a^b=\sum_{c} A_a^cB_c^b.$$
Notice that complex conjugation, transposition and taking the adjoint behave in the usual
way under this formal matrix product.

\begin{proposition}\label{pMatrixProduct} Let $\ga_2\ge\vark$. 
If $A\in \cM_{\ga,0}$ and $B\in \cM_{\ga,\vark}$,   then
$AB$ and $BA\in \cM_{\ga,\vark}$ and
$$
\ab{{AB} }_{\ga,\vark}\ \textrm{and}\ 
\ab{{BA} }_{\ga,\vark}\le \ab{A}_{\ga,0} \ab{B}_{\ga,\vark}.$$ 

\end{proposition}

\begin{proof}
(i) We have, by Lemma~\ref{lWeights}(i),
$$\sum_{b} \aa{( AB)_a^b} e_{\ga,\vark}(a,b)\le
\sum_{b,c} \aa{ A_a^c} \aa{ B_c^b} e_{\ga,\vark}(a,b)\le $$
$$\le\sum_{b,c}\aa{ A_a^c} \aa{ B_c^b} 
e_{\ga,0}(a,c)e_{\ga,\vark}(c,b)  $$
which is $\le \aa{A}_{\ga,0} \aa{B}_{\ga,\vark}$.
This implies in particular the existence of $(AB)_a^b$.

The sum over $a$ is shown to be  $\le \ab{A}_{\ga,0} \ab{B}_{\ga,\vark}$ in a similar way.
The estimate of $BA$ is the same.
\end{proof}

Hence $\cM_{\ga,0}$ is a Banach algebra, and $\cM_{\ga,\vark}$ is an ideal in $\cM_{\ga,0}$ when $\vark\le\ga_2$.
 
\subsubsection{The space $\cM_{\ga,\vark}^b$}
We define (formally) on $Y_\ga$ 
%(see section \ref{sThePhaseSpace})
$$(A\zeta )_a=\sum_{b} A_a^b \zeta _b.$$

\begin{proposition}\label{pMatrixBddOp}
Let $-\ga\le\tilde \ga\le\ga$.
If $A\in\cM_{\ga,\vark}$ and 
$\zeta \in Y_{\tilde \ga}$,  then $A\zeta \in Y _{\tilde \ga}$ and
$$
\aa{A\zeta }_{\tilde \ga}\le \ab{A}_{\ga,\vark}\aa{\zeta }_{\tilde\ga}.$$
\end{proposition}

\begin{proof}
Let $\zeta '=A\zeta $. We have
$$
\sum_{a} \ab{\zeta '_a}^2e_{\tilde \ga,0}(a,0)^2\le\sum_{a} 
\big(\sum_{b} \aa{A_a^b}\ab{\zeta _{b}} e_{\tilde \ga,0 }(a,0)\big)^2.$$
Write
$$\aa{A_a^b}\ab{\zeta _{b}}e_{\tilde \ga,0 }(a,0)
= I\times (I\ab{ \zeta _{b}}e_{\tilde \ga,0}(b,0))\times  J,$$
where
$$I=I_{a,b}=\sqrt{\aa{A_a^b}    e_{\ga,\vark}(a,b)}$$
and
$$J=J_{a,b}=
\frac{e_{\tilde \ga,0}(a,0)}{e_{\ga,\vark}(a,b)e_{\tilde\ga,0}(b,0)}.$$

Since,  by Lemma~\ref{lWeights}(ii), $J=\le 1$ we get,  by H\"older,
\begin{multline*}
\sum_{a} \ab{\zeta' _a}^2e_{\tilde \ga,0}(a,0)^2\le  
\sum_{a} (\sum_{b} I_{a,b}^2)( \sum_{b} I_{a,b}^2\ab{ \zeta _{b}}^2e_{\tilde\ga,0}(b,0)^2)\\
\le \ab{A}_{\ga,\vark}    \sum_{a,b} I_{a,b}^2\ab{\zeta _{b}}^2e_{\tilde\ga,0}(b,0)^2
\le  \ab{A}_{\ga,\vark}    \sum_{b} \ab{ \zeta _{b}}^2e_{\tilde \ga,0}(b,0)^2\sum_{a}I_{a,b}^2\le
\end{multline*}
$$\le \ab{A}_{\ga,\vark}^2      \aa{\zeta }_{\tilde \ga}^2. $$
This shows that $y_a$ exists  for all $a$, and it also proves the estimate. 
\end{proof}

We have thus, for any $-\ga\le \tilde\ga\le\ga$, a continuous embedding of $\cM_{\ga,\vark}$,
$$\cM_{\ga,\vark}\hookrightarrow \cM_{\ga,0}\to
 \cB(Y_{\tilde \ga},Y_{\tilde \ga}),$$
into the space of bounded linear operators on $Y_{\tilde \ga}$. Matrix multiplication in $\cM_{\ga,\vark}$
corresponds to composition of operators.

\smallskip

For our applications (see Lemma~\ref{lemP}) we shall consider a somewhat  larger sub algebra of $\cB(Y_\ga, Y_\ga)$ 
 with somewhat weaker decay properties.
Let
\be\label{b-space}
\cM_{\ga,\vark}^b= \cB(Y_{\ga},Y_{\ga})\cap \cM_{(\ga_1,\ga_2-m_*),\vark}\ee
which we provide with the norm
\be\label{b-matrixnorm}
\aa{A}_{\ga,\vark }=\aa{A}_{ \cB(Y_{\ga};Y_{\ga})}+ \ab{A}_{(\ga_1,\ga_2-m_*),\vark }.\ee
When $\ga=(\ga_1,\ga_2)\ge \ga_*=(0,m_*+\vark)$,
 Proposition \ref{pMatrixProduct} shows that this 
norm makes $\cM_{\ga,0}^b$ into a Banach sub-algebra of $\cB(Y_{\ga};Y_{\ga})$
and $\cM_{\ga,\vark}^b$ becomes an  ideal in $\cM_{\ga,0}^b$.

\subsection{Functions}

For $\sigma,\mu\in (0,1]$  let 
\be\label{domain}
\begin{split}
\O_\ga(\s,\mu)=&\\
\{x=(r_\A ,\theta_\A,w)\in \bA^{\cA} \times\pi_{\L}  Y_{\ga}: &
|r_\A|<\mu,\  \mid |\Im \theta_\A|<\sigma,\ \mid \|w\|_\ga <\mu \}.
\end{split}\ee
It is often useful to scale the action variables $r$ by $\mu^2$ and not by $\mu$, but in our case  $\mu$ will
be $\approx 1$, and then there is no difference (on the contrary, in Section~\ref{s_4.2} we scale $r_\A$ 
as $\mu^2$ to simplify the calculations we perform there). 
 The advantage with our scaling is that the Cauchy estimates becomes simpler.

Let
\be\label{gamma}\ga=(\ga_1,\ga_2)\ge \ga_*=(0,m_*+\vark),\ee
We shall consider perturbations 
$$f:\O_{\ga_*}(\s, \mu)\to \C$$
that are  {\it real holomorphic  and continuous up to the boundary} (rhcb).
This means that it gives real values to real arguments and extends continuously
to the closure of  $\O_{\ga_*}(\s, \mu)$. 
$f$ is clearly also rhcb on $\O_{\ga}(\s, \mu)$ for any $\ga\ge\ga_*$, and
$$d f:\O_{\ga}(\s, \mu)\to Y_{\ga}^* $$
and
$$
J_\Om d^2f:\O_\ga(\s, \mu)\to \cB(Y_\ga,Y_{-\ga}) $$
are rhcb. 

\begin{remark}
Identifying $Y_\ga^*$ with $Y_{-\ga}$ via the paring $\langle\cdot,\cdot\rangle$ we will interpret the 
differential $df(\zeta)$ as a gradient $\nabla f(\zeta)\in Y_{-\ga}$,
$$
df(\zeta) (\zeta') =\langle \nabla f(\zeta), \zeta'\rangle\qquad \forall\, \zeta'\in Y_\ga\,.
$$
As classically,  $\nabla f(\zeta)$ is the vector  $\nabla f(\zeta)=( \nabla_a f(\zeta), a\in\cZ)$, 
where for $ \zeta=\big(  \zeta_a = (p_a,q_a), a\in \cZ\big)$, 
$\nabla_a f$ is the 2-vector $(\p f/\p p_a, \p f/\p q_a)$. 

Similar we will interpret $d^2f$ as the Hessian $\nabla^2 f$, which is an operator 
 the matrix
$
((\nabla^2 f)_a^b, a,b\in \cZ),
$
formed by the $2\times2$-blocks $(\nabla^2 f)_a^b = \nabla_a\nabla_b f$. The Hessian defines bounded linear 
operators $\nabla^2 f(\zeta) : Y_\ga\to Y_{-\ga}$, and
$$
d^2 f(\zeta) (\zeta^1, \zeta^2) =\langle \nabla^2 f(\zeta)  \zeta^1, \zeta^2\rangle \quad \forall\ \zeta^1, \zeta^2 \in Y_\ga\,. 
$$
\end{remark}

We shall require that the mappings $df$ and $d^2f$ posses some extra smoothness: 

\begin{itemize}

\item[R1] {\it  -- first differential}.
There exists a $\ga\ge\ga_*$ such that
$$
Jd f=J\nabla f
:\O_{\ga'}(\s, \mu)\to Y_{\ga'} 
$$
is rhcb for any $\ga_*\le \ga'\le\ga$.

\end{itemize}
This is a natural smoothness condition on the space of holomorphic functions on $\O_{\ga_*}(\s, \mu)$, and it  implies, 
in particular, that   $Jd^2 f(x)=J\nabla^2 f(x)
%\in\cB(Y_{\ga'},Y_{-\ga'}^*)$ and, hence, $Jd^2 f(x)  
\in\cB(Y_{\ga'},Y_{\ga'})$ for any $x\in \O_{\ga'}(\s, \mu)$.  So 
$$
(\nabla^2 f(x))_a^b 
  \le\Cte e^{-\ga'_1\ab{ \ab{a}-\ab{b}}}\min(\frac{\langle a\rangle}{\langle b\rangle},
\frac{\langle b\rangle}{\langle a\rangle})^{\ga'_2}\qquad \forall\, a,b\in\cZ\,.
$$
%for any two  unit vectors $e_a\in(\C^2)^{\{a\}}$ and $e_b\in(\C^2)^{\{b\}}$.
%\footnote{\ here we have identified $\cB(Y_{\ga},Y_{\ga})$ with a bi-linear maping
%$Y_{\ga}\times Y_{\ga}\to \C$ in the usual way }
But many Hamiltonian PDE's verify other, and  stronger, 
decay conditions  in terms of
$[a-b]=\min(\ab{a-b},\ab{a+b}).$

Indeed we shall assume

\begin{itemize}

\item[R2] {\it  -- second differential}.
$$Jd^2f= J\nabla^2 f
:\O_{\ga'}(\s, \mu)\to  \cM_{\ga',\vark}^b$$
is rhcb for any $\ga_*\le \ga'\le\ga$.

\end{itemize}

Such decay conditions do not seem to be  naturally related to any smoothness condition of $f$, but they
 are instrumental in the KAM-theory for multidimensional PDE's:  see for example \cite{EK10} where
  such conditions were used to build a KAM-theory for some multidimensional non-linear Schr\"odinger equations.

\subsubsection{The function space  $\cT_{\ga,\vark} $ }\label{s_2.3.1}
Consider the space of  functions
$f:\O_{\ga_*}(\s, \mu)\to \C$
which are  real holomorphic  and continuous  up to the boundary (rhcb) 
of $\O_{\ga_*}(\s, \mu)$. We define  $\cT_{\ga,\vark}(\s,\mu) $ to be the space of all such functions which verify $R1$ and $R2$.

We provide  $\cT_{\ga,\vark} (\s, \mu)$ with the norm 
\be\label{norm}
|f|_{\begin{subarray}{c}\s,\mu\\ \ga, \vark  \end{subarray}}=
\max\left\{\begin{array}{l}
\sup_{x\in \O_{\ga_*}(\s, \mu)}|f(x)|\\ 
\sup_{\ga_*\le \ga'\le\ga}\sup_{x\in \O_{\ga'}(\s, \mu)} ||Jd f(x)||_{\ga'}=\|\nabla f(x)\|_{\ga'}   \\
\sup_{\ga_*\le  \ga'\le\ga}\sup_{x\in \O_{\ga'}(\s, \mu)}||Jd^2f(x)||_{\ga',\vark} =||\nabla^2f(x)||_{\ga',\vark}
\end{array}\right.
\ee
making it  into a Banach space. (It is even a Banach algebra with the constant function
$f=1$ as unit, but we shall be concerned with Poisson products rather than with products.)

This space is relevant for our application because

\begin{lemma}\label{lemP}
Let $\cZ=\L_\infty=\Z^d$ and  $\vark=2$. Then the Hamiltonian function $h_{\ge4}$, 
 defined in \eqref{H1}, belongs to $\cT_{\ga_g,2}(1, \mu_g) $
for suitable $\mu_g\in(0,1]$ and $\ga_g>\ga_*$. % If $G$ has no higher order terms (i.e. $G=u^4$), 
% then it  belongs to $\cT_{\ga_g,2}(1, \mu_g) $ for all $ \mu_g\in(0,1]$ and $\ga_g>\ga_*$
\end{lemma} 

The lemma in proven in Appendix A. 
Notice that we would not have been to prove this if we had used the matrix norm \eqref{matrixnorm} instead of \eqref{b-matrixnorm}. 

\medskip

The higher differentials $d^{k+2}f$  can be estimated by Cauchy estimates on some smaller domain in terms of this norm.

\begin{remark*}
The higher order differential  $d^{k+2}f(x)$, $x\in \O_{\ga}(\s, \mu)$, is canonically identified with three bounded symmetric
multi-linear maps
$$\begin{array}{l}
  (Y_{\ga})^{k+2}%=\underbrace{Y_{\ga}\times\dots\times Y_{\ga}}_{k+2\ times}
  \longrightarrow \C\,,\\
 (Y_{\ga})^{k+1} \longrightarrow  Y_{\ga}^*\,,\\
(Y_{\ga})^{k} \longrightarrow \cB(Y_\ga,Y_\ga^*).
\end{array}$$
Due to the smoothing condition R1 the second one takes its values in the subspace $Y_{-\ga}^*$.
Due  to the smoothing condition R2 $Jd^{k+2}f(x)$ is a bounded symmetric multi-linear map
\be\label{dk+2}
(Y_{\ga})^{k}  \longrightarrow  \cM_{\ga,\vark}^b.
\ee

Alternatively, identifying  $d^2f$ with the hessian $\nabla^2 f$, we may identify $d^{k+2}f$ with a continuous
symmetric multilinear mapping of the form \eqref{dk+2}. 

\end{remark*}

 \subsubsection{The function space  $\cT_{\ga,\vark,\D} $ }
 Let $\D$ be  an open set in $\R^{\P}$.  We shall consider functions
$$f:\O_{\ga^*}(\s, \mu)\times \D\to \C$$
which are of class $\cC^{{s_*}}$ for some integer 
$s_*\ge0$.  We say that $f\in \cT_{\ga,\vark,\D}(\s,\mu)  $ if, and only if,  
$$\frac{\p^j f}{\p \r^j}(\cdot,\r)\in \cT_{\ga,\vark}(\s,\mu) $$
for any $\r\in\D$ and any $\ab{j}\le {s_*}$. We provide this space by the norm
\be\label{normwithparameter}
|f|_{\begin{subarray}{c}\s,\mu\ \ \\ \ga, \vark,\D  \end{subarray}}=
\max_{\ab{j}\le {s_*}}\sup_{\r\in\D}
|\frac{\p^j f}{\p \r^j}(\cdot,\r)|_{\begin{subarray}{c}\s,\mu\\ \ga, \vark  \end{subarray}}.\ee
This norm makes $\cT_{\ga,\vark,\D}(\s,\mu)  $  a Banach space.

\subsubsection{ Jets of functions.} \label{ss5.1}

For any function $f\in \cT_{\ga,\vark,\D}(\s,\mu)$ we shall consider the following 
 Taylor polynomial of $f$ at $r=0$ and $w=0$
\be
\label{jet}
f^T(x)=f(0,\theta,0)+d_r f(0,\theta,0)[r] +d_w f(0,\theta,0)[w]+\frac 1 2 d^2_wf(0,\theta,0)[w,w] 
\ee
Functions of the form  $f^T$ will be called {\it jet-functions.}

\begin{proposition}\label{lemma:jet}
Let $f\in \cT_{\ga,\vark,\D}(\s,\mu)$. Then
$f^T\in \cT_{\ga,\vark,\D}(\s,\mu)$ and 
$$
|f^T|_{\begin{subarray}{c}\s,\mu \  \\ \ga, \vark, \D \end{subarray}}
\leq C |f|_{\begin{subarray}{c}\s,\mu\ \ \\ \ga, \vark,\D  \end{subarray}}.$$
($C$ is an absolute constant.)
\end{proposition}

\begin{proof} The first part follows by general arguments. Look for example on
$$g(x)=   d^2_wf\circ p(x)[ w,w],\quad x=(r,\theta,w),$$
where $p(x)$ is the projection onto $(0,\theta,0)$.
This function $g$  is rhcb on $\O_{\ga_*}(\s, \mu)$, 
being a composition of such functions.  A bound for its sup-norm is obtained by a Cauchy estimate of $f$:
$$\aa{d^2_wf(p(x))}_{ \cB(Y_{\ga_*},Y_{-\ga_*}^*)}\aa{w}^2_{\ga'}
\le\Cte\frac1{\mu^2} \sup_{\O_{\ga_*}(\s, \mu)}\ab{f(y)}\aa{w}_{\ga_*}^2\le 
\Cte \sup_{y\in \O_{\ga_*}(\s, \mu)}\ab{f(y)}.$$

Since $Jd g(x)[\cdot]$ equals
$$\big(J dd^2_w f\circ p(x)[w,w]\big)[dp[\cdot]]+ 2\big(Jd^2_w f\circ p(x)[ w]\big)[\cdot], $$
and
$$Jd^2_w f:\O_{\ga'}(\s, \mu)\to \cB(Y_{\ga'},Y_{\ga'})$$
and
$$Jd d^2_w  f=J d^2_w  df:\O_{\ga'}(\s, \mu)\to \cB(Y_{\ga'},\cB( Y_{\ga'},Y_{\ga'}))$$
are rhcb, it follows that $d g$ verifies R1 and is rhcb.  The norm $\aa{Jd g(x)}_{\ga'}$ is less than
$$\aa{J   d^2_w d f(p(x))}_{\cB(Y_{\ga'},\cB( Y_{\ga'},Y_{\ga'}))}\aa{w}^2_{\ga'}
+2\aa{Jd^2_wf(p(x))}_{ \cB(Y_{\ga'},Y_{\ga'})}\aa{w}_{\ga'},$$
which is 
$\le \Cte \sup_{y\in \O_{\ga'}(\s, \mu)}\aa{Jd f(y)}_{\ga'}$ --  this follows by  Cauchy estimates of 
derivatives of $Jd f$.

Since $Jd^2 g(x)[\cdot,\cdot]$ equals
$$\big(J d^2 d_w^2f\circ p(x)[w,w]\big)[dp[\cdot],dp[\cdot]]+ 
2J\big(  d d^2_w f\circ p(x)[w]\big)[\cdot,dp[\cdot]+ 2Jd_w^2 f\circ p(x)[\cdot,\cdot],$$
and
$$Jd^2_w f:\O_{\ga'}(\s, \mu)\to \cM_{\ga',\vark}^b,$$
$$J d d^2_w f=J d^2_w df:\O_{\ga'}(\s, \mu)\to \cB(Y_{\ga'},\cM_{\ga',\vark}^b)$$
and
$$J d^2 d_w^2f=J d_w^2 d^2f :\O_{\ga'}(\s, \mu)\to \cB(Y_{\ga'},\cB(Y_{\ga'},\cM_{\ga',\vark}^b))$$
are rhcb,  it follows that $Jd g^2$ verifies R2   and is rhcb. The norm $\aa{Jd^2 g}_{\ga',\vark}$ is less than
$$\aa{J d_w^2 d^2f(p(x))}_{ \cB(Y_{\ga'},\cB(Y_{\ga'},\cM_{\ga',\vark}^b))}\aa{w}^2_{\ga'}
+2\aa{J d_w d^2f(p(x))}_{\cB(Y_{\ga'},\cM_{\ga',\vark}^b)}\aa{w}_{\ga'}+$$
$$
+2\aa{Jd^2f(x)}_{\ga',\vark},$$
which is 
$\le \Cte \sup_{y\in \O_{\ga'}(\s, \mu)}\aa{Jd^2 f(y)}_{\ga',\vark}$ --  this follows by a Cauchy estimate of $Jd^2 f$.

The derivatives with respect to $\r$ are treated alike. 
\end{proof}

\subsection{Flows}
\subsubsection{ Poisson brackets.}  \label{ss5.2}
The Poisson bracket $\{f,g\}$  of two $\cC^1$-functions $f$ and $g$
is (formally) defined by
$$
\{f,g\}=
\Om(Jdf,Jdg) =\langle J\nabla f, \nabla g\rangle
=-df[Jdg]=dg[Jdf]$$
If one of the two functions verify condition R1, this product is well-defined.
Moreover, if both $f$ and $g$ are jet-functions, then $\{f,g\}$ is also a jet-function.

\begin{proposition}\label{lemma:poisson}
Let $f,g\in  \cT_{\ga,\vark,\D}(\s,\mu)$, and let $\s'<\s$ and $\mu'<\mu\le1$. Then
\begin{itemize}
\item[(i)] $\{g,f\}\in  \cT_{\ga,\vark,\D}(\s',\mu')$ and
$$
\ab{\{g,f\}}_{\begin{subarray}{c}\s',\mu' \  \\ \ga, \vark, \D \end{subarray}}\leq 
C_{\s-\s'}^{\mu-\mu'}\ab{g}_{\begin{subarray}{c}\s,\mu\ \  \\ \ga, \vark, \D \end{subarray}}
\ab{f}_{\begin{subarray}{c}\s,\mu \ \ \\ \ga, \vark, \D \end{subarray}}$$
for 
$$C_{\s-\s'}^{\mu-\mu'}=C \big(\frac1{(\sigma-\sigma')}  +   \frac1{ (\mu-\mu') }\big).$$

\item[(ii)] the n-fold Poisson bracket
$ P_g^n f \in  \cT_{\ga,\vark,\D}(\s,\mu)$ and
$$\ab{ P_g^n f }_{\begin{subarray}{c}\s',\mu' \ \\ \ga, \vark, \D \end{subarray}}\leq
\big(C_{\s-\s'}^{\mu-\mu'}\ab{g}_{\begin{subarray}{c}\s,\mu \ \ \\ \ga, \vark, \D \end{subarray}}\big)^n
\ab{f}_{\begin{subarray}{c}\s,\mu \ \ \\ \ga,  \vark, \D \end{subarray}}$$
where $P_g f=\{g,f\}$.
\end{itemize}
($C$ is an absolute constant.)

\end{proposition}

\begin{proof} (i)
We must first consider the function $h=\Om(Jdg,Jdf)$ on $\O_{\ga_*}(\s, \mu)$
Since $J dg,\ J df:\O_{\ga_*}(\s, \mu)\to Y_{\ga_*}$ are rhcb, it follows that
$h:\O_{\ga_*}(\s, \mu)\to \C$ is rhcb, and 
$$\ab{h(x)}\le \aa{J dg (x)}_{\ga_*}\aa{J df(x)}_{\ga_*}.$$

The vector $J d h(x)$ is a sum of
$$J \Om(Jd^2g(x),Jdf(x))=Jd^2g(x)[Jdf(x)]$$
and another term with $g$ and $f$ interchanged.
Since $Jd^2g:\O_{\ga'}(\s, \mu)\to \cB(Y_{\ga'},Y_{\ga'})$ and 
$Jdg,\ Jdf:\O_{\ga'}(\s, \mu)\to Y_{\ga'}$ are rhcb, it follows that $Jd h$
verifies R1 and is rhcb. Moreover 
$$\aa{Jd^2g(x)[Jdf(x),\cdot]}_{\ga'}\le
\aa{J d^2g (x)}_{\cB(Y_{\ga'},Y_{\ga'}) }\aa{J df(x)}_{\ga'}$$
and, by definition of $ \cM_{\ga,\vark}^b$ ,
$$\aa{J d^2g (x)}_{\cB(Y_{\ga'},Y_{\ga'}) }\le \aa{J d^2g (x)}_{\ga',0}.$$

The operator $Jd^2h(x)=d(Jdh) (x)$ is a sum of 
$$Jd^3g(x)[Jdf(x)]$$
and
$$Jd^2g(x)[Jd^2f(x)]$$
and two  other terms with $g$ and $f$ interchanged.

Since $Jd^3g:\O_{\ga'}(\s, \mu)\to \cB(Y_{\ga'},\cM_{\ga',\vark}^b)$ and 
$Jdf:\O_{\ga'}(\s, \mu)\to Y_{\ga'}$ are holomorphic functions, it follows that the first function
$\O_{\ga'}(\s, \mu)\to \cB(Y_{\ga'},\cM_{\ga',\vark}^b)$ also is holomorphic. It can be estimated
on a smaller domain using a Cauchy estimate for $Jd^3g(x)$.

The second term is treated  differently. Since
$$Jd^2f,\ Jd^2g:\O_{\ga'}(\s, \mu)\to \cM_{\ga,\vark}^b$$ 
are rhcb, and since, by Proposition \ref{pMatrixProduct}, taking products
is a bounded bi-linear maps with norm $\le 1$, it follows that the second function
$\O_{\ga'}(\s, \mu)\to \cM_{\ga',\vark}^b$ is rhcb and
$$\aa{Jd^2g(x)[Jd^2f(x)]}_{\ga',\vark}\le  \aa{Jd^2g(x)}_{\ga',\vark}\aa{Jd^2f(x)}_{\ga',\vark}.$$

The derivatives with respect to $\r$ are treated alike.

(ii) That $ g_n=P_g^n f \in  \cT_{\ga,\vark,\D}(\s',\mu')$ follows from (ii), but the estimate does not follow from
the estimate in (ii). The estimate follows instead from Cauchy estimates of  $n$-fold product $P_g^n f $ and from
the following statement:

for any $n\ge1$ and any $k\ge0$, $\ab{d^k g_n(x)}$, $x\in\O_{\ga'}(\s, \mu)$, is bounded by a sum of terms of the form
$$\ab{d^{m_1} g(x)}\dots\ab{d^{m_n} g(x)}\ab{d^{m_{n+1}}f(x)}
\quad\footnote{\ in the norms of the appropriate Banach spaces}
$$
with
$\sum m_j=n+1+k$
and each $m_j\ge1$. The number of terms in the sum is $\le 2^{nk}$.
[This is proven above for $n=1$ and $k\le2$. It follows for $k\ge 3$ by the product formula for derivatives. It follows then for
all $n\ge2$ and any $k\ge0$ by an easy induction.]

Let now $m'_j=2$ if $m_j\ge3$ and $=m_j$ if $m_j\le2$. Then the term above can be estimated by Cauchy estimates:
$$
\le (C_{\s-\s'}^{\mu-\mu'})^{\sum (m_j-m'_j)}
\ab{d^{m'_1} g(x)}\dots\ab{d^{m'_n} g(x)}\ab{d^{m'_{n_1}}f (x)}\le
$$
$$\le (C_{\s-\s'}^{\mu-\mu'})^{\sum (m_j-m'_j)}
(\ab{g}_{\begin{subarray}{c}\s,\mu\ \  \\ \ga, \vark, \D \end{subarray}})^n
\ab{f}_{\begin{subarray}{c}\s,\mu \ \ \\ \ga, \vark, \D \end{subarray}}$$
The result now follows by observing that $\sum (m_j-m'_j)\le \max(n+k-2,0)$ and taking $k=2$.
[Indeed, if $\sum (m_j-m'_j)$ were $\ge n+k-1$, then $\sum m'_j\le \sum m_j-(n+k-1)=2$. Since $m'_j\ge1$ this forces
$n$ to be $=1$ and all $m'_j$ to be $=1$. Hence $m_j=m'_j$ and $\sum( m_j- m'_j)=0$.]
\end{proof}

\begin{remark}
The proof shows that the assumptions can be relaxed when $g$ is a jet function: it suffices then to assume that
 $g\in  \cT_{\ga,0,\D}(\s,\mu)$ and
$g-\hat g(\cdot,0,\cdot)\in  \cT_{\ga,\vark,\D}(\s,\mu)$.
\footnote{\ $\hat g(\cdot,0,\cdot)$ this is the $0$:th Fourier coefficient of the function $\theta\mapsto g(\cdot,\theta,\cdot)$}

Then $\{g,f\}$ will still be in $\cT_{\ga,\vark,\D}(\s,\mu)$ but with the bound
$$
\ab{\{g,f\}}_{\begin{subarray}{c}\s',\mu' \  \\  \ga, \vark, \D \end{subarray}}\leq 
C_{\s-\s'}^{\mu-\mu'} 
\big(\ab{g}_{\begin{subarray}{c}\s,\mu \ \ \\ \ga, 0, \D \end{subarray}}
+\ab{g-\hat g(\cdot,0,\cdot)}_{\begin{subarray}{c}\s,\mu \ \ \\ \ga, \vark, \D \end{subarray}}\big)
\ab{f}_{\begin{subarray}{c}\s,\mu \ \ \\ \ga, \vark, \D \end{subarray}}.$$

To see this it is enough to consider a jet-function $g$ which does not depend on $\theta$. The only difference
with respect to case  (i) is for the second differential. The second term is fine
since, by Proposition \ref{pMatrixProduct},  $\cM_{\ga',\vark}^b$ is a two-sided ideal in  $\cM_{\ga',0}^b$ and
$$\aa{Jd^2g(x)[Jd^2f(x)]}_{\ga',\vark}\le  \aa{Jd^2g(x)}_{\ga',0}\aa{Jd^2f(x)}_{\ga',\vark}.$$
For the first term we must consider $Jd^3g(x)[Jdf(x)]$ which, a priori, takes its values in $\cM_{\ga',0}^b$ and not in
$\cM_{\ga',\vark}^b$. But since $g$ is a jet-function independent of $\theta$ this term is $=0$.

\end{remark}

 \subsubsection {Hamiltonian flows} \label{ss5.3}
 
 The Hamiltonian vector field of a $\cC^1$-function $g$ on (some open set in) $Y_\ga$  is $Jdg$. 
 Without further assumptions it is an element in $Y_{-\ga}$, but if $g\in\cT_{\ga,\vark}$, then it is an element
 in $Y_\ga$ and has a well-defined local flow $\{\Phi_g^t\}$. Clearly
 $(d/dt) f(\Phi^t_g) = \{f,g\}\circ \Phi^t_g$ for a $C^1$-smooth function $f$.

\begin{proposition}\label{Summarize}
Let $g\in  \cT_{\ga,\vark,\D}(\s,\mu)$, and let $\s'<\s$ and $\mu'<\mu\le1$.
If 
$$
\ab{g}_{\begin{subarray}{c}\s,\mu\ \ \\ \ga, \vark, \D  \end{subarray}}\leq\\
 \frac1{C}\min( \s-\s',\mu-\mu'),$$
 then
\begin{itemize}
\item[(i)]  the Hamiltonian flow map  $\Phi^t=\Phi^t_g$ is,
  for  all $\ab{t}\le1$ and all $\ga_*\le \ga'\le\ga$,   a $\cC^{{s_*}}$-map
$$\O_{\ga'}(\s',\mu')\times\D \to\O_{\ga'}(\s,\mu)$$
which is real holomorphic and symplectic for any fixed  $\rho\in \D$. 

Moreover,
$$\aa{ \p_\r^j (\Phi^t(x,\r)-x)}_{\ga'}\le 
C\ab{g}_{\begin{subarray}{c}\s,\mu\ \ \\ \ga, \vark, \D  \end{subarray}},$$
and
$$
\aa{ \p_\r^j(d\Phi^t(x)-I)}_{\ga',\vark}\le
C\ab{g}_{\begin{subarray}{c}\s,\mu\ \ \\ \ga, \vark, \D  \end{subarray}},$$
for any $x\in \O_{\ga'}(\s',\mu')$, $\ga_*\le \ga'\le\ga$,  and $0\le \ab{j}\le {s_*}$.

\item[(ii)] $f\circ \Phi_g^t\in \cT_{\ga,\vark}(\s',\mu',\D)$ for $\ab{t}\le1$ and
$$
\ab{ f\circ \Phi_g^t }_{\begin{subarray}{c}\s',\mu' \ \ \\ \ga, \vark, \D \end{subarray}}\leq C
\ab{f}_{\begin{subarray}{c}\s,\mu \ \ \\ \ga, \vark, \D \end{subarray}}$$
for any $f\in \cT_{\ga,\vark}(\s,\mu,\D)$.
\end{itemize}

($C$ is an absolute constant.)
\end{proposition}

\begin{proof} 
It follows by general arguments that  the local flow $\Phi=\Phi_g: U\to \O_{\ga}(\s,\mu)$
is real holomorphic in $(t,\zeta)$ in some $U\subset \C\times \O_{\ga}(\s,\mu)$, and that it depends smoothly on
any smooth parameter in the vector field.
Clearly,  for $\ab{t}\le 1$ and $x\in \O_{\ga}(\s',\mu')$
$$\aa{\Phi^t(x,\r)-x}_{\ga}\le  \sup_{x\in \O_{\ga}(\s,\mu)}\aa{Jdg(x)}_\ga 
\le \ab{g}_{\begin{subarray}{c}\s,\mu\ \ \\ \ga, 0, \D  \end{subarray}}$$
as long as $\Phi^t(x)$ stays in the domain $ \O_{\ga}(\s,\mu)$. It follows by classical arguments that this is the case if
$$\ab{g}_{\begin{subarray}{c}\s,\mu\ \ \\ \ga, 0, \D  \end{subarray}}
\le\cte \min(\s-\s',\mu-\mu').$$

{\it The differential}. We have
$$\frac{d}{dt} d\Phi^t(x)=-Jd^2g(\Phi^t(x))d\Phi^t(x)=B(t)d\Phi^t(x),$$
where $B(t)\in \cM_{\ga,\vark}^b$.
By re-writing this equation in  the integral form 
 $d\Phi^t(x)=\Id+\int_0^t B(s) d\Phi^s(x)\dd s$ and iterating this relation, we get that 
$ d\Phi^t(x)-\Id= B^{\infty}(t)$  with
 $$B^{\infty}(t)
 =\sum_{k\geq 1}\int_{0}^{t}\int_{0}^{t_{1}}\cdots \int_{0}^{t_{k-1}} \prod_{j=1}^{k}B(t_{j})\text{d}t_{k}  \cdots \text{d}t_{2}\,\text{d}t_{1}.$$

We get, by Proposition \ref{pMatrixProduct}, that $d\Phi^t(x)-\Id\in\cM_{\ga,\vark}^b$ and, for $\ab{t}\le1$ ,
$$
\aa{ d\Phi^t(x)-\Id}_{\ga,\vark}\le \sum_{k\geq 1} \aa{Jd^2g(\Phi^t(x))}^k_{\ga,\vark}\frac{t^k}{k!}\le
\aa{Jd^2g(\Phi^t(x))}_{\ga,\vark}.$$

In particular, $A=d\Phi^t(x)$ is a bounded bijective operator
on $Y_\ga$.  Since $Jd^2g$ is a Hamiltonian vector field we clearly have that
$$\Om(A\zeta,A\zeta')=\Om(\zeta,\zeta'),\quad\forall  \zeta,\zeta'\in Y_\ga,$$
so $A$ is symplectic.
 
 {\it Parameter dependence}. For $\ab{j}=1$, we have
$$\frac{d}{dt} Z(t)=\frac{d}{dt} \frac{\p^j\Phi^t(x,\r)}{\p\r^j}=
B(t,\r)Z(t) -\frac{\p^j Jdg(\Phi^t(x,\r),\r)}{\p\r^j}  =B(t)Z(t)+A(t).$$

Since
$$\aa{A(t)}_\ga+ \aa{B(t)}_{\ga,\vark}\le \Cte \ab{g}_{\begin{subarray}{c}\s,\mu\ \ \\ \ga, \vark, \D  \end{subarray}},$$
it  follows by classical arguments, using Gronwall, that
$$\aa{Z(t)}_{\ga,0}\le \Cte \ab{g}_{\begin{subarray}{c}\s,\mu\ \ \\ \ga, \vark, \D  \end{subarray}}\ab{t}.$$
The higher order derivatives (with respect to $\rho$) of   $\Phi^t(x,\r )$,  and the derivatives of $d\Phi^t(x,\r )$  are treated in the same way.

The same argument applies to any $\ga_*\le\ga'\le\ga$.

Since
$$
 f\circ \Phi_g^t =
\sum_{n\ge0}\frac1{n!}t^nP^n_{-g}f  ,$$
(ii) is a consequence of Proposition \ref{lemma:poisson}(ii).
\end{proof}

\begin{remark}\label{r_sigma}
If the set $\cZ$ is such that $\A=\F=\emptyset$ and $\L_\infty =\Z^d$ (so
$\cZ=\Z^d$), then the domains $\O_\ga(\sigma,\mu)$ and the functional spaces on these domains 
which we  introduced   do not depend on $\sigma$. In this case in our notation  we will chose the dumb parameter
$\sigma$ to be 1. The assertions of the Propositions~\ref{lemma:poisson} and \ref{Summarize} remain true if we 
there take $\s=\s'=1$ and drop the assumptions, related to $\s$ and $\s'$ (in particular, replace there 
$\min(\s-\s', \mu-\mu')$ by $\mu-\mu'$, and replace $1/(\s-\s')$ by 0). 
\end{remark}

\bigskip
\centerline{PART II.  A BIRKHOFF NORMAL FORM}
\section{Small divisors} \label{s_2} 
\subsection{Non resonance of  basic frequencies}

In this subsection we assume that the set $\A\subset\Z^d$ is admissible, i.e. it only 
contains  integer vectors with different norms (see Definition \ref{adm}).\\
We consider the  vector of basic frequencies 
\be\label{-om}
\om\equiv\om(m)=(\om_a(m))_{a\in\A}\,, \quad m\in [1,2]\,,
\ee
where $\om_a(m)=\lambda=
\sqrt{|a|^4+m}$.  The goal of  this section   is to prove the following result: 
\begin{proposition}
\label{NRom}
Assume that $\A$ is an admissible subset of $\Z^d$ of cardinality $n$ included in $\{a\in\Z^d\mid |a|\leq N\}$. 
Then for any $k\in\Z^\A\setminus\{0\}$, any $\ka>0$ and  any $c\in \R$ we have
\begin{equation*}
 \meas\ \left\{m\in[1,2]\ \mid \
\left|\sum_{a\in\A} k_a\omega_a(m)+c\right|\leq {\ka}\right\}\leq C_n \frac{N^{4n^2} \ka^{1/n}}{|k|^{1/n}}\,,
\end{equation*} 
where $|k|:=\sum_{a\in\A}|k_a|$ and $C_n>0$ is a constant, depending only on $n$. 
\end{proposition}

The proof follows  closely that of Theorem 6.5 in
\cite{Bam03} (also see  \cite{BG06}); a weaker form of the result was obtained
earlier in \cite{bou95}. Non of  the constants $C_j$ etc. in this section depend on the set $\A$. 

\begin{lemma}\label{l_det} Assume that $\A\subset\{a\in\Z^d\mid |a|\leq N\}$.
For any $p\leq n= |\A|$, consider $p$ points $a_1,\cdots,a_p$ in $\A$.
Then the modulus of the  following determinant 
\begin{equation*}
D:=\left|
\begin{matrix}
\frac{d \om_{a_1}}{dm} & \der \null {a_2} & .& .&.&\der \null {a_p}
\\
\frac{d^2 \om_{a_1}}{dm^2} & \frac{d^2 \om_{a_2}}{dm^2} & .& .&.&\frac{d^2 \om_{a_p}}{dm^2}
\\
.& .& .& .& .&.
\\
.& .& .& .& .&.
\\
\der{p}{a_1}& \der{p}{a_ 2}& .& .&.&\der {p}{a_p}
\end{matrix}
\right| 
\end{equation*}
is bounded from below:
$$
|D|\geq C N^{-3p^2+p}\,,
$$
where $C=C(p)>0$ is a constant
depending  only on $p$.
\end{lemma}

\proof First note that, by explicit computation, 
\be\label{-9}
\frac{d^j\omega_i}{dm^j}= (-1)^{j} \Upsilon_j\(|i|^4+m\)^{\frac12 -j}\,, \qquad \Upsilon_j=\prod_{l=0}^{j-1} \frac{2l-1}2\,.
\ee
Inserting this expression in $D$,  we deduce by factoring from each $l-th$ column the term
$(|a_\ell|^4+m)^{-1/2}=\om_\ell^{-1}$, and from each $j-th$ row the term $\Upsilon_j$
%$\frac{(2j-1)!}{2^{j-1}(j-1)!2^j}$
 that the determinant, up to a sign, equals
\begin{eqnarray*}
\left[\prod_{l=1}^{p}\omega_{a_\ell}^{-1}\right]
    \left[\prod_{j=1}^{p}  \Upsilon_j         %\frac{(2j-1)!}{2^{j-1}(j-1)!2^j} 
    \right]  
\times
\left|
\begin{matrix}
1& 1& 1 &. & . & . & 1
\cr
x_{a_1}& x_{a_2}& x_{a_3}&.&.&.&x_{a_p}
\cr
x_{a_1}^2& x_{ a_2}^2& x_{a_3}^2&.&.&.&x_{a_p}^2
\cr
.& .& .& .& .&.&.
\cr
.& .& .& .& .&.&.
\cr
.& .& .& .& .&.&.
\cr
x_{a_1}^{p}& x_{a_2}^{p}& x_{a_3}^{p}&.&.&.&x_{a_p}^{p}
\end{matrix}
\right|,
\end{eqnarray*}
where we denoted  $x_{a}:=(|a|^4+m)^{-1}= \omega_{a}^{-2}$.
Since $|\om_{a_k}|\le2|a_k|^2\le 2 N^2$ for every $k$, the first factor  is bigger than $(2N^2)^{-p}$. The second is a constant, 
 while the third  is the Vandermond determinant, equal to 
\begin{equation*}
\prod_{1\leq l<k\leq p}(x_{a_\ell}-x_{a_k})=\prod_{1\leq l<k\leq p}
\frac{|a_k|^4-|a_\ell|^4}{\omega_{a_\ell}^2\omega_{a_k}^2} =: V\,.
\end{equation*}
Since $\A$ is admissible, then
$$
|V| \ge \prod_{1\leq l<k\leq p}
\frac{|a_k|^2+|a_\ell|^2}{\omega_{a_\ell}^2\omega_{a_k}^2}  \ge \(\frac14\)^{p(p-1)} N^{-3p(p-1)}\,,
$$
where we used that each factor is bigger than $\frac1{16} N^{-6}$ using again that $|\om_{a_k}|\le2|a_k|^2\le 2 N^2$ for every $k$. 
This yields  the assertion.
\endproof

%From \cite{ben85}, appendix B, we learn
\begin{lemma}\label{m1.1}
 Let
$u^{(1)},...,u^{(p)}$ be $p$ independent vectors in $\R^p$  of norm at most one, and
%$\norma{u^{(i)}}_{\ell^1}\leq1$. 
 let $w\in\R^p$ be any non-zero vector. Then there exists  $i\in[1,...,p]$ such that
$$
| \langle u^{(i)}, w\rangle |\geq
 C_p |w|  |\det(u^{(1)},\ldots,u^{(p)})|\,.
$$
\end{lemma}
\begin{proof}
Without lost of generality we may assume that $|w|=1$. 

Let $| \langle u^{(i)}, w\rangle |  \le a$ for all $i$. Consider the $p$-dimensional parallelogram $\Pi$, 
generated by the vector  $u^{(1)},...,u^{(p)}$ in $\R^p$ (i.e., the set of all linear combinations
$\sum x_j u^{(j)}$, where $0\le x_j\le 1$ for all $j$). It lies in the strip of width $2pa$, perpendicular
to the vector $w$, and its projection to to the $p-1$-dimensional space, perpendicular to $w$, lies 
in the ball around zero of radius $p$. Therefore the volume of $\Pi$ is bounded by 
$C_p p^{p-1} (2pa)=C_p' a$. Since this volume equals $|\det(u^{(1)},\ldots,u^{(p)})|$, then 
$a\ge  C_p   |\det(u^{(1)},\ldots,u^{(p)})|$. This implies the assertion. 
\end{proof}

Consider vectors $\frac{d^i\omega}{dm^i}(m)$, $1\le i\le n$, denote 
$K_i= | \frac{d^i\omega}{dm^i}(m) |$ and set
$$
u^{(i)}= K_i^{-1}\frac{d^i\omega}{dm^i}(m) , \qquad 1\le i\le n\,.
$$
From \eqref{-9} we see that\footnote{In this section $C_n$ denotes any positive constant depending only 
on $n$.} 
$\ 
K_i\le  C_n$ for all\ $1\le i\le n\,
$
(as before, the constant does not depend on the set $\A$). Combining Lemmas~\ref{l_det} and \ref{m1.1}, we find that for any vector $w$ and any  $m\in[1,2]$ there exists $r=r(m)\le n$ such that 
\be\label{m1.2}
\begin{split}
\Big|  \langle \frac{d^r\omega}{dm^r}(m), w \rangle \Big| =K_r \big| \langle u^{(r)}, w\rangle \big| \ge 
K_r C_n |w| (K_1\dots K_n)^{-1}| D|\\
\ge C_n |w| N^{-3n^2+n}\,. 
\end{split}
\ee

Now we need the following result (see Lemma B.1 in \cite{E98}):

\begin{lemma}
\label{v.112}
Let  $g(x)$ be a $C^{n+1}$-smooth function on the segment [1,2] 
such that $|g'|_{C^n} =\beta$ and $\max_{1\le k\le n}\min_x|\p^k g(x)|=\sigma$. Then 
$$
\meas\{x\mid |g(x)|\le\rho\} \le C_n \(\frac{\beta}{\sigma}+1\) \(\frac{\rho}{\sigma}\)^{1/n}\,.
$$
\end{lemma}

Consider the function $g(m)=|k|^{-1}\sum_{a\in\A}k_a \om_a(m) +|k|^{-1}c$.
Then $|g'|_{C^n}\le C'_n$, and  
$\max_{1\le k\le n}\min_m|\p^k g(m)|\ge C_n N^{-3n^2+n}$ in view of  \eqref{m1.2}. 
Therefore, by Lemma~\ref{v.112}, 
\begin{equation*}
\begin{split}
\meas\{m\mid |g(m)|\le \frac{\kappa}{|k|}  \} \le  C_n N^{3n^2-n} \(\frac{\kappa}{|k|} N^{3n^2-n}\)^{1/n}
=
C_n N^{3n^2+2n-1}
\(\frac{\kappa}{|k|} \)^{1/n}\,.
\end{split}
\end{equation*}
This implies the assertion of the proposition.

\subsection{Small divisors estimates}

We recall the notation \eqref{L+},    \eqref{-om}, and note the elementary  estimates 
\be\label{estimla}
\langle a\rangle^2 <   \la(m)< \langle a\rangle^2  +  \frac{m}{2 \langle a\rangle ^2}\qquad \forall\, a\in\Z^d\,,\ m\in[1,2]\,,
\ee
where $\langle a\rangle =\max(1,|a|^2)$.
  In this section we study  four type of linear combinations of the frequencies $\la(m)$:
\begin{align*}
D_0=& \langle\om, k\rangle, \quad k\in \Z^\A\setminus\{0\}\\
D_1=&\langle\om, k\rangle +\la, \quad k\in \Z^\A,\; a\in \L\\
D_2^\pm=&\langle\om, k\rangle+\la\pm\lb, \quad k\in \Z^\A,\; a, b\in \L\,.
\end{align*}
In subsequent sections they will become  divisors for our constructions, so we call these linear combinations
``divisors".

\begin{definition}\label{Res-kab} Consider independent formal variables $x_0, x_1, x_2,\dots$. Now  take
any divisor of the form $D_0$, $D_1$ or $D_2^\pm$, write 
there each $\omega_a, a\in\A$,  as $\lambda_a$, and then replace every  $\lambda_a, a\in\Z^d$, by
$x_{|a|^2}$. Then the divisor is
called resonant if the obtained algebraical sum of the variables $x_j, j\ge0$,  is zero. Resonant 
divisors are also called  trivial resonances.
\end{definition}

Note that a $D_0$-divisor cannot be resonant since $k\ne0$ and the set $\A$ is admissible;
a $D_1$-divisor $(k;a)$ is  resonant only if $a\in {\L_f}$, $|k|=1$ 
 and $\langle\om, k\rangle=-\omega_b$, 
where  $|a|=|b|$. Finally, a $D_2^+$-divisor or a $D_2^-$ divisor with $k\ne0$ 
may be   resonant  only when $(a,b)\in {\L_f}\times {\L_f}$, while the divisors $D_2^-$
of the form $\la-\lb$, $|a|=|b|$, all are resonant. 
So there are   finitely many trivial resonances of the form $D_0, D_1, D_2^+$ and of the form $D_2^-$ with 
$k\ne0$, but infinitely many of them of the form $D_2^-$ with $k=0$. 
\smallskip

Our first  aim is to remove from the segment  $[1,2]=\{m\}$ a small subset to guarantee that for the remaining $m$'s 
  moduli of all non-resonant  divisors  admit  positive lower bounds.  
  Below in this section 
\be\label{agreement}
\begin{split}
&\text{constants $C, C_1$ etc. depend on the admissible set $\A$,}\\
&\text{while the exponents $c_1, c_2$ etc depend only on $|\A|$. Borel  }\\
&\text{sets $\Cc_\ka$ etc. depend on the indicated arguments and $\A$.}
\end{split}
\ee

 We begin with the easier divisors 
 $D_0$, $D_1$ and $D_2^+$. 

\begin{proposition}\label{D1D2}
Let $1\ge\ka>0$. There exists a Borel set $\Cc_\ka \subset[1,2]$ and positive constants 
$C$ (cf. \eqref{agreement}),  satisfying 
$\ 
\meas\  \Cc_\ka  \leq C \ka^{1/(n+2)} ,
$ 
 such that for all $m\notin\Cc_\ka$,  all $k$ and  all $a,b \in \L$ we have
\be\label{D0}
 |\langle\om, k\rangle|\geq  \ka {\langle k\rangle}^{-n^2}, \qquad
 \text{ except if  } k=0,
 \ee
 \be\label{D1}
 |\langle\om, k\rangle+\la|\geq \ka {\langle k\rangle}^{-3(n+1)^3}, \quad
\text{ except if the divisor  is a trivial resonance}, 
\ee
\be \label{D2}
|\langle\om, k\rangle +\la+\lb|\geq \ka{\langle k\rangle}^{-3(n+2)^3}, 
\text{ except if %$\langle\om, k\rangle+\la+\lb $
 the divisor   is a trivial resonance}.
%\text{ except if  }(k;a,b) \text{ is } D_2^+ \text{ resonant}\,.
\ee
Here 
${\langle k\rangle}=\max(|k|,1)$. 

Besides, for each $k\ne0$ there exists a set
${\mathfrak A}^k_\ka$ whose measure is $\ \le C\ka^{1/n}$ such that for $m\notin {\mathfrak A}^k_\ka$ 
we have 
\be\label{D22}|\langle\om, k\rangle+j  |\geq \ka{\langle k\rangle}^{-(n+1)n } 
\text{for all $j\in\Z$ }.
\ee
\end{proposition}
\proof 
We begin with the   divisors \eqref{D0}.  
By Proposition \ref{NRom} for any non-zero $k$   we have 
$$
\meas \{m\in[1,2]\mid |\langle\om, k\rangle|\leq  \ka  |k|^{-n^2} \} < C {\ka^{1/n}}{|k|^{-n-1/n}}  \,.
$$
Therefore the relation \eqref{D0} holds for all non-zero $k$ if $m\notin\mathfrak A_0$, 
where $\meas\mathfrak A_0\le  C  \ka^{1/n} \sum_{k\ne0}  |k|^{-n-1/n} =C  \ka^{1/n}$.

Let us consider the divisors \eqref{D1}. For $k=0$ the required estimate holds 
trivially. If $k\ne0$, then the relation, opposite to \eqref{D1} implies that $|\la|\le C|k|$.  So
we may assume that $|a|\le C|k|^{1/2}$. If $|a|\notin\{|s|\mid s\in\A\}$, then Proposition~\ref{NRom} with 
$n:=n+1$, $\A:=\A\cup \{a\}$ and $N=C|k|^{1/2}$ implies that 
\begin{equation*}
\begin{split}
\meas& \{m\in[1,2]\mid |\langle\om, k\rangle+\la|\leq \ka|k|^{-3(n+1)^3}\}\\
\le& C 
 \ka^{1/(n+1)} |k|^{2(n+1)^2- 3(n+1)^2 -\frac1{n+1}  }
\leq C\kappa^{1/(n+1)} |k|^{-(n+1)^2}\,. 
\end{split}
\end{equation*}
This relation with $n+1$ replaced by $n$ 
 also holds if $|a|=|s|$ for some $s\in\A$, but $\langle\om, k\rangle+\la$ is not a trivial resonant. 
Since for fixed $k$  the set$\{\la\mid |a|^2\leq C|k|  \}$ has cardinality less than $2C|k|$, then the relation 
$
|\langle\om, k\rangle+\la|\le\ka |k|^{-3(n+1)^3}
$
holds for a fixed $k$ and all $a$ if we remove from [1,2] a set of measure 
$\le C\ka^{1/(n+1)}|k|^{-(n+1)^2+1}\leq C\ka^{1/(n+1)}|k|^{-n-1}$. 
So we achieve that
 the relation \eqref{D1}  holds for all $k$ if we remove from $[1,2]$ a set $\mathfrak A_1$ whose 
measure is bounded by $C\ka^{1/(n+1)} \sum_{k\ne0} |k|^{-n-1} =C\ka^{1/(n+1)}$.

For a similar reason there exist  a Borel set $\mathfrak A_2$ whose
measure is bounded by $C\ka^{1/(n+2)}$ and  such that \eqref{D2} holds
for $m\notin \mathfrak A_2$. Taking 
$\Cc_\ka = \mathfrak A_0\cup  \mathfrak A_1\cup \mathfrak A_2$ we get \eqref{D0}-\eqref{D2}.
Proof of \eqref{D22} is similar. 
\endproof

Now we  control divisors $D_2^-=\langle\om, k\rangle+\la-\lb$.

\begin{proposition}\label{prop-D3}   
There exist positive constants $C,c, c_-$ and for $0<\ka$ 
there is  a Borel set $\Cc'_\ka \subset[1,2]$  (cf. \eqref{agreement}),  satisfying 
\be\label{meas-estim2}
\meas\ \Cc'_\ka \leq C \ka^{c},
\ee
 such that  for all $m\in [1,2]\setminus  \Cc'_\ka$, 
 all $k\ne0$ and  all $a,b \in \L$  we have
\be\label{D3}
R(k;a,b):=
|\langle\om, k\rangle +\la-\lb|\geq  \ka |k|^{-c_- }, 
\ee
 except if the divisor is a trivial  resonance 
\end{proposition}
\proof
We  may assume that 
 $|b|\geq |a|$. We get from  \eqref{estimla} that 
$$
|\la-\lb-(|a|^2-|b|^2)|\leq  {m}{|a|^{-2}}\leq 2 |a|^{-2}.
$$
Take any $\ka_0\in (0,1]$ and construct the set $\mathfrak A^k_{\ka_0}$ as in Proposition~\ref{D1D2}.
Then $\meas  {\mathfrak A}^k_{\ka_0}\le C\ka_0^{1/n}$ and  for  any 
 $m\notin {\mathfrak A}^k_{\ka_0}$  we have 
$$
R:=
R(k;a,b)\ge  \big| \langle\om, k\rangle +|a|^2-|b|^2\big|-2|a|^{-2}\ge
 \ka_0|k|^{-(n+1)n} - 2 |a|^{-2}\,.
$$
So $R \ge\tfrac12 \ka_0|k|^{-(n+1)n}$ and  \eqref{D3} holds if 
$$
|b|^2\ge |a|^2 \ge 4\ka_0^{-1}|k|^{(n+1)n}=:Y_1. 
$$

If $|a|^2\le Y_1$, then 
$$
R \ge \lb - \la -C|k| \ge |b|^2-Y_1-C|k|-1. 
$$
Therefore  \eqref{D3} also holds if $|b|^2\ge Y_1 +C|k|+2$, and 
 it remains to consider the case when $|a|^2\le Y_1 $ and 
  $|b|^2\le Y_1 +C|k|+2$. That is (for any fixed non-zero $k$), 
consider the pairs $(\la,\lb)$, satisfying 
\be\label{above}
|a|^2\le Y_1,\qquad  |b|^2\le  Y_1+2+C|k| =:Y_2 \,.
\ee
There are at most $CY_1Y_2$ pairs like that. Since the divisor $\langle\om, k\rangle +\la-\lb$
 is not  resonant, then in view of Proposition~\ref{NRom} with $N =Y_2^{1/2}$
and $|\A|\le n+2$,
for any  $\tilde\ka>0$ there exists
a set ${\mathfrak B}^k_{\tilde \ka}\subset [1,2]$, whose measure is bounded by
$$
C  \tilde\ka^{1/(n+2)} \ka_0^{-c_1} |k|^{c_2},\qquad c_j=c_j(n)>0,
$$
such that  $R \ge \tilde\ka$ if $m\notin   {\mathfrak B}^k_{\tilde \ka}\, $
for all pairs $(a,b)$ as in \eqref{above} (and $k$ fixed). 

Let us choose $\tilde\ka =\ka_0^{2c_1(n+2)}$. Then 
$
\meas {\mathfrak B}^k_{\tilde \ka}\le C \ka_0^{c_1}|k|^{c_2}
$
and $R\ge \ka_0^{2c_1(n+2)}$ for $a,b$ as in \eqref{above}. Denote
$  \mathfrak C^k_{\ka_0}= \mathfrak A^k_{\ka_0}\cup {\mathfrak B}^k_{\tilde \ka}\, $. Then
$\meas  \mathfrak C^k_{\ka_0}\le C \(\ka_0^{1/n} + \ka_0^{c_1}|k|^{c_2}\)$, and 
for $m$ outside this set and all $a,b$ (with $k$ fixed) 
 we have 
$
R\ge \min\(\tfrac12 \ka_0|k|^{-(n+1)n}, \ka_0^{2c_1(n+2)}\)\,.
$
We see that if $\ka_0=\ka_0(k)=2\ka^{c_3} |k|^{-c_4}$ with suitable $c_3,c_4>0$, then 
$$
\meas\(  \Cc'_\ka = \cup_{k\ne0}  \mathfrak C^k_{\ka_0}\) \le C\ka^{c_3}  \,,
$$
and, if $m$ is outside $\Cc'_\ka$, then 
$R(k;a,b)\ge \ka |k|^{-c_-}$ with a suitable $c_->0$. 
\endproof

It remains to consider the  divisors $D_2^-$ with $k=0$,  i.e. $D_2^-=\la-\lb$. Such a
divisor is resonant if $|a|=|b|$. 
\begin{lemma}\label{lem:D3-k=0}
Let $m\in [1,2]$ and the divisor $D_2^-=\la-\lb$ is non-resonant, i.e. $|a|\neq|b|$. 
Then
$
\left| {\la-\lb}\right|\ge \frac1 4.$
\end{lemma}
\proof  We have 
\begin{align*}
\left|\la-\lb\right|= \frac{\left||a|^4-|b|^4\right|}{\sqrt{|a|^4+m}+\sqrt{|b|^4+m}}\geq  
\frac{|a|^2+|b|^2}{\sqrt{|a|^4+m}+\sqrt{|b|^4+m}}\ge  \frac 1 4.
\end{align*}
\endproof

By construction the  sets $\Cc_\ka$ and $\Cc'_\ka$ decrease with $\ka$. 
Let us denote 
\be\label{setC}
\Cc = \bigcap_{\ka>0} (\Cc_\ka \cup \Cc'_\ka)\,.
\ee
From Propositions \ref{D1D2}, \ref{prop-D3} 
and Lemma~\ref{lem:D3-k=0}  we get:

\begin{proposition}\label{prop-m} 
The set $\Cc$ is a Borel subset of $[1,2]$ of zero measure. For any $m\notin\Cc$ there exists 
$\ka_0=\ka_0(m)>0$ such that the relations \eqref{D0}, \eqref{D1}, \eqref{D2} and \eqref{D3} hold
with $\ka=\ka_0$.
\end{proposition}

In particular, if $m\notin\Cc$ then any of the divisors 
$$
\langle\om, s\rangle,\;\; \langle\om, s\rangle\pm\la,\;\;\langle\om, s\rangle\pm\la \pm\lb ,\quad s\in\Z^d,\; a,b\in\L,
$$
vanishes only if this is a trivial resonance. If it is not, then its modulus 
 admits a qualified estimate from below.

 The zero-measure Borel set $\Cc$ serves a fixed admissible set $\A$, $\Cc=\Cc_\A$. But since the set of all
 admissible sets is countable, then replacing $\Cc$ by $\cup_\A\Cc_\A$ we obtain a zero-measure Borel set which 
 suits all admissible sets $\Cc$.  
 For further purposes we modify $\Cc$ as follows: \be\label{modif} \Cc=: \Cc\cup \{\tfrac43, \tfrac 53\}\,. \ee

\section{The Birkhoff normal form. I}\label{BNF}

In  Sections \ref{BNF} and \ref{s_4}  we construct a symplectic change of variable that puts the  Hamiltonian
 \eqref{H1} to a normal form.
 In Sections \ref{BNF} and  \ref{s_4} 
   constants in the estimates may depend on
   \be\label{depen}
   \text{
    $d$,  $G$, $\A$ and  constants with lower index $*$  (including $c_*$)
   }
   \ee
   without saying.   Their  dependence on other parameters will be  indicated. This does not contradicts Agreements 
    (see the end of Introduction) since in these sections the set $\F$ is defined in terms of $\A$ and $\P$ does not
   occur.

\subsection{Statement of the result}\label{s3.1}

The goal of this section is to get a normal form for the Hamiltonian $h=h_2+h_4+h_{\ge5}$ 
of the beam equation, written in the form 
\eqref{beam2},   in toroidal domains in the space 
which are complex  neighbourhoods of the $n$-dimensional  real tori $T_{I_\A}$ (see \eqref{ttorus}). We scale the parameters $I_\A$ as $\nu\r$ where
$\nu>0$ is small and $\yy=(\yy_a,a\in\A)$ belongs to the domain
\be\label{DDD}
\D=[c_*,1]^\A.
\ee
In this section  $c_*\in(0,\tfrac12]$ is regarded as a fixed parameter.
%will be fixed till Section \ref{s_10.2} (the last in our work). 

Consider the complex   vicinity of the torus $T_{\nu \rho\, \A}$ (see \eqref{ttorus})
\be\label{a-a}
\Tg(\nu,\sigma,{\mu},\gamma)=\{(p_\A,q_\A,p_\L,\zeta_\L): 
\left\{\begin{array}{lll}
| \tfrac12 (p_a^2+q_a^2) -\nu\r_a|<\nu  c_*^2 {\mu}^2 & a\in\A&\\
 |\Im\theta_a|<\sigma & a\in\A&\\
\|(p_\L, q_\L)\|_\ga<\nu^{1/2}c_*{\mu}& &,
\end{array}\right.\ee
where $\theta_a$ is related to $p_a,q_a$ through $\frac{p_a-{\bf i}q_a}{\sqrt{p_a^2+q_a^2}}=e^{{\bf i}\theta_a}$  ---  this is well-defined when $\mu\le1$ because then $p_a^2+q_a^2\not=0$ for all $a\in\A$ whenever the point belongs to this vicinity.

In this section we use the complex coordinates $(\xi_a, \eta_a), a\in\Z^d$, defined in \eqref{change}, denoting 
$(\xi_a, \eta_a) =\zeta_a$. So we will write  points of  $\Tg(\nu,\sigma,{\mu},\gamma)$ as 
$\zeta = (\zeta_\A, \zeta_\L)$. 
We recall (see \eqref{L+}) that we have split the set $\L=\Z^d\setminus \A$ into the union
$\L=\L_f\cup\L_\infty$. We will write  $\zeta_\L=(\zeta_f, \zeta_\infty)$ and will use the notation of 
Section~\ref{sThePhaseSpace} with 
$
\cZ=\Z^d,\, \Z=\A\cup \L_f\cup\L_\infty
$
(i.e. with $\F=\L_f$). 
 
\begin{proposition}\label{thm-HNF} There exists a  zero-measure Borel set $\Cc\subset[1,2]$
such that for any admissible   set  $\A$, any  $c_*\in(0,1/2]$ and   $m\notin\Cc$ we can find 
  real numbers   $\ga_g>\ga_*=(0,m_*+2)$ and $\nu_0>0$, where % $\ga_g$ depends only on the function $G$ 
  %$\s_*$ depends only on $\A$
   $\nu_0$ depends  on $m$, %and $G$, 
   with the following property.%for $m\notin\Cc$ we have:
 
 For any $0<\nu\le\nu_0$ and  $\yy\in [c_*,1]^\A$ there exists  al holomorphic  diffeomorphism (onto its image)
\be\label{mu*}
\Phi_\yy:  \O_{\ga_*} \big({\frac 12}, {\mu_*^2} \big)\to \Tg(\nu,  1,1,\ga_*)\,,\qquad 
{\mu_*}={\tfrac{c_*}{2\sqrt2}}\,,
\ee
which defines analytic transformations 
$$
\Phi_\yy:  \O_{\ga} \big({\tfrac 12}, {\mu_*^2} \big)\to \Tg(\nu,  1,1,\ga)\,,\quad \ga_*\le\ga\le \ga_g\,,
$$
such that
$$
\Phi_\r^*\big(-{\bf i}dp \wedge dq\big)=
\nu dr_\A\wedge d\theta_\A \ -{\bf i}\ \nu d\xi_\L\wedge d\eta_\L,
$$
and such that
\be\label{HNF}
\begin{split}
\frac1{\nu} h\circ\Phi_\yy(r,\theta,\xi_\L,\eta_\L) =\langle \Omega(\yy), r\rangle  &+ \sum_{a\in\L_\infty }\Lambda_a (\yy)\xi_a\eta_a+\\
&+\frac{\nu}2\, \langle K(\yy) \zeta_f, \zeta_f\rangle
+ f( r,\theta,\zeta_\L;\rho),
\end{split}
\ee
where $h$ is the Hamiltonian \eqref{H2}$+$\eqref{H1}, satisfies:

(i) $\Phi_\yy$  depends smoothly (even analytically) on $\yy$,
and
\be\label{boundPhi}
\begin{split}
\mid\mid \Phi_\r(r,\theta,\xi_\L,\eta_\L)-
(\sqrt{\nu\r}\cos(\theta),&\sqrt{\nu\r}\sin(\theta),0,0
%\sqrt{\nu\r}\xi_\L,\sqrt{\nu\r}\eta_\L   
 ) \mid\mid_\ga\le \\
&\le C(\sqrt\nu\ab{r}+\sqrt\nu\aa{(\xi_\L,\eta_\L)}_\ga+\nu^{\frac32} )
\end{split}
\ee
for all $(r,\theta,\xi_\L,\eta_\L)\in \O_{\ga} (\frac 12, \mu_*^2)\cap\{\theta\ \textrm{real}\}$ and all $\ga_*\le \ga\le\ga_g$.

(ii)  the vector  $\Om$ and the scalars $\La, a\in \L_\infty $ are affine functions of $\rho$, explicitly defined by
 \eqref{Om} and \eqref{Lam};

(iii) $K$ is a  symmetric  real matrix. It is a
quadratic polynomial of $\sqrt\yy=(\sqrt\yy_1,\dots,\sqrt\yy_n)$, explicitly defined 
by  relation \eqref{K};
%, and  satisfies \be\label{est11} \| K(\yy)\| \le C_2\qquad\forall\, \yy\in\D; \ee  

(iv) the remaining  term $f$ belongs  to $\cT_{\ga_g,\vark=2,\D}({\tfrac12}, \mu_*^2)$ and satisfies 
\be\label{est}
|f|_{\begin{subarray}{c}1/2,\mu_*^2 \  \\ \ga_g, 2, \D \end{subarray}}
 \le C\nu   \,, \qquad
 |f^T|_{\begin{subarray}{c}1/2,\mu_*^2 \  \\ \ga_g, 2, \D \end{subarray}}
  \le C
\nu^{3/2} \,.
\ee

Finally, $\Phi_\r$ is not a real diffeomorphism, but verifies the ``conjugate-reality'' condition:
$$\Phi_\yy(r,\theta,\xi_\L,\eta_\L)\quad\textrm{ is real if, and only if},\quad \eta_\L=\ov{\xi}_\L.$$

The constant $C$ depends on $m$ (we recall \eqref{depen}) but not on $\nu$. 
\end{proposition}

\begin{remark}\label{r_p4.1} 1)
 $\Phi_\yy$ is 
  close to the {  scaling by the factor $\nu^{1/2}$ } on the $\L_\infty $-modes 
%$$\|(\tilde\zeta_a)_{a\in\L_\infty}- (\zeta_a)_{a\in\L_\infty} \|_\ga = 0( \| (\zeta_a)_{a\in\L_\infty} \|^2_\ga)$$
but not on the $(\A\cup {\L_f})$-modes, where it is close to a certain affine 
  transformation, depending  on $\theta$. Moreover
  $$\Phi_\yy\big (\O_{\ga} ({\tfrac 12}, {\mu_*^2}) \big)\subset\Tg(\nu,  1,1,\ga),\qquad \ga_*\le \ga\le\ga_g.$$

2)  All the objects, involved in this proposition, except the remaining term $f$ in \eqref{HNF}, 
depend only on the  main part  $u^4$ of $G$, and not on the higher order correction.
\end{remark}

The rest of this section is devoted to the proof of Proposition \ref{thm-HNF}.  From now on we
arbitrarily enumerate the set $\A$ of excited modes, i.e.  we write $\A$ as 
\be\label{labA}
\A = \{a_1,\dots, a_n\}\,,
\ee
so that the cardinality of $\A$ is $n$,
and accordingly identify $\R^{\A}$ with $\R^n$ and identify various $\A$-valued maps with maps, valued in the set 
$\{1,\dots, n\}$.

\subsection{Resonances and the Birkhoff procedure}\label{s_4.2}

Instead of the domains $\O_\ga(\sigma,\mu)$, in this section we will use domains
\be\label{ddomain}
\O_\ga(\sigma,\mu^2,\mu) = \{(r,\theta,w): |r| <\mu^2, |\Im\theta|<\sigma, \|w\|_\ga<\mu\}\,,
\ee
more convenient for the normal form calculation. The space of functions on $\O_\ga(\sigma,\mu^2,\mu)$, defined similar to 
the space $\cT_{\ga,\vark}(\sigma,\mu)$, will be denoted $\cT_{\ga,\vark}(\sigma,\mu^2,\mu)$. The norm 
$ |f|_{\begin{subarray}{c}\s,\mu,\mu^2\\ \ga, \vark  \end{subarray}} $ in this space is defined by the relation \eqref{norm}, where the first 
line is given the weight $\mu^0=1$, the second line -- the weight $\mu^1$, and the third line -- $\mu^2$. Note that 
\be\label{O_relation}
\O_\ga(\sigma,\mu^2,\mu)\subset \O_\ga(\sigma,\mu)\subset \O_\ga(\sigma,\mu, \sqrt\mu),
\ee
and that $|\cdot   |_{\begin{subarray}{c}\s,\mu,\mu^2\\ \ga, \vark  \end{subarray}}$ 
and $|\cdot   |_{\begin{subarray}{c}\s,\mu\\ \ga, \vark  \end{subarray}}$ are equivalent if  $\mu\sim1$.

In the situation of Remark~\ref{r_sigma}, when $\cZ=\Z^d$ and  $\A=\F=\emptyset$, we have 
 $\cT_{\ga,\vark}(1,\mu^2,\mu) = \cT_{\ga,\vark}(1,\mu)$,  and 
\be\label{twonorms1}
| f   |_{\begin{subarray}{c}1,\mu,\mu^2\\ \ga, \vark  \end{subarray}}  \le 
 |f   |_{\begin{subarray}{c}1,\mu\\ \ga, \vark  \end{subarray}}  \le \mu^{-2}
 | f   |_{\begin{subarray}{c}1,\mu,\mu^2\\ \ga, \vark  \end{subarray}} 
\ee
for any  $0<\mu\le1$. 

\begin{example}[homogeneous functionals]
\label{ex_homog}
Let $\cZ=\Z^d$ and  $\A=\F=\emptyset$ 
and  let $f(w)\in \cT_{\ga,\vark}(1,1,1) = \cT_{\ga,\vark}(1,1)$ be an $r$-homogeneous
function,  $r\le 2$ integer. Then $df$ and $d^2f$ are, accordingly, $(r-1)\,$-- and $(r-2)$--homogeneous. 
So for any $0<\mu\le1$ we have 
\be\label{homog}
|f|_{\begin{subarray}{c}1,\mu,\mu^2\\ \ga, \vark  \end{subarray}} =
\mu^r
 |f|_{\begin{subarray}{c}1,1,1\\ \ga, \vark  \end{subarray}} \,.
\ee
If for $j=1,2\ $ 
$f_j(w)\in \cT_{\ga,\vark}(1,1,1)$ is an $r_j$--homogeneous functional, $r_j\ge2$, 
then the functional $ \{f_1,f_2\} $ is $r_1+r_2-2$--homogeneous. So the relation above 
and  Proposition~\ref{lemma:poisson} imply that 
\be\label{poiss_homog}
|\{f_1,f_2\} |_{\begin{subarray}{c}1,1,1\\ \ga, \vark  \end{subarray}} \le C
|f_1 |_{\begin{subarray}{c}1,1,1\\ \ga, \vark  \end{subarray}} \cdot 
|f_2|_{\begin{subarray}{c}1,1,1\\ \ga, \vark  \end{subarray}} \,.
\ee
\end{example}

Let us consider  the quartic part $h_2 + h_4$  of the Hamiltonian $h$, 
$$
h_2=
 \sum_{a\in\Z^d}\lambda_a \xi_a\eta_a,\quad
h_4= (2\pi)^{-d}\sum_{(i,j,k,\ell)\in\J}\frac{(\xi_i+\eta_{-i})(\xi_j+\eta_{-j})(\xi_k+\eta_{-k})(\xi_\ell+\eta_{-\ell})}{4\sqrt{\li\lj\lk\lel}}\,
$$
 (the variables $\xi, \eta$ are defined in \eqref{change}), where 
 $\J$ denotes the zero momentum set:$$\J:=\{(i,j,k,\ell)\subset\Z^d\mid i+j+k+\ell=0\}.$$

We decompose $h_4=h_{4,0}+h_{4,1}+h_{4,2}$ according to 
\begin{align*}
h_{4,0}=& \frac 1 4 (2\pi)^{-d}\sum_{(i,j,k,\ell)\in\J}\frac{ \xi_i\xi_j\xi_k\xi_\ell +\eta_i\eta_j\eta_k\eta_\ell}{\sqrt{\li\lj\lk\lel}},\\
h_{4,1}=& (2\pi)^{-d}\sum_{(i,j,k,-\ell)\in\J}\frac{ \xi_i\xi_j\xi_k\eta_\ell +\eta_i\eta_j\eta_k\xi_\ell}{\sqrt{\li\lj\lk\lel}},\\
h_{4,2}=& \frac 3 2 (2\pi)^{-d}\sum_{(i,j,-k,-\ell)\in\J}\frac{ \xi_i\xi_j\eta_k\eta_\ell }{\sqrt{\li\lj\lk\lel}}\,,
\end{align*}
and define
$$
\J_2= \{ (i,j,k,\ell)\subset\Z^d\mid
(i,j,-k,-\ell)\in\J, \; \sharp \{i,j,k,\ell\}\cap \A \geq 2\}\,.
$$
By Proposition~\ref{prop-m} we have 
\begin{lemma}\label{res-mon}If $m\notin\Cc$, then there exists $\ka(m)>0$ such that for all $(i,j,k,\ell)\in \J_2$
\begin{align*}
|\li+\lj+\lk-\lel|&\geq \ka(m)\, ;\\
|\li+\lj-\lk-\lel|&\geq \ka(m), \quad \text{except if } \{|i|,|j|  \}=\{|k|,|\ell| \}\, .
\end{align*}
\end{lemma}

For $\ga=(\ga_1, \ga_2)$, where $0\le\ga_1\le1$, $\ga_2\ge m_*$, and for $\cZ=\Z^d$ as above consider 
the space $Y_\ga$ as in Section~\ref{sThePhaseSpace}, written in terms of the complex coordinates
$\zeta_a=(\xi_a, \eta_a), a\in\Z^d$. In these variables the symplectic from $\Omega$ reads 
$\Om = -i\sum d\xi_a\wedge d\eta_a$. For $0<\mu\le1$ consider the ball 
$\O_\ga(1,\mu^2,\mu)=\O_\ga(1,\mu) =\{|\zeta|_\ga<\mu\}$. 

For any vector $\zeta=(\zeta_a=(\xi_a,\eta_a), a\in \Z^d)$, we will write $\zeta_a^+ = \xi_a$ and 
$\zeta_a^- = \eta_a$. For  an integer $r\ge2$ we abbreviate $a=(a_1,\dots, a_r)\in (\Z^d)^r$,
$\vs=(\vs_1,\dots,\vs_r)\in\{+,-\}^r$, and consider a homogeneous polynomial 
$$
P^r(\zeta) = M\sum_{a\in (\Z^d)^r} \sum_{\vs\in \{+,-\}^r} A_a^\vs\, \zeta_{a_1}^{\vs_1}\dots \zeta_{a_r}^{\vs_r}\,.
$$
Here $M$ is a positive constant, the moduli of all coefficients $A_a^\vs$ are bounded by 1, and 
$$
A_a^\vs = 0 \quad\text{unless} \quad a_1\vs_1^0+\dots a_r\vs_r^0=0
$$
for some fixed boolean vector $\vs^0\in\{+,-\}^r$. Denote by $D^-$ the block-diagonal operator 
\be\label{D-}
D^- = \diag \{ |\lambda_a|^{-1/2} I, a\in \Z^d\}\,,\qquad I\in M(2\times2)\,, 
\ee
and set $Q^r(\zeta) = P^r(D^-\zeta)$.

\begin{lemma}\label{XPanalytic}
For any $\ga$ as above, $Q^r\in \cT_{\ga,2}(1,1,1)$ and 
\be\label{z.1}
 |Q^r|_{\begin{subarray}{c}1,1,1\\ \ga, 2  \end{subarray}} \le CM\,,\qquad C=C(r)\,.
\ee
\end{lemma}
The lemma is proved in Appendix A. 

Note that by this lemma, \eqref{twonorms1} and \eqref{homog},
$
|Q^r  |_{\begin{subarray}{c}1,\mu\\ \ga, 2  \end{subarray}}  \le CM \mu^{r-2}.
$
  So by Lemma~\ref{Summarize} if the function $Q^r$ is real, then 
   the Hamiltonian flow-maps $\Phi^t = \Phi^t_{Q^r}$, $|t|\le1$,
  define real-holomorphic symplectic mappings 
  \be\label{homog_flow}
  \begin{split}
  \Phi^t:  \O_\ga(1,\mu^2,\mu) \to \O_\ga(1, 4\mu^2,2\mu)\quad
  \text{if $r\ge4$ and $\mu\le \mu_1,\ \mu_1=\mu_1(M)>0$,}\\
  \text{ or if $r=3$ and $M>0$ is sufficiently small
  }
  \end{split}
  \ee
  (we recall that now 
  $\O_\ga(1,\mu^2,\mu) = \O_\ga(1,\mu)$).

 \begin{proposition}\label{Thm-BNF}
For   $m\notin\Cc$  and $\mu_g>0$, $\ga_g> \ga_*$ as in Lemma~\ref{lemP} 
  there exists $\mu\in(0,\mu_g]$ and    a real holomorphic symplectomorphism 
  $$
  \tau: \O_{\ga_*} (1,\mu) =\{|\zeta|_{\ga_*}<\mu\}
  \to \O_{\ga_*} (1, 2\mu)
  $$
  which is a diffeomorphism on its image and which for $\ga_*\le\ga\le\ga_g$ 
  defines analytic mappings 
  $\tau: \O_{\ga} (1,\mu)  \to \O_{\ga} (1, 2\mu)$, such that  
 \be\label{esti1}
 \|\tau^{\pm1}(\zeta)-\zeta\|_{\ga} \le C \|\zeta\|_{\ga}^3
   \qquad \forall\,\zeta\in \O_{\ga} (1,\mu)   \,.
 \ee
 It transforms the Hamiltonian $h=h_2+h_4+ h_{\ge5}$  
  as follows:
\be\label{trans}
h \circ \tau= h_2 + z_4+ q_4^3+ r_6^0 +h_{\ge5}\circ \tau\,,
\ee
where 
\begin{align*}
z_4=&\frac 3 2(2\pi)^{-d} \sum_{\substack{(i,j,k,\ell)\in\J_2 \\ \{|i|,|j|  \}=\{|k|,|\ell| \}}}  \frac{ \xi_i\xi_j\eta_k\eta_\ell }{\li\lj},
\end{align*}
and 
$q_4^3=q_{4,1}+q_{4,2}$ with\footnote{The upper index 3 signifies that $q_4^3$ is at least cubic
in the transversal directions $\{\zeta_a, a\in \L\}$.}
\begin{align*}
q_{4,1}=&(2\pi)^{-d} \sum_{(i,j,-k,\ell)\not\in\J_2}\frac{ \xi_i\xi_j\xi_k\eta_\ell +\eta_i\eta_j\eta_k\xi_\ell}{\sqrt{\li\lj\lk\lel}},\\
q_{4,2}=& \frac 3 2(2\pi)^{-d}\sum_{(i,j,k,\ell)\not\in\J_2}\frac{ \xi_i\xi_j\eta_k\eta_\ell }{\sqrt{\li\lj\lk\lel}}\,.
\end{align*}
The functions $z_4, q_4^3, r_6^0, h_{\ge5}\circ\tau$ are real holomorphic on $\O_{\ga} (1,\mu)$
for each $\ga_*\le\ga\le\ga_g$.  Besides
 $r_6^0$ and $h_{\ge5}\circ\tau$ are, respectively, functions of order 6 and 5 at the origin. 
 For any $0<\mu'\le \mu$ the functions 
 %and ${\ga}>0$ (depending on $g$) |f|_{\begin{subarray}{c}\s,\mu,\mu^2\\ \ga, \vark  \end{subarray}} 
  $z_4, q_4^3, r_6^0$ and $h_{\ge5}\circ\tau$
belong to ${\Tc}_{\ga_g, 2}(1, (\mu')^2, \mu')$,  and 
\be\label{Z4}
|z_4|_{\begin{subarray}{c}1,\mu',(\mu')^2\\ \ga_g, 2  \end{subarray}} +
|q_4^3|_{\begin{subarray}{c}1,\mu',(\mu')^2\\ \ga_g, 2  \end{subarray}} \le C(\mu')^4\,,
%\big[z_4\big]_\ga^{{\mu},D} +\big[q_4^3\big]_\mu^{{\ga},D}\le C\mu^4\,,
\ee
\be\label{R6}
|r^0_6|_{\begin{subarray}{c}1,\mu',(\mu')^2\\ \ga_g, 2  \end{subarray}} 
\le C( \mu')^6\,,
\ee
\be\label{R66}
|h_{\ge5}\circ\tau|_{\begin{subarray}{c}1,\mu',(\mu')^2\\ \ga_g, 2  \end{subarray}}   \le C (\mu')^5\,.
%\big[h_{\ge5}\circ\tau\big]_\ga^{{\mu},D} \le C \mu^5\,,
\ee
The constants  $C$  and $\mu$ depend on $m$ (we recall \eqref{depen}). 
\end{proposition}

\proof
We use the classical Birkhoff normal form procedure. We construct the transformation 
$\tau$ as the time one flow $\Phi^1_{\chi_4}$ of  a Hamiltonian $\chi_4$,  given by
\begin{align}\label{chi4}\begin{split}
\chi_4=& -\frac{{\bf i} } 4 (2\pi)^{-d}\sum_{(i,j,k,\ell)\in\J}\frac{ \xi_i\xi_j\xi_k\xi_\ell -\eta_i\eta_j\eta_k\eta_\ell}{(\li+\lj+\lk+\lel)\sqrt{\li\lj\lk\lel}}\\
&-{\bf i}(2\pi)^{-d}\sum_{(i,j,-k,\ell)\in\J_2}\frac{ \xi_i\xi_j\xi_k\eta_\ell -\eta_i\eta_j\eta_k\xi_\ell}{(\li+\lj+\lk-\lel)\sqrt{\li\lj\lk\lel}}\\
&- \frac {3{\bf i}} 2 (2\pi)^{-d} \sum_{\substack{(i,j,k,\ell)\in\J_2 \\ \{|i|,|j|  \}\neq\{|k|,|\ell| \}}} \frac{ \xi_i\xi_j\eta_k\eta_\ell }{(\li+\lj-\lk-\lel)\sqrt{\li\lj\lk\lel}}
\end{split}\end{align}

The Hamiltonian $\chi_4$ is 4-homogeneous and real (its takes real values if $\xi_a = \bar \eta_a$ for each $a$). 
If $m\notin\Cc$, then by   Lemma~\ref{XPanalytic}    $\chi_4 \in \cT_{\ga,2}(1,1,1)$, 
and  by Lemma~\ref{Summarize} and \eqref{homog_flow} the time-one flow-map of this Hamiltonian, 
 $\tau=\Phi^1_{\chi_4}$ is a real holomorphic 
and  symplectic change of coordinates, defined  in  the $\mu$--neighbourhood  of the origin in $Y_{\ga}$
for any $\ga_*\le \ga\le \ga_g$ and 
 a suitable positive $\mu=\mu(m)$. The relation \eqref{homog} implies that on 
$  \O_{\ga} (1,2\mu)$ the norm of the Hamiltonian vector field is bounded by $C \mu^3$. 
This implies \eqref{esti1}. 

{ 
Since the Poisson bracket, corresponding to the symplectic form $-{\bf i}d\xi\wedge d\eta$  is
$
\{F,G\} ={\bf i}\langle \nabla_\eta F, \nabla_\xi G\rangle -{\bf i} \langle \nabla_\xi F, \nabla_\eta G\rangle,
$
and since $\nabla_{\eta_s}h_2=\lambda_s \xi_s$,   $\nabla_{\xi_s}h_2=\lambda_s \eta_s$, then} we 
 calculate
\be\label{h2chi4}
\begin{split}
\{\chi_4, h_2\}= & \frac{1 } 4 (2\pi)^{-d}\sum_{(i,j,k,\ell)\in\J}\frac{ \xi_i\xi_j\xi_k\xi_\ell +
\eta_i\eta_j\eta_k\eta_\ell}{\sqrt{\li\lj\lk\lel}}\\
+&(2\pi)^{-d}\sum_{(i,j,-k,\ell)\in\J_2}\frac{ \xi_i\xi_j\xi_k\eta_\ell +\eta_i\eta_j\eta_k\xi_\ell}{\sqrt{\li\lj\lk\lel}}\\
+& \frac {3} 2 (2\pi)^{-d}
\sum_{\substack{(i,j,k,\ell)\in\J_2 \\ \{|i|,|j|  \}\neq\{|k|,|\ell| \}}}   \frac{ \xi_i\xi_j\eta_k\eta_\ell }{\sqrt{\li\lj\lk\lel}}\,.
\end{split}
\ee
Therefore the transformed quartic part of the Hamiltonian $h$,  $(h_2+ h_4)\circ \tau$, equals 
\begin{align*}
 h_2+ \big(h_4  +  \{\chi_4, h_2\} \big) + \big(  \{\chi_4, h_4 \}
+&\int_0^1 (1-t) \{\chi_4,\{\chi_4, h_2+ h_4\}\}\circ \Phi_{\chi_4}^t \dd t \big)\\
=& h_2 +( z_4+ q_4^3)+ r_6^0
\end{align*}
with $z_4$ and $q_4^3$ as in the statement of the proposition and 
$$
r_6^0=  \{\chi_4, h_4 \} +\int_0^1 (1-t) \{\chi_4,\{\chi_4, h_2+ h_4\}\}\circ \Phi_{\chi_4}^t \dd t\,.
% \{h_4,\chi_4\}+ \int_0^1 (1-t)\{\{h_2+ h_4,\chi_4\},\chi_4\}\circ \Phi_{\chi_4}^t \dd t.
$$ 
 The reality of the functions $z_4$ and $q_4^3$ follow from the explicit formulas for them,
while the inclusion of these functions to ${\Tc}_{\ga_g, 2}(1, 1, 1)$ 
 and the estimate \eqref{Z4}  for any $0<\mu'\le\mu$  hold
by Lemma~\ref{XPanalytic} and \eqref{homog}. 

To verify \eqref{R6} we first note that $\{\chi_4, h_4\}$ is a 6-homogeneous function, 
belonging to ${\Tc}_{\ga_g, 2}(1, 1, 1)=:\Tc$
by \eqref{poiss_homog}. It satisfies the estimate in  \eqref{R6}  by \eqref{homog}. Next, 
 $\{\chi_4, h_2\}$ is a 4-homogeneous function, given by \eqref{h2chi4}.
By  Lemma~\ref{XPanalytic}  it belongs to  $\Tc$. The function $\{\chi_4, h_4\}$ is
6-homogeneous and belongs to $\Tc$ by \eqref{poiss_homog}. So $ \{\chi_4,\{\chi_4, h_2+ h_4\}\}$ is a sum
of  a 6\,- and 8-homogeneous functions, belonging to $\Tc$ by \eqref{poiss_homog}. Now the estimate 
 \eqref{R6} for the second component of $r^0_6$ follows from \eqref{homog_flow}, Lemma~\ref{Summarize} 
 and \eqref{homog}.

Finally,  the estimate \eqref{R66} follows by applying the argument above to homogeneous 
components of $h_{\ge5}$ and noting that the obtained sum converges, if $\mu$ is
sufficiently small. We skip the details. 
\endproof

Clearly $\Tg(\nu, 1,1,\ga)\subset \O_\ga(1,\mu)$ if $\nu\le C^{-1}\mu^2$ (see \eqref{a-a}). 
 Due to \eqref{esti1}, if $\zeta\in\Tg(\nu,1/2,1/2,{\ga})$ and $\ga_*\le\ga\le \ga_g$, 
 then $\|\tau^{\pm1}(\zeta)-\zeta\|_\ga\le C'(m)\nu^{\frac 3 2}$. Therefore
\be\label{prop1}
\tau^{\pm1} ( \Tg(\nu,1/2,1/2,\ga))\subset  \Tg(\nu,1 ,1,\ga) \subset \O_\ga(1,\mu)\,,
\ee
provided that $\nu\le C^{-1}\mu^2$, $\ga_*\le\ga\le \ga_g$ and $\rho\in [c_*, 1]^\A$. 

%$\nu\le C^{-1} \delta(m)^2$ and  $\yy\in\Da$, where $c_1=c_1(\A,m,g(\cdot),c_*)$ is sufficiently small. 

\subsection{Normal form, corresponding to  admissible   sets $\A$} \label{s_3.3}
Everywhere below  in  Sections \ref{BNF}--\ref{s_4}
 the set  $\A$ is assumed to be  admissible 
in the sense of Definition~\ref{adm}.

In the domains $\Tg=\Tg(\nu,\sigma,{\mu},\gamma)$ we  pass from the complex variables $(\zeta_a, a\in\A)$, 
to the corresponding complex
action-angles $(I_a, \theta_a)$, using the relations
\be\label{ac-an}
\xi_a=\sqrt I_a e^{{\bf i}\theta_a},\qquad \eta_a=\sqrt I_a e^{-{\bf i}\theta_a}\,,\quad a\in\A\,.
\ee
By   $ \Tg^{I,\theta}= \Tg^{I,\theta}(\nu,\sigma, \mu, \gamma) $ we will denote a domain
$\Tg(\nu,\sigma, \mu, \gamma)$, written in the variables $(I,\theta, \xi_\L\L, \eta_\L)$, and will denote
by $\iota$ the corresponding  change of variables,
\be\label{iota}
\iota: \Tg^{I,\theta} \to \Tg,\qquad (I,\theta, \xi_\L, \eta_\L) \mapsto \zeta. 
\ee
Thus, 
$
\iota^{-1} T_{\nu\rho\, \A} =\{(I,\theta, 0,0) : I=\nu\rho, \theta\in\T^n\}\,.
$

The Hamiltonian $z_4$ contains the integrable part,
 formed by monomials of the form $ \xi_i\xi_j\eta_i\eta_j=I_iI_j$ that only
  depend on the actions $I_n=\xi_n\eta_n$, $n\in\Z^d$. Denote it $z_4^+$ and denote the rest $z_4^-$. It is not hard to see 
that 
\be\label{Z4+}
z_4^+ \circ\iota
=\frac 3 2(2\pi)^{-d} \sum_{\ell\in\A,\ k\in\Z^d} (4-3\delta_{\ell,k})\frac{I_\ell I_k}{\lambda_\ell\lambda_k}.
\ee
To calculate  $z_4^-$, we  decompose it  according to the number of indices in $\A$:
a monomial $\xi_i\xi_j\eta_k\eta_\ell $ is in $z_4^{-r}$ ($r=0,1,2,3,4$) if $(i,j,-k,-\ell)\in\J$ and $\sharp \{i,j,k,\ell\}\cap\A=r$. 
We note that, by construction, $z_4^{-0}=z_4^{-1}=\emptyset$. 

Since  $\A$ is admissible, then  in view of Lemma~\ref{res-mon}
for  $m\notin\Cc$ the set $z_4^{-4}$ is empty. The set $z_4^{-3}$ is empty as well:
 
\begin{lemma}
If $m\notin\Cc$,  then
$z_4^{-3}=\emptyset.$
\end{lemma}
\proof
Consider any term $\xi_i\xi_j\eta_k\eta_\ell \in z_4^{-3}$, i.e.  $\{i,j,k,\ell\}\cap \A=3$. Without lost of generality we can assume that $i,j,k\in\A$ and $\ell\in\L$. 
Furthermore we know that $i+j-k-\ell=0$ and
$\{|i|,|j|\}=\{|k|,|\ell|\}$.
In particular  we must have $|i|=|k|$ or $|j|=|k|$ and thus, since $\A$ is admissible, $i=k$ or $j=k$.
Let for example, $i=k$. Then  $|j|=|\ell|$.
Since  $i+j=k+\ell$  we conclude that $\ell=j$ which contradicts our hypotheses.
\endproof

Recall that the finite set $ {\L_f}\subset\L$ was defined in \eqref{L+}.  The mapping
\be\label{lmap}
\ell: \L_f \to \A,\quad a\mapsto \ell(a)\in\A \text{ if } \ |a|=|\ell(a)|,
\ee
is well defined since the set $\A$ is admissible.  Now we define two subsets 
of $ {\L_f}\times {\L_f}$:
\begin{align}
\label{L++} ( {\L_f}\times {\L_f})_+=&\{(a,b)\in  {\L_f}\times {\L_f}\mid  \ell(a)+\ell(b)=a+b\}\\
\label{L+-} ( {\L_f}\times {\L_f})_-=&\{(a,b)\in  {\L_f}\times {\L_f}\mid a\neq b \text{ and }\ell(a)-\ell(b)=a-b\}.
\end{align}
\begin{example}\label{Ex39}
If $d=1$, then  $\ell(a)=-a$ and  the sets $( {\L_f}\times {\L_f})_\pm$ are empty.
If $d$ is any, but $\A$ is a one-point set $\A=\{b\}$, then $\L_f$ is the punched discrete sphere $\{a\in\Z^d\mid |a|=|b|, a\ne b\}$,
$\ell(a)= b$ for each $a$, and the sets $( {\L_f}\times {\L_f})_\pm$ again are empty. If $d\ge2$ and $|\A|\ge2$, then in general
the sets $( {\L_f}\times {\L_f})_\pm$ are non-trivial. See in Appendix~B.
\end{example}
  
   Obviously
 \be\label{obv}
 ( {\L_f}\times {\L_f})_+\cap ( {\L_f}\times {\L_f})_-=\emptyset\,.
 \ee
For further reference we note that
\begin{lemma}\label{L++-}
If $(a,b)\in ( {\L_f}\times {\L_f})_+\cup  ( {\L_f}\times {\L_f})_-$ then $|a|\neq|b|$.
\end{lemma}
\proof If $(a,b)\in ( {\L_f}\times {\L_f})_+$ and $|a|=|b|$ then $\ell(a)=\ell(b)$ and we have 
$$|a+b|=|2\ell(a)|=2|a|=|a|+|b|$$
which is impossible since $b$ is not proportional to $a$. 
 If $(a,b)\in ( {\L_f}\times {\L_f})_-$ and $|a|=|b|$ then $\ell(a)=\ell(b)$ and we get $a-b=0$ which is impossible in $ ( {\L_f}\times {\L_f})_-$. 
\endproof

Our notation now agrees with that of Section \ref{sThePhaseSpace}, where $\cZ=\Z^d$ is the disjoint union 
$\Z^d =\cA\cup \L_f\cup \L_\infty$. Accordingly, the space $Y_\ga=Y_{\ga \Z^d}$ decomposes as 
\be\label{YY}
Y_\ga = Y_{\A}\oplus Y_{\L_f}\oplus Y_{\ga \L_\infty},\qquad  Y_\ga=\{ \zeta = (\zeta_\A, \zeta_f, \zeta_\infty)\}\,,
\ee
where $Y_{\ga \A} = \, \text{span}\, \{\zeta_s, s\in\A\}$, etc. Below in this Section and in Section~\ref{s_4}, the domains
$\O_\ga(\s, \mu^2, \mu)$ and $\O_\ga(\s,  \mu)$, as well as the corresponding function spaces, refer the $\cZ$
as above.

\begin{lemma}\label{lem:adm}
For $m\notin\Cc$ the part $z_4^{-2}$ of the Hamiltonian $z_4$ equals 
\be\label{Z421}
\begin{split}
3{(2\pi)^{-d}} \Big(& \sum_{(a,b)\in  ( {\L_f}\times {\L_f})_+} \frac{ \xi_{\ell(a)}\xi_{\ell(b)}\eta_a\eta_b+ \eta_{\ell(a)}\eta_{\ell(b)}\xi_a\xi_b}{\la\lb}\\
+ 2&\sum_{(a,b)\in  ( {\L_f}\times {\L_f})_-}  \frac{ \xi_{a}\xi_{\ell(b)}\eta_{\ell(a)}\eta_b }{\la\lb}\Big)\,.
\end{split}
\ee
\end{lemma}
\proof
Let $\xi_i\xi_j\eta_k\eta_\ell $ be a monomial in $z_4^{-2}$. We know that  $(i,j,-k,-\ell)\in\J$ and  $\{|i|,|j|\}=\{|k|,|\ell|\}$. 
 If $i,j\in\A$ or $k,\ell\in\A$ then we obtain the finitely many monomials as in the first sum in  \eqref{Z421}.
 Now we assume that 
 $i,\ell\in\A$ and $ j,k\in\L.$
Then we have that, 
{ either }   $|i|=|k|$ and 
$|j|=|\ell|$  which leads to finitely many monomials as in the second sum in  \eqref{Z421}.
{ Or} $i=\ell$ and $|j|=|k|$. 
In this last case, the zero momentum condition implies that $j=k$ which is not possible in $z_4^{-}$.
\endproof

\subsection{Eliminating the non integrable terms}
For $\ell\in\A$ we introduce the variables $(I_a, \theta_a, \zeta_\L)$ as in \eqref{ac-an}, \eqref{iota}. 
Now the symplectic structure $-{\bf i}d\xi\wedge d\eta$ reads
\be\label{nsympl}
-\sum_{a\in\A} dI_a\wedge d\theta_a  -{\bf i}  d\xi_\L\wedge d\eta_\L\,.
\ee
In view of \eqref{Z4+}, \eqref{trans} and 
Lemma~\ref{lem:adm},  for $m\notin\Cc$ 
the  Hamiltonian  $h$, transformed by $\tau\circ\iota$, 
may be written as 
\begin{align*} 
h  \circ\tau\circ\iotaÂ  =&\langle \om, I\rangle  +\sum_{s\in\L}\ls \xi_s\eta_s+
\frac 3 2(2\pi)^{-d} \sum_{\ell\in\A,\ k\in\Z^d} (4-3\delta_{\ell,k})\frac{I_\ell \xi_k\eta_k}{\lambda_\ell\lambda_k}\\
+&3(2\pi)^{-d} \Big(\sum_{(a,b)\in  ( {\L_f}\times {\L_f})_+} \frac{ \xi_{\ell(a)}\xi_{\ell(b)}\eta_a\eta_b+ \eta_{\ell(a)}\eta_{\ell(b)}\xi_a\xi_b}{\la\lb}\\
+&2 \sum_{(a,b)\in  ( {\L_f}\times {\L_f})_-}  \frac{ \xi_{a}\xi_{\ell(b)}\eta_{\ell(a)}\eta_b }{\la\lb}\Big)
+ q_4^3 \circ\iota+r^0_5\,,
%\qquad r^0_5=(h_{\ge5}\circ \tau+r^0_6)\circ\iota
\,
\end{align*}
where $r^0_5=h_{\ge5}\circ \tau\circ\iota+r^0_6\circ\iota$
(recall that $\omega=(\lambda_a, a\in\A)$). 
The first line contains the integrable terms. The second and third lines  contain  the lower-order non integrable terms,
depending on the angles $\theta$;  there are finitely many of them. 
The last line contains the remaining high order  terms, where $q^3_4$ is of total order (at least) 4
and of order 3 in the {normal directions $\zeta$, while $r^0_5$ is of total order at least 5. 
 The latter is the sum of $r^0_6\circ\iota$ which comes from the Birkhoff normal form procedure
 (and is of order 6) and $h_{\ge5}\circ \tau\circ\iota$ which comes from the term of order 5  in the nonlinearity \eqref{g}. Here $I$ is regarded 
as a variable of order 2, while $\theta$  has zero order. The terms $q_4^3 \circ\iota$ and $r^0_5$ 
 should  be regarded as a perturbation. 

To deal with the non integrable terms in the second and third lines,
following the works on the finite-dimensional reducibility (see \cite{E01}),
we introduce a change of variables  
$$\Psi: (\tilde I, \tilde\theta,    \tilde\xi,\tilde\eta)  \mapsto( I, \theta,  \xi,  \eta)\,,$$
 symplectic with respect to 
\eqref{nsympl},  but such that its differential at the origin is 
not close to the identity. It is    defined by the following relations:
\begin{equation*}
\begin{split}
&  I_\ell=\tilde I_\ell-\sum_{\substack{|a|=|\ell| ,\  a\neq \ell}}{\tilde\xi}_a \tilde\eta_a,\quad   \tl=\tilde \tl\quad \ell\in\A\,;\\
&  \xi_a={\tilde\xi}_a e^{{\bf i} \tilde \theta_{\ell(a)}},\quad  \eta_a=\tilde\eta_a e^{-{\bf i} \tilde \theta_{\ell(a)}} \quad a\in {\L_f}\,;
\qquad \xi_a={\tilde\xi}_a , \quad  \eta_a=\tilde\eta_a  \quad a\in\L_\infty.
\end{split}
\end{equation*}

For any 
$(\tilde I, \tilde\theta, \tilde\zeta)\in \Tg^{I,\theta}(\nu,\sigma,\mu,\ga)$ denote by
$y=\{y_l, l\in\A\}$ the vector, whose $l$-th component equals 
$y_l=\sum_{|a|=|l|\,, a\ne l}\tilde\xi_a\tilde\eta_a$. Then
\be\label{Inu}
|I-   \nu\rho|\le |\tilde I-  \nu\rho |+|y|\le c_*^2 \nu\mu^2 +\sum_{a\in\L_f}|\tilde\xi_a\tilde\eta_a|\le 
2c_*^2\nu\mu^2\,.
\ee
This implies that 
\be\label{prop11}
\Psi^{\pm1}(  \Tg^{I,\theta}\big(\nu, \tfrac12, \cc ,\ga\big)  ) \subset \Tg^{I,\theta}\big(\nu,\tfrac12, \tfrac12,\ga\big)
=:  \Tg^{I,\theta}\,.
\ee
The transformation $\Psi$ is identity on each torus $\{ (I, \theta, \zeta_\L): I=\,$const$, \theta\in\T^n,\zeta_\L=0\}$. 
Writing it as $(I, \theta, \zeta_\L)\mapsto (\tilde I,\tilde\theta, \tilde\zeta_\L)$ we see that 
 \be\label{TheRem}
|\tilde I_a-I_a|\leq \| \zeta_\L\|_\ga^2\,,\ a\in\A ,\  \tilde\theta= \theta \text{ and } \| \tilde\zeta_\L\|_\ga=\|\zeta_\L\|_\ga\,,
\ee
and that $(\xi, \eta) = \iota  (\tilde I,\tilde\theta, \tilde\zeta_\L)$ satisfies 
\be\label{newxi}
\xi_l=\sqrt{I_l}\,e^{{\bf i}\theta_l} = \sqrt{\tilde I_l}\,e^{{\bf i} \tilde \theta_l} +O (\nu^{-1/2} )\,O(|\zeta_\L|^2)\,, 
\quad
 l\in\A\,.
\ee

Accordingly, dropping the tildes, we write the restriction to $\Tg^{I,\theta}$ of the 
 transformed Hamiltonian $h^1=
h\circ\tau\circ\iota\circ \Psi$   as
\begin{align*} 
h^1=&
\langle \om, I\rangle  +\sum_{a\in\L_\infty}\la {\xi}_a\eta_a
+6(2\pi)^{-d} \sum_{\ell\in\A,\ k\in\L} \frac{1}{\lambda_\ell\lambda_k}(I_\ell-\sum_{\substack{|a|=|\ell| \\ a\in {\L_f}}}\xi_a\eta_a) \xi_k\eta_k\\
&+\frac 3 2(2\pi)^{-d} \sum_{\ell,k\in\A} \frac{4-3\delta_{\ell,k}}{\lambda_\ell\lambda_k}(I_\ell-\sum_{\substack{|a|=|\ell| \\ a\in {\L_f}}}\xi_a\eta_a) (I_k-\sum_{\substack{|a|=|k| \\ a\in {\L_f}}}\xi_a\eta_a)\\
&+3(2\pi)^{-d}\sum_{(a,b)\in  ( {\L_f}\times {\L_f})_+} \frac{ \sqrt{I_{\ell(a)}I_{\ell(b)}}}{\la\lb}(\eta_a\eta_b+
 \xi_a\xi_b)\\
&+6(2\pi)^{-d}\sum_{(a,b)\in  ( {\L_f}\times {\L_f})_-} \frac{ \sqrt{I_{\ell(a)}I_{\ell(b)}}}{\la\lb}\xi_a\eta_b+
q_4^{3'}  +r^{0'}_5 +\nu^{-1/2}r^{4'}_5
\,.
\end{align*}
Here $q^{3'}_4$ and $r^{0'}_5$ are the function $q^{3}_4$ and $r^{0}_5$, transformed by
$\Psi$, so the former  satisfy the same estimates as the latter, while $r^{4'}_5$ is a function 
of forth order in the normal variables. The latter comes from  re-writing terms like $\xi_{\ell(a)} \xi_{\ell(b)}\eta_a\eta_b$,
using \eqref{newxi} and expressing $\eta_a, \eta_b$ via the tilde-variables. 
Or, after a  simplification:
\begin{align} \begin{split} \label{H-fin}
h^1= &\langle \om, I\rangle  +\sum_{a\in\L_\infty}\la \xi_a\eta_a
+\frac 3 2(2\pi)^{-d} \sum_{\ell,k\in\A} \frac{4-3\delta_{\ell,k}}{\lambda_\ell\lambda_k}I_\ell I_k\\
&+3(2\pi)^{-d} \Big( 2\sum_{\ell\in\A,\ a\in\L_\infty} \frac{1}{\lambda_\ell\lambda_a}I_\ell \xi_a\eta_a-  \sum_{\ell\in\A,\ a\in {\L_f}} \frac{(2-3\delta_{\ell,|a|})}{\lambda_\ell\lambda_a}I_\ell \xi_a\eta_a\Big)ÃÂÃÂÃÂÃÂ±ÃÂÃÂÃÂÃÂ \\
&+3(2\pi)^{-d} \sum_{(a,b)\in  ( {\L_f}\times {\L_f})_+} \frac{ \sqrt{I_{\ell(a)}I_{\ell(b)}}}{\la\lb}(\eta_a\eta_b+ \xi_a\xi_b)\\ 
&+6(2\pi)^{-d}\sum_{(a,b)\in  ( {\L_f}\times {\L_f})_-} \frac{ \sqrt{I_{\ell(a)}I_{\ell(b)}}}{\la\lb}\xi_a\eta_b
+ q_4^{3'}  +r^{0'}_5 +\nu^{-1/2}r^{4'}_5
\,.
\end{split}\end{align}

We see that the transformation $\Psi$ removed from $h\circ\tau\circ\iota$ the non-integrable lower-order terms on the price 
of introducing ``half-integrable" terms which do not depend on the angles $\theta$, but depend on the actions $I$ and 
quadratically depend   on the  finitely many variables $\xi_a,\eta_a$ with $a\in {\L_f}$.

The Hamiltonian $h\circ\tau\circ \Psi$ should be regarded as a function of the variables 
$(I,\theta,\zeta_\L )$. Abusing notation, below we often drop the lower-index $\L$ and  write 
$\zeta_\L= (\xi_\L,\eta_\L)$ as $\zeta= ( \xi,\eta)$.

\subsection{Rescaling the variables and defining  the transformation
 $\Phi$}
 Our aim is to study the transformed   Hamiltonian $h^1$ on the domains
  $ \Tg^{I,\theta} =\Tg^{I,\theta}(\nu, \frac12, \cc ,\ga) $, $0\le\ga\le\ga_g$
 (see \eqref{prop11}). 
 To do this we re-parametrise points of $\Tg^{I,\theta}$ by mean of the scaling 
 \be\label{scaling}
 \chi_\yy :  (\tilde r,\tilde\theta,\tilde\xi, \tilde\eta) \mapsto (I,\theta,\xi,\eta)\,,
 \ee
 where
 $\ 
I=\nu\rho+\nu \tilde r,\quad \theta=\tilde\theta,\quad
\xi=\sqrt\nu\, \tilde \xi,\quad \eta=\sqrt\nu\, \tilde\eta\,.
$
Clearly,
$$
\chi_\yy: \O_{\ga}(\tfrac 12, \mu_*^2, \mu_*  )   \to  \Tg^{I,\theta}\,
$$
for $0\le\ga\le\ga_g$, where $\mu_*$ is defined in \eqref{mu*}, 
 and in the new variables the symplectic structure reads 
$$
-\nu\sum_{\ell\in\A}\tilde dr_\ell\wedge d\tilde\tl \ -{\bf i}\ \nu\sum_{a\in\L}d\tilde\xi_a\wedge d\tilde\eta_a.
$$
Denoting
$$
\Phi =\Phi_\yy=\tau\circ\iota
\circ\Psi\circ\chi_\yy,
$$ 
we see that this transformation is analytic in $\yy\in\D$. In view of \eqref{TheRem},
$\zeta=(\xi,\eta) = \Phi(\tilde r, \tilde\theta, \tilde\zeta)$ satisfies
$$
\| \zeta- \zeta' \|_\ga \le C(\sqrt\nu\, (|\tilde r| +\|\zeta\|_\ga))\,,\qquad 
\zeta'=\big( \sqrt{\nu\rho} \,e^{{\bf i}\tilde\theta},  \sqrt{\nu\rho} \,e^{{\bf i}\tilde\theta}, 0\big)\,.
$$
This relation and  \eqref{esti1} imply \eqref{boundPhi}, so the assertion (i) of the proposition holds.

Dropping the tildes and forgetting the irrelevant   constant $\nu\langle \om, \rho\rangle$, we have 
\begin{align} \begin{split} \label{H-rescall}
h \circ \Phi(r,\theta,\zeta) &=\nu\Big[ \langle \om, r\rangle + \sum_{a\in\L_\infty}\la  \xi_a\eta_a
+(2\pi)^{-d}\nu\,
\Big(\,
\frac 3 2
  \sum_{\ell,k\in\A} \frac{4-3\delta_{\ell,k}}{\lambda_\ell\lambda_k}\rho_\ell r_k\\
+6   &\sum_{\ell\in\A,\ a\in\L_\infty} \frac{1}{\lambda_\ell\lambda_a}\rho_\ell \xi_a\eta_a-3  \sum_{\ell\in\A,\ a\in {\L_f}} \frac{(2-3\delta_{\ell,|a|})}{\lambda_\ell\lambda_a}\rho_\ell \xi_a\eta_a\\
+3  &\sum_{(a,b)\in  ( {\L_f}\times {\L_f})_+} \frac{  {\sqrt{\rho_{\ell(a)}}\sqrt{\rho_{\ell(b)}}}}{\la\lb}(\eta_a\eta_b+ \xi_a\xi_b)\\
+6 & \sum_{(a,b)\in  ( {\L_f}\times {\L_f})_-} \frac{  {\sqrt{\rho_{\ell(a)}}\sqrt{\rho_{\ell(b)}}}}{\la\lb}\xi_a\eta_b\Big) \Big]  \\
&+\Big(\big(q_4^{3'}  +r^{0'}_5+\nu^{-1/2}r^{4'}_5\big)(I,\theta,\sqrt\nu\zeta)\Big)\mid_{I=\nu\rho+\nu r}\,,
\end{split}\end{align}
where
$\zeta=\zeta_\L=(\zeta_a)_{a\in\L},\ \zeta_a=(\xi_a,\eta_a)$, and  $\zeta_f=(\zeta_a)_{a\in {\L_f}}$.
So, 
\be\label{f} {\nu}^{-1} h \circ \Phi = \tilde h_2+ f\,,
\ee
 where $f$ is the perturbation, given by the last line in \eqref{H-rescall},
 \be\label{ff}
 f=\nu^{-1}\Big(\big(q_4^{3'}  +r^{0'}_5 +\nu^{-1/2}r^{4'}_5\big)(I,\theta,\nu^{1/2}\zeta)\Big)\mid_{I=\nu\rho+\nu r}\,,
 \ee
and 
 $\tilde h_2= \tilde h_2( I, \xi,\eta;\yy,\nu)$ is the quadratic part of the Hamiltonian, which is 
 independent from  the angles $\theta$:
$$
\tilde h_2=\langle \Omega, r\rangle  + \sum_{a\in\L_\infty}\Lambda_a \xi_a\eta_a+\nu \langle K(\yy)\zeta_f,\zeta_f\rangle\,.
$$
Here $\Omega=(\Omega_k)_{k\in\A}$ with 
\begin{align}\label{Om}\Omega_k=\Omega_k(\yy,\nu)&=\om_k+ \nu\sum_{\ell\in\A} M^\ell_k \rho_l,\quad 
%(2\pi)^{-d}\frac{3\nu}{\lambda_k}\sum_{\ell\in\A} \frac{4-3\delta_{\ell,k}}{\lambda_\ell}\rho_\ell,
M^\ell_k=\frac{3(4-3\delta_{\ell,k})}{(2\pi)^d\lk\lambda_\ell} \,,
\\
\label{Lam}\Lambda_a=\Lambda_a(\yy,\nu)&= \la+6\nu(2\pi)^{-d}\sum_{\ell\in\A}\frac{\rho_\ell}{\lambda_\ell \la}\,,
\end{align}
and  $K(\yy) $ is a symmetric complex matrix,  acting in the space 
\be\label{Yf}
Y_{ \L_f}=\{\zeta_f\}\simeq \C^{2|\L_f|}\,,
\ee
such that the corresponding quadratic form is 
\begin{align} \begin{split} \label{K}
\langle K(\yy)\zeta_f,\zeta_f\rangle=\,&3(2\pi)^{-d} 
\Big(\sum_{\ell\in\A,\ a\in {\L_f}} \frac{(
3\delta_{\ell,|a|}-2
)}{\lambda_\ell\lambda_a}
\rho_\ell \xi_a\eta_a\\
+ \sum_{(a,b)\in  ( {\L_f}\times {\L_f})_+}& \frac{ {\sqrt{\rho_{\ell(a)}}\sqrt{\rho_{\ell(b)}}}}{\la\lb}(\eta_a\eta_b
+ \xi_a\xi_b)+ \\
2 \sum_{(a,b)\in  ( {\L_f}\times {\L_f})_-} &\frac{  {\sqrt{\rho_{\ell(a)}}\sqrt{\rho_{\ell(b)}}}}{\la\lb}\xi_a\eta_b \Big).
\end{split}\end{align}
Note that the matrix $M$ in \eqref{Om} is invertible since 
$$
\det M={3^n}{(2\pi)^{-dn}}\big(\Pi_{k\in\A}\lk\big)^{-2}\det \left(4-3\delta_{\ell,k}\right)_{\ell,k\in\A}\ne0\,.
$$
The explicit formulas \eqref{Om}-\eqref{K} imply the assertions (ii) and (iii). 

The transformations $\Psi\circ\chi_\rho$ and $\tau\circ\iota$ both are real if we use in the spaces $Y_\ga$ 
and $Y_{\ga\,\L}$ the real coordinates $(p_a, q_a)$, see \eqref{change}. This implies the stated ``conjugate-reality"
of $\Phi_\rho$.

  It remains to verify (iv). By Proposition \ref{Thm-BNF}  the function $f$ belongs to the class 
  $ {\Tc}_{\ga_g, 2} ({\tfrac 12},\mu_*^2,  \mu_*) $. 
 Since the reminding term $f$ has the form \eqref{ff} then in view of \eqref{Z4}-\eqref{R66} 
for $(r,\theta,\zeta)\in\O_{\ga}({\tfrac 12}, \mu_*^2, \mu_* )$ it satisfies the estimates
\begin{equation*}
\begin{split}
 |f|\le C  \nu  \,,\quad
\| \nabla_\zeta f\|_{\ga} \le C  \nu  \,,\quad
\|\nabla^2_\zeta f\|_{\ga, 2}^b \le C  \nu \,.
\end{split}
\end{equation*}
Now consider the $f^T$-component of $f$. Only the second term in \eqref{ff} contributes to it
and we have that 
$$
|f^T| + \| \nabla_\zeta f^T\|_\ga +\|\nabla_\zeta^2 f^T\|^b_{\ga,2}\le C\nu^{3/2}\,.
$$
 This implies the assertion (iv) of the proposition in view of \eqref{O_relation} and  \eqref{twonorms1}. 
\endproof
\medskip

We will provide the domains 
$ \O_{\ga} \big({\frac 12},  \mu_*^2 \big)\subset \O_{\ga} \big({\frac 12}, \mu_*^2, \mu_* \big) = \{(r,\theta,\xi,\eta)\}$ with the 
symplectic structure $- \sum_{\ell\in\A}dr_\ell\wedge d\tl \ -{\bf i} \sum_{a\in\L}d\xi_a\wedge d\eta_a$. Then
the transformed Hamiltonian system, constructed in Proposition~\ref{thm-HNF} has the 
Hamiltonian, given by the r.h.s. of \eqref{HNF}.

%%%%%%%%%%%%%%%%%%%%%%%%%%%%%%%%%%%%%%%%%%%%%

\section{The Birkhoff normal form. II} \label{s_4}

 In this section we  shall refine the normal form \eqref{HNF} further. We shall construct a $\rho$-dependent transformation which diagonalises the Hamiltonian  operator (modulo the term $f$) and  shall examine its  
 smoothness in $\yy$. So here we are concerned with 
 analysis of the finite-dimensional linear Hamiltonian operator ${\mathbf i}JK(\r)$ defined by the Hamiltonian 
  \eqref{K}. To do this  we will have to restrict $\r$ to some (large) subset  $Q\subset \D= [c_*,1]^\A$. 
  In this section and below $c_*$ is regarded as a parameter of the construction, belonging to an
   interval $(0,\tfrac12 c_0]$, where $c_0>0$ depends on $m$ and on the constants in \eqref{depen}.
This $c_0$   is introduced in  Lemma~\ref{laK} and is fixed after it.
    The parameter $c_*$ will be fixed till Section~\ref{s_10.2} (the last in our work),
    where we will vary it.

  In this section we  shall also shift from the conjugate-reality to the ordinary reality, 
  thus restoring the original real character of the system.

\begin{theorem}\label{NFT} There exists a zero-measure Borel set $\Ca \subset[1,2]$ such that for any 
admissible set  $\A$ and any  $m\notin \Cc$ 
% any $c_*\in(0, \tfrac12 c_0]$ (see \eqref{DD})  and any 
  there exist 
 real numbers $\ga_g>\ga_*=(0,m_*+2)$ and
  $\beta_0,  \nu_0,c_0 >0$, 
   where $c_0$, $\beta_0$, $\nu_0$ 
depend  on $ m$,  such that,
for any $0<c_*\le c_0$, $0<\nu\le\nu_0$ and $0<\bb \le\beta_0$
 there exists an open  set 
$Q=Q(c_*,\bb,\nu)  \subset [c_*,1]^\A$, increasing as  $\nu\to0$ and satisfying
\be\label{mesmes}
\Leb ([c_*,1]^\A \setminus  Q
)
\le C\nu^{\bb}\,,
\ee
with the following property.

For any $\r\in Q$ there exists  a real holomorphic  diffeomorphism (onto its image)
\be\label{mu*bis}
\Phi_\yy:  \O_{\ga_*} \big({\tfrac 12}, {\mu_*^2} \big)\to \Tg(\nu,  1,1,\ga_*)\,,\qquad 
{\mu_*}={\tfrac{c_*}{2\sqrt2}},
\ee
which defines analytic diffeomorphisms 
$
\Phi_\yy:  \O_{\ga} \big({\frac 12}, {\mu_*^2} \big)\to \Tg(\nu,  1,1,\ga)$, $ \ga_*\le\ga\le\ga_g\,,
$
such that
\be\label{new_symplec}
\Phi_\yy^*\big(d\xi\wedge d\eta\big)=
\nu dr_\A\wedge d\theta_\A \  +\ \nu d p_\L\wedge d q_\L,
\ee
and 
\be\label{HNFbis}
\begin{split}
\frac1{\nu} \,h\circ&\Phi_\yy(r,\theta,p_\L,q_\L) =\langle \Omega(\yy), r\rangle  + \frac12 \sum_{a\in\L_\infty }\Lambda_a (\yy)
( p_{a}^2 + q_{a}^2)+\\
&+ \frac12\sum_{b\in\L_f\setminus \F} \Lambda_b(\yy)      ( p_{b}^2 + q_{b}^2)
 %\sum_{j=1}^{M_0}  \mu(b_j,\rho)  \Big( u_{b_j}^2 + v_{b_j}^2\Big)
 +  \nu \langle K(\yy) _\F,  \zeta_\F\rangle
%\sum_{j\in \L_f^e} \Lambda_j^e(\rho) (u_j^2+v_j^2)
%+\frac\nu2 \langle H_0^h(\rho)\tilde\zeta^{fh}, \tilde\zeta^{fh}\rangle
+ f(r,\theta, \zeta_\L; \yy),
\end{split}
\ee
where $\F=\F_\r\subset \L_f$ (only depending on the connected component of $Q$ containing $\yy$),
and $h$ is the Hamiltonian \eqref{H2}$+$\eqref{H1}. $\Phi_\yy$ satisfies:

(i) $\Phi_\yy$  depends smoothly on $\yy$ and 
\be\label{boundPhibis}
\begin{split}
\mid\mid \Phi_\r(r,\theta,\xi_\L,\eta_\L)-(\sqrt{\nu\r}\cos(\theta),&\sqrt{\nu\r}\sin(\theta),\sqrt{\nu\r}\xi_\L,\sqrt{\nu\r}\eta_\L    ) \mid\mid_\ga\le \\
&\le C(\sqrt\nu\ab{r}+\sqrt\nu\aa{(\xi_\L,\eta_\L)}_\ga+\nu^{\frac32} )\nu^{-\hat c\bb}
\end{split}
\ee
for all $(r,\theta,\xi_\L,\eta_\L)\in \O_{\ga} (\frac 12, \mu_*^2)\cap\{\theta\ \textrm{real}\}$ and all $\ga_*\le \ga\le\ga_g$.

(ii)  the vector  $\Om$ and the scalars 
$\La, a\in\L_\infty $, are affine functions of $\rho$, explicitly 
defined  \eqref{Om}, \eqref{Lam};
 
 (iii)  the functions $\Lambda_b(\yy)$, $b\in\L_f\setminus \F$, are smooth in $Q $,
\be\label{Lambdab}
\|\Lambda_b\|_{C^j(Q )}\le C_j\nu^{- \bb\beta(j)}\nu ,\qquad \forall j\ge0,
\ee
where $0<\beta(1)\le\beta(2)\le\dots$, 
and  satisfy \eqref{K4}. In some open subset of $[c_*,1]^\A$ they
also satisfy \eqref{aaa}. 

(iv) $ K$ is a  symmetric  real matrix that depends smoothly on $\yy\in Q $, and
 \be\label{normK}
 \sup_{\yy\in Q }
  \|\p^j_\yy   K(\yy)\|  \le C_j\nu^{-\bb\beta(j) },\qquad  \forall j\ge0\,.
\ee
 The set $\F=\F_\r$ is void for some $\r$ 
 (in which case the operator $K(\yy)$ is trivial).
 
 (v) the eigenvalues $\{ \pm{\bf i}\Lambda_{a},  a\in \F\}$ of $J K$ are smooth in $Q $,
 satisfy \eqref{Lambdab} and 
\be\label{hyperb}
 \inf_{\yy\in Q }  | \Im \Lambda_a(\yy) | \ge C^{-1} \nu^{ \bar c \bb},\qquad \forall  a\in\F\,.
\ee

(vi) There exists a complex symplectic operator $U(\yy)$ such that
$$
U(\yy)^{-1} J K(\yy) U(\yy) ={\bf i}\diag \{\pm \Lambda_{a}(\yy),  a\in \F\}\,.
$$
The operator $U(\yy)$ smoothly depends on $\yy$ and satisfies 
\be\label{Ubound}
\sup_{\yy\in  Q }
 \big( \|\p^j_\yy U (\yy)\| + \|\p^j_\yy U  (\yy)^{-1} \|) \le C_j\nu^{-\bb\beta(j) },
 \qquad \forall j\ge0\,.
 \ee

vii)    $ f$ belongs  to $\cT_{\ga,\vark=2,Q}({\tfrac12}, \mu_*^2)$   and satisfies 
\be\label{estbis}
| f|_{\begin{subarray}{c}1/2,\mu_*^2 \  \\ \ga_g, 2, Q \end{subarray}}
 \le C\nu^{- \hat c \bb}\nu  \,, \qquad
 | f^T|_{\begin{subarray}{c}1/2,\mu_*^2 \  \\ \ga_g, 2, Q \end{subarray}}
  \le C\nu^{- \hat c \bb}\nu^{3/2} \,.
\ee

The set $Q$ and the matrix $K(\rho)$ do not depend on the function $G$
(having the form \eqref{g}). 
The constants $C, C_j$ are as in \eqref{depen}, while 
 the exponents $\bar c, \hat c$ and $\beta(j)$ depend on $m$ (we recall 
\eqref{depen}). 
\end{theorem}

\begin{remark*}  1) By \eqref{new_symplec} the transformation $\Phi_\yy$ transforms the beam
equation, written in the form \eqref{beam2}, to a system, which has the Hamiltonian \eqref{HNFbis} 
with respect to the symplectic structure $dr_\A\wedge d\theta_\A  +\nu d p_\L\wedge d q_\L$. 

2) We also have
$\ 
\Phi_\yy\big (\O_{\ga} ({\frac 12}, {\mu_*^2}) \big)\subset\Tg(\nu,  1,1,\ga)$ for 
$ \ga_*\le \ga\le\ga_g.
$
\end{remark*}

The remaining part of this section is devoted to the proof of this result.
 
\subsection{Matrix $K(\yy)$} \label{s_3.6}
Recalling \eqref{labA} and \eqref{DDD}, we write the 
 symmetric  matrix $K(\yy)$, defined by  relation \eqref{K}, as  a block-matrix, polynomial in  
$
\sqrt\rho = (\sqrt\rho_1, \dots, \sqrt\rho_n)  \,.
$
We write it as  $ K(\yy) =  K^d(\yy) + \tkn(\yy)$. 
Here $\tkd$ is the block-diagonal matrix 
\be\label{.1}
\begin{split}
\tkd (\yy) &=\text{diag}\,\Big( \left(\begin{array}{ll}
0& \mu(a,\rho)\\ \mu(a,\rho)& 0 \end{array}\right), \ a\in \L_f\Big), \\
\mu(a,\rho) &= C_*  \big(\frac32\,\rho_{\ell(a)}\la^{-2}- \la^{-1} \sum_{l\in\A} \rho_l\lambda_l^{-1}\big)\,,\quad C_*=3(2\pi)^{-d}. 
\end{split}
\ee
Note that\footnote{Here and in similar situations below we do not mention the obvious dependence 
on the parameter $m\in[1,2]$.}
\be\label{munew}
\mu(a,\rho)\quad \text{is a function of $|a|$ and $\rho$\,.
}
\ee

The non-diagonal 
matrix $\tkn$ has zero diagonal blocks, while for $a\ne b$ its block $\tkn (\yy)_a^b $ equals 
$$
 C_*\frac{\sqrt{ {\yy_{l(a)} \yy_{l(b)}}}}{\la \lb}  \, \left(
 \left(\begin{array}{ll}
1& 0\\
0& 1
\end{array}\right) \chi^+(a,b)+
 \left(\begin{array}{ll}
0& 1\\
1& 0
\end{array}\right)  \chi^-(a,b)
\right)\,,
$$
where
\ben
\chi^+(a,b)=
 \left\{\begin{array}{ll}
1, \;\; (a,b)\in (\L_f\times\L_f)_+,
\\
0,  \;\; \text {otherwise},
\end{array}\right.
\een
and $\chi^-$ is defined similar in terms of the set $ (\L_f\times\L_f)_-$. In view of \eqref{obv},
$$
\chi^+(a,b)\cdot \chi^-(a,b)\equiv0. 
$$

Accordingly, the Hamiltonian matrix 
 $\H(\yy) =={\bf i}JK(\yy)$ equals
$ \big( \H^d(\yy) + \H^{n/d}(\yy)\big)$,  where
\begin{align}\begin{split}\label{diag}
\H^d(\yy) &={\bf i}\,\text{diag}\, \left( 
 \left(\begin{array}{ll}
\mu(a,\rho)& 0\\
0& -\mu(a,\rho)
\end{array}\right) ,\; a\in\L_f
\right),\\
 \H^{n/d}(\yy) _a^b &={\bf i}C_* \frac{{\sqrt{\rho_{\ell(a)}} \sqrt{\rho_{\ell(b)}}}}{\la \lb}\Big[
 J
 \chi^+{(a,b)}
+\left(\begin{array}{ll}
1& 0\\
0& -1
\end{array}
\right) \chi^-{(a,b)}
 \Big] \,.
\end{split}\end{align}
Note that   all elements of  the matrix  $\H(\yy)$ are pure imaginary, and 
\be\label{Lf+=0}
\text{ if 
 $( {\L_f}\times {\L_f})_+=\emptyset$, then  $-{\bf i}\H(\yy)$ is real symmetric},
 \ee
 in which case   all 
eigenvalues of  $\H(\yy)$  are pure imaginary. In Appendix~B we show that if $d\ge2$,
then, in general, the set 
$( {\L_f}\times {\L_f})_+ $ is not empty and  the matrix $\H(\yy)$ may have hyperbolic eigenvalues. 
\smallskip

\begin{example}\label{Ex41}
In view of Example \ref{Ex39}, if $d=1$ then the operator $\H^{n/d}$ vanishes. We see immediately
that in this case $\H^{d}$ is a diagonal operator with simple spectrum. 
\end{example}

Let us introduce in $\L_f$ the relation $\sim$, where 
\be\label{class}
a\sim b\;\; \text{if and only if }\;\; a=b\;\;\text{or}\;\; (a,b)\in (\L_f\times \L_f)_+\cup  (\L_f\times \L_f)_-\,.
\ee
It is easy to see that this is an equivalence relation. By Lemma~\ref{L++-}
\be\label{.2}
a\sim b,\ a\ne b \;\Rightarrow \; |a| \ne |b|\,\,.
\ee
The equivalence $\sim\,$, as well as the sets $ (\L_f\times \L_f)_\pm$, depends only on the lattice $\Z^d$
and the set $\A$, not on the eigenvalues $\la$ and the vector $\rho$. It
 is trivial if $d=1$ or  $|\A|=1$ (see Example~\ref{Ex39})) and, in general,  is non-trivial otherwise. If  $d\ge2$ and $|\A|\ge2$
 it is rather complicated.

The equivalence relation divides   $\L_f$ into equivalence classes,
$
\L_f = \L_f^1\cup\dots \cup \L_f^M\,.
$
The set $\L_f$ is a union of the punched spheres $\Sigma_a=\{b\in\Z^d\mid |b|=|a|, b\ne a\}$, $a\in\A$, 
and by \eqref{.2}  each equivalence class $\L_f^j$ intersects every punched sphere $\Sigma_a$ 
 at at most one point.

Let us order the sets $\L_f^j$ in such a way that for a suitable $0\le M_0\le M$ we have

-- $\L_f^j = \{b_j\}$ (for a suitable point $b_j\in\Z^d)$ if $j\le M_0$;

-- $|\L_j^f|=n_j\ge2$ if $j>M_0$.\\

Accordingly the complex  space $Y_{\L_f}$ (see \eqref{YY}) decomposes as
\be\label{deco}
Y_{\L_f}= Y^{f1}\oplus\dots \oplus Y^{fM},\quad Y^{fj}= \,\text{span}\, \{\zeta_s, 
s\in\L_f^j\}\,.
\ee
Since each $\zeta_s, s\in\L_f$, is a 2-vector, then 
$$
\dim Y^{f j}=2|\L_j^j|:= 2n_j\,,\qquad \dim Y_{\L_f}=2|\L_f| = 2\sum_{j=1}^M n_j
:= 2\mathbf N\,.
$$
So dim$\,Y^{f j}=2$ for $j\le M_0$ and  dim$\,Y^{f j}\ge4$ for $j> M_0$.
In view of \eqref{.2},
\be\label{.0}
|\L^j_f  | = n_j \le |\A|\qquad \forall\, j\,.
\ee

We readily see from the formula for the matrix $\H(\yy)={\bf i}J K(\yy)$ that the spaces $ Y^{fj}$ are 
invariant for the operator  $\H(\yy)$. So
\be\label{decomp}
\H(\yy)= 
\H^1(\yy)\oplus\dots \oplus\H^M(\yy)
\,,\qquad 
\H^j = \H^{j\,d} + \H^{j\,n/d},
\ee
where $\H^j$ operates in the space $Y^{fj}$, so this is a block of the matrix $\H(\yy)$.
The operators  $ \H^{j\,d}$ and $ \H^{j\,n/d}$ are given by the formulas \eqref{diag}
with  $a,b\in\L_f^j$. The Hamiltonian operator $\H^j(\yy)$ polynomially depends on $\sqrt\yy$, so
its eigenvalues form an algebraic function of $\sqrt\yy$. Since the spectrum of $\H^j(\yy)$ is an even set,
then we can write branches of  this algebraic function as 
$\{ \pm{\bf i}\Lambda_1^j(\yy),  \dots,  \pm{\bf i}\Lambda_{n_j}^j(\yy)\}$ (the factor ${\bf i}$ is convenient for further purposes). 
The eigenvalues of $\H(\rho)$ are given by another algebraic function and we write  its branches  as 
$\{\pm{\bf i}\Am_m(\yy), 1\le m\le \mathbf N=|\L_f|\}$. Accordingly, 
\be\label{spectrum}
\{\pm\Am_1(\yy),\dots,\pm \Am_{\mathbf N}(\yy)
 \} = \cup_{j\le M} \{\pm\Am_k^j(\yy),  k\le n_j\}\,,
 \ee
 and $\Am_j=\Am^j_1$ for $j\le M_0$. 
 
 The functions $\Am_k$ and $\Am^j_k$ are defined up to multiplication by $\pm1$.\footnote{More precisely, if $\Am_k$
 is not real, then well defined is the quadruple $\{\pm\Am_k, \pm\bar\Am_k\}$; see below Section~\ref{s_bd}.
 }
  But if $j\le M_0$, then 
  $\L^j_f = \{b_j\}$  and  $\H^j = \H^{jd}$,  so the spectrum of this operator is $\{\pm{\bf i}\mu(b_j,\rho)\}$, 
 where $\mu(b_j,\rho)$ is a well defined analytic function of $\rho$, given by the explicit formula \eqref{.1}. 
  In this case we specify the choice of
 $\Am^j_1$:
 \be\label{normaliz}
 \text{ if  $\ \L^j_f=\{b_j\}$,  we choose $\Am^j_1(\rho) = \mu(b_j,\rho)$. }
 \ee
 So for $j\le M_0$, 
 $ \Am_j(\yy) = \mu(b_j,\rho)$  is a polynomial of $\sqrt\yy$, which depends only on $|b_j|$ and $\yy$.

 Since the norm of the operator $K(\yy)$ satisfies \eqref{esti1}, then 
 \be
\label{N1} 
 | \Lambda^j_r(\yy)|\le C_2 \quad \forall\,\yy, \ \forall\,r\,,\; \forall\,j\,.
\ee

 \begin{example}\label{n=1}
 In view of \eqref{.0}, 
 if  
 $\A=\{a_*\}$,  then all sets $|\L^j_f|$ are one-point. So $M_0=M=\mathbf N$ and 
 $$
 \{\pm\Am_1(\yy),\dots,\pm \Am_{\mathbf N}(\yy)\} = 
 \{ \pm\mu(a,\rho)\mid a\in\Z^d, |a| = |a_*| , a\ne a_*\}.
 $$
 In  this case the spectrum of the Hamiltonian operator $\H(\rho)$ is pure imaginary and multiple. It analytically 
 depends on $\rho$. 
 \end{example}

 %Let us numerate the points of $\A$ as $\{a_1,\dots, a_n\}$ and accordingly  write points of $\D$ as $\rho=(\rho_1, \dots, \rho_n)$. 
   Let $1\le j_*\le n$ and 
 $\D^{j_*}_0$ be the set 
\be\label{DD}
  \D^{j_*}_0 = \{\rho=(\rho_1,\dots, \rho_n)\mid c_* \le \rho_l\le c_0  \;\;\text{if}\;\;
  l\ne j_* \;\;\text{and}\;\;  1-c_0  \le \rho_{j_*}\le 1\}\,,
  \ee
where $0<c_*\le\tfrac12 
c_0<1/4$. Its measure satisfies
$$
\meas \D_0^{j_*} \ge \tfrac12 c_0^n\,.
$$
This is a subset of $\D=[c_*,1]^n$ which  lies in the (Const$\,c_0$)-vicinity of the point 
 $\yy_*=(0,\dots,1,\dots,0)$ in $[0,1]^n$, where 1 stands on the $j_*$-th place. Since 
 $\tkn(\yy_*)=0$, then $K(\yy_*) = K^d(\yy_*)$. 
 Consider any equivalence class $\L_f^j$ and  enumerate its 
elements   as $b^j_1,\dots,b^j_{n_j}$ $(n_j\le n)$. 
 For $\yy=\yy_*$ the  matrix $\H^j(\yy_*)$ is diagonal with the eigenvalues 
 $\pm{\bf i}\mu(b^j_r,\yy_*), 1\le r\le n_j$. This suggests that for $c_0$ sufficiently small we may uniquely 
  numerate the 
 eigenvalues $\{\pm{\bf i}\Lambda^j_r(\yy)\} \ (\yy\in \D^{j_*}_0)$ of the matrix $\H^j(\yy)$ in such a way that 
 $\Lambda^j_r(\yy)$ is close to  $\mu(b^j_r,\yy_*)$. Below we justify this possibility.

 Take any $b\in\L_f$ and denote $\ell(b)=a_b\in\A$. 
If $a_b = a_{j_*}$, then 
\be\label{mu1}
\mu(b, \yy_*) = C_*(\frac32  \lambda_{a_{j_*}}^{-2}-\lambda_{a_{j_*}}^{-2})= \tfrac12C_*  \lambda_{a_{j_*}}^{-2}\,.
\ee
If $a_b \ne a_{j_*}$, then
\be\label{mu2}
\mu(b, \yy_*) = -C_* \lambda_{a({b})}^{-1} \lambda_{a_{j_*}}^{-1}.  %
\ee
If $m\in[1,2]$ is different from $4/3$ and $ 5/3$, then   it is easy  to see that $2\la\ne \pm \lambda_{a'}$ 
 for any $a, a'\in\A$. By \eqref{modif} this 
  implies that for $m\in[1,2]\setminus\Cc$ and for $ b, b'\in \L_f$ such that 
 $|b|\ne |b'|$ we have 
$$
|\mu(b, \yy_*)  | \ge 2c^{\#}(m)>0\,,\quad
|\mu(b, \yy_*) \pm \mu(b', \yy_*) | \ge 2c^{\#}(m)\,,\quad
%\forall\, b, b'\in \L_f,\; |b|\ne |b'|\,,
$$
and
\be\label{N110}
|\mu(b,\yy)|\ge c^{\#}(m)>0\,,\qquad |\mu(b,\yy) \pm \mu(b',\yy)  |\ge c^{\#}(m)\;\;\;\text{for}\;\;
\yy\in \D_0^{j_*}\,,
\ee
if  $c_0$ is small. In particular, for each $j$
the spectrum $\pm{\bf i}\mu(b^j_r,\yy_*), 1\le r\le n_j$ of the matrix $\H^j(\yy_*)$
is simple.

\begin{lemma}\label{laK} 
If $c_0 \in(0,1/2)$ is sufficiently small,\footnote{Its smallness only depends on  $\A, m$ and $g(\cdot)$.}
 then there exists $c^o=c^o(m)>0$ %(independent from $j_*$ and $\nu$)
such that for each $r$ and $j$, 
 $\Lambda^j_r(\yy)$ is a real analytic function of $\yy\in  \D^{j_*}_0$, satisfying 
\begin{equation}\label{N111}
\begin{split}
|\Lambda^j_r(\rho)  - \mu(b^j_r,\yy)|
% \frac{1}2C_*  \lambda_{a_{j_*}}^{-2}|
 \le C\sqrt{c_0}\qquad \forall\, \rho\in D^{j_*}_0\,,
 %\;\;\text{or} \;\;  |\Lambda^j_r(\rho)  +C_* \lambda_{a({\sigma_r})}^{-1} \lambda_{a_{j_*}}^{-1}  |\le Cc_*
%\\ 
%\text{for each $j$ and $r$ if } \; \rho\in D^{j_*}_0\,;
\end{split}
\end{equation}
and
\begin{equation}\label{N11}
|\Lambda^j_r(\rho)|\ge c^o(m)>0
\;\;\text{and}\;\; |\Lambda^j_r(\rho)\pm\Lambda^j_l(\rho)
|\ge c^o(m)\;\; \forall \, r\ne l, \forall\, j, \forall\, \rho\in D^{j_*}_0\,,
\end{equation}
\be\label{aaa}
 |  \Lambda^{j_1}_{r_1}(\yy) + \Lambda^{j_2}_{r_2}(\yy)  | \ge  c^{0}(m)\quad %\text{for}\quad
\forall \, j_1, j_2, r_1, r_2 \quad \text{and}\;\;
\rho\in D^{j_*}_0\,.
\ee
In particular, 
\begin{equation}
\label{N3}
\Lambda^j_r \not\equiv 0\quad \forall r;\quad \Lambda^j_r \not\equiv \pm\Lambda^j_l \quad
\forall\, r\ne l\,.
\end{equation}
\end{lemma}

The estimate \eqref{N111} assumes that for $\yy\in\D_0^{j_*}$ we fix the sign of the function $\Lambda_r^j$ by
the following agreement:
\be\label{agreem}
\Lambda_j^r(\yy) \in\R\;\;\text{and \ \ sign}\, \Lambda_r^j(\yy) =  \ \text{sign}\,\mu(b_r^j,\yy)\;\;
\forall \yy\in \D_0^{j_*}\,,\; \forall 1\le j_*\le n\,,\; \forall\, r, j\,,
\ee
see \eqref{mu1}, \eqref{mu2}. 

Below we fix any $c_0 =c_0(\A, m, g(\cdot))\in(0,1/2)$ such that the lemma's assertion holds, but the parameter 
$c_*\in(0, \tfrac12 c_0]$ will vary at the last stage of our proof, in Section~\ref{s_10.2}.

\begin{proof}
Since the 
 spectrum of $\H^j(\rho_*)$ is 
simple and the matrix $\H^j(\rho)$ and the numbers  $\mu(b^j_r,\yy)$ are polynomials of $\sqrt\yy$, then the basic perturbation theory 
implies that the functions $\Lambda^j_r(\yy)$ are real analytic in  $\sqrt\yy$ in the vicinity of
$\rho_*$ and we have
$$
|\mu(b^j_r,\yy_*) - \mu(b^j_r,\yy)|\le C\sqrt{c_0}\,,\quad
|\Lambda^j_r(\yy_*) - \Lambda^j_r(\yy)| \le C\sqrt{c_0}\,.
$$
So \eqref{N111} holds. It is also clear that the functions $\Lambda^j_r(\yy)$ are analytic in 
$\yy\in\D_0^{j_*}$. Relations   \eqref{N111} and \eqref{N110} (and the fact that $\mu(b,\yy)$ depends only on
$|b|$ and $\yy$) imply \eqref{N11} and \eqref{aaa}
 if $c_0>0$ is sufficiently small. 
\end{proof}

\begin{remark}\label{r_m}
The differences $|2\la - \lb|$ can be estimated from below uniformly in $a,b$ in terms of the distance from $m\in[1,2]$
to the points $4/3$ and $5/3$. So the  constants $c^{\#}$ and $c^o$ depend only on this distance, and they  can be chosen independent from $m$ if the latter   belongs to the smaller segment $[1, 5/4]$. 
\end{remark}

Contrary to \eqref{aaa}, in general a difference of two eigenvalues $\Lambda^{j_1}_{r_1}- \Lambda^{j_2}_{r_2}$ may vanish identically.
Indeed, if $j,k\le M_0$, then   $\L^k_f$ and $\L^j_f$ are 
one-point sets, $\L^k_f=\{b_k\}$ and $\L^j_f=\{b_j\}$, and $\Am^j_1=\mu(b_j,\cdot)$, $\Am^k_1=\mu(b_k,\cdot)$. 
 So if $|b_j|=|b_k|$, then $\Lambda^j_1 \equiv  \Lambda^k_1$ due to  \eqref{munew}. In particular, in 
  view of Example~\ref{n=1}, if $n=1$ then each $\L^j_f$ is a  one-point set, corresponding to some point
$b_j$ of the same length. In this case all functions $\Lambda_k(\rho)$ 
coincide identically. But if $j\le M_0 <k$, or if $\max{j,k}>M_0$ and the set $\A$ is\sa\ (recall that everywhere in this section 
it is assumed to be admissible), then $\Lambda^{j_1}_{r_1}- \Lambda^{j_2}_{r_2}\not\equiv0$.
This is the assertion of the  non-degeneracy lemma below, proved  in Section~\ref{s_2d}. 

\begin{lemma}\label{l_nond}
Consider any two spaces $Y^{f\,r_1}$ and $Y^{f\,r_2}$ such that $r_1\le r_2$ and 
$r_2>M_0$. Then 
\be\label{single}
 \Lambda_j^{r_1} \not\equiv \pm\Lambda_k^{r_2} \qquad \forall\, (r_1,j)\ne (r_2,k)\,,
\ee
provided  that either $r_1\le M_0$, or the set $\A$ is strongly admissible.
\end{lemma}

We recall that for $d\le2$ all admissible sets are\sa. For $d\ge3$ non-\!\!\sa\ sets exist. 
In Appendix~B we give an example \eqref{AAA} of such a set  for $d=3$ 
and show that for it the relation \eqref{single} does not hold.

\subsection{Removing singular values of the parameter $\rho$}\label{rho_sing}
We recall that the Hamiltonian operator $\H(\yy)$ equals ${\bf i}JK(\yy)$; so $\{\Lambda^j_l(\yy)\}$ are the 
eigenvalues of the real matrix $JK(\yy)$. Accordingly, the numbers $\{\Lambda^j_l(\yy), 1\le l\le n_j\}$,
are  eigenvalues of the real matrix $\tfrac1{{\bf i}} \H^j(\yy)=:L^j(\yy)$. 
Due to Lemma~\ref{laK} we know that for each $j$  the eigenvalues $\{\pm \Am_k^j(\yy)$,  $k\le n_j\}$,
  do not vanish identically
in $\yy$ and do not identically coincide. Now our goal is to quantify these statements by removing certain 
singular values 
of the parameter $\yy$. To do this let us first  denote   
$P^j(\yy)=(\prod_l\Lambda^j_l(\yy))^2 = \pm \det L^j(\yy)$  and consider the determinant 
$$
 P(\yy)=\prod _jP^j(\yy) =\pm \det JK(\yy)\,. %\qquad D(\yy)= \prod _j D^j(\yy)\,.
 $$
 
 Recall that for an $R\times R$-matrix with eigenvalues $\ka_1,\dots,\ka_R$ (counted with their multiplicities) 
 the discriminant of the determinant of this matrix 
 equals the product $\prod_{i\ne j}(\ka_1-\ka_j)$. This is a polynomial of the matrix' elements. 
 
 Next we define a ``poly-discriminant" $D(\rho)$, which is another   polynomial 
 of the matrix elements of  $JK(\rho)$. Its definition is motivated by Lemma~\ref{l_nond}, and it is 
 different for the admissible and \sa\ sets $\A$.  Namely, if $\A$ is \sa, then 
 \smallskip

 -- for $r=1,\dots, M_0$  define  $D^r(\rho)$ as the discriminant of the determinant of the matrix
 $ L^r(\yy)\oplus L^{M_0+1}(\yy)\oplus\dots\oplus  L^M(\yy)$; %equal to 
% $$D^r(\yy)= D^0(\rho)\cdot 
%\prod_{j = M_0+1}^M 
 %\prod_{l} (\pm\Lambda^r_1(\yy)\pm\Lambda^j_l(\yy))^2\,$$
 %(this also is a polynomial of the matrix coefficients);  
 
 -- set $D(\rho) =% D^0(\rho)\cdot 
 D^1(\rho)\cdot\dots\cdot D^{M_0}(\rho)$.
 
 \noindent 
 This is a polynomial in the matrix coefficients of $JK(\yy)$, so a polynomial of  $\sqrt\rho$.
It  vanishes if and only if 
 $\Am^r_m(\rho)$ equals $\pm \Am^l_k(\rho)$ for some $r, l, m$ and $k$, where either $r,l \ge M_0+1$ and
 $m\ne k$ if $r=l$, or $r\le M_0$ and $m=1$.
  \medskip
 
  If $\A$ is admissible, then  we:
 
 %-- for $r\ge M_0+1$ define $D^r(\rho)$ as  the discriminant of the determinant of the matrix
 %$ L^r(\yy)$; % equal to $ \prod_{i<l}(\pm\Lambda^r_i(\yy)\pm\Lambda^r_l(\yy))^2$;
 
 -- for $l\le M_0, r\ge M_0+1$ define $D^{l,r}(\rho)$ as  the discriminant of the determinant of the matrix
 $ L^l(\yy)\oplus  L^r(\yy)$; % equal to  $D^r(\rho)\cdot  \prod_{l}(\pm\Lambda^l_1(\yy)\pm\Lambda^r_l(\yy))^2$;
 
 -- set $D(\rho) = %\prod_{r\ge M_0+1} D^r(\rho) 
 \prod_{l\le M_0,   r\ge M_0+1} D^{l,r}(\rho) $.

 \noindent 
 This is a polynomial in the matrix coefficients of $JK(\yy)$, so a polynomial in $\sqrt\rho$. It 
 vanishes if and only if 
 $\Am^r_1(\rho)$ equals $\pm \Am^l_k(\rho)$ for some $r\le M_0$, some $l\ge M_0+1$ and some $k$,
 or if $\Am^l_k(\rho)$ equals $\pm\Am^l_m(\rho)$ for some $l\ge M_0+1$ and some $k\ne m$.
 \medskip
 
 Finally, in the both cases we set
 $$
 M(\rho) =  \prod_{b\in \L_f }  \mu(b,\yy) 
  \prod_{\substack{b,b'\in \L_f \\ |b| \ne |b'| }} \big( \mu(b,\yy) - \mu(b',\yy)\big).
 $$
 This also is a polynomial in $\sqrt\rho$ which does not vanish identically due to \eqref{N110}.

 The set
 $$
 X=\{\yy\mid  P(\yy)\,D(\yy)\, M(\yy)
 =0\}
 $$
 is an algebraic variety, if written in the variable $\sqrt\rho$ (analytically diffeomorphic to the variable 
 $\rho\in[c_*,1]^\A$), and is  non-trivial by Lemma~\ref{laK}. 
 The open set $\D\setminus X$ is dense in $\D$ 
 and is formed by finitely many connected components. Denote them $Q_1,\dots, Q_L$.
 For any component $Q_{l}$ its boundary is a stratified  analytic manifold with finitely many smooth analytic components
 of dimension $<n$, see  \cite{ KrP}. The eigenvalues $\Lambda_j(\yy)$ and the corresponding eigenvectors are locally 
 analytic functions on the domains $Q_l$, but since some of these domains 
   may be not simply connected, then the functions may have 
 non-trivial monodromy, which would be inconvenient for us. But since each $Q_l$ is a domain with a regular boundary, then 
 by removing from it  finitely many  smooth closed hyper-surfaces we cut $Q_l$
  to a finite system of simply connected domains $Q^1_l, \dots, Q_l^{\hat n_l}$ 
 such that their union has the same measure as 
  $Q_l$ and  each domain  $Q_l^\mu$ lies on one side of its boundary.\footnote{For
 example, if $n=2$ and $\DD$ 
  is the annulus $A=\{1<\yy_1^2+\yy_2^2<2\}$, then we remove
from $A$ not the interval $\{\yy_2=0, 1<\yy_1<2\}=:J$ (this would lead to a simply connected domain which lies on both 
parts of the boundary $J$), but two intervals, $J$ and $-J$. 
}
We may realise these cuts (i.e. the
 hyper-surfaces)  as the zero-sets of certain 
 polynomial functions of $\yy$. Denote by $R_1(\yy)$ the 
product of the polynomials, corresponding to the cuts made,  and remove from $\DD\setminus X$ 
the zero-set of $R_1$. This zero-set contains all the cuts we made  (it may be bigger than the union of the cuts), 
 and still has zero measure. Again, $(\tilde Q_l\setminus X)\setminus \{\text{zero-set of} \ R_1\}$ 
  is a finite union of 
domains, where each one lies in some domain $Q^r_l$. 

Intersections of these new domains with the sets $\D^{j_*}_0$ (see \eqref{DD}) will be important for us by virtue of Lemma~\ref{laK},
and any fixed set $\D^{j_*}_0$, say $\D^{1}_0$, will be sufficient for out analysis. To agree the domains with $\D^1_0$ we note that 
the boundary of $\D^1_0$ in $\D$ is the zero-set of the polynomial
$$
R_2(\yy) = (\yy_1-(1-c_0))(\yy_2-c_0)\dots(\yy_n-c_0)\,,
$$
and  modify the set $X$ above to the set $\tilde X$, 
$$
\tilde X = \{\yy\in\D\mid \Rc(\yy)=0\}\,,\qquad \Rc (\yy) = P(\yy) D(\yy) M(\yy) R_1(\yy) R_2(\yy)\,. 
$$
As before, $\D\setminus \tilde X$ is a finite union of open domains with regular boundary. We still denote them $Q_l$:
   \be\label{components}
 \D\setminus \tilde X = Q_1\cup\dots\cup Q_{\mathbb J}\,,\qquad  \bJ<\infty\,.
 \ee
 A domain $Q_j$ in \eqref{components} may be non simply connected, but since each $Q_j$
 belongs to some domain $Q^r_l$, then the eigenvalues $\La(\yy)$ and the corresponding
 eigenvectors define in these domains single-valued analytic functions. Since every domain 
 $Q_l$ lies either in $\D^1_0$ or in its complement, we may enumerate the domains $Q_l$ in
 such a way that 
 \be\label{newJ}
\D_0^1 \setminus \tilde X = Q_1\cup\dots\cup Q_{\bJ_1}\,,\quad 1\le \bJ_1\le \bJ\,.
\ee
The domains $Q_l$  with  $l\le \bJ_1$ will play a special role in our argument. 
\smallskip
 
 Let us take $c_1 = \tfrac12 c_*$ and consider the complex vicinity $\Da$ of $\D$, 
 \be\label{Dset}
\Da=\{ \yy\in \C^\A\mid |\Im\yy_j|<c_1, \, c_*-c_1<\Re\yy_j<1+c_1\ \forall j\in\A\}
 \,.
\ee
 We naturally extend $\tilde X$ to a complex-analytic subset $\tilde X^c$ of $\D_{c_1}$
 (so  $\tilde X=\tilde X^c\cap \D$),  consider the set $\D_{c_1}\setminus \tilde X^c$, and for 
 any $\delta>0$ consider its open sub-domain $\Da(\delta)$, 
$$
\Da(\delta) = \{ \yy\in\Da\mid |\Rc(\yy)| >   \delta
%|P(\yy)|>\delta, \; |D(\yy)|>\delta , \; |M(\yy)|>\delta
\}\subset \Da\setminus \tilde X^c\,.
$$
Since the factors, forming $\Rc$, are polynomials with bounded coefficients, then they are bounded in $\D_{c_1}$:
 \be\label{det}
 \|P\|_{C^1(\D_{c_1})} \le C_1\,, \dots,   \|R_2\|_{C^1(\D_{c_1})} \le C_1\,.
 \ee
So in the domain  $ \Da(\delta)$ the norms of the factors $P,\dots, R_2$, making $\Rc$, are bounded from 
below by $C_2\delta$, and similar estimates hold for the factors, making $P$, $D$ and $M$. 
Therefore, by the Kramer rule
\be\label{Kram}
\|(JK)^{-1}(\rho)\| \le C_1 \delta^{-1}\qquad \forall \rho\in \Da(\delta)\,.
\ee
Similar  for   $\yy \in \Da(\delta)$ we have 
\be\label{K4}
  |\Am^j_{k}(\yy)  |\ge C^{-1} \delta \qquad \forall j,k\,,
\ee
\be\label{K44}
 |\mu(b,\rho) |
  \ge C^{-1} \delta\,, \quad
 |\mu(b,\rho) - \mu(b', \rho)|
  \ge C^{-1} \delta \quad \text{if $\ b,b'\in\L_f\ $ and $|b|\ne |b'|$\,,
  }
\ee
and 
\be\label{K04}
   |\Am^j_{k_1}(\yy) \pm\Am^r_{k_2} (\yy)|\ge C^{-1} \delta \quad \text{where}\;\; (j,k_1)\ne (r,k_2)\,.
\ee
In \eqref{K04} if the set $\A$ is\sa, then 
 the index $j$ is any and $r\ge M_0+1$, while if $\A$ is admissible, then  either $j\le M_0$ (and so 
 $k_1=1$)  and  $r\ge M_0+1$,
or $j=r\ge M_0+1$. The functions $\Am^j_{k}(\yy)$ are algebraic  functions on the complex
domain $\Da(\delta)$, but their restrictions to the real parts of these domains
split to branches which are well defined analytic functions. 

We have
\be\label{K1}
\meas (\D\setminus \Da(\delta))\le C \delta^{\beta_{4}},
\ee
for some positive $C$ and $\beta_{4}$  --   this follows easily from  Lemma~\ref{lTransv1} and Fubini since
$\Rc$ is a polynomial in $\sqrt\r$ (also see Lemma~D.1 in \cite{EGK1}). 
Denote $c_2=c_1/2$, define set $\Db$  as
in \eqref{Dset} but replacing there $c_1$ with  $c_2$, and denote $\Db(\delta) =\Da(\delta) \cap \Db$. 
Obviously,
\be\label{two_sets}
\text{the set $\D_{c_2}(2\delta)$ lies in $\Da(\delta)$ with its $C^{-1}\delta$-vicinity\,.
}
\ee

Consider the eigenvalues $\pm{\bf i}\Am_k(\yy)$. They analytically depend on $\yy\in\Da(\delta)$, where $|\Am_k|\le C_2$ for
each $k\le \mathbf N$ by \eqref{N1}. In view of \eqref{two_sets}, 
\be\label{K2}
|\frac{\p^l}{\p\yy^l} \Am_k (\yy)| \le C_l\delta^{-l}\qquad \forall\,\yy\in\Db(2\delta)\,,\ l\ge0\,,\  k\le \mathbf N\,,
\ee
by the Cauchy estimate.

\subsection{Block-diagonalising and the end of the proof of Theorem \ref{NFT}} \label{s_bd}

We shall block-diagonalise the operator ${\mathbf i}JK(\yy)$ for $\yy \in \Da(\delta)$. By \eqref{decomp} this operator is a direct sum
of operators, each of which has a simple spectrum with eigenvalues that are separated by $\ge C^{-1}\delta$. Let us denote one of these blocks   by ${\mathbf i}JK_1(\yy)$. Let its dimension be $2N$ and let
$I(\xi,\eta)=(\bar\eta,\bar\xi)$.
Notice that since ${\mathbf i}JK_1(\yy)$  is ``conjugate-real'' we have
$${\mathbf i}JK_1(\yy)I( z)=I({{\mathbf i}JK_1(\yy)z}).$$

Fix now a $\r_0 \in \Da(\delta)$.
Then, by \eqref{two_sets} with $\delta$ replaced by $\delta/2$, for $\ab{\r-\r_0}\le C^{-1}\delta^{4N}$ 
the operator  ${\mathbf i}JK_1(\yy)$ has a single spectrum. Consider a  (complex) matrix
$$
U(\r)=\big(z_1(\yy), \dots, z_{2N}(\yy)\big),
$$
whose  column vectors 
$
\aa{z_j(\r)}=1
$
are eigenvectors of  ${\mathbf i}JK_1(\yy)$. It  diagonalises ${\mathbf i}JK_1$:
\be\label{diago}
U(\yy)^{-1}\big( {\mathbf i}JK_1(\yy) \big)U(\yy)={\bf i}\,\text{diag}\,\{\pm\Am_1(\yy),\dots,\pm\Am_{{N}}(\yy)\}. 
\ee
The operator $U$ is  smooth in $\r$ with estimates
\be\label{K8}
\sup_{\yy} 
%\max_{j=0,1}
 \big( \|\p^j_\yy U(\yy)\| + \|\p^j_\yy U(\yy)^{-1} \|) \le C_j\delta^{-\beta(j) }
 \qquad \forall\, j\ge0\,,
 \ee
 and
 \be\label{K8bis}
\inf_{\yy} 
\ab{\det(U(\yy))}\ge \frac1{C_0}\delta^{\beta(0) }\,,
 \ee
for some $0<\beta(0)\le\beta(1)\le\dots$. See Lemma A.6 in \cite{E98} and Lemma~C.1 in \cite{EGK1}. 

Since the spectrum is simple, then   the pairing $\langle {\mathbf i}Jz_{k}(\r), z_{l}(\r)\rangle =
{}^t\!z_{l}(\r)\big({\mathbf i}J\big)z_{k}(\r)$ 
is zero unless the eigenvalues of $z_{k}(\r)$ and $z_{l}(\r)$ are equal but of opposite sign. We therefore enumerate the eigenvectors so that $z_{2j-1}(\r)$ and $z_{2j}(\r)$ correspond to eigenvalues of opposite sign. If now $\pi_j(\yy)=\langle {\mathbf i}Jz_{2j-1}(\yy), z_{2j}(\yy)\rangle $, then, for each $j$,
$$
 \frac1{C_0}\delta^{\beta(0) }\le \ab{\det(U)}=\sqrt{\ab{\det({}^t\!U{\mathbf i}JU)}}=\prod_l \ab{\pi_l}\le \ab{\pi_j}\le 1\,,
 $$
%It is clear that $|\pi_j(\yy)|\le1$.  Since determinant of ${}^tU(\r)L(\r)U(\r)$ equals one it follows that $(\pi_1(\yy)\cdots\pi_N(\yy))^2=1$ and therefore each $|\pi_k(\yy)|=1$. 
since the matrix elements of ${}^t\!U{\mathbf i}JU$ are $\langle {\mathbf i}Jz_{k}(\r), z_{l}(\r)\rangle $.

Replacing each eigenvector
$z_{2j}$ by $\frac{1}{\pi_{2j}}z_{2j}$, we can assume without restriction that $U$ verifies
\be\label{eigenvectorsbis}
 \frac1{C_0}\delta^{\beta(0) }\le \aa{z_j(\r)}\le C_0\delta^{-\beta(0) }\ee
 and \eqref{diago}-\eqref{K8bis} (for some choice of constants) and, moreover,
\be\label{matrixU}
{}^t\!U\big({\mathbf i}J\big)U=J.\ee

Suppose now that some $\Lambda_j$,   $\Lambda_1$ say, is real. Then $z_2$ and $I(z_1)$ are parallel, so $z_2={\bf i}\al I(z_{1})$ for some complex number 
$\alpha\in\C^*$ satisfying the bound \eqref{eigenvectorsbis} (for some choice of constants). Since 
$\langle {\mathbf i}Jz_{1}, z_{2}\rangle =1$, we have that $\al=  \langle Jz_{1}, I(z_{1})\rangle^{-1} $ is real, and, by  eventually interchanging
$z_1$ and $z_2$, we can assume that $\alpha=\beta^2>0$. Replacing now $z_1,z_2$ by $\beta z_1, \frac1{\beta}z_{2}$
we can assume without restriction that $U$ verifies
\eqref{diago}-\eqref{matrixU} (for some choice of constants), and $z_2={\bf i} I({z_{1}})$.

Suppose then that some $\Lambda_j$,   $\Lambda_1$ again say, is purely imaginary. Then $z_1$ and $I(z_1)$ are parallel, so
$z_1=\alpha I( z_1)$ for some unit $\alpha$. 
Similarly, $z_2=\beta I( z_2)$ for some unit $\beta$. 
Since
$\langle {\mathbf i}Jz_{1}, z_{2}\rangle =1$, we have that $1= \alpha\beta \langle {\mathbf i}JI(z_{1}), I(z_{2})\rangle=\alpha\beta$.
Let now $\al=\ga^2$,  and by replacing
$z_1,z_2$ by $\bar \ga z_1, \frac1{\bar\ga}z_{2}$
we can assume without restriction that $U$ verifies
\eqref{diago}-\eqref{matrixU} (for some choice of constants), and $z_1= I({z_{1}})$ and $z_2= I({z_{2}})$.

%implies that either the real part of $z_1$ or the imaginary part of $z_1$ satisfies the bound \eqref{eigenvectorsbis} (for some choice of constants), and, by eventually replacing $z_1$ by ${\mathbf i}z_1$, we can assume that it is the real part. Replacing now $z_1$ by $\Re z_1$we can assume that $z_1$ is real. Similarly we can assume that $z_2$ is real.

Suppose finally that some $\Lambda_j$, $\Lambda_1$ say,  is neither real nor purely imaginary. Then $-{\mathbf i}\ov{\Lambda_1}$  also is  an eigenvalue,\footnote{An example, considered in Appendix B, shows that
quadruples of eigenvalues $\{\pm {\bf i}\Lambda, \pm {\bf i} \bar\Lambda\}$ indeed may occur in the spectra
of operators ${\bf i}JK$.}
 and, hence, equals to $\pm {\mathbf i}\Lambda_2$ say.  Let us assume it is ${\mathbf i}\Lambda_2$, the other case being similar. Then 
$z_3=\al I({z_{1}})$ for some unit $\al$, and $z_2=\beta I({z_{4}})$ for some  $\beta\in \C^*$, both satisfying the bound \eqref{eigenvectorsbis} (for some choice of constants). Since $\langle {\mathbf i}Jz_{1}, z_{2}\rangle =
\langle {\mathbf i}Jz_{3}, z_{4}\rangle =1$, $\al\beta$ must be $=1$.  
%Replacing $z_3$ and $z_4$ by $\bar\al z_3$ and $\bar\beta z_4$,  we can assume without restriction that $U$ verifies\eqref{diago}-\eqref{matrixU} (for some choice of constants), and $z_1=\overline{z_{3}}$ and $z_2=\overline{z_{4}}$.
Let now $\al=\ga^2$,  and by replacing
$z_1,z_3$ by $\bar \ga z_1, \bar\ga z_{3}$ and $z_2,z_4$ by $\frac1{\bar \ga} z_2, \frac1{\bar\ga}z_{4}$
we can assume without restriction that $U$ verifies
\eqref{diago}-\eqref{matrixU} (for some choice of constants), and $z_3= I({z_{1}})$ and $z_4= I({z_{2}})$.

Now we define a new matrix
$$\tilde U(\r)=\big(p_1(\yy)\ q_1(\r)\dots p_{N}(\yy)\ p_{N}(\yy)\big)$$
in the following way. If $\Lambda_1$ is real, then we take
$$p_1=-\frac{{\bf i}}{\sqrt2}(z_{1}+{\bf i}z_2),\quad q_1=-\frac{1}{\sqrt2}(z_{1}-{\bf i}z_2),$$
so that $I(p_1)=p_1$, $I(q_1)=q_1$ and $\langle {\mathbf i}Jp_{1}, q_{1}\rangle=  1$.
We do similarly for all $\Lambda_j$  real.
 If $\Lambda_1$ is purely imaginary, then we take $p_1=z_1$ and $q_1=z_2$,
and similarly for all $\Lambda_j$  purely imaginary. If $\Lambda_1$ is neither real nor purely imaginary, 
and $z_1=I({z_{3}})$ and $z_2=I({z_{4}})$, then
$$p_1=-\frac{{\bf i}}{\sqrt2}(z_{1}+{\bf i}z_3),\quad p_2=-\frac{1}{\sqrt2}(z_{1}-{\bf i}z_3)$$
and
$$q_1=-\frac{{\bf i}}{\sqrt2}(z_{2}+{\bf i}z_4),\quad q_2=-\frac{1}{\sqrt2}(z_{2}-{\bf i}z_4),$$
similarly for all $\Lambda_j$ neither real nor purely imaginary.

Then the matrix $\tilde U(\r)$ verifies
 \eqref{K8}-\eqref{matrixU} (for some choice of constants) and the mapping
 $$w\mapsto \tilde U(\r)w$$
 takes any real vector $w$ into the subspace $\{I(w)=w\}$. By doing this for each ``component'' ${\mathbf i}JK_1(\r)$ of the operator \eqref{decomp} and taking the direct sum we find a matrix $\hat U(\r)$ which 
  transforms the Hamiltonian of ${\mathbf i}JK(\r)$ to the form
\be\label{rham}
\frac12
 \sum_{j=1}^{M_0}  \mu(b_j,\rho)  \Big( p_{b_j}^2 + q_{b_j}^2\Big)+ 
 \frac12
 \sum_{j=M_0+1}^{M_{00}}  \Lambda_j(\yy)  \Big( p_{b_j}^2+ q_{b_j}^2\Big) + \frac12
  \langle \widehat  K(\yy) \zeta_h,  \zeta_h\rangle\, ,
 \ee
 where $\zeta_h$ denotes the the remaining $\{(p_{b_j},q_{b_j}):  M_{00}+1\le j\le \mathbf{N}\}$.
The Hamiltonian operator $J \widehat  K(\yy)$ is formed by the hyperbolic eigenvalues of
the operator ${\bf i} J\widetilde K(\yy)$.

Since $\Lambda_a (\yy)\xi_a\eta_a$ is transformed to $\frac12 \Lambda_a(\yy)  \Big( p_{a}^2+ q_{a}^2\Big)$ by a 
matrix $\hat U_a$, independent of $\r$,  that verifies ${}^t\tilde U_a(iJ_a)\tilde U_a=J_a=J$ (see \eqref{change}) ,
the full Hamiltonian \eqref{HNF} gets transformed to
\be\label{hhak}
\begin{split}
\langle \Omega(\yy), r\rangle  +\frac12 \sum_{a\in\L_\infty}\Lambda_a (\yy)\Big( p_{a}^2 + q_{a}^2\Big)
+\frac12
 \sum_{j=1}^{M_0}  \mu(b_j,\rho)  \Big( p_{b_j}^2 + q_{b_j}^2\Big)+ \\
+ \frac12
 \sum_{j=M_0+1}^{M_{00}}  \Lambda_j(\yy)  \Big( p_{b_j}^2+ q_{b_j}^2\Big) + \frac12
  \langle \widehat  K(\yy) \zeta_h,  \zeta_h\rangle
\end{split}
\ee
plus the error term $\tilde f(r,\theta, p_\L,q_\L; \yy)=f(r,\theta, \xi_\L,\eta_\L; \yy)$.

Note that in difference with the normal form \eqref{HNF}, the variable $\zeta_h$ belongs to a subspace of the linear space,
formed by the vectors $\{(p_a, q_a), a\in\L_f\}$, with the usual reality condition.

 We choose any subset $\F\subset\L_f$ 
of cardinality $|\F|={\bf N}- M_{00}$, and identify the space, where acts the operator $\hat K(\yy)$, with the space 
$\L_\F =\big\{ \zeta_\F =\{(p_a,q_a), a\in\F\} \big\}
$. We denote the operator $\hat K(\yy)$, re-interpreted as an operator in $\L_\F$, as $K(\yy)$. Finally, we identify the set
of nodes $\{1,\dots, M_{00}\}$ with $\L_\F\setminus\F$, and write the collection of frequencies
$
\{\mu(b_j,\yy), 1\le j\le M_{0}\} \cup \{\Lambda_j(\yy),   M_0+1\le\yy \le M_{00}\}
$
as $\{\Lb(\yy), b\in \L_\F\setminus \F\}$. 
After that the Hamiltonian \eqref{hhak} takes the form \eqref{HNFbis}, required by Theorem~\ref{NFT}. 
We denote by $\bf\hat U_\yy$ the constructed linear symplectic change of variables which transforms 
the Hamiltonian \eqref{HNF} to  \eqref{HNFbis}

 For convenience we denote 
\be\label{barc}\bar c= 1/{\beta_{4}}\quad \text{and} \quad \hat c=\beta(0)   \bar c .
\ee
With an eye on the relation \eqref{K1}, for 
  $\bb>0$ and any $\nu>0$ we denote  $\delta(\nu) = C^{\bar c} \nu^{\bar c\bb}$. Then 
\be\label{delta}
C\delta^{\beta_{4}} = \nu^{\bb}\,.
\ee
For any $\nu>0$ we set
$$
Q(c_*,\bb,\nu) = \D\cap \D_{c_1} (\delta(\nu))\,. 
$$
This is a monotone in $\nu$ 
 system of subdomains of $\D$, and 
$
 Q(c_*,\bb,\nu) 
 \nearrow( \D\setminus \tilde X)$ as  $ \nu\to0$.
In view of \eqref{K1} the measures of these domains satisfy \eqref{mesmes}. 

For $\yy\in Q(c_*,\bb,\nu) $ the operator $\tilde \Phi_\yy = \Phi_\yy \circ{\bf\hat U}_\yy$ 
transforms the Hamiltonian $\nu^{-1} h$ to \eqref{HNFbis}. Re-denoting this transformation back to 
$\Phi_\yy$, we see that the constructed objects satisfy the assertions (i)-(v) and (vii) of the theorem. 
To prove (vi) we recall (see \eqref{diago}) that the operator $U(\yy)$ (complex-)diagonalises one block of 
those, forming the operator ${\bf i}JK(\yy)$. Denote by ${\bf U}(\yy)$ the direct sum of the operators $U(\yy)$,
corresponding to all blocks of  ${\bf i}JK(\yy)$.  It diagonalises the whole operator  ${\bf i}JK(\yy)$. Accordingly, the
operator 
$
{\bf U}(\yy) \circ {\bf\hat U}^{-1}(\yy)
$
diagonalises $JK(\yy)$. Denoting it $U(\yy)$ we see that this operator satisfies the assertion (vi)

\subsection{Proof of the non-degeneracy Lemma~\ref{l_nond}}\label{s_2d}

Consider the decomposition  \eqref{decomp} of  the  Hamiltonian operator $\H(\rho)$.
To simplify notation, in this section we suspend the agreement that $|L^r_f|=1$ for $r\le M_0$,
and changing the order of 
the direct summands  achieve that the indices $r_1$ and $r_2$, involved in \eqref{single}, are $r_1=1$ and
$r_2=2$. For $r=1,2$ we  will write elements 
of the set $\L^r_f$ as $a^r_j, 1\le j\le n_r$, and vectors of the space $Y^{fr}$ as
\be\label{vectors}
\zeta=
\big(\zeta_{a^r_j}=(\xi_{a^r_j} ,  \eta_{a^r_j} ), 1\le j\le n_r\big) = 
\big( ( \xi_{a^r_1} , \eta_{a^r_1}),\dots, ( \xi_{a^r_{n_r}} , \eta_{a^r_{n_r}})\big)\,.
\ee

Using \eqref{labA} and abusing notation, we will regard the mapping 
$\ell: \L_f\to \A$ also as a mapping $\ell: \L_f\to \{1,\dots,n\}$. 
Consider the points $ \ell(a^1_1),\dots, \ell(a^1_{n_1}) $ (they are different by \eqref{.2}). 
Changing if needed the
labelling \eqref{labA}  we achieve that
\be\label{achieve}
\{ \ell(a^1_1),\dots, \ell(a^1_{n_1}) \ni 1\,. 
\ee

 We  write  the operator $\H^r$ as  $\H^r={\bf i}M^r$,  where 
$$
M^r (\rho)= JK^r(\rho) = JK^{r\,d}(\rho) +JK^{r\,n/d}(\rho) =: M^{r\,d}(\rho) +M^{r\,n/d}(\rho)\,,
$$
and  the real block-matrices $M^{r\, d}={\bf i}^{-1}\H^{r,\,d}$, $\ M^{r\, n/d}={\bf i}^{-1}\H^{r,\,n/d}$ are given
by \eqref{diag}. Then $\{\pm\Lambda^r_j(\rho)\}$ are the eigenvalues of $M^r(\rho)$, and
$$
M^{r\,d}(\rho) = \text{diag}\ \left( 
 \left(\begin{array}{ll}
%\mu^r_j(\rho) 
 \mu(a^r_j,\rho) & 0\\
0& - \mu(a^r_j,\rho)
\end{array}\right) ,\;  1\le j\le n_r
\right),
$$
where 
$
  \mu(a^r_j, \rho)
$
is given by \eqref{.1}. 

 Renumerating the eigenvalues we achieve that in \eqref{single} (with $r_1=1, r_2=2$) we have 
 $\Lambda^1_j=\Lambda^1_1$ and $\Lambda^2_k=\Lambda^2_1$.
  %As before, we enumerate the points of $\A$ as $\{a_1,\dots,a_n\}$. % and identify the space $\R^\A$ with $\R^n$. 
  As in the proof of Lemma~\ref{laK},  consider the vector $\rho_*= (1,0,\dots,0)$.  Let us abbreviate 
$$
\mu(a, \rho_*)= \mu(a)\qquad \forall\, a\,,
$$
where $\mu(a)$ depends only on $|a|$ by  \eqref{munew}.
   In view of \eqref{diag} $M^r(\rho_*)=M^{r\,d}(\rho_*)$ and thus
 $\Lambda^1_1(\rho_*)=\mu(a^1_1)$ and
$\Lambda^2_1(\rho_*)= \mu(a^2_1)$, if we numerate the elements of 
 $\L^1_f$ and $\L^2_f$ accordingly. As in the proof of  Lemma~\ref{laK}, $\mu(|a^r_1|)$ equals
 $\tfrac12 C_*\lambda^{-2}_{a^r_1}$ or $- C_* \lambda^{-1}_{\ell(a^r_1)} \lambda^{-1}_{a^r_1}$. Therefore
 the relation 
 $\mu(a^1_1) =\pm \mu(a^2_1) $
 is possible only if the sign is ``+" and $|a^1_1| = |a^2_1|$. So it remains to verify that under the lemma's assumption
 \be\label{sin}
 \Lambda^1_1(\rho)\not\equiv \Lambda^2_1(\rho)\quad\text{if}\quad |a^1_1| = |a^2_1|\,.
 \ee
  Since $ |a^1_1| = |a^2_1|$, then 
$$
\ell(a^1_1) = \ell(a^2_1)=:{a_{j_{\#}}}\in\A \;\; 
\text{ and } \Lambda^1_1(\rho_*)=\Lambda^2_1(\rho_*)=: \Lambda\,.
$$

To prove that $\Lambda^1_1(\rho)\not\equiv \Lambda^2_1(\rho)$
 we compare  variations of the two functions   around $\rho=\rho_*$. To do this it 
   is convenient to pass from $\yy$ to the  new parameter $y=(y_j)_1^n$, defined by
 $$
 y_j=\sqrt{\rho_{j}}, \quad j=1,\cdots ,n.
 $$
 Abusing notation we will sometime write $y_{a_j}$ instead of $y_j$.
 Take  any vector $x=(x_1, \dots,x_n)\in\R^n $, where $x_1=0$ and $x_j > 0$ if $j\ge2$,
 and consider  the following variation $y(\eps)$ of $y_*=(1,0,\cdots,0)$:
\be\label{yx}
 y_j(\eps) =
\begin{cases}
1& \text{ if} \ \  j=1,\\
 \eps x_j &  \text{ if} \ \  j\ge2.
        \end{cases} 
\ee
By \eqref{N11}, for  small $\eps$ the real matrix  $M^r(\eps):=
M^r(\rho(\eps))$ $(r=1,2)$ has a simple eigenvalue
$\Lambda^r_1(\eps)$, close to $\Lambda$. 
%It is an analytic function of $\eps$. 
 We will show that for a suitable choice of vector $x$ the  functions  $ \Lambda^1_1(\eps)$ and
$ \Lambda^2_1(\eps)$ are different. More specifically, that  their jets at zero  of sufficiently high order are different.

Let $r$ be 1 or 2. We denote $\Lambda(\eps) = \Lambda^r_1(\rho(\eps))$,
 $M(\eps)= M^r(\rho(\eps))$ and denote by $M^d(\eps)$ and $M^{n/d}(\eps)$ the diagonal and 
non-diagonal parts of $M(\eps)$. The matrix
 $M^{n/d}(\eps)$   is formed by $2\times2$-blocks
\be\label{M}
\Big( M^{n/d}(\eps)\Big)^{a^r_j}_{a^r_k}  = 
 C_*
 \frac{{ {y_{\ell(a^r_k)} y_{\ell(a^r_j)}}}}{ \lambda_{ a^r_k}  \lambda_{a^r_j }}  \,\left(
 \left(\begin{array}{ll}
0& 1\\
-1& 0
\end{array}\right) \chi^+( a^r_k,a^r_j )+
 \left(\begin{array}{ll}
1& 0\\
0& -1
\end{array}\right)  \chi^-( a^r_k,a^r_j )
\right)\,,
\ee
(note that if $j=k$, then the block  vanishes). 

For $\eps=0$, $M(0) = M^{rd}(0)$ is a matrix with the single eigenvalue 
$\Lambda(0) = \mu(a^r_1, \rho_*)$, corresponding
to the eigen-vector  $\zeta(0) = (1,0,\dots,0)$. For small $\eps$ they analytically 
extend to a real eigenvector $\zeta(\eps)$ 
 of $M(\eps)$ with the eigenvalue $\Lambda(\eps)$, i.e. 
$$
M(\eps)\zeta(\eps) = \Lambda(\eps) \zeta(\eps)\,,\qquad |\zeta(\eps)|\equiv 1\,.
$$

We abbreviate
$\zeta= \zeta(0), 
M=M(0)  
$
and define similar $ \dot \zeta, \ddot\zeta, \Lambda,\dot\Lambda\dots$ etc,
where  the upper dot  stands for $d/d\eps$. 
We have
\be\label{y1}
M=M^d
=\text{diag}\big(\mu(a^r_1), -\mu(a^r_1), \dots, -\mu(a^r_{n_r}) \big)\,,
\ee
\be\label{y01}
 \dot M^d=0\,.
\ee

Since $ (M(\eps) - \Lambda(\eps))\zeta(\eps)\equiv 0$, then
\be\label{y3}
(M(\eps) - \Lambda(\eps)) \dot\zeta(\eps) = -\dot M(\eps) \zeta(\eps) +\dot\Lambda(\eps) \zeta(\eps). 
\ee
Jointly with \eqref{y1} and  \eqref{y01} this relation with $\eps=0$  implies that 
\be\label{y4}
(M^d-\Lambda)\dot \zeta =-\dot M^{n/d} \zeta +\dot\Lambda \zeta.
\ee
In view of \eqref{y1}  we have $\langle(M^d-\Am)\dot\zeta, \zeta\rangle =0$. 
We derive from here and from \eqref{y4} that 
\be\label{y5}
\dot\Am=\langle \dot M^{n/d}\zeta,\zeta\rangle=0\,.
\ee
Let us denote by $\pi$ the linear projection 
$ \ 
\pi: \R^{2n_r} \to \R^{2n_r}  
$
which makes zero the first component of a vector to which it applies. Then $M^d-\Lambda$ is an isomorphism of the space 
$\pi  \R^{2n_r} $, and the vectors $\dot\zeta$ and $-\dot M \zeta+ \dot\Lambda\zeta = \dot M^{n/d}\zeta$ belong to
$\pi  \R^{2n_r} $. So we get from \eqref{y4} that 
\be\label{y6}
\dot\zeta= -(M^d-\Am)^{-1} \dot M^{n/d}\zeta\,,
\ee
where the equality holds in the space $\pi  \R^{2n_r} $.  
 Differentiating \eqref{y3} we find that 
\be\label{y7}
(M(\eps)-\Am(\eps))\ddot\zeta(\eps)= -\ddot M(\eps) \zeta(\eps) -2\dot M(\eps)\dot\zeta(\eps)
+\ddot\Am(\eps)\zeta(\eps)  +2\dot\Am(\eps)\dot\zeta(\eps)\,.
\ee
Similar to the derivation of \eqref{y5} (and using that $\langle\zeta, \dot\zeta\rangle=0$ since $|\zeta(\eps)|\equiv1$), 
we get from \eqref{y7} and \eqref{y5} that 
\be\label{y9}
\begin{split}
\ddot\Am =\langle\ddot M\zeta,\zeta\rangle+2\langle\dot M\dot\zeta,\zeta\rangle
=\langle\ddot M\zeta,\zeta\rangle +2\langle(M-\Am)^{-1}\dot M^{n/d}\zeta, {}^t(\dot M)\zeta\rangle\,.
\end{split}
\ee
Since for each $\eps$ and every $j$
$$
\frac{d^2}{d\eps^2}\rho_j(\eps)=\frac{d^2}{d\eps^2} y^2_{j} (\eps) =2 x^2_{j}\,,
\qquad
\frac{d^2}{d\eps^2} y_{1}(\eps) y_j(\eps)=0\,,
$$
and since $ \langle \ddot M \zeta, \zeta\rangle =\langle \ddot M^d \zeta, \zeta\rangle $, then 
\be\label{ddot}
\langle \ddot M \zeta, \zeta\rangle =%\langle \ddot M^d \zeta, \zeta\rangle = 
\frac{d^2}{d\eps^2}  \mu(a^r_1, \rho(\eps))\!\mid_{\eps=0} \,= 
C_* \lambda^{-1}_{a_{j_{\#}}}  \big( 3  \lambda^{-1}_{a_{j_{\#}}} x_{j_{\#}}^2
 -2 \sum^{n}_{j=2} x_{j}^2\lambda^{-1}_{a_j}\big)=:k_1\,. 
\ee
Note that $k_1$ does not depend on  $r$.

Now consider the second term in the r.h.s. 
 \eqref{y9}. For any $a,b \in\L^r_f$ we see that 
$\ 
%\frac{d}{d\eps} (\sqrt{\rho_{\ell(a)} (\eps)}\sqrt{ \rho_{\ell(b)}(\eps)})\mid_{\eps=0} 
\frac{d}{d\eps} (y_{\ell(a)}(\eps)y_{\ell(b)}(\eps))\mid_{\eps=0} \ 
%=\frac{d}{d\eps}   y_{\ell(a)}(\eps) = x_{\ell(a)}\,,
$
  is non-zero if exactly one  of the numbers $\ell(a), \ell(b )$ is $a_1$, and this derivative 
  equals $x_{\ell (c)}$, 
where $c\in\{a,b\}$, $\ell (c)\ne a_1$. 
Therefore,  by  \eqref{M}, 
\be\label{y100}
\begin{split}
&(\dot M^{n/d}\zeta )_{{a^r_j}} =%(1- \delta_{1,j} )  
\frac{C_*}{\lambda_{a_{j_{\#}}}}  (\xi^o_{{a^r_j}}, -\eta^o_{{a^r_j}}),\qquad {a^r_j}\in\L^r_f\,,
\\ % j=1,\dots,n_r\big),\\
& \xi^o_{{a^r_j}}= \frac{  \phi(a_1^r, a_j^r)  %x_{\ell({a^r_j})   }
}{\lambda_{{a^r_j}}}\chi^-(a^r_1, {a^r_j}), \quad   
\eta^o_{{a^r_j}}= \frac{   \phi(a_1^r, a_j^r)   }{\lambda_{{a^r_j}}}\chi^+(a^r_1, {a^r_j})\,,
\end{split}
\end{equation}
where $ \phi(a_1^r, a_1^r) =0$ and for $j\ne1$ 
$$
 \phi(a_1^r, a_j^r) =
\begin{cases}
x_{\ell(a^r_j)} & \text{ if} \ \ {j_{\#}}=1\,, \\
 x_{{{j_{\#}}}}  &  \text{ if} \ \  \ell(a^r_j)=a_1   \,,\\
         0 &  \text{ if} \ \  {j_{\#}}\ne1,\ \ell(a^r_j)\ne a_1\,. 
        \end{cases} \qquad %r=\pi r_1+{\pi\/2},        \ r_1\in \N.
$$
Since $\chi^\pm(a^r_1, a^r_1)=0$, then $\xi^o_{a^r_1} = \eta^o_{a^r_1}=0$.

In view of \eqref{obv}, at most one of the numbers $\xi^o_{{a^r_j}}, \eta^o_{{a^r_j}}$ is non-zero. 
By \eqref{y100}, 
\be\label{y11}
((M-\Lambda)^{-1}\dot M^{n/d}\zeta)_{{a^r_j}} =%(1-\delta_{\{1, \ell(a)\}})
\frac{C_*}{\lambda_{a_{j_{\#}}}} (\xi^{oo}_{{a^r_j}}, \,\eta^{oo}_{{a^r_j}}), 
\ee
where $\xi^{oo}_{{a^r_j}}= \eta^{oo}_{{a^r_j}}=0$ if $j=1$, 
and otherwise 
$$
 \xi^{oo}_{{a^r_j}}= 
\frac{ \phi(a_1^r, a_j^r) \chi^-(a^r_1, {a^r_j})  }{  \lambda_{{a^r_j}} (  \mu({a^r_j})-\mu(a^r_1)  )} , \;\;\;\;
\eta^{oo}_{{a^r_j}}=% - \frac{x_{\ell(a)}}{\lambda_{a}}\chi^+(b, a)
\frac{ \phi(a_1^r, a_j^r)  \chi^+(a^r_1, {a^r_j})  }{  \lambda_{{a^r_j}} ( \mu({a^r_j})+\mu(a^r_1))} \,.
$$
Here 
$\ 
\mu({a^r_j}) =\frac12 C_*\lambda^{-2}_{a_1}$ if $\ell(a^r_j)=a_1$ and 
$\ \mu(a^r_j) = -C_*\lambda^{-1}_{a^r_l} \lambda^{-1}_{a_1}$
 if $\ell(a^r_j)\ne a_1$.

Similar,
$$
({}^t\dot M\zeta  )_{{a^r_j}} =
\frac{C_*}{\lambda_{a_{j_{\#}}}}  (\xi^o_{{a^r_j}}, \eta^o_{{a^r_j}}),
$$ 
so  the second term in the r.h.s. of \eqref{y9} equals 
\be\label{k2}
 \frac{C_*^2}{\lambda^2_{a_{j_{\#}}}}\sum_{j=2}^{n_r}   \frac{   \phi(a_1^r, a_j^r)^2 }{ \lambda^2_{a^r_j}}
 \Big(\frac{\chi^-(a_1^r, a^r_j)}{\mu(a^r_j)-\mu(a^r_1)  } +
 \frac{\chi^+(a_1^r, a^r_j)}{  \mu(a^r_j)+\mu(a^r_1)  }
 \Big)=: k_2(r)\ .
\ee

Finally, we  have seen that 
$$
\Am^r_1(\rho(\eps)) = \Lambda^1_1(\rho_*)+\tfrac12 \eps^2 k_1+ \tfrac12 \eps^2k_2(r) +O(\eps^3),\quad r=1,2,
$$
where $k_1$  does not depend on $r$. Since $a_1^r\sim a^r_j$ for each $r$ and each  $j$ (see \eqref{class}), 
% $n_1, n_2\ge2$, 
 then for $j>1$ 
  at least one of the coefficients
$\chi^\pm(a_1^r,a^r_j)$ is non-zero. As  $\chi^+\cdot\chi^-\equiv0$, then
\be\label{non-zer}
  \frac{\chi^-(a_1^r, a^r_j)}{\mu(a^r_j)-\mu(a^r_1)  } +
 \frac{\chi^+(a_1^r, a^r_j)}{  \mu(a^r_j)+\mu(a^r_1)  }
  \ne0 \qquad \forall\,r,\;\; \forall\,j >1  \,.
\ee
We see that  the sum, defining $k_2(r)$, is a  non-trivial quadratic polynomial of the quantities 
$\phi(a_1^r, a_j^r)$ if  $n_r\ge2$, and vanishes if $n_r=1$. 

The following lemma is crucial for the proof. 
\begin{lemma}\label{l_triv}
If the set $\A$ is\sa \  and $|a|=|b|$, $a\ne b$, 
and $\chi^+(a,a')\ne0$,  $\chi^+(b,b')\ne0$, 
or  $\chi^-(a,a')\ne0$,  $\chi^-(b,b')\ne0$, 
then
$|a'|\ne |b'|$. 
\end{lemma}

\begin{proof}
 Let first consider the case when $\chi^+\ne0$.  \\
We know that $\ell(a)=\ell(b)=:{a_{j_{\#}}}$.  Assume that $|a'|=|b'|$. Then $\ell(a')=\ell(b')=: {{a_{j_\flat}}} \in\A$. 
Denote ${a_{j_{\#}}}+{{a_{j_\flat}}}=c$. Then $c\ne0$ since the set $\A$ is admissible.  As
$(a,a'), (b,b') \in (\L_f\times \L_f)_+$, then we have 
$\ 
|{a_{j_{\#}}}-c| = |a-c| = |b-c|\,. 
$
As $|{a_{j_{\#}}}| = |a| = |b|$, then the three points ${a_{j_{\#}}}, a$ and $b$ lie  in the intersection of two circles, one centred in
the origin and another centred in $c ={a_{j_{\#}}}+{{a_{j_\flat}}}$. Since $\A$ is\sa, then  ${a_{j_{\#}}}\an c$ (see 
\eqref{ddd}). So among the 
three point  two are equal, which is a contradiction. Hence,   $|a'|\ne|b'|$ as stated.

The case  $\chi^-\ne0$ is similar.  
\end{proof}

We claim that this lemma implies that 
\be\label{hi}
\Am_1^1(\rho(\eps))\not\equiv \Am^2_1(\rho(\eps)) \quad\text{for a suitable choice of the vector $x$ in
} \eqref{yx}  \,,
\ee
so \eqref{sin} is valid and Lemma~\ref{l_nond} holds. To prove \eqref{hi} 
 we  consider two cases. 
\smallskip

\noindent 
{\it Case 1:  ${j_{\#}}=1$.}   Then $ \phi(a_1^r, a_j^r) = x_{\ell(a_j^r)}$. Denoting 
$\ 
 \frac{C_*^2}{\lambda^2_{a_1}}  \,  \frac{x^2_{\ell(a^r_j)}}{ \lambda^2_{a^r_j}} =: z_{\ell(a^r_j)}
$
we see that 
 $k_2(1)$ and $k_2(2)$  are linear functions of the variables 
 $z_{a_1},\dots, z_{\ell_n}$. 

i) Assume that $\chi^-(a_1^r, a_j^r)=1$  for some $r\in \{1,2\}$ and some $j>1$.  Denote $\ell(a_j^r)=a_{{j_*}}$.  Then 
${j_*}\ne j_{\#}$ and 
$$
k_2(r) =  \frac{z_{a_{{j_*}}}}{\mu(a_j^r) -\mu(a_1^r)} +\dots\,, 
$$
where $\dots$ is independent from $ z_{{j_*}} $.  Now let $r'= \{1,2\}\setminus \{r\}$, and find $j'$ such that 
$\ell (a^{r'}_{j'}) =a_{{j_*}}$.  If such $j'$ does not exist, then $k_2(r')$ does not depend on $z_{{j_*}}$. 
Accordingly, for a suitable  $x$ we have $k_2(r)\ne k_2(r')$, and \eqref{hi} holds. If $n_2=1$, then $r=1$ and $r'=2$. So
$j'$ 
does not exists and \eqref{hi} is established. 

If $j'$ exists, then  $n_1, n_2\ge2$, so the set $\A$ is\sa. 
By Lemma~\ref{l_triv}
 $\chi^-  (a^{r'}_{1} ,  a^{r'}_{j'}   ) =0 $  since $\chi^-(a_1^r, a_j^r)=1$ and 
\be\label{kkk}
|a_1^r| = |a_1^{r'}|, \qquad |a_j^r| = |a_{j'}^{r'}|\,.
\ee
 So
 $$
k_2(r') =  z_{{j_*}}\,\frac{ \chi^+  (a^{r'}_{1} ,  a^{r'}_{j'}   )  }{\mu(a_j^{r'}) +\mu(a_1^{r'})} +\dots\,.
$$
 Since $\chi^+$ equals 1 or 0, then using again \eqref{kkk} and the fact that $\mu(a)$ only depends on $|a|$, 
we see that $k_2(r)\ne k_2(r')$ for a suitable $x$, so  \eqref{hi}  again  holds. 
\smallskip

ii) If $\chi^-(a_1^r, a_j^r)=0$ for all $j$ and $r$,  then $\chi^+(a_1^r, a_j^r)=1$ for some $r$ and $j$. Define $z_{{j_*}}$ as above. 
Then the coefficient in $k_2(r)$ in 
 front of $z_{{j_*}}$ is non-zero, while for $k_2(r')$ it vanishes. This is obvious if $n_{r'}=1$. Otherwise $\A$ is\sa\ 
 and it holds   by Lemma~\ref{l_triv}
(and since $\chi^-\equiv0$). So \eqref{hi} again holds. 

\medskip

\noindent 
{\it Case 2:  ${j_{\#}}\neq 1$. Then by \eqref{achieve} there exists $a^1_j\in \L^r_f$ such that   $\ell(a^r_j)=a_1$. So 
$\chi^+(a ^1_1, a ^1_j)\ne 0$ or $\chi^-(a ^1_1, a ^1_j)\ne 0$. }
Then $ \phi(a_1 ^1, a_j ^1) = x_{a_{j_{\#}} }$, 
 the sum in \eqref{k2} is non-trivial and 
 for the same reason as in Case~1  \eqref{hi} holds.
 
   This completes the proof of   Lemma~\ref{l_nond}.

\bigskip\bigskip
\centerline{PART III. A KAM THEOREM}
\section{KAM normal form Hamiltonians}
\subsection{Block decomposition, normal form matrices.}
In this subsection we recall two notions introduced in 
\cite{EK10} for the nonlinear Schr\"odinger equation. 
They are essential  to overcome the problems of small divisors 
in a multidimensional context. Since the  structure of the spectrum for the beam equation,
 $\{\sqrt{|a|^4+m},\ a\in \Z^{d}\}$, is similar to that  for the NLS equation,
 $\{|a|^2+\hat{V}_a,\ a\in \Z^{d}\}$, then to study the beam equation we will use 
 tools, similar to those used to study the NLS equation.
\medskip

\subsubsection{Partitions} \label{blockdecomp}
 For any $\Delta\in\N\cup \{\infty\}$ 
we define an equivalence relation on $\Z^{d}$, generated by the pre-equivalence relation
$$ a\sim b \Longleftrightarrow \left\{\begin{array}{l} |a|=|b| \\   {[a-b]}
 \leq \Delta. \end{array}\right.$$
(see \eqref{pdist}). 
Let $[a]_\Delta$ denote the equivalence class of $a$  -- the {\it block} of $a$. 
For further references we note that 
\be\label{a-b} 
|a|=|b| \text{ and }   [a]_{\Delta}\neq [b]_{\Delta}
\Rightarrow [a-b]\geq \Delta
\ee
The crucial fact is that the blocks have a finite maximal ``diameter''
$$d_\Delta=\max_{[a]=[b]} [a-b]$$
which do not depend on $a$ but only on $\Delta$. This is the content of

\begin{proposition}\label{blocks}
\be\label{block}
d_\Delta\leq C \Delta^{\frac{(d+1)!}2}. 
\ee
The constant $C$ only depends on $d$.
\end{proposition}

\begin{proof} In \cite{EK10} it was considered the equivalence relation on $\Z^{d}$, generated by the 
pre-equivalence 
$$
a\approx b\quad\text{if}\quad |a|=|b|\quad \text{and}\quad |a-b|\le\Delta. 
$$
Denote by $[a]^o_\delta$ and $d^o_\Delta$ the corresponding equivalence class and its diameter (with respect to the 
usual distance). Since $a\sim b$ if and only if $a\approx b$ or $a\approx -b$, then 
\be\label{union}
[a]_\Delta = [a]^o_\Delta \cup -[a]^o_\Delta,
\ee 
provided that the union in the r.h.s. is disjoint. It is proved in \cite{EK10} that 
$d_\Delta^o\le D_\Delta=:C \Delta^{\frac{(d+1)!}2}$. Accordingly,
if $|a|\ge D_\Delta$, then the union above is disjoint, \eqref{union} holds and diameter of $[a]_\Delta$ satisfies 
\eqref{block}. If  $|a|< D_\Delta$, then $[a]_\Delta$ is contained in a sphere of radius $< D_\Delta$. So the block's
  diameter is at most 
$2D_\Delta$. This proves \eqref{block} if we replace there $C_{d}$ by $2C_{d}$.
\end{proof}

If $\Delta=\infty$ then the block of $a$ is the sphere $\{b: |b|=|a|\}$. 
Each block decomposition is a sub-decomposition of the trivial decomposition formed by the spheres $\{|a|=\text{const}\}$.

\medskip

\subsubsection{Normal form matrices} \label{normalformmatrices} 
On $\L_{\infty}\subset\Z^{d} $ we define the partition
$$[a]_\Delta=
\left\{\begin{array}{ll}
 [a]_\Delta\cap\L_{\infty} & \ \textrm{if}\ a\in \L_{\infty}\ \textrm{and}\ \ab{a}>c\\
 \{ b\in \L_{\infty}: \ab{b}\le c\} &  \ \textrm{if}\  a\in \L_{\infty}  \ \textrm{and}\ \ab{a}\le c.\end{array}\right.$$
On $\L=\F\sqcup\L_{\infty}$ we
define the partition, denoted $\E_\Delta$,
\be\label{Fcluster}
[a]=[a]_\Delta=
\left\{\begin{array}{ll}
 [a]_\Delta\cap\L_{\infty} & a\in \L_{\infty}\\
\F & a\in \F.\end{array}\right.
\ee

\begin{remark}\label{remark-blocks}
Now the diameter of each block $[a]$ is bounded as in \eqref{a-b}
if  we just let $C\gsim\max( \#\F,c^{d})$.
\end{remark}

If $A:\ \L\times \L\to gl(2, \C)$ we define its {\it block components} 
$$
A_{[a]}^{[b]}:[a]\times[b]\to gl(2, \C)$$
to be the restriction of $A$ to $[a]\times[b]$. $A$ is {\it block diagonal} over $\E_\Delta$ if, and only if, 
$A_{[b]}^{[a]}=0$ if $[a]\neq [b]$. Then we simply write $A_{[a]}$ for $A_{[a]}^{[a]}$.

On the space of $2\times 2$ complex matrices we introduce  a projection 
$$
\Pi: gl(2, \C)\to \C I+\C J,
$$
 orthogonal with respect to the Hilbert-Schmidt  scalar product. Note that 
$\C I+\C J$ is the space  of matrices, commuting with the symplectic matrix $J$. 
\begin{definition}\label{d_31}
We say that a matrix $A:\ \L\times \L\to gl(2, \C)$ is on normal form with respect to 
$\Delta$, $\Delta\in\N\cup \{\infty\}$,  and  write  $A\in  \NF_\Delta$, if
 \begin{itemize}
 \item[(i)] $A$ is real valued,
 \item[(ii)] $A$ is symmetric, i.e. $A_b^a\equiv {}^t\hspace{-0,1cm}A_a^b$,
 \item[(iii)] $A$ is block diagonal over $\E_\Delta$, 
 \item[(iv)] $A$ satisfies $\Pi A^a_b\equiv A^a_b$ for all $a,b\in\L_{\infty}$.
 %\item[(v)] $$ {}^t\!U_{\L_{\mathbf f}} A_{\L_{\mathbf f}} (\r)U_{\L_{\mathbf f}}= \diag\Big(\Lambda_a(\r)\left(\begin{array}{cc} 1&0\\0&1\end{array}\right), a\in\L_{\mathbf f}\Big)$$

 \end{itemize}
 \end{definition}
 
Any real quadratic form ${\mathbf q}(w)= \frac 1 2\langle w,Aw \rangle$, $w=(p,q)$, 
can be written as
$$
\frac 1 2\langle p,A_{11}p \rangle+\langle p,A_{12}q \rangle+\frac 1 2\langle q,A_{22}q \rangle 
+\frac 1 2\langle w_{\F}, H(\r) w_{\F}\rangle 
$$
where $A_{11},\ A_{22}$ and $H$ are real symmetric matrices and $A_{12}$ is a real matrix. We now pass from the real variables $w_a=(p_a,q_a)$ to the 
complex variables $z_a=(\xi_a,\eta_a)$ by the transformation $w=U z$ defined through
\be\label{transf}
\xi_a=\frac 1 {\sqrt 2} (p_a+{\mathbf i}q_a),\quad \eta_a =\frac 1 {\sqrt 2} (p_a-{\mathbf i}q_a),\ee
for $a\in\L_\infty$, and acting like the identity on $  (\C^2)^\F$.
Then we have
$$
{\mathbf q}(Uz)=\frac 1 2\langle \xi,P\xi\rangle+ \frac 1 2\langle \eta,{\overline P}\eta\rangle+\langle \xi,Q\eta\rangle
+\frac 1 2\langle z_{\F}, H(\r) z_{\F}\rangle ,$$
where
$$P=\frac12\Big( (A_{11}-A_{22})-{\mathbf i}(A_{12}+{}^t A_{12})\big)$$
and
$$Q=\frac12\Big( (A_{11}+A_{22})+{\mathbf i}(A_{12}-{}^t A_{12})\big).$$
Hence $P$ is a complex symmetric matrix and $Q$ is a Hermitian matrix. If $A$ is on normal form, then $P=0$. 

Notice that this change of variables is not symplectic but changes the symplectic form slightly:
$$
U^*\Omega={\mathbf i}\sum_{a\in\L}d\xi_a\wedge d\eta_a+\sum_{a\in\F}d\xi_a\wedge d\eta_a.$$

\subsection{The unperturbed Hamiltonian}\label{ssUnperturbed}

Let $h_{\textrm up}(r,w,\r)$ be a function of the form
\be\label{unperturbed}  \langle r,\Omega_{\textrm up}(\r)\rangle +\frac12\langle w,A_{\textrm up}(\r)w \rangle=
\langle r,\Omega_{\textrm up}(\r)\rangle +\frac12\langle w_{\F},H_{\textrm up}(\r)w_{ \F} \rangle
+\frac12\sum_{a\in \L_{\infty}} \Lambda_a (p_a^2+q_a^2),\ee
where $w_a=(p_a,q_a)$ and
\be\label{properties}\left\{\begin{array}{ll}
\Omega_{\textrm up}:\D\to\R^{\cA}&\\
\Lambda_a:\D\to \R,&\quad  a\in \L_\infty\\
H_{\textrm up}:\D\to gl(\R^{\F}\times \R^{\F}),&\quad {}^t\! H_{\textrm up}=H_{\textrm up}
\end{array}\right.\ee
are $\cC^{{s_*}}$-functions, ${s_*}\ge 1$. $\D$ is an open ball or a cube
of diameter at most $1$ in the space $\R^{\P}$, parametrised by  some finite subset $\P$
of $\Z^{d}$. 

We can write
$$\langle w,A_{\textrm up}(\r)w\rangle =\langle w_{\F},H_{\textrm up} (\r)w_{\F} \rangle
+\frac12\big( \langle p_{\infty},Q_{\textrm up} (\r)p_{\infty} \rangle + \langle q_{\infty},Q_{\textrm up} (\r)q_{\infty} \rangle  \big)$$
and
$$Q_{\textrm up} (\r)=\diag\{\Lambda_a(\r): a\in\L_\infty\}.$$

\begin{definition}\label{definitionup} A function $h_{\textrm {up}}$ of the form \eqref{unperturbed}+\eqref{properties} will be called un {\it unperturbed Hamiltonian} if it verifies Assumptions A1-3 (given below)
described by the positive constants 
 $$c',c, \de_0,\beta=(\beta_1,\beta_2,\beta_3),\tau.$$
 \end{definition}
To formulate these assumptions we shall use the partition $[a]=[a]_{\infty}$ of $\F\sqcup\L_{\infty}$ defined in \eqref{Fcluster}. Notice
that this partition depend on a (possibly quite large) constant $c$.

\smallskip

\subsubsection{A1 -- spectral asymptotics.} There exist  a constant  
$0< c'\le c$  and exponents $\beta_1\ge0,\beta_2>0$  such that for all $\r\in\D$:

\be\label{la-lb-ter} 
|\Lambda_a(\r)-\ab{a}^{2} |\leq c \frac1{\langle a\rangle^{\beta_1}}\quad a\in \L_{\infty};
\ee

\begin{multline}\label{la-lb}
\qquad |(\Lambda_a(\r )-\Lambda_b(\r))-(\ab{a}^{2}-\ab{b}^{2}) | 
 \le c\max( \frac1{\langle a\rangle^{\beta_2}},\frac1{\langle b\rangle^{\beta_2}}),\quad a,b\in \L_{\infty}\,;\quad\end{multline}

\be\label{laequiv}
\left\{\begin{array}{l}
| \Lambda_a(\r )|\geq  c' \ \quad a\in\L_\infty\\
||(JH_{\text{up}}(\rho))^{-1}||\leq \frac1{c'};
\end{array}\right.\ee

\be\label{laequiv-bis}
| \Lambda_a(\r )+\Lambda_b(\r)|\geq  c' \ \quad a,b\in\L_\infty\ee

\be\label{la-lb-bis}
\left\{\begin{array}{ll}
|(\Lambda_a(\r )-\Lambda_b(\r))) |\ge c' & a,b\in\L_\infty,\ [a]\not=[b]\\
||(\Lambda_a(\r)I-{\mathbf i}JH_{\text{up}}(\rho))^{-1}||\leq \frac1{c'} & a\in \L_{\infty},
\end{array}\right.
\ee
Notice that if $\beta_1\ge\beta_2$, then \eqref{la-lb-ter}  implies \eqref{la-lb} (if $c$ is large enough).

\smallskip

\subsubsection{A2 -- transversality.}
Denote by $(Q_{\textrm up})_{[a]}$  the restriction of the matrix $Q_{\textrm up}$
to $[a]\times [a]$ and let $(Q_{\textrm up})_{[\emptyset]}=0$. 
Let also $JH_{\text{up}}(\r)_{[\emptyset]}=0$.

\medskip

There exists a  $1\ge\delta_0>0$ such that 
%\marginpar{$ \qquad  \delta_0 $}
for all $\cC^{{s_*}}$-functions
\be\label{o}
\Omega:\D\to \R^n,\quad |\Omega-\Omega_{\textrm up}|_{\cC^{{s_*}}(\D)}<\delta_0,\ee
and for all $k\in\Z^n\setminus 0$
 there exists a unit vector ${\mathfrak z}$ such that
$$\ab{\p_{\mathfrak z}\langle k,  \Omega(\r)\rangle}  \ge \delta_0, \quad \forall \r\in\D$$
\footnote{\ $\p_{\mathfrak z} $ denotes here the directional derivative in the direction ${\mathfrak z}\in\R^p$}
and the following dichotomies hold for each  $k\in\Z^n\setminus 0$:

\begin{itemize}

\item[$(i)$] for any $a,b\in\L_\infty\cup \{\emptyset\} $ let
$$L(\r):X\mapsto \langle k,\Omega(\r) \rangle X+(Q_{\textrm up})_{[a]}(\r)X\pm X(Q_{\textrm up})_{[b]}:$$ 
then either $L(\r) $ is {\it $\de_0$-invertible} for all $ \r\in\D$  , i.e.
\be\label{invert}
\aa{L(\r)^{-1}}\le\frac1{\delta_0}\qquad\forall \r\in\D,\ee
or there exists a unit vector ${\mathfrak z}$ such that
$$\ab{\langle  v,\p_{\mathfrak z}  L(\r) v\rangle}  \ge \delta_0, \quad \forall \r\in\D
$$
and for any unit-vector $v$ in the domain of $L(\r)$
\footnote{\ $L$ is a linear operator acting on $([a]\times[b])$-matrices};

\item[$(ii)$]   let 
$$L(\r,\lambda):X\mapsto \langle k,\Omega(\r) \rangle X+\lambda X+ {\mathbf i}XJH_{\textrm up}(\r)$$
and
$$P_{\textrm up}(\r,\lambda)=  \det L(\r,\lambda):$$
then  either $L(\r,\Lambda_a(\r))$ is $\de_0$-invertible for all $ \r\in\D$  and $a\in[a]_\infty$, or
there exists a unit vector ${\mathfrak z}$ such that, with $m=2\#\F$,
$$
\ab{\p_{\mathfrak z}P_{\textrm up}(\r,\Lambda_a(\r))+\p_\lambda P_{\textrm up}(\r,\Lambda_a(\r))
\langle v,\p_{\mathfrak z} Q_{\textrm up}(\r) v\rangle}
\ge
 \delta_0\aa{L(\cdot,\Lambda_a(\cdot))}_{\cC^{1}(\D)}\aa{L(\cdot,\Lambda_a(\cdot))}_{\cC^{0}(\D)}^{m-2} $$
for all $\r\in \D$,  $a\in[a]_\infty$
and for any unit-vector $v\in (\C^2)^{[a]}$
\footnote{\  $L$ is a linear operator acting on $(1\times m)$-matrices};

\item[$(iii)$] for any $a,b\in \F\cup \{\emptyset\} $
let
$$
L(\r):X\mapsto \langle k,\Omega(\r) \rangle X-{\mathbf i}JH_{\textrm up}(\r)_{[a]}X+ {\mathbf i}XJH_{\textrm up}(\r)_{[b]}:$$ 
then  either $L(\r)$ is $\de_0$-invertible for all $ \r\in\D$, or
 there exists a unit vector ${\mathfrak z}$ and an integer  $1\le j\le {s_*}$ such that
 \be\label{altern1}
\ab{ \p_{\mathfrak z}^j \det L(\r) }\ge \delta_0 \aa{L}_{\cC^{j}(\D)}\aa{L}_{\cC^{0}(\D)}^{m^2-2}, \quad \forall \r\in \D,\ee
where  $m^2=(2\#\F)^2$ if both $[a]$ and $[b]$ are $\not=\emptyset$
and  $m^2=2\#\F$ if one of $[a]$ and $[b]$  $=\emptyset$
\footnote{\  in the first case $L$ is a linear operator acting on $(m\times m)$-matrices, and in the second case
$L$ is a linear operator acting on $(1\times m)$-matrices or  $(m\times 1)$-matrices. }
\end{itemize}
\medskip

 \begin{remark}\label{remro} The dichotomy 
in A2 is imposed not only on  $\Omega_{\textrm up}$ but also on
$\cC^{s_*}$-perturbations of $\Omega_{\textrm up}$, because, in general, the 
dichotomy for $\Omega_{\textrm up}$ does not imply that for perturbations.

If, however,  any $\cC^{{s_*}}$ perturbation 
of $\Omega_{\textrm up}$ can be written as $\Omega_{\textrm up}\circ  f$
for some diffeomorphism $ f=id+\O(\delta_0)$ --   this is for example the case when $\Omega(\r)=\r$  -- 
then the dichotomy on $\Omega$ implies a dichotomy on 
$\cC^{{s_*}}$-perturbations.
\end{remark}

\subsubsection{A3 --  a Melnikov condition.}

There exist constants
$\beta_3,\tau>0$ such that
\be\label{melnikov}
|\langle k,\Omega(0)\rangle-(\Lambda_a(0)-\Lambda_b(0))) |\ge\frac{\beta_3}{|k|^\tau}\ee
for all $k\in\Z^\P\setminus 0$ and all $a,b\in \L_{\infty}\setminus [0]$.

\subsection{KAM normal form Hamiltonians}

Consider now an unperturbed Hamiltonian $h_{\textrm {up}}$ defined on the set $\D$ (see Definition \ref{definitionup}).
The essential properties of  this function  are described by  the positive constants 
 $$c',c, \de_0,\beta=(\beta_1,\beta_2,\beta_3),\tau$$
(occurring in assumptions A1-3), and by the constant
\be\label{chi}
\chi=
 |\nabla_\r \Omega_{\textrm up} |_{\cC^{ {{s_*}}-1 } (\D)}+\sup_{a\in\L_\infty} |\nabla_\r \Lambda_a|_{\cC^{ {{s_*}}-1 } (\D)}
 + ||\nabla_\r H_{\textrm up} ||_{\cC^{ {{s_*}}-1 } (\D) }.\ee
Notice that, by Assumption A2, $\chi\ge \delta_0$,
and in order to simplify the estimates a little we shall assume that  
 \be\label{Conv}0<c'\le \de_0\le\chi\le  c.\ee

\medskip

We shall consider a somewhat larger class of functions.

\begin{definition} A function  of
the form
\be\label{normform}
h(r,w,\r)=\langle \Omega(\r), r\rangle +\frac 1 2\langle w, A(\r)w\rangle\ee  
is said to be on {\it KAM normal form} with respect to the  unperturbed Hamiltonian $h_{\textrm {up}}$, satisfying \eqref{Conv},  if

\noindent 
({\bf Hypothesis $\Omega$})
$\Omega$ is of class $\cC^{{s_*}}$ on $\D$ and
\be\label{hyp-omega}
|\Omega-\Omega_{\textrm up}|_{\cC^{{s_*}}(\D)}\le \delta.
\ee

\medskip

\noindent
({\bf Hypothesis B})   
$A-A_{\textrm up}:\D\to \cM_{(0, m_*+\vark),\varkappa}^b$ is of class $\cC^{{s_*}}$, 
$A(\r)$ is on normal form $\in \NF_{\Delta}$ for all $\r\in\D$
 and
  \be\label{hypoB}
 || \p_\r^j (A(\r)-A_{\textrm up}(\r))_{[a]} || \le  \delta\frac{1}{\langle a \rangle^\varkappa} \ee
for $ |j| \le {{s_*}}$, $a\in\L$  and $\r\in \D$
\footnote{\ here it is important that $||\cdot ||$ is the matrix operator norm}. Here we require that
\be\label{varkappa}
0<\varkappa.\ee

We denote this property by
$$h\in \NF_{\varkappa}(h_{\textrm up},\Delta,\delta).$$
Since the unperturbed Hamiltonian $h_{\textrm up}$ will be fixed in Part III
we shall often suppress it,  writing simply  $h\in \NF_{\varkappa }(\Delta,\delta)$.

\end{definition}

\subsection{The KAM  theorem}
\ 
In this section we state an abstract KAM result for perturbations of a certain KAM normal form Hamiltonians.

Let 
$$h_{\textrm {up}}= h_{\textrm {up}, \chi,c',\de_0,c}$$
be a fixed unperturbed Hamiltonian 
satisfying \eqref{Conv}. ($h_{\textrm {up}}$ also depends on $\beta, \tau$ but we shall not track this dependence.)

Let $h$ be a KAM normal form Hamiltonian,
$$h\in\NF_{\varkappa}(h_{\textrm {up}, \chi,c', \de_0,c},\Delta,\de),$$
and recall \eqref{varkappa}. We shall also assume $\Delta\ge1$.

The perturbation will belong to $ \cT_{\ga,\varkappa,\D }(\s,\mu)$ with
$$0<\s,\mu,\ga_1\le 1$$
and (recall \eqref{gamma})
$$\ga=(\ga_1,m_*+\vark)> \ga_*=(0,m_*+\vark).$$

These bounds will be, often implicitly, assumed in the rest of Part III.

\begin{theorem}\label{main}
There exist positive constants $C$, $\alpha$ and  $\exp$ such that, for any
$h\in\NF_{\varkappa,h_{\textrm up}}(\Delta,\de)$ and for any 
$f\in \cT_{\ga,\varkappa,\D }(\s,\mu)$,
$$ \eps=\ab{f^T}_{\begin{subarray}{c}\s,\mu\ \ \\ \ga, \varkappa,\D  \end{subarray}}\ \textrm{and}\  
 \xi=\ab{f}_{\begin{subarray}{c}\s,\mu\ \ \\ \ga, \varkappa,\D  \end{subarray}},$$
if
$$\delta \le \frac1{2C} c'$$
and
\be\label{epsi}
\eps(\log \frac1\eps)^{\exp}\le
\frac1{C}\big( \frac {\s\mu}{\max(\ga_1^{-1} ,d_{\Delta})}\frac{c'}{\chi+\xi}\big)^{\exp}c',
\ee
then there exist a closed subset $\D'=\D'(h, f)\subset \D$,
\be\label{measure}
\Leb (\D\setminus \D')\leq 
C\big(\log\frac1{\eps} \frac{ \max(\ga_1^{-1} ,d_{\Delta})}{\s\mu} \big)^{\exp}
\frac{\chi}{\delta_0}((\chi+\xi) \frac{\eps}{\chi})^{\alpha},\ee
and a $\cC^{{s_*}}$ mapping
$$\Phi:\O_{ \ga_*}(\s/2,\mu/2)\times\D\to \O_{ \ga_*}(\s,\mu),$$
real holomorphic and symplectic for each parameter $\r\in\D$, such that 
$$(h+ f)\circ \Phi= h'+f'$$
with
\begin{itemize}
\item[(i)] $$h'\in\NF_{\vark}(\infty,\de'),\quad\de'\le \frac{c'}2,$$
  and
$$\ab{ h'- h}_{\begin{subarray}{c}\s/2,\mu/2\ \ \\ \ga_*, \varkappa,\D  \end{subarray}}\le C;$$

\item[(ii)] 
for any $x\in \O_{ \ga_*}(\s/2,\mu/2)$, $\r\in\D$ and $\ab{j}\le{s_*}$
$$|| \p_\r^j (\Phi(x,\r)-x)||_{ \ga_*}+ \aa{ \p_\r^j (d\Phi(x,\r)-I)}_{ \ga_*,\vark} \le C$$
and $\Phi(\cdot,\r)$ equals the identity for $\r$ near the boundary of $\D$;

\item[(iii)] for $\r\in\D'$ and $\zeta=r=0$
$$d_r f'=d_\theta f'=
d_{\zeta} f'=d^2_{\zeta} f'=0.$$
\end{itemize}

Moreover,
\begin{itemize}
\item[(iv)]  if  $\tilde \r=(0,\r_2,\dots,\r_p)$ and $f^T(\cdot,\tilde \r)=0$ for all $\tilde \r$,  then $h'=h$ and $\Phi(x,\cdot)=x$
 for all $\tilde \r$.
\end{itemize}

The exponent
$\alpha$ is a positive constant only depending on  $d,s_*,\vark$ and $\beta_2 $.
The exponent $\exp$ only depends  on $d$, $\#\cA$ and $\tau,\beta_2,\vark$. 
${C}$ is an absolute constant that depends on  $c,\tau,\beta_2,\beta_3$ and $\vark$. ${C}$ also depend on $\sup_\D\ab{\Omega_{\textrm up}}$ and $\sup_\D\ab{H_{\textrm up}}$, but stays bounded
when these do.
\end{theorem}

The condition on $\Phi$ and $h'-h$ may look bad but it is not. 

\begin{corollary}\label{cMain} Under the assumption of Theorem~\ref{main}, let $\eps_*$ be the largest positive number such that \eqref{epsi}
holds. Then,
for any  $\r\in\D$ and $\ab{j}\le{s_*}-1$,
\begin{itemize}
\item[$(i)'$] 
$$\ab{  \p_\r^j (h'(\cdot,\r)- h(\cdot,\r))}_{\begin{subarray}{c}\s/2,\mu/2\ \ \\ \ga_*, \varkappa,\ \ \  \ \end{subarray}}\le 
\frac{C}{\eps_*}\ab{f^T}_{\begin{subarray}{c}\s,\mu\ \ \\ \ga, \varkappa,\D  \end{subarray}};$$

\item[$(ii)'$] 
$$|| \p_\r^j (\Phi(x,\r)-x)||_{\ga_*}+ \aa{ \p_\r^j (d\Phi(x,r)-I)}_{\ga^*,\vark} \le \frac{C}{\eps_*}\ab{f^T}_{\begin{subarray}{c}\s,\mu\ \ \\ \ga, \varkappa,\D  \end{subarray}},$$
for any $x\in \O_{\ga_*}(\s/2,\mu/2)$.
\end{itemize}
\end{corollary}

\begin{proof} Let us denote $\r$ here by  $ \r_1$. If  $|f^T|_{\begin{subarray}{c}\s,\mu\ \ \\ \ga, \vark,\D  \end{subarray}}\le \eps_*$, then we can apply the theorem to $\eps f$ for any $|\eps|\le 1$. Let now
$\r=(\eps,\r_1)$ and consider  $h_{\mathrm up}$, $h$ and $f$ as functions depending on 
this new parameter $\r$ --  they will still verify the assumptions of the theorem, which will provide us with a mapping $\Phi$ with a $\cC^{{s_*}}$ dependence in 
$\r=(\eps,\r_1)$ and equal to the identity when $\eps=0$. The bound on the derivative together with assertion $(iv)$ now implies that
$$|| \Phi(x,\eps,\tilde\r)-x  ||_{\ga_*}\le C\eps\le
\frac C{\eps_*}|f^T|_{\begin{subarray}{c}\s,\mu\ \ \\ \ga, \vark,\D  \end{subarray}}
$$
for any $x\in \O_{\ga_*}(\s/2,\mu/2)$. 
The same estimate holds for all derivatives with respect to 
$\tilde \r$ up to order ${{s_*}}-1$. Take now $\eps=1$ and we get $(ii)'$.

The argument for $h'-h$ is the same.
\end{proof}

A special case
that will interest us in particular is the following.

\begin{corollary}\label{cMain-bis} 
Let $h_{\textrm up}=h_{\textrm {up}, \chi,c', \de_0,c}$ be an
unperturbed Hamiltonian,  satisfying 
$$ 
\;\;a) \qquad\qquad\qquad\qquad
\de_0^{1+\aleph}\le c'\le \delta_0\le\chi \le C' \delta_0^{1-\aleph} \le  c,
\qquad\qquad
$$
and be $f\in \cT_{\ga,\varkappa,\D }(\s,\mu)$  with
$$  
\;\; b) \qquad\qquad\qquad\qquad
\xi=\ab{f}_{\begin{subarray}{c}\s,\mu\ \ \\ \ga, \varkappa,\D  \end{subarray}}\le C'\delta_0^{1-\aleph}.
\qquad\qquad\qquad\qquad\qquad\qquad\qquad\qquad
$$
 for some $1>\aleph>0$ and $C'>0$.

 Then there exist constants $\eps_0>0$, $\al$ and $\ka$ --  independent of $c',\delta_0,\chi$ and $\aleph$  -- such that
 if $ \eps=\ab{f^T}_{\begin{subarray}{c}\s,\mu\ \ \\ \ga, \varkappa,\D  \end{subarray}}$ satisfies
\be\label{epsi-bis}
\eps(\log \frac1\eps)^{\ka}\le  \eps_0\delta_0^{1+ \aleph \ka},\ee

then there exist a closed subset $\D'=\D'(h, f)\subset \D$,
\be\label{measure-bis}
\Leb (\D\setminus \D')\leq 
\frac1{\eps_0}\delta_0^{-\aleph\ka}\eps^{\alpha},\ee
and a $\cC^{{s_*}}$ mapping $\Phi$ 
$$\Phi:\O_{ \ga_*}(\s/2,\mu/2)\times\D\to \O_{ \ga_*}(\s,\mu),$$
real holomorphic and symplectic for each parameter $\r\in\D$, 
such that  
$$(h_{\text{up}}+ f)\circ \Phi(r,w,\r)= \langle \Omega'(\r), r\rangle +\frac 1 2\langle w, A'(\r)w\rangle+f'(r,w,\r)$$
with
\begin{itemize}
\item[$(i)$] the frequency vector $\Om'$ satisfies
$$|\Omega'-\Omega_{\textrm up}|_{\cC^{{s_*-1}}(\D)}\le c'$$
and, for each  $ |j| \le {{s_*}}$ and $\r\in \D$, the matrix 
$$A'(\r)=A'_\infty(\r)\oplus H'(\r)\in   \NF_{\infty}$$
and satisfies
$$ || \p_\r^j (H'(\r)-H_{\textrm up}(\r) || \le  c';$$

\item[$(ii)'$]  for any  $x\in \O_{\ga_*}(\s/2,\mu/2)$, $\r\in\D$ and $\ab{j}\le{s_*}-1$,
$$|| \p_\r^j (\Phi(x,\r)-x)||_{\ga_*}+ \aa{ \p_\r^j (d\Phi(x,r)-I)}_{\ga^*,\vark} \le 
\frac1{\eps_0}\frac{\eps}{\delta_0^{1+ \aleph \ka}}(\log \frac1{\delta_0})^{\ka} $$
and $\Phi(\cdot,\r)$ equals the identity for $\r$ near the boundary of $\D$;
\item[(iii)] for $\r\in\D'$ and $\zeta=r=0$
$$d_r f'=d_\theta f'=
d_{\zeta} f'=d^2_{\zeta} f'=0.$$
\end{itemize}

The exponent
$\alpha$ is a positive constant only depending on  $d,s_*,\vark$ and $\beta_2 $.
The exponent $\ka$ also depends  on $\#\cA$ and $\tau$. The constant
${\eps_0}$  depends on everything except, as already said, $c',\delta_0,\chi$ and $\aleph$.

\end{corollary}

\begin{proof} 
We apply the theorem with $h=h_{\textrm up}$, i.e. $\delta=0$ and $\Delta=1$.
The condition \eqref{epsi} is implied by
$$\eps(\log \frac1\eps)^{\exp}\le \frac1{C''}\big(\frac{c'}{\chi+\xi}\big)^{\exp}c'$$
for some $C''$ depending on $C,\ga_1,\s,\mu$. With the choice of $c',\xi,\chi$ this is now implied
by \eqref{epsi-bis} if $\ka\ge1+2\exp$.

The estimate of the measure becomes, from \eqref{measure}, 
$$\frac1{\eps_0}\big(\log\frac1{\eps})^{\exp}\delta_0^{-\aleph(1+\alpha)}\eps^{\alpha}
\le \frac1{\eps_0}\delta_0^{-\aleph(1+\alpha)}\eps^{\frac{\alpha}2},$$
which is what is claimed if we replace $\frac{\alpha}2$ by $\al$, and take $\ka\ge(1+\al)$.

(i) is just a consequence of $h'\in\NF(\infty,c')$. The bound in $(ii)$ follows from the bound $(ii)'$ in Corollary~\ref{cMain} plus an easy estimate of $\eps_*$.
\end{proof}

\section{Small divisors}\label{s4}

Control of the small divisors is essential for solving the homological equation (next section). In this section we shall
control these divisors for $k\not=0$ using Assumptions A2 and A3.

For a mapping $L:\D\to gl(\dim,\R)$ define, for any $\ka>0$,
$$\Sigma(L,\kappa)=\{\r\in \D:   ||L^{-1}(\rho)| |>\frac1\ka\}.$$
Let  
$$h(r,w,\r)=\langle r,\Omega(\r)\rangle+ \frac12\langle w,A(\r) w \rangle$$
be a  normal form Hamiltonian in $\NF_{\varkappa}(\Delta,\delta)$.
Recall the convention \eqref{Conv} and assume $\vark>0$ and
\be\label{ass} \delta \le \frac{1}{C}c' ,\ee
where $C$ is to be determined.

\begin{lemma}\label{lSmallDiv1}
Let
$$L_{k}=\langle k,\Omega(\r)\rangle.$$
There exists  a constant $C$ such that if \eqref{ass} holds, then
$$\Leb\big(\bigcup_{0<\ab{k}\le N} \Sigma(L_k,\ka)\big)
\le C N^{\exp} \frac{\ka}{\de_0}$$
and
$$\dist(\D\setminus \Sigma(L_k,\ka),\Sigma(L_k,\frac\ka2))>\frac{1}{C}\frac{\ka}{N\chi}$$
 \footnote{\ this is assumed to be fulfilled if $\Sigma_{L_k}(\frac\ka2)=\emptyset$}
for any $\ka>0$ .

(The exponent $\exp$ only depends on $\#\cA$. $C$ is an absolute constant.)
\end{lemma}

\begin{proof} 
We only need to consider $\ka\le\delta_0$ since otherwise the result is trivial.
Since $\delta\le\delta_0$, 
using  Assumption A2$(i)$, with $a=b=\emptyset$,  we have, for each $k\not=0$, either that
$$ |\langle\Omega(\r), k\rangle|\ge \de_0\ge \ka\quad \forall \r\in \D$$
or that
$$ \p_{\mathfrak z} \langle\Omega(\r), k\rangle  \ \geq \delta_0\quad \forall \r\in \D
$$
(for some suitable choice of a unit vector $\mathfrak z$). The first case implies  
$\Sigma(L_{k},\ka)=\emptyset$. The second case implies  that 
 $\Sigma(L_k,\kappa)$
 has Lebesgue measure  $\lsim  \frac{\ka}{\de_0}$.
 Summing up over all $0<\ab{k}\le N$ gives the first statement. The second statement
 follows from the mean value theorem and the bound
$$\ab{\nabla_\r L_k(\r)}\le N(\chi+\delta).$$
\end{proof}

\begin{lemma}\label{lSmallDiv2}
Let 
$$L_{k,[a]}=\big(  \langle k,\Omega\rangle I - {\mathbf i} J A \big)_{[a]}.$$
There exists  a constant $C$ such that if \eqref{ass} holds, then,
$$\Leb\big(\bigcup_{\begin{subarray}{c} 0<\ab{k}\le N\\  [a] \end{subarray}} \Sigma(L_{k,[a]}(\ka)\big)
\le C N^{\exp} (\frac{\ka}{\delta_0})^{\frac1{{s_*}}}$$
and
$$\dist(\D\setminus \Sigma(L_{k,[a]},\ka),\Sigma(L_{k,[a]},\frac\ka2))>\frac{1}{C}\frac\ka{N\chi},$$
for any $\ka>0$.

(The exponent $\exp$ only depends  on $d$ and $\#\cA$. $C$ is an absolute constant that depends
on  $c$. $C$ also depend on $\sup_\D\ab{\Omega_{\textrm up}}$ and $\sup_\D\ab{H_{\textrm up}}$, but stays bounded
when these do.)

\end{lemma}

\begin{proof}  
Consider first $a\in\L_\infty$. Then  $L_{k,[a]}$ is conjugate to a sum of two Hermitian operators of the form
$$L=\langle k,\Omega\rangle I + Q_{[a]},$$
where $Q_{[a]}$  is  the restriction of $Q$ to $[a]\times [a]$ (see the discussion in section \ref{normalformmatrices}) . 

If we let
$$L_{\textrm up}= \langle k,\Omega\rangle I + (Q_{\textrm up}) _{[a]},$$
where $ Q_{\textrm u p}$ comes from the unperturbed Hamiltonian,
then it follows, from \eqref{hypoB} and  \eqref{ass},   that   
$$\aa{L-L_{\textrm up}}_{\cC^{{1}}(\D)} \leq  \de\leq \cte\delta_0.$$
If now $L_{\textrm up}$ is $\delta_0$-invertible, then  this  implies that
$L$ is $\frac{\delta_0}2$-invertible. 

Otherwise,  by  assumption A2$(i)$, there exists a unit vector ${\mathfrak z}$ such that 
$$\ab{\langle v,\p_{\mathfrak z}  L_{\textrm up}(\r) v\rangle}\ge \delta_0$$
for any unit vector $v$. Since $Q_{[a]}$ is Hermitian we have, for any eigenvalue  $\Lambda(\rho)$, $\cC^1$ in the direction 
${\mathfrak z}$, and any associated unit eigenvector $v(\rho)$,
$$\p_{\mathfrak z}\big( \langle k,\Omega(\r)\rangle+\Lambda(\r)\big)=
\langle v(\r),\p_{\mathfrak z}  L(\r) v(\r)\rangle%=$$
%$$=
=\langle v(\r),\p_{\mathfrak z}  L_{\textrm up}(\r) v(\r)\rangle +\O( \delta).
$$

Hence
$$\ab{\p_{\mathfrak z}\big( \langle k,\Omega(\r)\rangle+\Lambda(\r)\big)}\ge \delta_0-\Cte \delta \ge \frac{\delta_0}2,$$
which implies that $\ab{ \langle k,\Omega(\r)\rangle+\Lambda(\r)}$ is larger than $\ka$ outside a set
of  Lebesgue measure $\lsim \frac\ka{\delta_0}$. Since $L(\rho)$ is Hermitian this implies that 
$$\Leb \Sigma(L,\ka))\lsim   \ab{a}^{d}\frac\ka{\delta_0}$$
--  the dimension of $L$ is $\lsim   \ab{a}^{d}$. (This argument is valid  if $\Lambda(\r)$ is $\cC^1$ in the direction $\mathfrak z$ which can always be assumed
when $Q$ is analytic in $\r$. The non-analytic case follows by analytical approximation.)

We still have to sum up over, a priori, infinitely many $[a]$'s. However, since
$|\langle k, \Omega(\r)\rangle |\lsim \ab{k}\lsim N$, it follows,  by \eqref{la-lb-ter},  that
$$|\langle k, \Omega(\r)\rangle  + \Lambda(\r)|\geq \ab{\Lambda_a(\r)}-\delta-\Cte\ab{k}
\ge \ab{a}^{2}- c \langle a \rangle ^{-\beta2} -\delta-\Cte\ab{k}$$
 for some appropriate $a\in[a]$. Hence $\ab{ \langle k,\Omega(\r)\rangle+\Lambda(\r)}$ is larger than $\ka$
for  $ |a |\gsim N^{\frac1{2}}$. Summing up  over all $0<\ab{k}\le N$ and 
all  $ |a |\lsim N^{\frac1{2}}$ gives a set whose complement $\Sigma$ verifies the estimate.

\medskip

Consider now $a\in\F$ and let $L(\r)=\big(\langle k,\Omega\rangle I- {\mathbf i}JH\big)$. It follows, by \eqref{hypoB} and  \eqref{ass}, that
$$\aa{L-L_{\textrm up}}_{\cC^{{s_*}}}\le   \de\leq  \frac12 \delta_0,$$ 
where $L_{\textrm up}(\r)=\big(\langle k,\Omega\rangle I- {\mathbf i}JH_{\textrm up}\big)$  -- now we are not dealing with
an Hermitian operator.

If now $L_{\textrm up}$ is $\delta_0$-invertible, then
$L$ will be  $\frac{\delta_0}2$-invertible. Otherwise, by  assumption A2(iii),
there exists a unit vector ${\mathfrak z}$ and an integer  $1\le j\le {s_*}$ such that
$$\ab{ \p_{\mathfrak z}^j \det L_{\textrm up}(\r) }\ge \delta_0 
\aa{L_{\textrm up}}_{\cC^{j}(\D)}\aa{L_{\textrm up}}_{\cC^{0}(\D)}^{m-2}, \quad \forall \r\in \D.$$
Since, by convexity estimates (see \cite{Ho}),
$$\ab{ \p_{\mathfrak z}^j \det L_{\textrm up}(\r) }\le \Cte \aa{L_{\textrm up}}_{\cC^{j}(\D)}\aa{L_{\textrm up}}_{\cC^{0}(\D)}^{m-1}$$
and
$$\ab{ \p_{\mathfrak z}^j(\det L(\r)-\det L_{\textrm up}(\r))} \le \Cte  \delta\big(\aa{L_{\textrm up}}_{\cC^{j}}+\delta \big)
( \aa{ L}_{\cC^{0}(\D)}     +\delta)^{m-2},$$
this implies that 
$$\ab{ \p_{\mathfrak z}^j \det L(\r) }\ge (\delta_0-\Cte\delta)\aa{ L_{\textrm up}}_{\cC^{1}(\D)}\aa{ L_{\textrm up}}_{\cC^{0}(\D)}^{m-1}, \quad \forall \r\in \D,$$
which is $\ge\frac{\delta_0}2$ if $\delta$ is sufficiently small.

Then, by Lemma  \ref{lTransv1},
$$
{\ab{\det L(\r)}}\ge \ka\,  { \aa{ L }_{\cC^{j}}^{m-1} } \,,
$$
outside a set of Lebesgue measure
$$\le \Cte (\frac\ka{\de_0})^{\frac1j}.$$
Hence, by Cramer's rule, 
$$
\Leb \Sigma(L,\ka)\le \Cte (\frac\ka{\de_0})^{\frac1j}
\le  \Cte (\frac\ka{\de_0})^{\frac1j}.$$
Summing up  over all $\ab{k}\le N$  gives the first estimate.

\medskip

The second estimate follows from the mean value theorem and the bound
$$\ab{\nabla_\r L_{k,[a]}(\r)}\le N(\chi+\delta).$$
\end{proof}

\begin{lemma}\label{lSmallDiv3}
Let  
$$L_{k,[a],[b]}=(\langle k,\Omega\rangle I-{\mathbf i}\cAd_{JA})_{[a]}^{[b]}.$$
There exists  a constant $C$ such that if \eqref{ass} holds, then,
$$\bigcup_{\begin{subarray}{c}0<\ab{k}\le N\\ [a],[b]\end{subarray}} \Sigma(L_{k,[a],[b]},\ka)
\le C (N \Delta)^{\exp}(\frac{\ka}{\delta_0})^{\alpha}(\frac{\chi}{\delta_0})^{1-\alpha}$$
and
$$\dist(\D\setminus \Sigma(L_{k,[a],[b]},\ka),\Sigma(L_{k,[a],[b]},\frac\ka2))>
\frac{1}{C}\frac\ka{\Delta^{\exp}} N\chi,$$
for any $\ka>0$. Here
$$\alpha=\min\big(\frac{\beta_2\vark}{\beta_2\vark+2d(\beta_2+\vark)},\frac1{s_*}\big).$$

(The exponent $\exp$ only depends  on $d$, $\#\cA$ and $\tau,\beta_2,\vark$. 
$C$ is an absolute constant that depends on $c,\tau,\beta_2,\beta_3$ and $\vark$. $C$ also depend on $\sup_\D\ab{\Omega_{\textrm up}}$ and $\sup_\D\ab{H_{\textrm up}}$, but stays bounded
when these do.)

\end{lemma}

\begin{proof} 
Consider first  $a,b \in\F$. This case is treated as the operator
$L(\r)=\big(\langle k,\Omega\rangle I- {\mathbf i}JH\big)$  in the previous lemma.

\medskip

Consider then  $a\in\L_\infty$ and $b \in\F$. Then  $L_{k,[a]}$ is conjugate to a sum of two operators of the form
$$X\mapsto \langle k,\Omega(\r) \rangle X + Q_{[a]}(\r)X+X{\mathbf i}JH(\r)$$
(see the discussion in section \ref{normalformmatrices}). This operator in not Hermitian, but only ``partially'' Hermitian: it decomposes as an orthogonal sum of operators of the form $L(\r,\Lambda(\rho))$, where
$$L(\r,\lambda):X\mapsto \langle k,\Omega(\r) \rangle X+\lambda X+ {\mathbf i}XJH(\r),$$
and $\Lambda(\rho) $ is an eigenvalue of $Q_{[a]}(\r)$.

If we let
$$L_{\textrm up}(\r,\lambda):X\mapsto \langle k,\Omega(\r) \rangle X + \lambda X+X{\mathbf i}JH_{\textrm up}(\r),$$
then it follows, from \eqref{hypoB} and  \eqref{ass}, that   
$$\aa{L(\cdot,\lambda)-L_{\textrm up}(\cdot,\lambda)}_{\cC^{{1}}(\D)} \leq  \de\leq \cte\delta_0.$$
If $L_{\textrm up}(\rho,\Lambda_a(\rho))$ is $\delta_0$-invertible for all $a\in[a]$, then this  implies that, for any
eigenvalue $\Lambda(\rho)$ of $Q_{[a]}(\r)$,
$L(\rho,\Lambda(\rho))$ is $\frac{\delta_0}2$-invertible.

Otherwise, by  Assumption A2$(ii)$, there exists a unit vector ${\mathfrak z}$ such that
$$
\ab{\p_{\mathfrak z}P_{\textrm up}(\r,\Lambda_a(\r))+\p_\lambda P_{\textrm up}(\r,\Lambda_a(\r))
\langle v,\p_{\mathfrak z} Q_{\textrm up}(\r) v\rangle}
\ge \delta_0\aa{L_{\textrm up}}_{\cC^{1}(\D)}\aa{L_{\textrm up}}_{\cC^{0}(\D)}^{m-2}$$
for all $\r\in \D$,  all $a\in[a]$ and for any unit-vector $v\in (\C^2)^{[a]}$.
If now
$$P(\r,\lambda)=  \det L(\r,\lambda),$$
then, for any eigenvalue  $\Lambda(\rho)$, $\cC^1$ in the direction 
${\mathfrak z}$, and any associated unit eigenvector $v(\rho)$,
$$\frac{d}{d_{\mathfrak z}}P(\r,\Lambda(\r))=
\p_{\mathfrak z}P(\r,\Lambda(\r))+\p_\lambda P(\r,\Lambda(\r))\langle v(\r),\p_{\mathfrak z}Q(\r)v(\r)\rangle=$$
$$=\p_{\mathfrak z}P_{\textrm up}(\r,\Lambda_a(\r))+\p_\lambda P_{\textrm up}(\r,\Lambda_a(\r))
\langle v(\r),\p_{\mathfrak z}Q_{\textrm up}(\r)v(\r)\rangle
+\O(\delta  \aa{L_{\textrm up}}_{\cC^{1}(\D)}\aa{L_{\textrm up}}_{\cC^{0}(\D)}^{m-1}).$$
Hence
$$
\ab{\frac{d}{d_{\mathfrak z}}P(\r,\Lambda(\r))}\ge\frac{\delta_0}2  \aa{L_{\textrm up}}_{\cC^{1}(\D)}\aa{L_{\textrm up}}_{\cC^{0}(\D)}^{m-2}.$$

Then 
$$\ab{ \frac {P(\r,\Lambda(\r))} {|| L||_{\cC^{0}(\D)}^{m-1}}}\ge\ka$$
outside a set of Lebesgue measure $\lsim\frac\ka{\de_0}$.
Hence, by Cramer's rule, 
$$
\Leb \Sigma(L,\ka)\le \Cte \frac\ka{\de_0}.$$
Since $|\langle k, \Omega(\r)\rangle |\lsim \ab{k}\lsim N$, it follows,  by \eqref{la-lb-ter},  that for any eigenvalue $\alpha(\r)$ of $JH(\r)$,
$$|\langle k, \Omega(\r)\rangle  + \Lambda(\r)+\alpha(\r)|\geq \ab{\Lambda_a(\r)}  -\delta-\Cte\ab{k}
\ge \ab{a}^{2}-c \langle a \rangle ^{-\beta_1} -\delta -\Cte\ab{k}$$
 for some appropriate $a\in[a]$. Hence, 
$\Sigma(L,\ka)=\emptyset$ for $ |a |\gsim N^{\frac1{2}}$.

Summing up  over all $0<\ab{k}\le N$ and 
all  $ |a |\lsim N^{\frac1{2}}$ gives the first estimate.

\medskip

Consider finally $a,b \in\L_\infty$. Then  $L_{k,[a],[b]}$  is conjugate to a sum of four operators of the forms
$$X\mapsto \langle k,\Omega\rangle X + Q_{[a]}X+X{}^tQ_{[b]}$$
and
$$X\mapsto \langle k,\Omega\rangle X + Q_{[a]}X-XQ_{[b]}.$$
These operators are Hermitian with respect to the Hilbert-Schmidt norm on the space of matrices $X$. Changing from the operator norm to the Hilbert-Schmidt norm  (and conversely) changes any estimate by a factor that depends on the dimension of the space of matrices $X$, which, we recall, is bounded by some power of $\Delta$.

With this modification, the first operator is treated exactly as the operator $X\mapsto \langle k,\Omega\rangle X + Q_{[a]}X$
in the previous lemma, so let us concentrate on the second one, which we shall call $L=L_{k,[a],[b]}$. 
It follows as in the previous lemma
that  the Lebesgue measure of $\Sigma(L,\ka)$ is 
$\lsim   (\ab{a}\ab{b})^{d} \frac{\ka}{\delta_0}$  --
recall that the operator is of dimension $\lsim  (\ab{a}\ab{b})^{2d} $.

The problem now is the measure estimate of $\bigcup\Sigma(L_{k,[a],[b]},\ka)$ since, a priori, there may be infinitely many 
$\Sigma(L_{k,[a],[b]},\ka)$ that are non-void. 
We can assume without restriction that $\ab{a}\le \ab{b}$.
Since $|\langle k, \Omega(\r)\rangle |\le\Cte \ab{k}\le\Cte N$, it is enough to 
consider  $\ab{b}-\ab{a} \le\Cte N$.

Suppose first that $[a]$ and $[b]$ are $\not=[0]$. Let $\alpha(\rho)$ and $\beta(\rho)$ be eigenvalues of $Q_{[a]}(\rho)$ and $Q_{[b]}(\rho)$ respectively, and chose
$a,b$ such that
$$\ab{\alpha(\rho)-\Lambda_a(\rho)}\le \de\frac1{\langle a \rangle^\varkappa},\quad
\ab{\beta(\rho)-\Lambda_b(\rho)}\le \de\frac1{\langle b \rangle^\varkappa}.$$
 Using Assumption A3 now gives
   $$|\langle k,\Omega(\r)\rangle \ +\alpha(\r)-\beta(\r)|\ge
   |\langle k,\Omega_{\textrm up}(\r)\rangle \ +\Lambda_a(\r)-\Lambda_b(\r)|-\ab{k}\de-2\de\frac1{\langle a\rangle^\vark} $$
   $$
  \ge |\langle k,\Omega_{\textrm up}(0)\rangle \ +\Lambda_a(0)-\Lambda_b(0)|- \chi(\ab{k}+2)-\de(\ab{k}+2)\ge
   \frac{\beta_4}{\ab{k}^\tau}-6\ab{k}\chi,$$
   and this is $\ge\ka$ unless
   $$\ab{k}\ge K\approx(\frac{\beta_3}{\chi})^{\frac1{\tau+1}}.$$
Recall that $\chi\ge\de_0$, by convention, and that $\ka\le\de_0$, because otherwise the lemma is trivial.

From now on we only consider  $K\le\ab{k}\le N$. By Assumption A2, there exists a unit vector ${\mathfrak z}$ such that
$$
\ab{\p_{\mathfrak z}\langle k,\Omega(\r)\rangle }\ge\de_0.$$
Since  $\ab{k}\le N$ and  $\ab{a}^2-\ab{b}^2$ are integers, it follows that (for any $\ka'$)
 $$|\langle k,\Omega(\r)\rangle \ +\ab{a}^2-\ab{b}^2|\ge 2\ka'$$
 for  all $a,b$ and all $\r$ outside a set of Lebesgue measure $\lsim N\frac{\ka'}{\de_0}$. Summing up over all 
 $K\le\ab{k}\le N$  gives a set $\Sigma_1$ of  Lebesgue measure 
 $$\lsim N^{\exp}\frac{\ka'}{\de_0}.$$
 
 By \eqref{la-lb} it follows that, for
 $\r$ outside of $\Sigma_1$,
  $$|\langle k,\Omega(\r)\rangle \ +\Lambda_a(\r)-\Lambda_b (\r)|\ge \ka',$$
  if just
  $$\ab{a}^{\beta_2}\ge2\frac{c}{\ka'}.$$
Then
 $$|\langle k,\Omega(\r)\rangle \ +\alpha(\r)-\beta(\r)|\ge
 \ka'-2\de\frac1{\langle a\rangle^\vark}$$
 which is $\ge\ka$ if $\ka'\ge 2\ka$ and
   $$\ab{a}^\vark\ge 2(\frac{\de}{\ka'}).$$
   
  Let
  $$M=2\max(   (\frac{c}{\ka'})^{\frac1{\beta_2}} ,   (\frac{\de_0}{\ka'})^{\frac1{\vark}}   ).$$
  Then it only remains to consider $[a]$ and $[b]$ with $\ab{a}\le M$ and $\ab{b}\le M+\Cte N$.
  We have seen above that the the Lebesgue measure of each $\Sigma(L_{k,[a],[b]},\ka)$ is 
$\lsim   (\ab{a}\ab{b})^{d} \frac{\ka}{\delta_0}$. Summing up over all these $a$ and $b$ gives 
a set $\Sigma_2$ of  Lebesgue measure 
$$\lsim N^{\exp}M^{2d}\frac{\ka}{\de_0}.$$ 

Suppose now that $[a]$ or $[b]$ is $=[0]$. Then $\ab{a}$ and $\ab{b}$ are $\lsim c+N\lsim N$.
Summing up over all these $a$ and $b$ gives 
a set $\Sigma_3$ of  Lebesgue measure 
$$\lsim N^{\exp}\frac{\ka}{\de_0}.$$

The union of $\Sigma_1$, $\Sigma_2$ and $\Sigma_3$ has Lebesgue measure
$$\lsim N^{\exp} \big(      \frac{\ka'}{\de_0} + M^{4d}\frac{\ka}{\de_0}\big) \lsim 
N^{\exp} \big(      \frac{\ka'}{\de_0} +  (\frac1{\ka'})^\theta\frac{\ka}{\de_0}\big)\qquad \theta=4d(\frac1{\beta_2}+\frac1{\vark}).$$ 
Take now $\ka'=\ka^{\frac1{1+\theta}}$ and observe that $N\chi^{\frac1\tau}\gsim 1$ (because $N\ge K$). Then the bound becomes
$$\lsim N^{\exp}    (\frac{\ka}{\de_0})^{\frac1{1+\theta}}  (\frac{\chi}{\de_0})^{\frac\theta{1+\theta}} $$
 ( with a new and larger exponent $\exp$).

\end{proof}

\section{Homological equation}\label{s5}

Let $h$ be a  normal form Hamiltonian \eqref{normform},
$$
h(r,w,\r)=\langle \Omega(\r), r\rangle +\frac 1 2\langle w, A(\r) w\rangle\in\NF_{\varkappa}(\Delta,\delta)$$
-- recall the convention \eqref{Conv} -- and assume $\vark>0$ and 
 \be\label{ass1}
 \delta \le \frac{1}{C}c' ,\ee
 where $C$ is to be determined. Let
$$\ga=(\ga,m_*)\ge \ga_*=(0,m_*).$$

\begin{remark}\label{rAbuse}
Notice the abuse of notations here. It will be clear from the context when $\ga$ is a two-vector, like in $\aa{\cdot}_{\ga,\vark}$,
and when it is a scalar, like in $e^{\ga d}$.
\end{remark}

Let $f\in \cT_{\ga,\varkappa,\D}(\s,\mu)$. In this section we shall construct a jet-function  $S$ that solves the 
{\it non-linear
\footnote{\ ``non-linear'' because the solution depends non-linearly on $f$}
 homological equation}
\be\label{eqNlHomEq}
\{ h,S \}+ \{ f-f^T,S \}^T+f^T=0\ee
as good as possible  -- the reason for this will be explained in the beginning of the next section. 
In order to do this we shall start by analysing  the {\it homological equation}
\be \label{eqHomEq}
\{ h,S \}+f^T=0.
\ee
We shall solve this equation modulo some ``cokernel'' and modulo an ``error''.

\medskip 

\subsection{Three components of the homological equation}\label{ssFourComponents}

Let us write 
$$f^T(\theta,r,w)=f_r(r,\theta)+\langle f_w(\theta),w\rangle+\frac 1 2 \langle f_{ww}(\theta)w,w \rangle$$
and recall that, by Proposition \ref{lemma:jet}, $f^T\in \cT_{\ga,\varkappa,\D}(\s,\mu)$.
Let
$$
S(\theta,r,w)=S_r(r,\theta)+\langle S_w(\theta),w\rangle+
\frac 1 2 \langle S_{ww}(\theta)w,w \rangle,$$
where  $f_r$ and $S_r$ are affine functions in $r$ -- here we have not indicated the dependence on $\r$.

Then the Poisson bracket $\{h, S\}$ equals
\begin{multline*}
-\big( \p_{\Omega} S_r(r, \theta)  + \langle \p_{\Omega} S_w(\theta), w\rangle
+ \frac12 \langle \p_{\Omega} S_{ww}(\theta),w\rangle + \\
+\langle AJ S_w(\theta),w\rangle + \frac12\langle AJ S_{ww}(\theta)w,w\rangle  - \frac12\langle  S_{ww}(\theta)JAw,w\rangle
\end{multline*}
where $\p_{\Omega}$ denotes the derivative of the angles $\theta$ in direction $\Omega$.
Accordingly the  homological equation \eqref{eqHomEq}  
 decomposes into three linear equations:
$$\left\{\begin{array}{l}      
\p_{\Omega} S_r(r,\theta) =f_r(r,\theta),\\  
 \p_{\Omega} S_w(\theta)   -AJ  S_w(\theta)= f_w(\theta),\\ 
 \p_{\Omega} S_{ww}(\theta)   - AJS_{ww}(\theta)  +S_{ww}(\theta)JA=f_{ww}(\theta).
\end{array}\right.$$

\subsection{The first equation}\label{homogene}

\begin{lemma}\label{prop:homo12}
There exists   constant $C$ such that if \eqref{ass1} holds, then,
 for any  $N\ge1$ and $\ka>0$,
 there exists  a closed  set $\D_1= \D_1(h,\ka,N)\subset \D$,  satisfying
$$\Leb (\D\setminus \D_1)\leq C N^{\exp} \frac{\ka}{\de_0}$$
and there exist $\cC^{{s_*}}$ functions $S_r$ and $R_r$ on $\C^{\cA}\times\T^\cA\times \D\to\C$,  
real holomorphic in $r,\theta$, such that for all $\r\in\D_1$
 \be\label{homo1}
 \p_{\Omega(\r)}S_r (r,\theta,\r)  =f_r(r,\theta,\r)-\hat f_r(r,0,\r)
 -R_r(\theta,\r)
\quad  \footnote{\ $\hat f_r(r,0,\r)$ is the $0$:th Fourier coefficient, or the mean value, of the function $\theta\mapsto  f_r(r,\theta,\r)$}
\ee
and for all $(r,\theta,\r)\in \C^{\cA}\times \T^\cA_{\s'}\times \D$, $\ab{r}<\mu$, $\s'<\s$, and $|j|\le{{s_*}}$
 \begin{align} \label{homo1S}
 |\p_\r^jS_r(r,\theta,\r)|\leq &
C \frac{1}{\ka(\s-\s')^{n}}\big(N\frac{\chi}{\ka}\big)^{|j|} 
  |f^T|_{\begin{subarray}{c}\s,\mu\ \ \\ \ga, \vark,\D  \end{subarray}} ,\\ \label{homo1R}
 |\p_\r^j R_r(r,\theta,\r)|\leq & C\frac{  e^{- (\s-\s')N}}  {  (\s-\s')^{n}}|f^T|_{\begin{subarray}{c}\s,\mu\ \ \\ \ga, \vark,\D  \end{subarray}}\,. 
\end{align}
Moreover, $S_r(\cdot,\r)=0$ for $\r$ near the boundary of $\D$.

(The  exponent $\exp$  only depends on $n=\#\cA$, and $C$ is an absolute constant.)
\end{lemma}

\begin{proof} Written in Fourier components the 
equation \eqref{homo1} then becomes, for $k\in \Z^{\cA}$,
$$L_k(\r)\hat S(k)=:
 \langle k, \Omega(\r) \rangle \hat S(k)=-{\mathbf i}(\hat F(k)-\hat R(k))$$
where we have written $S,F$ and $ R$ for $S_r, (f_r-\hat f_r)$ and $R_r$ respectively.
Therefore \eqref{homo1} has the (formal) solution 
$$S(r,\theta,\r)=\sum\hat S (r,k,\r) e^{{\mathbf i}\langle k,\theta\rangle}\quad\textrm{and}\quad
R(r,\theta,\r)=\sum\hat F (r,k,\r) e^{{\mathbf i}\langle k,\theta\rangle}$$
with
$$\hat S(r,k,\r)= 
\left\{\begin{array}{ll}
-L_k(\r)^{-1}{\mathbf i}\hat F(r,k,\r) & \textrm{ if }  0< |k|\le N\\
0 & \textrm{ if not}  
\end{array}\right.$$
and
$$\hat R(r,k,\r)= 
\left\{\begin{array}{ll}
\hat F(r,k,\r ) & \textrm{ if }   |k|> N\\
0& \textrm{ if not}.
\end{array}\right.$$
%Here $\Omega:\D\to\R^n$  is $\cC^{{s_*}}$ and verifies  $$|\Omega-\Omega_{\textrm{up}}|_{\cC^{{s_*}}(\D)}\le \delta.$$

By Lemma~\ref{lSmallDiv1}
$$ ||(L_k(\r))^{-1}||\le \frac1\ka\,
 $$
for all $\r$ outside some  set $\Sigma(L_k,\kappa)$ such that
$$\dist(\D\setminus \Sigma(L_k,\ka),\Sigma(L_k,\frac\ka2))\ge \cte\frac\ka{N\chi}$$
and
$$\D_1=\D\setminus \bigcup_{0<|k|\le N}\Sigma(L_k,\ka)$$
fulfils the estimate of the lemma.

 For $\r\notin \Sigma(L_k,\frac\kappa2)$ we get
 $$ |\hat S (r,k,\r)| \le \Cte\frac{1}{\ka}|\hat F(r,k,\r)|\,.$$
 Differentiating the formula for $\hat S(r,k,\r)$ once we obtain
 $$\p^j_\r \hat S(r,k,\r)=
 \Big(
 \ -\frac{{\mathbf i}}{ \langle\Omega, k\rangle}\p^j_\r \hat F(r,k,\r)+
 \ \frac{{\mathbf i}}{ \langle\Omega, k\rangle^2} \langle \p^j_\r\Omega, k\rangle\hat F(r,k,\r)\Big)
 $$
 which gives, for $\r\notin \Sigma(L_k,\frac\kappa2)$,
 $$
 |\p^j_\r\hat S(r,k,\r)|\le \Cte \frac1\kappa( N\frac{\chi}{\kappa})\max_{0\le l\le j}|\p^l_\r \hat F(r,k,\r)|.
 $$
(Here we used that $|\p_\r \Omega(\rho)|\le \chi+\de$. )
The higher order derivatives are estimated in the same way and this gives
 $$
 |\p_\r^j\hat S(r,k,\r)|\le \Cte \frac1\kappa(N\frac{\chi}{\kappa})^{ |j |}\max_{0\le l\le j} |\p^l_\r\hat F(r,k,\r)|
 $$
 for any $ |j |\le{{s_*}}$,  where $\Cte$ is an absolute constant. 
 
 By Lemma  \ref{lExtension},
there exists a $\cC^\infty$-function $g_k:\D\to\R$, being $=1$ outside $\Sigma(L_k,\ka)$ and $=0$ on
$\Sigma(L_k,\frac\ka2)$ and such that for all $j\ge 0$
$$| g_k |_{\cC^j(\D)}\le (\Cte\frac{N\chi}{\ka})^j.$$
 Multiplying $\hat S(r,k,\r)$ with $g_k(\r)$ gives a $\cC^{{s_*}}$-extension of $\hat S(r,k,\r)$ from 
 $\D\setminus \Sigma(L_k,\ka)$  to  $\D$ satisfying the same bound \eqref{homo1S}.

It follows now, by a classical argument,  that the formal solution converges and that
$| \p_\r^j  S(r,\theta,\r)|$ and $|\p_\r^j  R(r,\theta,\r)|$
fulfils the estimates of the lemma.
When summing up  the series for  $|\p_\r^j  R(r,\theta,\r)|$ we get a term $e^ {-\frac1C(\s-\s')N}$ (because of truncation of Fourier modes), but the factor $\frac1C$
disappears by replacing $N$ by $CN$.  

By construction $S$ and $R$ solve equation \eqref{homo1} for any $\r\in \D_1$.

If we multiply  $\hat S(r,k,\r)$ by a second $\cC^\infty$ cut-off function $h_k:\D\to\R$ --  which is $=1$ at a distance 
$\ge \frac{\ka}{N\chi}$ from the boundary of $\D$ and $=0$ near this boundary  --  then the new function will satisfy the bound \eqref{homo1S},
it will solve the  equation \eqref{homo1} on a new domain, smaller but still satisfying the measure bound of the Lemma, and it will vanish near the boundary of $\D$.
\end{proof}

\subsection{The second equation}\label{s5.3} 
Concerning the second component  of the homological equation we have

\begin{lemma}\label{prop:homo3}
There exists  an absolute constant $C$ such that if \eqref{ass1} holds, then,
for any  $N\ge1$ and 
$$0<\ka\le c',$$
there exists  a closed set $\D_2=\D_2(h,\ka,N)\subset \D$,  satisfying
 $$
 \Leb (\D\setminus \D_2)\leq   C N^{\exp} 
 (\frac{\ka}{\delta_0})^{\frac1{{s_*}}}, $$
and there exist $\cC^{{s_*}}$-functions $S_w$ and $R_w$ $:\T^\cA  \times \D\to Y_\ga$,  
real holomorphic in $\theta$, such that for $\r\in \D_2$
\be\label{homo2}\p_{\Omega(\r)} S_w(\theta,\r) -A(\r) JS_w(\theta,\r)=
f_w(\theta,\r)-R_w(\theta,\r)
\ee
and for all $(\theta,\r)\in \T^\cA_{\s'}\times \D$, $\s'<\s$, and $|j|\le {{s_*}}$
 \begin{align} \label{homo2S}
|| \p_\r^j S_w(\theta,\r)||_{\ga}\leq &
 C\frac{1}{\ka (\s-\s')^{n}}\big( N\frac{\chi  }{\ka}\big)^{|j|}
  |f^T|_{\begin{subarray}{c}\s,\mu\ \ \\ \ga, \vark,\D  \end{subarray}} \\ \label{homo2R}
   ||  \p_\r^j R_w(\theta,\r)||_{\ga}\leq & C\frac{ e^{-(\s-\s')N} } {(\s-\s')^{n}}
 |f^T|_{\begin{subarray}{c}\s,\mu\ \ \\ \ga, \vark,\D  \end{subarray}}.
\end{align}
Moreover, $S_w(\cdot,\r)=0$ for $\r$ near the boundary of $\D$.

(The exponent $\exp$ only depends  on $d$ and $\#\cA$. $C$ is an absolute constant that depends
on  $c$. $C$ also depend on $\sup_\D\ab{\Omega_{\textrm up}}$ and $\sup_\D\ab{H_{\textrm up}}$, but stays bounded
when these do.)
\end{lemma}

\proof
Let us re-write \eqref{homo2} in the complex variables  $,z=(\xi\eta)$ described in section
\ref{ssUnperturbed}. The quadratic form  $(1/2)\langle w, A(\rho) w\rangle\ $ gets transformed, by $w=Uz$,  to
$$ \langle\xi, Q(\rho)\eta\rangle + \frac12 \langle z_{\F}, H'(\rho) z_{\F}\rangle,$$ 
where $Q'$ is a Hermitian matrix and  $H'$ is a real symmetric matrix. Then we make in \eqref{homo2} the 
substitution   $ S={}^t\!US_w$, $R={}^t\! UR_w$ 
and $F={}^t\!Uf_w$, where $S={}^t(S_\xi, S_\eta,S_{\F})$, etc. 
In this notation eq.~\eqref{homo2}  decouples into the equations 
\begin{align*}
 \p_{\Omega}  S_\xi + {\mathbf i} QS_\xi= F_\xi-R_\xi,\\
 \p_{\Omega}  S_\eta-  {\mathbf i}{}^t\!QS_\eta= F_\eta-R_\eta\\
\p_{\Omega} S_{\F}-  HJS_{\F}= F_{\F}-R_{\F}.
\end{align*}

Let us consider the first equation. Written in the  Fourier components it becomes
\be\label{homo2.10}
( \langle k, \Omega(\r) \rangle I  + Q) \hat S_\xi (k)=-{\mathbf i}(\hat F_\xi(k)-\hat R_\xi(k)).\ee
This equation decomposes into its ``components'' over the  blocks $[a]=[a]_\Delta$ and takes the form
\be\label{homo2.1}L_{k,[a]}(\r)\hat S_{[a]}(k)=:
( \langle k, \Omega(\r) \rangle  + Q_{[a]}) \hat S_{[a]}(k)=-{\mathbf i}(\hat F_{[a]}(k)-\hat R_{[a]}(k))\ee
--  the matrix $Q_{[a]}$ being the restriction of $Q_\xi$ to $[a]\times [a]$, the vector $F_{[a]}$ being 
the restriction of $F_\xi$ to $[a]$ etc. 

Equation  \eqref{homo2.1} has the (formal) solution
$$\hat S_{[a]}(k,\r)= 
\left\{\begin{array}{ll}
-(L_{k,[a]}(\r))^{-1}{\mathbf i}\hat F_{[a]}(k,\r) & \textrm{ if }   |k|\le N\\
0 & \textrm{ if not}  
\end{array}\right.$$
and 
$$\hat R_{a}(k,\r)= 
\left\{\begin{array}{ll}
\hat F_{a}(k,\r ) & \textrm{ if }   |k|> N\\
0& \textrm{ if not}.
\end{array}\right.$$

For $k\not=0$, by Lemma~\ref{lSmallDiv2},
$$ ||(L_{k,[a]}(\r))^{-1}||\le \frac1\ka\,
 $$
for all $\r$ outside some  set $\Sigma(L_{k,[a]},\kappa)$ such that
$$\dist(\D\setminus \Sigma(L_{k,[a]},\ka),\Sigma(L_{k,[a]},\frac\ka2))\ge \cte\frac\ka{N\chi}$$
and
$$\D_2=\D\setminus \bigcup_{\begin{subarray}{c} 0< |k|\le N\\ [a]\end{subarray}}\Sigma_{k,[a]}(\ka),$$
fulfils the required estimate. 

For $k=0$, it follows by \eqref{ass1} and \eqref{laequiv} that
$$ ||(L_{k,[a]}(\r))^{-1}||\le \frac1{c'}\le \frac2{\ka}\, . $$

We then get, as in the proof of Lemma~\ref{prop:homo12}, that $\hat S_{[a]}(k,\cdot)$ and 
$\hat R_{[a]}(k,\cdot)$ have  $\cC^{{s_*}}$-extension to $\D$ satisfying
$$|| \p_\r^j \hat S_{[a]} (k,\r)|| \leq \Cte\frac{1}{\ka}\big(N\frac{\chi}{\ka}\big)^{|j|}
\max_{0\le l\le j} || \p_\r^l \hat F_{[a]}(k,\r)||
$$
and
$$||  \p_\r^j R_{[a]}(k,\r)||\leq  \Cte || \p_\r^j \hat F_{[a]}(k,\r)||,$$
and satisfying \eqref{homo2.1} for $\r\in\D_2$.

These estimates imply that
$$|| \p_\r^j \hat S_\xi(k,\r)||_\ga\leq \Cte\frac{1}{\ka}\big(N\frac{\chi}{\ka}\big)^{|j|}
\max_{0\le l\le j} || \p_\r^l \hat F_\xi (k,\r)||_\ga
$$
and
$$||  \p_\r^j R_\xi (k,\r)||_\ga\leq  \Cte || \p_\r^j F_\xi(k,\r)||_\ga.$$
Summing up the Fourier series, as in Lemma~\ref{prop:homo12}, we get
$$|| \p_\r^j S_\xi (\theta,\r)||_\ga\leq \Cte\frac{1}{\ka (\s-\s')^{n}}\big(N\frac{\chi}{\ka}\big)^{|j|}
\max_{0\le l\le j} \sup_{|\Im\theta|<\s}|| \p_\r^l F_\xi(\cdot,\r)||_\ga
$$
and
$$||  \p_\r^j R_\xi(\theta,\r)||_\ga \leq  \Cte\frac{ e^{-\frac1{\Cte}(\s-\s')N} } {(\s-\s')^{n}}
\sup_{|\Im\theta|<\s}|| \p_\r^j F_\xi(\cdot,\r)||_\ga
$$
for $(\theta,\r)\in \T^{\cA}_{\s'}\times \D$, $0<\s'<\s$, and $|j|\le{{s_*}}$. This implies the estimates \eqref{homo2S} and \eqref{homo2R}  --  the factor $\frac1{\Cte}$
disappears by replacing $N$ by $\Cte N$. 

The other two equations are treated in exactly the same way.
\endproof

\subsection{The third equation}\label{s5.4}
 Concerning the third component of the homological equation, \eqref{eqHomEq}, we have the following result.
 
\begin{lemma}\label{prop:homo4}
There exists  an absolute constant $C$ such that if \eqref{ass1} holds, then,
for any  $N\ge1$, $\Delta'\ge \Delta\ge 1$,  and 
$$\ka\le\frac1C c',$$
there exist a closed subset $\D_3=\D_3(h, \ka,N)\subset \D$, satisfying 
$$\Leb(\D\setminus {\D_3})\le C (\Delta N)^{\exp_1}  (\frac{\ka}{\delta_0})^{\alpha}(\frac{\chi}{\delta_0})^{1-\alpha}$$
and there exist real $\cC^{{s_*}}$-functions 
$B_{ww}: \D\to \cM_{\ga,\varkappa} \cap \NF_{\Delta'} $ and 
$S_{ww}$, $R_{ww}=R_{ww}^F+R_{ww}^s:\T^{\cA}\times \D\to \cM_{\ga,\varkappa}$,
real holomorphic  in $\theta$, such that for all $\r\in\D_3$
\begin{multline}\label{homo3} 
\p_{\Omega(\r)} S_{ww}(\theta,\r) -A(\r)JS_{ww}(\theta,\r)+
S_{ww}(\theta,\r)JA(\r)=\\
f_{ww} (\theta,\r)-B_{ww}(\r)-R_{ww}(\theta,\r)
 \end{multline}
and for all $(\theta,\r)\in \T^\cA_{\s'}\times \D$, $\s'<\s$, and $|j|\le {{s_*}}$

\be\label{homo3S}
\aa{\p_\r^j S_{ww}(\theta,\r)}_{\ga,\varkappa}\leq 
C\Delta'\frac{\Delta^{\exp_2}e^{2\ga d_\Delta}}{\ka(\s-\s')^{n}}
\big(N\frac{\chi+\de}{\ka}\big)^{|j|}
\ab{f^T}_{\begin{subarray}{c}\s,\mu\ \ \\ \ga, \vark,\D  \end{subarray}},\ee

\be \label{homo3B}
\aa{\p_\r^j B_{ww}(\r)}_{\ga',\vark}\leq  C \Delta' \Delta^{\exp_2} \ab{f^T}_{\begin{subarray}{c}\s,\mu\ \ \\ \ga, \vark,\D  \end{subarray}},\ee
and 
\be\label{homo3R}
\left\{\begin{array}{l}
\aa{ \p_\r^j  R_{ww}^F(\theta,\r)}_{\ga,\vark}\leq  
C\Delta' \Delta^{\exp_2}\left(\frac{e^{-(\s-\s')N}}{ (\s-\s')^{n}}\right)
\ab{f^T}_{\begin{subarray}{c}\s,\mu\ \ \\ \ga, \vark,\D  \end{subarray}}\\
\aa{ \p_\r^j  R_{ww}^s(\theta,\r)}_{\ga',\vark}\leq  
C\Delta' \Delta^{\exp_2}e^{-(\ga-\ga')\Delta'}
\ab{f^T}_{\begin{subarray}{c}\s,\mu\ \ \\ \ga, \vark,\D  \end{subarray}}
\end{array}\right.,\ee
for any $\ga_*\le \ga'\le\ga$.

Moreover, $S_{ww}(\cdot,\r)=0$ for $\r$ near the boundary of $\D$.

The exponent
$\alpha$ is a positive constant only depending on  $d,s_*,\vark$ and $\beta_2 $.
\footnote{\ $\alpha$ is the exponent of Lemma~\ref{lSmallDiv3}}
.

(The exponent $\exp$ only depends  on $d$, $n=\#\cA$ and $\tau,\beta_2,\vark$. The exponent $\exp_2$ only depends  
on $d,m_*,s_*$. 
$C$ is an absolute constant that depends on  $c,\tau,\beta_2,\beta_3$ and $\vark$. $C$ also depend on $\sup_\D\ab{\Omega_{\textrm up}}$ and $\sup_\D\ab{H_{\textrm up}}$, but stays bounded
when these do.)

\end{lemma}

\proof It is also enough to find complex solutions $S_{ww}$, $R_{ww}$
and $B_{ww}$ verifying the estimates, because then their real parts will do the job.

As in the previous section, and using the same notation, we re-write \eqref{homo3} in complex variables. 
So we introduce 
$S={}^t\! US_{\zeta,\zeta} U$, $R={}^t\!  UR_{\zeta,\zeta} U$, $B={}^t\! UB_{\zeta,\zeta} U$ and $F={}^t\! UJf_{\zeta,\zeta} U$.
In appropriate notation  \eqref{homo3}  decouples into  the equations
\begin{align*}
&\p_{\Omega}  S_{\xi\xi} +{\mathbf i} Q S_{\xi\xi}+ {\mathbf i} S_{\xi\xi}\ {}_tQ= F_{\xi\xi}-B_{\xi\xi}- R_{\xi\xi},\\ 
&\p_{\Omega}  S_{\xi\eta}  + {\mathbf i}Q S_{\xi\eta} -  {\mathbf i}S_{\xi\eta} Q= F_{\xi\eta}-B_{\xi\eta}- R_{\xi\eta},\\
&\p_{\Omega}  S_{\xi z_{\F}} +{\mathbf i} Q S_{\xi z_{\F}}+  S_{\xi z_{\F}}\ JH
= F_{\xi z_{\F}}-B_{\xi\xi}- R_{\xi z_{\F}},\\ 
&\p_{\Omega}  S_{z_{\F}z_{\F}} 
+HJ S_{z_{\F}z_{\F}}- S_{z_{\F}z_{\F}}JH= F_{z_{\F}z_{\F}}-
B_{z_{\F}z_{\F}}- R_{z_{\F}z_{\F}},
\end{align*}
and  equations for $ S_{\eta\eta},S_{\eta\xi}, S_{z_{\F}\xi }, S_{\eta z_{\F}},S_{z_{\F} \eta}$. Since those latter equations are of the same type as the first four, we shall concentrate on these first.

\smallskip

{\it First equation. } 
 Written in the  Fourier components it becomes
\be\label{homo3.10}
( \langle k, \Omega(\r) \rangle I  + Q) \hat S_{\xi\xi} (k)+\hat S_{\xi\xi} (k){}^tQ
=-{\mathbf i}(\hat F_{\xi\xi}(k)-\de_{k,0} B-\hat R_{\xi\xi}(k)).\ee
This equation decomposes into its ``components'' over the  blocks $[a]\times[b]$, $[a]=[a]_\Delta$,
and takes the form
\begin{multline}\label{homo3.1}
L(k,[a],[b],\r)\hat S_{[a]}^{[b]}(k)=:\langle k, \Omega(\r)\rangle \ \hat S_{[a]}^{[b]}(k) + Q_{[a]}(\r) \hat S_{[a]}^{[b]}(k)+ \\
\hat S_{[a]}^{[b]}(k)\  {}^tQ_{[b]}(\r)=
-{\mathbf i}(\hat F_{[a]}^{[b]}(k,\r) -\hat R_{[a]}^{[b]}(k)-\de_{k,0} B_{[a]}^{[b]})
\end{multline}
--  the matrix $Q_{[a]}$ being the restriction of $Q_{\xi\xi}$ to $[a]\times [a]$, the vector $F_{[a]}^{[b]}$ being 
the restriction of $F_{\xi\xi}$ to $[a]\times[b]$ etc. 

Equation \eqref{homo3.1} has the (formal)  solution:
$$\hat S_{[a]}^{[b]}(k,\r)=
\left\{\begin{array}{ll}
-L(k,[a],[b],\r)^{-1}{\mathbf i}\hat F_{[a]}^{[b]}(k,\r) &
\textrm{ if } \dist([a],[b])\le\Delta'\ \textrm{and}\ \  |k|\le N\\
0 &  \textrm{ if not }, 
\end{array}\right.
$$
$B_{[a]}^{[b]}=0$ and
$$\hat R_{a}^{b}(k,\r)=  
\left\{\begin{array}{ll}
\hat F_{a}^{b}(k,\r )& \textrm{ if } \dist([a],[b])\ge\Delta'\ \textrm{or}\ \ |k|>N\\
0 & \text{ if not}.
\end{array}\right.
$$
We denote $\hat R_{a}^{b}(k,\r)$ by $\widehat {(R^s)}_{a}^{b}(k,\r)$ if $\dist([a],[b])\ge\Delta'$  --  truncation off ``diagonal'' in space modes  -- 
and by $\widehat {(R^F)}_{a}^{b}(k,\r)$ if $ |k|>N$ --    truncation in Fourier modes.

For $k\not=0$, by Lemma~\ref{lSmallDiv3},
$$ ||(L_{k,[a],[b]}(\r))^{-1}||\le \frac1\ka\,
 $$
for all $\r$ outside some  set $\Sigma_{k,[a],[b]}(\kappa)$ such that
$$\dist(\D\setminus \Sigma_{k,[a],[b]}(\ka),\Sigma_{k,[a],[b]}(\frac\ka2))\ge \cte\frac\ka{N\chi},$$
and
$$\D_3=\D\setminus \bigcup_{\begin{subarray}{c} 0<|k|\le N\\ [a],[b]\end{subarray}}\Sigma_{k,[a],[b]}(\ka)$$
fulfils the required estimate.  For $k=0$, it follows by \eqref{ass1} and \eqref{laequiv-bis} that
$$ ||(L_{k,[a],[b]}(\r))^{-1}||\le \frac1{c'}\le \frac1{\ka}\, . $$

We then get, as in the proof of Lemma~\ref{prop:homo12}, that $\hat S_{[a]}^{[b]}(k,\cdot)$ and 
$\hat R_{[a]}^{[b]}(k,\cdot)$ have  $\cC^{{s_*}}$-extension to $\D$ satisfying
$$|| \p_\r^j \hat S_{[a]}^{[b]} (k,\r)|| \leq \Cte\frac{1}{\ka}\big(N\frac{\chi}{\ka}\big)^{|j|}
\max_{0\le l\le j} || \p_\r^l \hat F_{[a]}^{[b]}(k,\r)||
$$
and
$$||  \p_\r^j R_{a}^{b}(k,\r)||\leq  \Cte || \p_\r^j \hat F_{a}^b(k,\r)||,$$
and satisfying \eqref{homo3.1} for $\r\in\D_3$.

These estimates imply that, for any $\ga_*\le \ga'\le\ga$,
$$|| \p_\r^j \hat S_{\xi\xi}(k,\r)||_{\cB(Y_{\ga'},Y_{\ga'})}\leq \Cte\Delta'\frac{\Delta^{\exp}e^{2\ga d_\Delta}}{\ka}\big(N\frac{\chi}{\ka}\big)^{|j|}
\max_{0\le l\le j} || \p_\r^l \hat F_{\xi\xi} (k,\r)||_{\cB(Y_{\ga'},Y_{\ga'})} $$
and
$$|| \p_\r^j \hat R_{\xi\xi}(k,\r)||_{\cB(Y_{\ga'},Y_{\ga'})}\leq \Cte\Delta' \Delta^{\exp}  || \p_\r^j \hat F_{\xi\xi} (k,\r)||_{\cB(Y_{\ga'},Y_{\ga'})}.$$
The factor $\Delta^{\exp}e^{2\ga d_\Delta}$ occurs because the diameter of the blocks $\le d_{\Delta}$ interferes
with the exponential decay and influences the equivalence between the $l^1$-norm and the operator-norm.
The factor $\Delta' \Delta^{\exp} $ occurs because the truncation $\lsim \Delta'+ d_{\Delta}$ of diagonal 
influences the equivalence between the sup-norm and the operator-norm.

The estimates of the ``block components'' also gives estimates for the matrix norms and, for any $\ga_*\le \ga'\le\ga$,
$$|| \p_\r^j \hat S_{\xi\xi}(k,\r)||_{\ga,\vark}\leq \Cte\Delta'\frac{\Delta^{\exp}e^{2\ga d_\Delta}}{\ka}\big(N\frac{\chi}{\ka}\big)^{|j|}
\max_{0\le l\le j} || \p_\r^l \hat F_{\xi\xi} (k,\r)||_{\ga,\vark} $$
and
$$||  \p_\r^j R_{\xi\xi} (k,\r)||_{\ga,\vark}\leq  \Cte || \p_\r^j F_{\xi\xi}(k,\r)||_{\ga,\vark}.$$

Summing up the Fourier series, as in Lemma~\ref{prop:homo3}, we get that 
$S_{\xi\xi}(\theta,\r)$ satisfies the estimate \eqref{homo3S}. 
$R_{\xi\xi}(\theta,\r)$ decompose naturally into a sum of a factor $R^F_{\xi\xi}(\theta,\r)$, which is truncated in Fourier modes
and therefore satisfies the first estimate of \eqref{homo3R}, and a factor $R^s_{\xi\xi}(\theta,\r)$, which is truncated in off ``diagonal'' in space modes and therefore satisfies the second estimate of \eqref{homo3R}.
 \smallskip

{\it The third equation. } 
We  write the equation in  Fourier components and decompose it into its
 ``components'' on each product block $[a]\times[b]$, $[b]=\F$:
\begin{multline*}
L(k,[a],[b],\rho) \hat S_{[a]}^{[b]}(k) := \langle k, \Omega(\r)\rangle \ 
\hat S_{[a]}^{[b]}(k) +Q_{[a]}(\r)\hat S_{[a]}^{[b]}(k) -\\
 {\mathbf i}\hat S_{[a]}^{[b]}(k) JH(\r)= 
-{\mathbf i}( \hat F_{[a]}^{[b]}(k,\r)-\delta_{k,0}B_{[a]}^{[b]}-\hat R_{[a]}^{[b]}(k))
\end{multline*}
--  here we have suppressed the upper index ${\xi z_{\F}}$.

The formal solution is the same as in the previous case and it converges to functions verifying
\eqref{homo3S} and \eqref{homo3R}, by Lemma~\ref{lSmallDiv3}, and by \eqref{la-lb-bis}.

 \smallskip

{\it The fourth equation. } 
We  write the equation in  Fourier components:
\begin{multline*}
L(k,[a],[b],\rho) \hat S_{[a]}^{[b]}(k) := \langle k, \Omega(\r)\rangle \ 
\hat S_{[a]}^{[b]}(k) -  {\mathbf i}HJ(\r)\hat S_{[a]}^{[b]}(k) +\\
 {\mathbf i}\hat S_{[a]}^{[b]}(k) JH(\r)= 
-{\mathbf i}( \hat F_{[a]}^{[b]}(k,\r)-\delta_{k,0}B_{[a]}^{[b]}-\hat R_{[a]}^{[b]}(k)),
\end{multline*}
where $[a]=[b]=\F$  --  here we have suppressed the upper index ${z_{\F} z_{\F}}$.

The equation is solved (formally) by 
$$\hat S_{[a]}^{[b]}(k,\r)= 
\left\{\begin{array}{ll}
-L(k,[a],[b],\r)^{-1} {\mathbf i}\hat F_{[a]}^{[b]}(k,\r) & \textrm{ if }   0<|k|\le N\\
0 & \textrm{ if not} ,
\end{array}\right.
$$
$$\hat R_{[a]}^{[b]}(k,\r)= 
\left\{\begin{array}{ll}
\hat F_{[a]}^{[b]}(k,\r ) & \textrm{ if }   |k|> N\\
0& \textrm{ if not};
\end{array}\right.
$$
and 
$$B^{[b]}_{[a]}(\r)=\hat F_{[a]}^{[b]}(0,\r).
$$

The formal solution now converges to a  solution  verifying
\eqref{homo3S}, \eqref{homo3B} and \eqref{homo3R} by Lemma~\ref{lSmallDiv3}. The factor $R^s$ is here $=0$.

 \smallskip

{\it The second equation. } 
We  write the equation in  Fourier components and decompose it into its
 ``components'' on each product block $[a]\times[b]$:
\begin{multline*}
L(k,[a],[b],\r)\hat S_{[a]}^{[b]}(k)=:\langle k, \Omega(\r)\rangle \ \hat S_{[a]}^{[b]}(k) + Q_{[a]}(\r) \hat S_{[a]}^{[b]}(k)- \\
\hat S_{[a]}^{[b]}(k)Q_{[b]}(\r)=
-{\mathbf i}(\hat F_{[a]}^{[b]}(k,\r) -\hat R_{[a]}^{[b]}(k)-\de_{k,0} B_{[a]}^{[b]})
\end{multline*}
--  here we have suppressed the upper index $\xi\eta$.
This equation is now  solved (formally) by 
$$
S_{[a]}^{[b]}(\theta,\r)=\sum\hat S_{[a]}^{[b]} (k,\r) e^{{\mathbf i}k\cdot \theta}
\quad\textrm{and}\quad
R_{[a]}^{[b]}(\theta,\r) =\sum\hat R_{[a]}^{[b]} (k,\r) e^{{\mathbf i}k\cdot \theta},$$
with
$$\hat S_{[a]}^{[b]}(k,\r)=
\left\{\begin{array}{ll}
L(k,[a],[b],\r)^{-1}{\mathbf i}\hat F_{[a]}^{[b]}(k,\r) &
\textrm{ if } \dist([a],[b])\le\Delta'\ \textrm{and}\ \ 0< |k|\le N\\
0 &  \textrm{ if not }, 
\end{array}\right.
$$
$$B_{a}^{b}(\r)=  
\left\{\begin{array}{ll}
\hat F_{a}^{b}(0,\r ) & \textrm{ if } \dist([a],[b])\le\Delta'\ \textrm{and}\ \ k=0\\
0& \text{ if not}
\end{array}\right.
$$
and
$$\hat R_{a}^{b}(k,\r)=  
\left\{\begin{array}{ll}
\hat F_{a}^{b}(k,\r )& \textrm{ if } \dist([a],[b])\ge\Delta'\ \textrm{or}\ \ |k|>N\\
0 & \text{ if not}.
\end{array}\right.
$$
We denote again $\hat R_{a}^{b}(k,\r)$ by $\widehat {(R^s)}_{a}^{b}(k,\r)$ if $\dist([a],[b])\ge\Delta'$ 
and by $\widehat {(R^F)}_{a}^{b}(k,\r)$ if $ |k|>N$.

We have to distinguish two cases, depending on when $k= 0$ or not.

\smallskip

{\it The case $k\not=0$. }  

We have, by Lemma~\ref{lSmallDiv3},
$$ ||(L_{k,[a],[b]}(\r))^{-1}||\le \frac1\ka\,
 $$
for all $\r$ outside some  set $\Sigma_{k,[a],[b]}(\kappa)$ such that
$$\dist(\D\setminus \Sigma_{k,[a],[b]}(\ka),\Sigma_{k,[a],[b]}(\frac\ka2))\ge \cte\frac\ka{N\chi},
$$
and
$$\D_3=\D\setminus \bigcup_{\begin{subarray}{c} 0<|k|\le N\\ [a],[b]\end{subarray}}\Sigma_{k,[a],[b]}(\ka)$$
fulfils the required estimate.  

\smallskip

{\it The case $k=0$.}
In this case we consider the block decomposition $\E_{\Delta'}$  and  we distinguish  
whether $|a|=|b|$ or not.

If  $|a|>|b|$, we use   \eqref{ass1} and \eqref{la-lb-bis} to get 
$$|\alpha(\r)-\beta(\r)|\geq c'-\frac{\de}{\langle a\rangle^\varkappa}
-\frac{\de}{\langle b\rangle^\varkappa}\geq \frac{c'}{2}\ge \ka.$$

This estimate allows us to solve the equation by choosing 
$$B_{[a]}^{[b]}=\hat R_{[a]}^{[b]}(0) =0$$
and
$$\hat S_{[a]}^{[b]}(0,\r)= L(0,[a],[b],\r)^{-1}\hat F_{[a]}^{[b]}(0,\r)$$
with 
$$
||\p_\r^j\hat S_{[a]}^{[b]}(0,\r)||\le\Cte  \frac{1}{\ka} (N\frac{\chi}{\kappa})^{ |j |}
\max_{0\le l\le j}\aa{ \p_\r^l\hat F_{[a]}^{[b]}(0,\r)},$$
which implies \eqref{homo3S}.

If $|a|=|b|$, we cannot control $|\alpha(\r)-\beta(\r)|$ from below, 
so then we define
$$\hat S_{[a]}^{[b]}(0)=0$$
and
\begin{align*}
B_{a}^{b}(\r)=\hat F_{a}^{b}(0,\r)) ,\quad \hat R_{a}^{b}(0)  =0\quad &\text{for }  [a]_{\Delta'}= [b]_{\Delta'}\\
 \hat R_{a}^{b}(0,\r)  =\hat F_{a}^{b}(0,\r)\quad B_{a}^{b}=0,\quad& \text{for }  [a]_{\Delta'}\not= [b]_{\Delta'}.
\end{align*}
Clearly $R$ and $B$ verify the estimates \eqref{homo3R} and \eqref{homo3B}. 

\smallskip

Hence, the formal solution converges to functions verifying
\eqref{homo3S}, \eqref{homo3B} and \eqref{homo3R} by Lemma~\ref{lSmallDiv3}.
Moreover, for $\r\in \D'$, these functions are a solution of the fourth equation.
 \endproof
 
\subsection{The homological equation}

For simplicity we shall restrict ourselves here to $\s,\mu,\ga\le1$.

\begin{lemma}\label{thm-homo}
There exists  a constant $C$ such that if \eqref{ass1} holds, then,
for any  $N\ge1$, $\Delta'\ge \Delta\ge 1$  and 
$$\ka\le\frac1C c',$$
there exists a closed subset $\D'=\D(h, \ka, N)\subset \D$, satisfying 
$$\Leb(\D\setminus {\D'})\le C (\Delta N)^{\exp_1} (\frac{\ka}{\delta_0})^{\alpha}(\frac{\chi}{\delta_0})^{1-\alpha}$$
and there exist real jet-functions 
$S, R=R^F+R^s\in \cT_{\ga,\varkappa,\D}(\s,\mu)$ and $h_+$
verifying, for $\r\in\D'$,
\be\label{eqHomEqbis}
\{ h,S \}+ f^T= h_++R,\ee
and such that
$$h+h_+\in\NF_{\varkappa}(\Delta',\de_+)$$
and, for all  $0<\s'<\s$,

\be\label{estim-B}
\ab{ h_+}_{\begin{subarray}{c}\s',\mu\ \ \\ \ga, \varkappa,\D  \end{subarray}}\le  X
\ab{f^T}_{\begin{subarray}{c}\s,\mu\ \ \\ \ga, \varkappa,\D  \end{subarray}}
\ee

\begin{equation}\label{estim-S}
\ab{S}_{\begin{subarray}{c}\s',\mu\ \ \\ \ga, \varkappa,\D  \end{subarray}}\\
\leq \frac{1}\ka X (N\frac{\chi}{\ka})^{{{s_*}}}
\ab{f^T}_{\begin{subarray}{c}\s,\mu\ \ \\ \ga, \varkappa,\D  \end{subarray}}
\end{equation}
  and
\be\label{estim-R}
\left\{\begin{array}{l}
\ab{R^F}_{\begin{subarray}{c}\s',\mu\ \ \\ \ga, \varkappa,\D  \end{subarray}}\leq 
Xe^{-(\s-\s')N}\ab{f^T}_{\begin{subarray}{c}\s,\mu\ \ \\ \ga, \varkappa,\D  \end{subarray}}\\
\ab{R^s}_{\begin{subarray}{c}\s',\mu\ \ \\ \ga', \varkappa,\D  \end{subarray}}\leq 
Xe^{-(\ga-\ga')\Delta'}
\ab{f^T}_{\begin{subarray}{c}\s,\mu\ \ \\ \ga, \varkappa,\D  \end{subarray}}
\end{array}\right.,\ee
for $\ga_*\le\ga'\le \ga$, 
where
$$
X=C\Delta'\big(\frac{\Delta}{\s-\s' }\big)^{\exp_2} e^{2\ga d_\Delta}.$$

Moreover, $S_r(\cdot,\r)=0$ for $\r$ near the boundary of $\D$.

The exponent
$\alpha$ is a positive constant only depending on  $d,s_*,\vark$ and $\beta_2 $.
%\footnote{\ $\alpha$ is the exponent of Lemma~\ref{lSmallDiv3}}

(The exponent $\exp_1$ only depends  on $d$, $n=\#\cA$ and $\tau,\beta_2,\vark$. The exponent $\exp_2$ only depends  
on $d,m_*,s_*$. 
$C$ is an absolute constant that depends on  $c,\tau,\beta_2,\beta_3$ and $\vark$. $C$ also depend on $\sup_\D\ab{\Omega_{\textrm up}}$ and $\sup_\D\ab{H_{\textrm up}}$, but stays bounded
when these do.)

\end{lemma}

\begin{remark}
The estimates \eqref{estim-B} provides an estimate of $\de_+$. Indeed, let $\frac 1 2 \langle w, Bw\rangle$ denote the quadratic part of $h_+$. Then, for any $a,b\in [a]_{\Delta'}$,
$$
\ab{\p_\r^j B_a^b}\le  \frac1 C||\p_\r^j  B||_{(\ga,m_*),\vark}e_{(\ga,\vark),\vark}(a,b)^{-1}\le \Cte  (\Delta')^{\vark} 
\ab{f^T}_{\begin{subarray}{c}\s,\mu\ \ \\ \ga, \varkappa,\D  \end{subarray}} \frac1{\langle a\rangle^{\vark}}
$$
-- recall the definition of the matrix norm \eqref{b-matrixnorm} and of the exponential weight \eqref{weight}. By
 \eqref{estim-B} this is
 $$\le \Cte  (\Delta')^{\vark} 
\ab{f^T}_{\begin{subarray}{c}\s,\mu\ \ \\ \ga, \varkappa,\D  \end{subarray}} \frac1{\langle a\rangle^{\vark}}.$$
Since $\# [a]_{\Delta'}\lsim (\Delta')^{\exp}$ we get
$$
 || \p_\r^j B(\r)_{[a]_{\Delta'}} || \le  \Cte(\Delta')^{\exp} 
\ab{f^T}_{\begin{subarray}{c}\s,\mu\ \ \\ \ga, \varkappa,\D  \end{subarray}} \frac1{\langle a\rangle^{\vark}}.
$$
This gives the estimate 
$$\delta_+-\delta\le \Cte(\Delta')^{\exp} 
\ab{f^T}_{\begin{subarray}{c}\s,\mu\ \ \\ \ga, \varkappa,\D  \end{subarray}}.$$
\end{remark}

\begin{proof}
The set $\D'$ will now be given by the intersection of the sets in the three previous lemmas of this section.
We set 
$$ h_+(r,w)=\hat f_r(r,0) + \frac 1 2 \langle w, Bw\rangle$$
$$
S(r,\theta,w)=S_r(\theta,r)+\langle S_w(\theta)w\rangle+
\frac 1 2 \langle S_{ww}(\theta)w,w \rangle$$ 
and
$$
R(r,\theta,w)= R_r(r,\theta)+\langle R_w(\theta),w\rangle+
\frac 1 2 \langle R_{ww}(\theta)w,w \rangle,$$
with $R_{ww}=R_{ww}^F+R_{ww}^s$.
These functions also depend on $\r\in\D$ and they verify equation \eqref{eqHomEqbis}
for $\r\in\D'$.

If $x=(r,\theta,w)\in \O_{\ga_*}(\s,\mu)$, then 
$$| h_+(x)|\le \ab{f^T}_{\begin{subarray}{c}\s,\mu\ \ \\ \ga, \varkappa,\D  \end{subarray}}
+  \frac 1 2  || Bw ||_{\ga_*}||w ||_{\ga_*}.$$
Since
$$\aa{B}_{\ga,\vark} \ge \aa{B}_{\ga_*,\vark}\ge \aa{B}_{\cB(Y_{\ga_*},Y_{\ga_*})}$$
it follows that
$$| h_+(x)|\le  \Cte \ab{f^T}_{\begin{subarray}{c}\s,\mu\ \ \\ \ga, \varkappa,\D  \end{subarray}}.$$

We also have for any $x=(r,\theta,w)\in \O_{\ga'}(\s,\mu)$, $\ga_*\le \ga'\le\ga$, 
$$||Jd  h_+(x)||_{\ga'}\le 
 \Cte \ab{f^T}_{\begin{subarray}{c}\s,\mu\ \ \\ \ga, \varkappa,\D  \end{subarray}}+ 
||Bw||_{\ga'}.$$
Since
$$\aa{B}_{\ga,\vark} \ge \aa{B}_{\ga',\vark}\ge \aa{B}_{\cB(Y_{\ga'},Y_{\ga'})}$$
it follows that
$$||Jd  h_+(x)||_{\ga'}\le
\Cte \ab{f^T}_{\begin{subarray}{c}\s,\mu\ \ \\ \ga, \varkappa,\D  \end{subarray}}.$$
Finally $Jd^2  h_+(x)$ equals $ JB$ which satisfies the required bound.

The estimates of the derivatives with respect to $\r$  are the same and obtained in the same way.

The functions $S(\theta,r,\zeta)$, $R^F(\theta,r,\zeta)$ and $R^s (\theta,r,\zeta)$ 
are estimated in the same way.
\end{proof}

\subsection{The non-linear homological equation}

The equation \eqref{eqNlHomEq} can now be solved easily. We restrict ourselves again to $\s,\mu,\ga\le1$.

\begin{proposition}\label{thm-Eq}
There exists a constant $C$ such that for any
$$h\in\NF_{\varkappa}(\Delta,\de),\quad\delta \le \frac1C c',$$
and for any 
$$N\ge 1,\quad \Delta'\ge \Delta\ge 1,\quad \ka\le\frac1C c'$$ 
there exists a closed subset $\D'=\D(h, \ka,N)\subset \D$, satisfying 
$$\Leb(\D\setminus {\D'})\le C  (\Delta N)^{\exp_1} (\frac{\ka}{\delta_0})^{\alpha}(\frac{\chi}{\delta_0})^{1-\alpha},$$
and, for any $f\in \cT_{\ga,\varkappa}(\s,\mu,\D)$ 
$$\eps=\ab{f^T}_{\begin{subarray}{c}\s,\mu\ \ \\ \ga, \vark,\D  \end{subarray}}\quad
\textrm{and}\quad 
\xi=\ab{f}_{\begin{subarray}{c}\s,\mu\ \ \\ \ga, \vark,\D  \end{subarray}},$$
there exist real jet-functions 
$S,R=R^F+R^s\in\cT_{\ga,\varkappa,\D}(\s,\mu)$ and $h_+$ 
verifying, for $\r\in\D'$,
\be\label{eqNlHomEqbis}
\{ h,S \}+\{ f-f^T,S \}^T+ f^T=  h_++R\ee
and such that
$$h+ h_+\in\NF_{\varkappa}(\Delta',\delta_+)$$
and, for all $ \s'<\s$ and $\mu'<\mu$,

\be\label{estim-B2}
\ab{ h_+}_{\begin{subarray}{c}\s',\mu'\ \ \\ \ga, \varkappa,\D  \end{subarray}}\le  CXY\eps\ee

\be\label{estim-S2}
\ab{S}_{\begin{subarray}{c}\s',\mu'\ \ \\ \ga, \varkappa,\D  \end{subarray}}\\
\leq C \frac1\ka XY\eps\ee
 and
\be\label{estim-R2}
\left\{\begin{array}{l}
\ab{R^F}_{\begin{subarray}{c}\s',\mu'\ \ \\ \ga, \varkappa,\D  \end{subarray}}\leq 
Ce^{-(\s-\s')N}XY\eps\\
\ab{R^s}_{\begin{subarray}{c}\s',\mu'\ \ \\ \ga', \varkappa,\D  \end{subarray}}\leq 
Ce^{-(\ga-\ga')\Delta'}XY\eps,
\end{array}\right.\ee
for $\ga_*\le\ga'\le \ga$, 
where
$$X=(\frac{N\Delta' e^{\ga  d_\Delta} }{(\s-\s')(\mu-\mu')})^{\exp_2}$$
and
$$
Y=(\frac{\chi+\xi}{\ka})^{4{s_*}+3}.$$

Moreover, $S_r(\cdot,\r)=0$ for $\r$ near the boundary of $\D$.

Moreover, if  $\tilde \r=(0,\r_2,\dots,\r_p)$ and $f^T(\cdot,\tilde \r)=0$ for all $\tilde \r$,  then $S=R=0$ and $h_+=h$ for all $\tilde \r$.

The exponent
$\alpha$ is a positive constant only depending on  $d,s_*,\vark$ and $\beta_2 $.
%\footnote{\ $\alpha$ is the exponent of Lemma~\ref{lSmallDiv3}}

(The exponent $\exp_1$ only depends  on $d$, $n=\#\cA$ and $\tau,\beta_2,\vark$. The exponent $\exp_2$ only depends  
on $d,m_*,s_*$. 
$C$ is an absolute constant that depends on  $c,\tau,\beta_2,\beta_3$ and $\vark$. $C$ also depend on $\sup_\D\ab{\Omega_{\textrm up}}$ and $\sup_\D\ab{H_{\textrm up}}$, but stays bounded
when these do.)

\end{proposition}

\begin{remark}
Notice that the ``loss'' of $S$ with respect to $\ka$ is of ``order''  $4{{s_*}}+4$.
However, if $\chi$, $\de$  and 
$\xi= \ab{f}_{\begin{subarray}{c}\s,\mu\ \ \\ \ga', \varkappa,\D  \end{subarray}}$ are of size $\lsim \ka$, then the loss is only of ``order'' 1.
\end{remark}

\begin{proof}
Let $S=S_0+S_1+S_2$ be a jet-function such that $S_1$ starts with terms of degree $1$ in $r,w$
and $S_2$ starts with terms of degree $2$ in $r,w$  --  jet functions are polynomials in $r,w$ and we
give (as is usual) $w$ degree $1$ and $r$ degree $2$.

Let now $\s'=\s_5<\s_4<\s_3<\s_2<\s_1<\s_0=\s$ be a (finite) arithmetic progression, i.e.
$\s_j-\s_{j+1}$ do not depend on $j$,  and let
  and $\mu'=\mu_5<\mu_4<\mu_3<\mu_2<\mu_1<\mu_0=\mu$  be another  arithmetic progressions.

Then $\{ h',S\}+\{ f-f^T,S \}^T+ f^T=  h_++R$ decomposes into three homological equations
$$\{ h',S_0 \}+ f^T= ( h_+)_{0}+R_0,$$
$$\{ h',S_1 \}+f_1^T= ( h_+)_{1}+R_1,\quad f_1=\{ f-f^T,S_0 \},$$
$$\{ h',S_2 \}+ f_2^T= ( h_+)_{2}+R_2,\quad f_2=\{ f-f^T,S_1 \}.$$

By Lemma~\ref{thm-homo} we have for the first equation
$$\ab{( h_+)_0}_{\begin{subarray}{c}\s_1,\mu\ \ \\ \ga, \varkappa,\D  \end{subarray}}\le  X\eps,
\quad
\ab{S_0}_{\begin{subarray}{c}\s_1,\mu\ \ \\ \ga, \varkappa,\D  \end{subarray}}\\
\leq \frac{1}\ka X Y\eps
$$
where
$$
X=C\Delta'\big(\frac{5\Delta}{\s-\s' }\big)^{\exp} e^{2\ga_1 d_\Delta}.$$
and where $Y,Z$ are defined by the right hand sides in the estimates  \eqref{estim-S} and  \eqref{estim-R}.
 
By Proposition \ref{lemma:poisson} we have
$$
\xi_1=\ab{f_1}_{\begin{subarray}{c}\s_2,\mu_2\ \ \\ \ga, \varkappa,\D  \end{subarray}}\le
\frac{1}\ka X YW\xi \eps$$
where
$$W=C \big(\frac5{(\sigma-\sigma')}  +   \frac5{ (\mu-\mu') }\big).$$
By Proposition \ref{lemma:jet} 
$\eps_1=\ab{f_1^T}_{\begin{subarray}{c}\s_2,\mu_2\ \ \\ \ga, \varkappa,\D  \end{subarray}}$
satisfies the same bound as $\xi_1$

By Lemma~\ref{thm-homo} we have for the second equation
$$\ab{( h_+)_1}_{\begin{subarray}{c}\s_3,\mu_2\ \ \\ \ga, \varkappa,\D  \end{subarray}}\le  X\eps_1,\quad
\ab{S_1}_{\begin{subarray}{c}\s_3,\mu_2\ \ \\ \ga, \varkappa,\D  \end{subarray}}\\
\leq \frac{1}\ka X Y\eps_1.
$$

By Propositions \ref{lemma:jet} and \ref{lemma:poisson} we have
$$
\xi_2=\ab{f_2}_{\begin{subarray}{c}\s_4,\mu_4\ \ \\ \ga, \varkappa,\D  \end{subarray}}\le
\frac1\ka X YW\xi_1 \eps_1,$$
and $\eps_2=\ab{f_2^T}_{\begin{subarray}{c}\s_4,\mu_4\ \ \\ \ga, \varkappa,\D  \end{subarray}}$
satisfies the same bound.

By Lemma~\ref{thm-homo} we have for the third equation
$$\ab{( h_+)_2}_{\begin{subarray}{c}\s_5,\mu_4\ \ \\ \ga, \varkappa,\D  \end{subarray}}\le  X\eps_2,\quad
\ab{S_2}_{\begin{subarray}{c}\s_5,\mu_4\ \ \\ \ga, \varkappa,\D  \end{subarray}}\\
\leq \frac{1}\ka X Y\eps_2.
$$
 
Putting this together we find that
$$\eps+\eps_1+\eps_2 \le (1+\frac1\ka X YW\xi)^3\eps= T\eps$$
and
$$\ab{h_+}_{\begin{subarray}{c}\s',\mu'\ \ \\ \ga, \varkappa,\D  \end{subarray}}\le  
XT\eps,\quad
\ab{S}_{\begin{subarray}{c}\s',\mu'\ \ \\ \ga, \varkappa,\D  \end{subarray}}\\
\leq \frac1\ka X YT\eps.
$$

Renaming $X$ and $Y$ gives now the estimates for $h_+$ and $S$. 
$R=R_0+R_1+R_2$ and its estimates follows immediately from the homological equation.

The final statement does not follow from Lemma~\ref{thm-homo}. However, if one follows the whole construction
through the proofs of Lemmas \ref{prop:homo12} to \ref{thm-homo} one sees that it holds. For example in
Lemma~\ref{prop:homo12} it is seen immediately that this holds for $\tilde\r\notin \Sigma(L_k,\frac\kappa2)$. The only arbitrariness in the construction is the extension, but we have chosen it so that $S_r$ and $R_r$ are $=0$ on
$\Sigma(L_k,\frac\kappa2)$. The construction Lemmas \ref{prop:homo3} and \eqref{prop:homo4} displays the same feature.
\end{proof}

\section{Proof of the KAM Theorem}

Theorem~\ref{main} is proved by an infinite sequence of change of variables typical for KAM-theory.
The change of variables will be done by the classical Lie transform method
which is based on a well-known relation between composition of a function with a Hamiltonian flow
$\Phi^t_S$ and Poisson brackets:
$$\frac{d}{d t} f\circ \Phi^t_S=\{f,S\}\circ \Phi^t_S
$$
from which we derive
$$f \circ \Phi^1_S=f+\{f,S\}+\int_0^1 (1-t)\{\{f,S\},S\}\circ \Phi^t_S\ \dd t.$$
Given now three functions $ h,k$ and $f$. Then
\begin{multline*}( h+k+f )\circ \Phi^1_S=\\
 h+k+f +\{ h+k+f ,S\}+\int_0^1 (1-t)\{\{ h+k+f ,S\},S\}\circ \Phi^t_S\ \dd t.
\end{multline*}
If now $S$ is a solution of the equation
\be \label{eq-homobis}
\{  h,S \}+\{ f-f^T,S \}^T+ f^T= h_++R^F+R^s,
\ee
then 
$$
( h+k+f )\circ \Phi^1_S= h+k+h_++f_+ +R^s$$
with
\begin{multline}\label{f+}
f_+=R^F+ (f-f^T)+\{k+f^T ,S\}+\{f-f^T,S\}-\{f-f^T,S\}^T+\\ +\int_0^1 (1-t)\{\{ h+k+f ,S\},S\}\circ \Phi^t_S\ \dd t
\end{multline}
and
\be \label{f+T}
f_+^T= R^F+ \{k+f^T ,S\}^T+(\int_0^1 (1-t)\{\{h+k+f ,S\},S\}\circ \Phi^t_S\ \dd t)^T.
\ee

If we assume that $S$ and $R^F$ are ``small as''  $f^T$, then  $f_+^T$ is is ``small as'' $k f^T$ --  this is the basis
of a linear iteration scheme with (formally) linear convergence.
\footnote{\ it was first used by Poincar\'e, credited by him to the astronomer Delauney, and it has been used many times since then in
different contexts. } 
But if also $k$ is of the size $f^T$, then $f^+$ is ``small as'' the square of $f^T$  --  this is the basis
of a quadratic iteration scheme with (formally) quadratic convergence. We shall combine both of them.

First we shall give a rigorous version of the change of variables described above. We restrict ourselves to the case
when $\s,\mu,\ga\le1$.

\subsection{The basic step}
Let $h\in\NF_{\varkappa}(\Delta,\de)$ and  assume  $\vark>0$ and 
 \be\label{ass2}
 \delta \le \frac{1}{C}c' .\ee
 Let
$$\ga=(\ga,m_*)\ge \ga_*=(0,m_*)$$
and recall Remark~\ref{rAbuse} and the convention \eqref{Conv}. Let  $N\ge 1$, $\Delta'\ge \Delta\ge 1$ and
$$\ka\le\frac1C c'.$$ 
The constant  $C$ is to be determined.

Proposition \ref{thm-Eq} then gives, for any $f\in \cT_{\ga,\varkappa,\D}(\s,\mu)$,
$$\eps=\ab{f^T}_{\begin{subarray}{c}\s,\mu\ \ \\ \ga, \vark,\D  \end{subarray}}\quad
\textrm{and}\quad 
\xi=\ab{f}_{\begin{subarray}{c}\s,\mu\ \ \\ \ga, \vark,\D  \end{subarray}},$$
a set
$\D'=\D'(h, \ka,N)\subset \D$
and functions $h_+,S, R=R^F+R^s,$ satisfying \eqref{estim-B2}+\eqref{estim-S2}+\eqref{estim-R2}
and solving  the equation \eqref{eq-homobis},
$$\{h,S \}+\{ f-f^T,S \}^T+ f^T= h_++R,$$
for any $\r\in\D'$. Let now $0<\s'=\s_4<\s_3<\s_2<\s_1<\s_0=\s$  and 
$0<\mu'=\mu_4<\mu_3<\mu_2<\mu_1<\mu_0=\mu$ be  (finite) arithmetic progressions.

\medskip

\noindent{\it The flow $\Phi^t_S$.}\ 
We have, by \eqref{estim-S2},
$$\ab{S}_{\begin{subarray}{c}\s_1,\mu_1\ \ \\ \ga, \varkappa,\D  \end{subarray}}\\
\leq \Cte \frac1\ka XY\eps$$
where $X,Y$ and $\Cte$ are given in Proposition \ref{thm-Eq}, i.e. 
$$X=(\frac{\Delta' e^{\ga  d_\Delta}N}{(\s_0-\s_1)(\mu_0-\mu_1)})^{\exp_2}=
(\frac{4^2\Delta' e^{\ga  d_\Delta}N}{(\s-\s')(\mu-\mu')})^{\exp_2}\,,\qquad 
Y=(\frac{\chi+\xi}{\ka})^{4{{s_*}}+3}$$
--  we can assume without restriction that $\exp_2\ge 1$.

If 
\be\label{hyp-f1}
\eps\leq \frac1C \frac{\ka}{X^2Y}, \ee
and $C$ is sufficiently large, then we can apply Proposition \ref{Summarize}(i).
By this proposition  it follows  that for  any $0\le t\le1$ the Hamiltonian
flow map  $\Phi^t_S$ is a $\cC^{{s_*}}$-map 
$$\O_{\ga'}(\s_{i+1},\mu_{i+1})\times \D\to\O_{\ga'}(\s_i,\mu_i),\quad \forall \ga_*\le\ga'\le\ga,\quad i=1,2,3, $$
 real holomorphic and symplectic for any fixed  $\rho\in \D$.
Moreover,
$$|| \p_\r^j (\Phi^t_S(x,\cdot)-x)||_{\ga'}\le \Cte\frac1\ka XY \eps$$
and
$$\aa{ \p_\r^j (d\Phi^t_S(x,\cdot)-I)}_{\ga',\vark}\le \Cte\frac1\ka XY \eps$$
for any $x\in \O_{\ga'}(\s_2,\mu_2)$, $\ga_*\le \ga'\le\ga$,  and $0\le \ab{j}\le {s_*}$.

\medskip

\noindent{\it A transformation.} \ 
Let now $k\in \cT_{\ga,\varkappa,\D}(\s,\mu)$  and set
$$\eta=\ab{k}_{\begin{subarray}{c}\s,\mu\ \ \\ \ga, \vark,\D  \end{subarray}}.$$
 Then we have
$$(h+k+f )\circ \Phi^1_S= h+k+h_++f_+ +R$$
where $f_+$ is defined by \eqref{f+}, i.e.
\begin{multline*}
f_+= (f-f^T)+\{k+f^T ,S\}+\{f-f^T,S\}-\{f-f^T,S\}^T+\\ +\int_0^1 (1-t)\{\{ h+k+f ,S\},S\}\circ \Phi^t_S\ \dd t.
\end{multline*}
The integral term is the sum
$$
\int_0^1 (1-t)\{h_++R-f^T,S\}\circ \Phi^t_S\ \dd t
+\int_0^1 (1-t)\{\{k+f ,S\}-\{f-f^T,S\}^T,S\}\circ \Phi^t_S\ \dd t.$$

\medskip

\noindent{\it The estimates of $\{k+f^T,S \}$ and $\{f-f^T,S \}$.} \ 
By Proposition \ref{lemma:poisson}(i)
$$
\ab{\{k+f^T,S\}}_{\begin{subarray}{c} \s_2,\mu_2 \  \\  \ga, \alpha, \D \end{subarray}}
\leq  \Cte X
\ab{S}_{\begin{subarray}{c}\s_1,\mu_1\ \ \\ \ga, \varkappa,\D  \end{subarray}}
 \ab{k+f^T}_{\begin{subarray}{c} \s_1,\mu_1 \  \\  \ga, \alpha, \D \end{subarray}}.$$
 Hence
 \be
\ab{\{k+f^T,S\}}_{\begin{subarray}{c} \s_2,\mu_2 \  \\  \ga, \alpha, \D \end{subarray}}\leq
\Cte  \frac1\ka X^2Y(\eta+\eps) \eps.\ee

Similarly,
\be
 \ab{\{f-f^T,S\}}_{\begin{subarray}{c} \s_2,\mu_2 \  \\  \ga, \alpha, \D \end{subarray}}\leq
\Cte \frac1\ka X^2Y\xi\eps.\ee

\medskip

\noindent{\it The estimate of $\{h_+-f^T,S \}\circ\Phi^t_S$.} \ 
The estimate of $h_+$  is given by \eqref{estim-B2}:
$$\ab{ h_+}_{\begin{subarray}{c}\s_1,\mu_1\ \ \\ \ga, \varkappa,\D  \end{subarray}}\le  \Cte XY\eps.$$
This gives, again by Proposition \ref{lemma:poisson}(i),
$$
\ab{\{h_+-f^T,S\}}_{\begin{subarray}{c} \s_2,\mu_2 \  \\  \ga, \alpha, \D \end{subarray}}\leq 
\Cte \frac1\ka X^3Y^2\eps^2.$$

Let now $F=\{h_+-f^T,S \}$.  If $\eps$ verifies \eqref{hyp-f1} for
a  sufficiently large constant $C$, then we can apply Proposition \ref{Summarize}(ii). By this
proposition, for $\ab{t}\le1$, the function
$F\circ \Phi_S^t\in \cT_{\ga,\vark,\D}(\s_3,\mu_3)$ and
\be
\ab{\{h_+-f^T,S \}\circ \Phi_S^t}_{\begin{subarray}{c} \s_3,\mu_3 \  \\  \ga, \vark, \D \end{subarray}}\leq 
\Cte \frac1\ka X^3Y^2\eps^2.\ee

\medskip

\noindent{\it The estimate of $\{R,S \}\circ\Phi^t_S$.} \ 
The estimate of  $R$ is given by \eqref{estim-R2}. It implies that

$$\ab{R}_{\begin{subarray}{c}\s_1,\mu_1\ \ \\ \ga, \varkappa,\D  \end{subarray}}\leq 
\Cte XY\eps.$$
Then, as in the previous case,
\be
\ab{\{R,S \}\circ \Phi_S^t}_{\begin{subarray}{c} \s_3,\mu_3 \  \\  \ga, \vark, \D \end{subarray}}\leq 
\Cte \frac1\ka X^3Y^2\eps^2.\ee

\medskip

\noindent{\it The estimate of $\{\{k+f,S \}-\{f-f^T,S\}^T,S\}\circ\Phi^t_S$.} \ 
This function is estimated as above. If $F=\{\{k+f,S \}-\{f-f^T,S\}^T,S\}$, then, by 
Proposition \ref{lemma:jet} and Proposition \ref{lemma:poisson}(i),
$$
\ab{F}_{\begin{subarray}{c} \s_3,\mu_3 \  \\  \ga, \alpha, \D \end{subarray}}\leq 
\Cte (\frac1\ka X^2Y)^2(\eta+\xi)\eps^2$$
and by Proposition \ref{Summarize}(ii)
\be
\ab{\{\{k+f,S \}-\{f-f^T\}^T,S\}\circ \Phi_S^t}_{\begin{subarray}{c} \s_4,\mu_4 \  \\  \ga, \vark, \D \end{subarray}}\leq
\Cte (\frac1\ka X^2Y)^2(\eta+\xi)\eps^2 .\ee

\medskip

\noindent{\it The estimates of $R^F$ and $R^s$.}\ 
These estimates are given by \eqref{estim-R2}:
$$\ab{R^F}_{\begin{subarray}{c}\s_1,\mu_1\ \ \\ \ga, \varkappa,\D  \end{subarray}}\leq 
\Cte XY e^{-(\s-\s')N}\eps$$
and
$$\ab{R^s}_{\begin{subarray}{c}\s_1,\mu_1\ \ \\ \ga, \varkappa,\D  \end{subarray}}\leq 
\Cte XY e^{-(\ga-\ga')\Delta'}\eps.$$

Renaming now $X$ and $Y$ and denoting $R^s$ by $R_+$ gives the following lemma.

\begin{lemma}\label{basic}
There exists an absolute constant $C_1$ such that, for any
$$h\in\NF_{\varkappa}(\Delta,\de),\quad\vark>0,\quad \delta \le \frac1{C_1} c',$$
and for any 
$$N\ge 1,\quad \Delta'\ge \Delta\ge 1,\quad \ka\le\frac1{C_1} c',$$ 
there exists a closed subset $\D'=\D( h, \ka,N)\subset \D$, satisfying 
$$\Leb(\D\setminus {\D'})\le {C_1} (\Delta N)^{\exp_1}  (\frac{\ka}{\delta_0})^{\alpha}(\frac{\chi}{\delta_0})^{1-\alpha}$$
and, for any $f\in \cT_{\ga,\varkappa,\D}(\s,\mu)$,  
$$\eps=\ab{f^T}_{\begin{subarray}{c}\s,\mu\ \ \\ \ga, \vark,\D  \end{subarray}}\quad \textrm{and}\quad
\xi=[f]_{\s,\mu,\D}^{\ga,\varkappa},$$
satisfying
$$\eps \leq \frac1{C_1} \frac{\ka}{XY}, \qquad 
\left\{\begin{array}{ll}
X=(\frac{N\Delta' e^{\ga  d_\Delta}}{(\s-\s')(\mu-\mu')})^{\exp_1},& \s'<\s,\ \mu'<\mu\\
Y= (\frac{\chi+\xi}\ka)^{\exp_1},&\ \end{array}\right. 
$$
and for any  $k\in \cT_{\ga,\varkappa,\D}(\s,\mu)$,
$$\eta=\ab{k}_{\begin{subarray}{c}\s,\mu\ \ \\ \ga, \varkappa,\D  \end{subarray}},$$
there exists a  $\cC^{{s_*}}$ mapping
$$\Phi:\O_{\ga'}(\s',\mu')\times \D\to\O_{\ga'}(\s-\frac{\s-\s'}2,\mu-\frac{\mu-\mu'}2),\quad \forall \ga_*\le\ga'\le\ga,$$
real holomorphic and symplectic  for each fixed parameter $\r\in\D$, and functions
$f_+,R_+\in \cT_{\ga,\varkappa,\D}(\s',\mu')$ and
$$h+h_+\in\NF_{\varkappa}(\Delta',\delta_+),$$
such that
$$(h+k+f )\circ \Phi= h+k+ h_++f_++R_+,\quad \forall \r\in\D',$$
and
$$\ab{ h_+}_{\begin{subarray}{c}\s',\mu'\ \ \\ \ga, \varkappa,\D  \end{subarray}}\le 
 {C_1}XY\eps,$$
$$\ab{ f_+-f}_{\begin{subarray}{c}\s',\mu'\ \ \\ \ga, \varkappa,\D  \end{subarray}}\le 
 {C_1}XY(1+\eta+\xi)\eps,$$
$$\ab{ f_+^T}_{\begin{subarray}{c}\s',\mu'\ \ \\ \ga, \varkappa,\D  \end{subarray}}\le 
{C_1}\frac1\ka XY
(\eta+\ka e^{-(\s-\s')N}+\eps )\eps
$$
and
$$\ab{ R_+}_{\begin{subarray}{c}\s',\mu'\ \ \\ \ga', \varkappa,\D  \end{subarray}}\le 
{C_1} XYe^{-(\ga-\ga')\Delta'}\eps$$
for any $\ga_*\le\ga'\le\ga$.

Moreover,
$$|| \p_\r^j (\Phi(x,\r)-x)||_{\ga'}+ \aa{ \p_\r^j (d\Phi(x,\r)-I)}_{\ga',\vark} \le {C_1}\frac1\ka XY \eps$$
for any $x\in \O_{\ga'}(\s',\mu')$, $\ga_*\le\ga'\le\ga$, $\ab{j}\le{s_*}$ and  $\r\in\D$,
and $\Phi(\cdot,\r)$ equals the identity for $\r$ near the boundary of $\D$.

Finally, if  $\tilde \r=(0,\r_2,\dots,\r_p)$ and $f^T(\cdot,\tilde \r)=0$ for all $\tilde \r$,  then $f_+-f=R_+=h_+=0$ and $\Phi(x,\cdot)=x$
 for all $\tilde \r$.

\end{lemma}

\begin{remark}\label{constants}
The exponent
$\alpha$ is a positive constant only depending on  $d,s_*,\vark$ and $\beta_2 $.
%\footnote{\ $\alpha$ is the exponent of Lemma~\ref{lSmallDiv3}}
The exponent $\exp_1$ only depends  on $d$, $n=\#\cA,s_*$ and $\tau,\beta_2,\vark$. 
${C_1}$ is an absolute constant that depends on  $c,\tau,\beta_2,\beta_3$ and $\vark$. ${C_1}$ also depend on $\sup_\D\ab{\Omega_{\textrm up}}$ and $\sup_\D\ab{H_{\textrm up}}$, but stays bounded
when these do.

\end{remark}

\subsection{A finite induction}
We shall first make a finite iteration without changing the normal form in order to decrease
strongly the size of the perturbation. We shall restrict ourselves to the case when $N=\Delta'$.

\begin{lemma}\label{Birkhoff}
There exists a constant ${C_2}$ such that, for any
$$h\in\NF_{\varkappa}(\Delta,\de),\quad\vark>0,\quad\delta \le \frac1{C_2} c',$$
and for any 
$$\Delta'\ge \Delta\ge 1,\quad \ka\le\frac1{C_2} c',$$ 
there exists a closed subset $\D'=\D( h, \ka,\Delta')\subset \D$, satisfying 
$$\Leb(\D\setminus {\D'})\le {C_2} (\Delta')^{\exp_2}
(\frac{\ka}{\delta_0})^{\alpha}(\frac{\chi}{\delta_0})^{1-\alpha}$$
and,  for any $f\in \cT_{\ga,\varkappa,\D}(\s,\mu)$, 
$$\eps=\ab{f^T}_{\begin{subarray}{c}\s,\mu\ \ \\ \ga, \vark,\D  \end{subarray}}\quad \textrm{and}\quad
\xi=[f]_{\s,\mu,\D}^{\ga,\varkappa},$$
satisfying
$$\eps \leq \frac1{C_2} \frac{\ka}{XY},\quad \left\{\begin{array}{ll}
X=(\frac{\Delta' e^{\ga  d_\Delta}}{(\s-\s')(\mu-\mu')}\log\frac1{\eps})^{\exp_2},& \s'<\s,\ \mu'<\mu\\
Y= (\frac{\chi+\xi}\ka)^{\exp_2},&\ \end{array}\right. $$
there exists a  $\cC^{{s_*}}$ mapping
$$\Phi:\O_{\ga'}(\s',\mu')\times \D\to\O_{\ga'}(\s-\frac{\s-\s'}{2},\mu-\frac{\mu-\mu'}{2}),
\quad \forall \ga_*\le\ga'\le\ga,$$
real holomorphic and symplectic  for each fixed parameter $\r\in\D$, and  functions
$f'\in \cT_{\ga,\varkappa,\D}(\s',\mu')$ and
$$h'\in\NF_{\varkappa}(\Delta',\de'),$$
such that  
$$(h+f )\circ \Phi= h'+f',\quad \forall \r\in\D',$$
and
$$\ab{ h'- h}_{\begin{subarray}{c}\s',\mu'\ \ \\ \ga, \varkappa,\D  \end{subarray}}\le {C_2} XY \eps,$$
$$\xi'=\ab{ f'}_{\begin{subarray}{c}\s',\mu'\ \ \\ \ga', \varkappa,\D  \end{subarray}} \le  \xi+ C_2XY(1+\xi)\eps$$
and 
$$ \eps'=\ab{ (f')^T}_{\begin{subarray}{c}\s',\mu'\ \ \\ \ga', \varkappa,\D  \end{subarray}} \le 
{C_2} XY(e^{-(\s-\s')\Delta'}+ e^{-(\ga-\ga')\Delta'})\eps,$$
for any $\ga_*\le\ga'\le \ga$.
%\be\label{constraint} \Delta'\ge4 K \max(\frac1{\ga-\ga'}, \frac1{\s-\s'})\log\frac1{\eps}.\ee

Moreover,
$$|| \p_\r^j (\Phi(x,\r)-x)||_{\ga'}+ \aa{ \p_\r^j (d\Phi(x,\r)-I)}_{\ga',\vark} \le {C_2}\frac1\ka XY \eps$$
for any $x\in \O_{\ga'}(\s',\mu')$, $\ga_*\le\ga'\le\ga$, $\ab{j}\le{s_*}$, and  $\r\in\D$,
and $\Phi(\cdot,\r)$ equals the identity for $\r$ near the boundary of $\D$.

Finally, if  $\tilde \r=(0,\r_2,\dots,\r_p)$ and $f^T(\cdot,\tilde \r)=0$ for all $\tilde \r$,  then $f'-f=h'=0$ and $\Phi(x,\cdot)=x$
 for all $\tilde \r$.
 
( The exponents $\alpha$, $\exp_2$ and the constant ${C_2}$ have the same properties as those in Remark
 \ref{constants}.)
  
\end{lemma}

\begin{proof}
Let $N=\Delta'$ and $\ka\le\frac{c'}{C_1}$. Let $\s_1=\s-\frac{\s-\s'}2$, $\mu_1=\mu-\frac{\mu-\mu'}2$ and  $\s_{K+1}=\s'$, $\mu_{K+1}=\mu'$, and let 
$\{\s_j\}_1^{K+1}$ and $\{\mu_j\}_1^{K+1}$ be arithmetical progressions. Let
$$(\s-\s')\Delta'\le K\le (\s-\s')\Delta'(\log\frac\ka{\eps})^{-1}.$$
This implies that
 $$\ka e^{-(\s_{j}-\s_{j+1})N}\le \eps.$$

We let $f_1=f$ and $k_1=0$, and we let
$\eps_1= [f_1^T]_{\begin{subarray}{c}\s,\mu\ \ \\ \ga, \vark,\D  \end{subarray}}=\eps$, 
$\xi_1=  [f_1]_{\begin{subarray}{c}\s,\mu\ \ \\ \ga, \vark,\D  \end{subarray}}=\xi$, 
$\delta_1=\delta$ and
$\eta_1=[k_1]_{\begin{subarray}{c}\s,\mu\ \ \\ \ga, \vark,\D  \end{subarray}}=0$. 

Define now
$$\eps_{j+1}=C_1\frac1\ka X_jY_j(\eta_j+\eps_1+\eps_j)\eps_j,$$
$$\xi_{j+1}=\xi_j+ C_1 X_jY_j(1+\eta_j+\xi_j) \eps_j,\quad \eta_{j+1}=\eta_j+C_1X_jY_j\eps_j,$$
with
$$X_j=(\frac{N\Delta' e^{\ga d_\Delta}}  {(\s_j-\s_{j+1})(\mu_j-\mu_{j+1})})^{\exp_1},\quad Y_j=(\frac{\chi+\xi_j}{\ka})^{\exp_1},$$
where $C_1, \exp_1$  are given in Lemma~\ref{basic}. Notice that $X_j=X_1$.

\begin{sublem*}
If
$$\eps_1\le\frac 1{C_2}
 \frac\ka{ X_1^2Y_1^2},\quad C_2=3eC_1 2^{\exp_1},$$
 then, for all $j\ge1$, 
 $$\eps_j\le \frac1{C_1} \frac\ka{ X_j^2Y_j^2}\quad\textrm{and}\quad\eps_{j}\le (\frac{C_2}2 \frac{ X^2_1Y^2_1}\ka  \eps_1)^{j-1}\eps_1\le e^{-(j-1)}\eps_1 ,$$
$$
\xi_{j} - \xi_1  \le 2C_1 X_1 Y_1(1+\xi_1)\eps_1
\quad\textrm{and}\quad
\eta_{j} \le 2C_1 X_1 Y_1\eps_1.$$
\end{sublem*}

This sublemma shows that we can  apply Lemma~\ref{basic} K times  to get
a sequence of mappings
$$\Phi_j:\O_{\ga'}(\s_{j+1},\mu_{j+1})\times \D'\to
\O_{\ga'}(\s_j-\frac{\s_j-\s_{j+1}}2,\mu_j-\frac{\mu_j-\mu_{j+1}}2),\quad \ga_*\le\ga'\le\ga_{j}$$
and functions $f_{j+1}$ and $R_{j+1}$ such that, for $\r\in\D'$,
$$(h+k_j+f_j )\circ \Phi_j= h+k_{j+1}+ f_{j+1}$$
with $k_{j+1}=k_j+ h_{j+1}+R_{j+1}$. 

Let $f'=f_{K+1}+R_1+\dots+R_{K+1}$ and $h'=h_1+\dots+h_{K+1}$. Then
$$\ab{h'-h}_{\begin{subarray}{c}\s',\mu'\ \ \\ \ga, \varkappa,\D  \end{subarray}}\le
C_1\sum X_jY_j\eps_j\le\eta_{K+1}\le  2C_1X_1Y_1\eps_1,$$
$$\ab{f'-f}_{\begin{subarray}{c}\s',\mu'\ \ \\ \ga, \varkappa,\D  \end{subarray}}\le 
C_1\sum X_jY_j(1+\xi_j+\eta_j)\eps_j \le4C_1X_1Y_1(1+\xi_1)\eps_1$$
and
$$\ab{(f')^T}_{\begin{subarray}{c}\s',\mu'\ \ \\ \ga, \varkappa,\D  \end{subarray}}\le \eps_{K+1}+
C_1\sum X_jY_je^{(\ga-\ga')\Delta'}\eps_j\le $$
$$ 
e^{-K}\eps_1+2C_1X_1Y_1e^{(\ga-\ga')\Delta'}\eps_1\le e^{(\s-\s')\Delta'}\eps_1+
2C_1X_1Y_1e^{(\ga-\ga')\Delta'}\eps_1.$$

We then take $\Phi=\Phi_1\circ\dots\circ \Phi_K$.
For the estimates of $\Phi$, write $\Psi_j=\Phi_j\circ\dots\circ \Phi_K$ and  $\Psi_{K+1}=id$.
For $(x,\r)\in \O_{\ga'}(\s',\mu')\times \D$ we then have 
$$||\Phi(x,\r)-x||_{\ga'}\le 
\sum_{j=1}^K  ||\Psi_j(x,\r)-\Psi_{j+1}(x,\r)||_{\ga'}.$$
Then
$$
||\Psi_j(x,\r)-\Psi_{j+1}(x,\r)||_{\ga'}=||\Phi_j(\Psi_{j+1}(x,\r),\r)-\Psi_{j+1}(x,\r)||_{\ga'}$$
is
$$
\le C_1 \frac1\ka X_jY_j \eps_j.$$
It follows that
$$||\Phi(x,\r)-x||_{\ga'}\le   2C_1 \frac1\ka X_1Y_1 \eps_1.$$
The estimate of $||d\Phi(x,\r)-I||_{\ga'}$ is obtained in the same way. 

The derivatives with respect to $\r$ depends on higher order differentials which can be estimated by Cauchy
estimates.

The result now follows if we take $C_2$ sufficiently large and increases the exponent $\exp_1$.
\end{proof}

\noindent{\it  Proof of sublemma.}\ 
The estimates are true for $j=1$ so we proceed by induction on $j$. Let us assume the estimates 
hold up to $j$. Then, for $k\le j$,
$$Y_{k}\le (\frac{\chi+\xi_1+2C_1X_1Y_1(1+\xi_1)\eps_1}{\ka})^{\exp_1}= 2^{\exp_1} Y_1$$
and 
$$\eps_{j+1}\le   2^{\exp_1} \frac{ X_1Y_1}\ka  
[2C_1 X_1 Y_1\eps_1+\eps_1+\eps_1]\eps_j\le    C'\frac{ X^2_1Y^2_1}\ka  \eps_1\eps_j,$$
$C'=3C_1 2^{\exp_1} $. Then
$$\xi_{j+1}-\xi_1\le2^{\exp_1}  X_1Y_1(1+\xi_1+4C_1 X_1 Y_1(1+\xi_1)\eps_1   )(\eps_1+\dots+\eps_{j+1})\le $$
$$
2^{\exp_1}  X_1Y_1(1+\xi_1)(1+4C_1 X_1 Y_1\eps_1   )2\eps_1
\le 2^{\exp_1} 4X_1Y_1(1+\xi_1)\eps_1,
$$
if $4C_1 X_1 Y_1\eps_1\le1$ and $C'\frac{ X^2_1Y^2_1}\ka  \eps_1\le\frac1e\le\frac12$  --
and similarly for $\eta_{j+1}$.
\subsection{The infinite induction}

We are now in position to prove our main result, Theorem~\ref{main}. 

Let $h$ be a normal form Hamiltonian in $\NF_{\varkappa}(\Delta,\delta)$ and let 
$f\in \cT_{\ga,\varkappa,\D}(\s,\mu)$ be a perturbation such that
$$ 0<\eps= \ab{f^T}_{\begin{subarray}{c}\s,\mu\ \ \\ \ga, \vark,\D  \end{subarray}},\quad   
\xi=\ab{f}_{\begin{subarray}{c}\s,\mu\ \ \\ \ga, \vark,\D  \end{subarray}}. $$
We  construct the transformation $\Phi$ as the composition of infinitely many transformations $\Phi$ as in Lemma~\ref{Birkhoff}. We first specify the choice of all the parameters for $j\geq 1$. 

Let $C_2,\exp_2$ and $\alpha$ be the constants given in Lemma~\ref{Birkhoff}.

\subsubsection{Choice of parameters}
We have assumed  $\ga,\s,\mu\le1$ and we take $\Delta\ge1$.
By decreasing $\ga$ or increasing $\Delta$ we can also assume $\ga=(d_{\Delta})^{-1}$.

We  choose for $j\ge1$
$$ \mu_j=\big(\frac 12 +\frac 1 {2^j}\big)\mu\quad\textrm{and}\quad\s_{j}=\big(\frac 1 2 +\frac 1 {2^j}\big)\s.$$
We define inductively the sequences $\eps_j$, $\Delta_j$, $\de_j$ and $\xi_j$ by
\be\left\{\begin{array}{ll}
\eps_{j+1}= \eps^{K_j}\eps & \eps_1=\eps\\
\Delta_{j+1} =4K_j
\max(\frac {1} {\s_j-\s_{j+1}},d_{\Delta_j})\log \frac1{\eps}&\Delta_1=\Delta\\
\ga_{j+1}=(d_{\Delta_{j+1}})^{-1}& \ga_1=\ga\\
\de_{j+1}=\de_j+ C_2   X_jY_j\eps_j& \de_1=\de\ge0 \\
\xi_{j+1}= \xi_j+C_2X_jY_j(1+\xi_j)\eps_j&\xi_1=\xi,
\end{array}\right.\ee
where
$$\left\{\begin{array}{ll}
X_j=(\frac{\Delta_{j+1} e^{\ga_j  d_{\Delta_j}}}{(\s_j-\s_{j+1})(\mu_j-\mu_{j+1})}\log \frac1{\eps_j})^{\exp_2}
&=(\frac{K_j\Delta_{j+1} e4^{j+1}}{\s\mu}\log \frac1{\eps})^{\exp_2}\\
Y_j=( \frac{\chi+\xi_j}{\ka_j})^{\exp_2}&
\end{array}\right.$$
--for $d_\Delta$ see \eqref{block}.The $\ka_j$ is defined implicitly by
$$2^j\eps_j=\frac 1{C_2}
 \frac{\ka_j}{ X_jY_j},$$
 
 These sequences depend on the choice of $K_j$. We shall let $K_j$ increase like 
$$K_{j}=K^{j}$$
for some $K$ sufficiently large.

\begin{lemma}\label{numerical2} 
There exist  constants $C'$ and $\exp'$  such that, if
$$K\ge C' $$
and
$$
\eps(\log\frac1\eps)^{\exp'}\le\frac1{C'}\big( \frac{\s\mu}{(\chi+\xi)K\Delta}\big)^{\exp'},$$
then 
\begin{itemize} 
\item[(i)] 
$$\delta_j-\delta,\quad \xi_j-\xi,\quad \ka_j\ \le\  2C_2X_1Y_1\eps;$$

\item[(ii)] 

$$\eps_{j+1}\ge C_2 X_jY_j(e^{-\frac12(\s_j-\s_{j+1})\Delta_{j+1}}+ e^{-\frac12(\ga_j-\ga_{j+1})\Delta_{j+1}}  )\eps_j;$$

\item[(iii)]
 $$
  \sum_{j\ge1} \Delta_{j+1}^{\exp_2}   \ka_j^\alpha \le 2\Delta_{2}^{\exp_2}   \ka_1^\alpha\le
  C'\big(\frac{Kd_{\Delta}\log\frac1\eps}{\s\mu})^{\exp_2}((\chi+\xi)\eps)^\alpha .
 $$
    
\end{itemize}

( The exponents $\alpha$, $\exp'$ and the constant ${C'}$ has the same properties as those in Remark
 \ref{constants}.)
  \end{lemma}

\begin{proof}
$\Delta_{j+1} $ is equal to
$$4K_j\max(\frac {1} {\s_j-\s_{j+1}},d_{\Delta_j})\log \frac1{\eps}
\le
(\Cte \frac {1} {\s} \log\frac1{\eps})(2K)^{j}\Delta_j^{a},$$
where $a$ is some exponent depending on $d$. By a finite induction one sees that this is 
$$\le (\Cte \frac {1} {\s} \log\frac1{\eps})(2K)\Delta)^{a^j},$$
if, as we shall assume, $a\ge2$.
Now $X_j$ equals
$$
(\frac{K_j\Delta_{j+1} e4^{j+1}}{\s\mu}\log \frac1{\eps})^{\exp_2}\le
\big((\Cte \frac {1} {\s\mu} \log\frac1{\eps})(4K)^{j^2}\Delta_j^{a}\big)^{2\exp_2}.
$$
which, by assumption on $\eps$, is
$$\le \big((\Cte \frac {1} {\s\mu} \log\frac1{\eps})K\Delta\big)^{4\exp_2 a^j}\le (\frac1{\eps})^{4\exp_2 a^j},$$
if, as we shall assume, $a\ge3$.

(i) holds trivially for $j=1$, (i) , so assume it holds up to $j-1\ge1$. Then $\de_j\le\de+2C_2X_1Y_1\eps$ and $\xi_j\le\xi+2C_2X_1Y_1\eps$, and hence
$$
Y_j\le ( \frac{\chi+\xi+ 2C_2X_1Y_1 \eps}{\ka_j})^{\exp_2}\le
2^{\exp_2} Y_1(\frac{\ka_1}{\ka_j})^{\exp_2}.$$
By definition of $\ka_j$,
$$\ka_j^{1+\exp_2}=2^jC_2X_jY_j\eps_j \ka_j^{\exp_2} \le 2^{\exp_2} C_2Y_1 \ka_1^{\exp_2}2^jX_j\eps_j \le 2^jX_j\eps^{K_{j-1}}$$
by assumption on $\eps$. Hence
$$2^jC_2X_jY_j\eps_j=\ka_j\le 2^jX_j\eps^{2b K_{j-1}}\le \eps^{2b K_{j-1}-4\exp_2 a^j-j\log2}  ,\quad b=\frac1{2(1+\exp_2)}.$$
If $K$ is large enough   --  notice that $j\ge2$ --   this is $\le  \eps^{b K_{j-1}}$. 

Hence
$$\ka_j\le \eps^{b K_{j-1}}\le \eps^{b K}\le\eps\le 2C_2X_1Y_1\eps,$$
if $K$ is large enough. Moreover
$$\delta_j-\delta=\sum_{k=2}^j C_2 X_kY_k\eps_k  \le  \eps^{b K_{1}}\le 2C_2X_1Y_1\eps_1$$
if $K$ is large enough. From these estimates one also obtains the required bound for $\xi_j-\xi$ if $K$ is large enough.
This concludes the proof of (i).

To see (ii), notice that
$$e^{-(\s_j-\s_{j+1})\Delta_{j+1}}\le e^{-4K_j\log\frac1\eps}\le\eps^{K_j}\eps.$$
Notice also that $\Delta_{j+1}$ is much larger then $\Delta_{j}$ so that $\ga_{j+1}$ is much smaller than $\ga_j$ and,
hence,
$$e^{-(\ga_j-\ga_{j+1})\Delta_{j+1}}\le e^{-4K_j\frac{\ga_j-\ga_{j+1}}{\ga_j} \log\frac1\eps}\le\eps^{K_j}\eps.$$
This implies that
$$C_2 X_jY_j(e^{-\frac12(\s_j-\s_{j+1})\Delta_{j+1}}+ e^{-\frac12(\ga_j-\ga_{j+1})\Delta_{j+1}}  )\eps_j\le
\eps^{K_j} \eps=\eps_{j+1}.$$

To see (iii)  we have for $j\ge 2$
   $$ \Delta_{j+1}^{\exp_2}\ka_j^{\alpha}\le  X_j^{\exp_2} \ka_j^{\alpha}
   \le (\frac1{\eps})^{4\exp_2^2 a^j}\ka_j^{\alpha}\le
   e^{-4\exp_2^2 a^j\log\frac1{\eps}}e^{\alpha bK_{j-1}\log\frac1{\eps}}  $$
  which is
  $$\le \eps^{\frac12b K_{j-1}\alpha }\le 2^{-j}\eps,$$
  if $K$ is large enough (depending on $\alpha$). This implies the first inequality in (iii). The second one is a simple computation.
\end{proof}
\subsubsection{The iteration}

\begin{proposition}
There exist positive constants $C_3$, $\alpha$ and  $\exp_3$ such that, for any
$h\in\NF_{\varkappa}(\Delta,\de)$ and for any 
$f\in \cT_{\ga,\varkappa,\D}(\s,\mu)$,
$$ \eps=\ab{f^T}_{\begin{subarray}{c}\s,\mu\ \ \\ \ga, \varkappa,\D  \end{subarray}},\quad 
 \xi=\ab{f}_{\begin{subarray}{c}\s,\mu\ \ \\ \ga, \varkappa,\D  \end{subarray}},$$
if
$$\delta \le \frac1{C_3} c'$$
and
$$
\eps(\log \frac1\eps)^{\exp_3}\le\frac1{C_3}\big( \frac{\s\mu}
{(\chi+\xi)\max(\frac1\ga,d_{\Delta})}c'\big)^{\exp_3}c',$$
then there exist a closed subset $\D'=\D'(h, f)\subset \D$,
$$\Leb (\D\setminus \D')\leq 
C_3\big(\frac{\max(\frac1\ga,d_{\Delta})\log\frac1\eps}{\s\mu})^{\exp_3}\frac{\chi}{\delta_0}((\chi+\xi)\frac\eps\chi)^\alpha$$
and a $\cC^{{s_*}}$ mapping 
$$\Phi :\O_{\ga_*}(\s/2,\mu/2)\times \D \to\O_{\ga_*}(\s,\mu),$$
real holomorphic and symplectic for given parameter $\rho\in\D$,
and
$$h'\in \NF_{\varkappa}(\infty,\de'),\quad \de'\le \frac{c'}2,$$
such that
$$(h+f)\circ \Phi=h'+f'$$
verifies
$$\ab{ f'-f}_{\begin{subarray}{c}\s/2,\mu/2\ \ \\ \ga_*, \varkappa,\D  \end{subarray}} \le  C_3$$
and, for $\r\in\D '$, $(f')^T=0$.

Moreover,
$$\ab{ h'- h}_{\begin{subarray}{c}\s/2,\mu/2\ \ \\ \ga_*, \varkappa,\D  \end{subarray}}\le C_3$$
and
$$|| \p_\r^j (\Phi(x,\cdot)-x)||_{\ga_*}+ \aa{ \p_\r^j (d\Phi(x,\cdot)-I)}_{\ga_*,\vark} \le C_3$$
for any $x\in \O_{(0,m_*)}(\s',\mu')$, $\ab{j}\le{s_*}$, and $\r\in\D$, 
and $\Phi(\cdot,\r)$ equals the identity for $\r$ near the boundary of $\D$.

Finally, if  $\tilde \r=(0,\r_2,\dots,\r_p)$ and $f^T(\cdot,\tilde \r)=0$ for all $\tilde \r$,  then $h'=h$ and $\Phi(x,\cdot)=x$
for all $\tilde \r$.

( The exponents $\alpha$, $\exp_3$ and the constant $C_3$ have the same properties as those in Remark
 \ref{constants}.)
  
\end{proposition}

\begin{proof} Assume first that $\ga=d_\Delta^{-1}$.

%Fix a $\beta_5<........$ and take
%$$K=\cte\log(\frac {C}{\ga\s\mu}\Delta\log\frac1\eps ),$$

Choose the number $\mu_j,\s_j,\eps_j,\Delta_j,\ga_j,\de_j,\xi_j,X_j,Y_j,\ka_j$ as 
above in Lemma~\ref{numerical2} with $K= C' $.
Let $h_1=h$, $f_1=f$. 

Since
$$\ka_j,\ \delta_j-\delta\ \le 2C_2X_1Y_1\eps\le\frac1{2C_2}c'$$
by Lemma~\ref{numerical2} and by assumption on $\eps$ we can apply Lemma~\ref{Birkhoff}
iteratively. It gives, for all $j\ge1$,
a set $\D_{j}\subset\D$,
$$\Leb (\D\setminus \D_{j})\leq    {C_2} \Delta_{j+1}^{\exp_2} (\frac{\ka_j}{\delta_0})^{\alpha}(\frac{\chi}{\delta_0})^{1-\alpha},$$
and a $\cC^{{s_*}}$ mapping
$$\Phi_{j+1} :\O^{\ga'}(\s_{j+1},\mu_{j+1})\times \D_{j+1}\to
\O^{\ga'}(\s_j-\frac{\s_j-\s_{j+1}}{2},\mu_j-\frac{\mu_j-\mu_{j+1}}{2}),\quad \forall \ga_*\le\ga'\le\ga_{j+1},$$
real holomorphic and symplectic  for each fixed parameter $\r$, 
and  functions
$f_{j+1}\in \cT_{\ga,\varkappa,\D}(\s_{j+1},\mu_{j+1})$ and
$$h_{j+1}\in \NF_{\varkappa}(\Delta_{j+1},\de_{j+1})$$
such that
$$(h_j+f_j)\circ \Phi_{j+1}=h_{j+1}+f_{j+1},\quad\forall \r\in \D_{j+1},$$
with
$$\ab{ f_{j+1}^T}_{\begin{subarray}{c}\s_{j+1},\mu_{j+1}\ \ \\ \ga_{j+1}, \varkappa,\D  \end{subarray}} 
\le  \eps_{j+1}$$
and
$$\ab{ f_{j+1}}_{\begin{subarray}{c}\s_{j+1},\mu_{j+1}\ \ \\ \ga_{j+1}, \varkappa,\D  \end{subarray}} 
\le  \xi_{j+1}. $$

Moreover,
$$\ab{ h_{j+1}- h_j}_{\begin{subarray}{c}\s_{j+1},\mu_{j+1}\ \ \\ \ga_{j+1}, \varkappa,\D  \end{subarray}}
\le C_2 X_jY_{j} \eps_j$$
and
$$|| \p_\r^l (\Phi_{j+1}(x,\cdot)-x)||_{\ga'}+ \aa{ \p_\r^l (d\Phi_{j+1}(x,\cdot)-I)}_{\ga',\vark} \le C_2\frac1{\ka_j} X_jY_j \eps_j$$
for any $x\in \O_{\ga'}(\s_{j+1},\mu_{j+1})$, $\ga_*\le \ga'\le\ga_{j+1}$ and $\ab{l}\le{s_*}$.

We let $h'=\lim h_j$, $f'=\lim f_j$ and  $\Phi=\Phi_2\circ\dots\circ \Phi_3\circ\dots $.
Then $(h+f)\circ \Phi=h'+f'$ and $h'$ and $f'$ verify the statement. The convergence of $\Phi$
and its estimates follows  as in the proof of Lemma~\ref{Birkhoff}.

Let $\D'=\bigcup\D_j$. Then, by Lemma~\ref{numerical2}, 
$$\Leb (\D\setminus \D')\leq {C_2}
\frac{\chi^{1-\alpha}}{\delta_0}\sum_j \Delta_{j+1}^{\exp_2}\ka_j^{\alpha}\le
C_3\frac{\chi^{1-\alpha}}{\delta_0}
\big(\frac{d_{\Delta}\log\frac1\eps}{\s\mu})^{\exp_2}((\chi+\xi)\eps)^\alpha.$$

The last statement is obvious.

 If 
$\ga<(d_{\Delta})^{-1}$, then we increase $\Delta$ and we obtain the same result. If  $\ga>(d_{\Delta})^{-1}$, then we can just decrease $\ga$ and we obtain the same result.\end{proof}

Theorem~\ref{main} now follows from this proposition.

\bigskip
\bigskip
\begin{samepage}
\centerline{PART IV. SMALL AMPLITUDE SOLUTIONS}
\section{Proofs  of Theorems  \ref{t72}, \ref{t73}}\label{s11}
\end{samepage}
We shall now treat the beam equation by combining the Birkhoff normal form theorem \ref{NFT} 
and the KAM theorem \ref{main} or, more precisely, its Corollary~\ref{cMain-bis}.  
In order to apply
 Corollary~\ref{cMain-bis} we need to verify, first  that  the quadratic part of the Hamiltonian \eqref{HNFbis}
 is a KAM normal form Hamiltonian and, second that  the perturbation $f$  is sufficiently small. 

 We recall the agreement about constants made in the introduction.

 \subsection{A KAM normal form Hamiltonian}
 \ 
 
  Let $h$ be the Hamiltonian \eqref{H2}$+$\eqref{H1}.

 \begin{theorem}\label{p_KAM} There exists a zero-measure Borel set $\Ca \subset[1,2]$ such that for any 
strongly admissible set  $\A$ and any  $m\notin \Cc$ 
% any $c_*\in(0, \tfrac12 c_0]$ (see \eqref{DD})  and any 
  there exist 
 real numbers $\ga_g>\ga_*=(0,m_*+2)$ and
  $\beta_0,  \nu_0,c_0 >0$,    where   $c_0$, $\beta_0$, $\nu_0$ depend  on  $ m$,  such that,
for any $0<c_*\le c_0$, $0<\bb\le\beta_0$ and $0<\nu\le\nu_0$ 
 there exists an open  set 
$Q=Q(c_*,\bb,\nu)  \subset [c_*,1]^\A$, increasing as  $\nu\to0$ and satisfying 
\be\label{}
\Leb ([c_*,1]^\A \setminus  Q
)
\le C\nu^{\bb}\,,
\ee
with the following property:

For any $\r\in Q$ there exists  a real holomorphic  diffeomorphism (onto its image)
\be\label{}
\Psi_\yy:  \O_{\ga_*} \big({\tfrac 12}, {\mu_*^2} \big)\to \Tg(\nu,  1,1,\ga_*)\,,\qquad 
{\mu_*}={\tfrac{c_*}{2\sqrt2}},
\ee
such that
$$
\Psi_\yy^*\big(dp\wedge dq\big)=
\nu dr_\A\wedge d\theta_\A \  +\ \nu d u_\L\wedge d v_\L,
$$
and such that
$$\frac1{\nu} (h\circ\Psi_\yy)=h_{\textrm{up}}+f,$$
\be\label{unperturbedbis}
h_{\textrm{up}}(r,\theta,p_\L,q_\L) =\langle \Omega(\yy), r\rangle  + \frac12 \sum_{a\in\L_\infty}\Lambda_a (\yy)
( p_{a}^2 + q_{a}^2) +  \nu\langle K(\yy) \zeta_\F,  \zeta_\F\rangle
\ee
where $\F=\F_\r\subset \L_f$, with the following properties:

(i) $\Psi_\yy$  depends smoothly on $\yy$ and 
$$
\Psi_\yy\big (\O_{\ga} ({\tfrac 12}, {\mu_*^2}) \big)\subset\Tg(\nu,  1,1,\ga),\qquad \ga_*\le \ga\le\ga_g;
$$

(ii) $h_{\textrm{up}}$ satisfies, on any ball (or cube) $\D\subset Q$, the Hypotheses~A1-A3 of Section~\ref{ssUnperturbed} for some constants $c',c,\delta_0,\beta,\tau$ satisfying
\be\label{choice1}
c' \ge   \nu^{1+ \bb}  \,, \quad
c=2 \max\{\langle a\rangle^3, a\in\A\},\quad \beta_1=\beta_2=2\,,\quad
\ee
\be\label{choice2}
\delta_0\ge \nu^{1+\bb   } \,, \quad s_*=4\, (\# \F)^2
\ee

\be\label{choice3}
\beta_3=\beta_3(m)>0  \,, \quad \tau=\tau(m)>0\,;
\ee

(iii)
$$
\chi=
 |\nabla_\r \Omega |_{\cC^{ {{s_*}}-1 } (\D)}+\sup_{a\in\L_\infty} |\nabla_\r  \Lambda_a|_{\cC^{ {{s_*}}-1 } (\D)}
 + ||\nu \nabla_\r K ||_{\cC^{ {{s_*}}-1 } (\D) }\le 
 {\color{red} C}\nu^{1-\bb};$$
 
 (iv) $ f$ belongs  to $\cT_{\ga,\vark=2,Q}({\tfrac12}, \mu_*^2)$   and satisfies 
$$
\xi=| f|_{\begin{subarray}{c}1/2,\mu_*^2 \  \\ \ga_g, 2, \D \end{subarray}}
 \le C\nu^{1- \bb}  \,, \qquad
\eps= | f^T|_{\begin{subarray}{c}1/2,\mu_*^2 \  \\ \ga_g, 2, \D \end{subarray}}
  \le C\nu^{3/2-\bb} \,.$$

 If  $\A$ is admissible but not strongly admissible, then the same thing is true with the difference that $(ii)$ only holds for 
 balls (or cubes) $\D\subset Q\cap\D_0$, where $\D_0\subset [0, 1]^\A$ is an open set,
 independent of $c_*,\bb$ and $ \nu$, such that
 \be\label{hren2}
\Leb (\D_0)\ge \tfrac12\, c_0^{\#\A}.
\ee
 
The constant $ C$  
depends on $m, c_*,\bb$, but not on $\nu$.

\end{theorem}

\begin{proof}
We apply Theorem~\ref{NFT} and denote the constructed there 
 symplectic transformation by $\Psi$. We
let $\L_\infty=\L\setminus \F= (\L\setminus \L_f)\cup (\L_f\setminus\F)$ (this is a slight abuse 
of notation since in Part~II we denoted by $\L_\infty$ the set $\L\setminus \L_f$). For $\beta_0, \nu_0$ and 
$\eps_0$ we take the same constants as in Theorem~\ref{NFT}. If  $\A$ is only admissible, we take for 
$\D_0$ the set $\D_0=\D_0^1$, see \eqref{DD}.

The assertion (i) of the theorem holds by Theorem~\ref{NFT}. 

To prove (ii) and (iii) we will first verify (ii) for a  smaller $c'$,
\be\label{choice0}
c' \ge \nu^{1+2\bb(\beta(0)+\bar c)}\,,
\ee
and in (iii) will replace  the exponent for $\nu$  by a bigger number. 

By \eqref{Om},  \eqref{Lam}, \eqref{Lambdab}  and  \eqref{normK} we have that
$$
\chi=
 |\nabla_\r \Omega |_{\cC^{ {{s_*}}-1 } (Q)}+\sup_{a\in\L_\infty} |\nabla_\r  \Lambda_a|_{\cC^{ {{s_*}}-1 } (Q)}
 + ||\nu \nabla_\r K ||_{\cC^{ {{s_*}}-1 } (Q) }\le 
 \cte \nu^{1-\bb \beta(s_*-1)},
 $$
 which implies (iii) with a modified exponent.
 %The assertion (iv) follows from  \eqref{estbis}. 
 Now let us  consider (ii). 
 We will check the validity of the three hypotheses A1--A3 (with $c'$ as in \eqref{choice0}).

First we note that using \eqref{Lam}, \eqref{estimla},  \eqref{N1}, \eqref{K4} and \eqref{delta} we get 
\be\label{basic1}
\tfrac12 +\tfrac12 |a|^2\le \La \le 2|a|^2 +1\,,\quad
|\La-\la|_{C^j(\D_0)} \le C_3 \nu |a|^{-2}\quad \forall\, j\ge1 \,,\;\forall\,
a\in \L\setminus \L_f\,,
\ee
\be\label{basic2}
C_1\nu^{1+\bar c \bb} \le |\La| \le C_2\nu\qquad \forall\, a\in\L_f\setminus\F\,.
\ee
It is convenient to re-denote
\be\label{rel2}
\la=:0\quad\text{if}\quad a\in\L_f\setminus\F\,;
\ee
then the second relation in \eqref{basic1} holds for all $a$. 
We recall that the numbers $\{\pm \la, a\in\F\}$ are the eigenvalues of the operator $JK$.
They satisfy the estimates \eqref{hyperb}.

 The vector--function  $\Omega(\yy)\in\R^n$ is defined in \eqref{Om}, so 
\be\label{ddet}
\Omega(\yy) = \omega +\nu M\yy,\qquad \det M\ne0\,,
\ee
and $K$ is a symmetric real 
linear operator in the space $Y_\F$.   Its norm satisfies
\be\label{basic3} 
\| {\nu K}(\rho)\|_{C^j}
 \le C_j \nu^{1-\bb \beta(j)}\,,\qquad j\ge0\,.
\ee
See Theorem~\ref{NFT}, items (ii)-(iv). 
\medskip

\noindent
{\it Hypothesis~A1}. Relations
\eqref{la-lb-ter} and \eqref{la-lb} and the first relation in \eqref{laequiv} 
  immediately follow from \eqref{basic1} and \eqref{basic2}.

  To prove the second relation in \eqref{laequiv}  note that by Theorem~\ref{NFT}
   the operator $U$ conjugates   $JK$ with the diagonal operator with the eigenvalues $\pm{\bf i}\Lambda_j^h(\rho)$. 
   So by  \eqref{basic2} and  \eqref{Ubound} the norm of $(JH)^{-1}$ is bounded by $C\nu^{-1 -\bb(\bar c +2\beta(0)}$,
   and the required estimate follows from \eqref{choice0}. The second relation in \eqref{la-lb-bis} follows by the
   same argument from \eqref{hyperb}, which implies that the norms of the eigenvalues of 
   $\La I-{\bf i}JH$ are $\ge C^{-1}\nu^{\bar c\bb}$.  The first relation in  \eqref{la-lb-bis}  is a consequence of 
   \eqref{basic1},  \eqref{basic2} and \eqref{Fcluster}.

 Now consider \eqref{laequiv-bis}.\footnote{This is the only condition of Theorem~\ref{main} 
 which we cannot verify for any  $\rho\in Q$ without assuming that the set $\A$ is\sa.
 }
  If $a\in \L_\infty$ and $b\in\L\setminus\L_f$, then again the relation follows from  \eqref{basic1} and \eqref{basic2}.
    Next, let $a,b\in\L_f\setminus\F$. Let us write $\La$ and $\Lb$ as 
 $\Lambda^j_r$ and $\Lambda^k_m$, $j\le k$. If $j=k$, then the condition follows from \eqref{K04}, \eqref{delta}
 (from \eqref{K4} if $m=r$).
  If $j\le M_0<k$, then  again it follows from \eqref{K04}. If $j,k\le M_0$,
 then $\Lambda^j_r=\Lambda^j_1=\mu(b_j,\yy)$  and $\Lambda^k_m=\mu(b_m,\yy)$, so the relation follows
 from \eqref{K44}. Finally, let $j,k> M_0$. Then if the set $\A$ is\sa,   the required 
 relation follows from \eqref{K04}, while if $\yy \in \D_0=\D_0^1$, then it follows from \eqref{aaa}. 
\smallskip

\smallskip
\noindent
{\it Hypothesis~A2}. By \eqref{ddet}, $\p_\zz\Omega(\rho) = \nu M\zz$. Choosing 
\be\label{zet}
\zz= \frac{{}^t\!M k}{ |{}^t\!M k|}
\ee
and using that $|\Omega' -\Omega|_{C^{s_*}}\le \delta_0$
we achieve that  $\p_\zz\langle k, \Omega'(\rho)\rangle \ge C\nu$, so \eqref{o} holds.  

To verify (i) we restrict ourselves to the more complicated case when $a,b\ne\emptyset$. Then
$L(\rho)$ is a diagonal operator with the eigenvalues 
$$
\lambda_{a b}^k :=\langle k,\Omega'(\rho)\rangle +\Lambda_a(\rho) \pm \Lambda_b(\rho)\,\quad
a\in[a],\; b\in[b]\,.
$$
Clearly 
$$
|\lambda_{a b}^k -( \langle k, \omega\rangle +\lambda_a \pm \lambda_b)| \le C\nu |k|\,.
$$
(we recall \eqref{rel2}). 
Therefore by Propositions \ref{D1D2} and \ref{prop-D3} the first alternative in (i) holds, unless 
\be\label{unless}
|k|\ge C \nu^{-\bar\beta} 
\ee
for some (fixed) $\bar\beta>0$. But if we choose $\zz$ as in \eqref{zet}, then $\p_\zz L(\rho)$ becomes a 
diagonal matrix with the diagonal elements  bigger than $|{}^tMk| - C\nu |k| - C_1\nu$. 
So if $k$ satisfies \eqref{unless}, then the second alternative in (i) holds. 
\medskip

To verify (ii) we write $L(\rho, \Lambda_a)$ as the multiplication from the right by the matrix 
$$
L = (\langle k,\Omega'\rangle +\Lambda_a(\rho) )I +{\bf i}\nu J  \widehat K\,.
$$
%Since the operator $JH_0^h$ is a block of the operator $JH_0$, obtained from the smooth operator of order 
%one $J\widetilde K(\rho)$ using the symplectic change of variable $\widetilde U(\rho)$, satisfying \eqref{LL}, \eqref{K8},
%then 
 The transformation  $U$ conjugates  $L$ with  the diagonal 
 operator with the eigenvalues 
 $\lambda^k_{a j}=:  \langle k,\Omega'\rangle +\Lambda_a(\rho) \pm \nu{\bf i}\Lambda^h_j$. 
 In view of  \eqref{hyperb}, 
 $ |\lambda^k_{a j}|\ge |\Im \lambda^k_{a j}|\ge C^{-1}   \nu^{ 1+\bar c\bb}$. This 
 implies (ii) by \eqref{Ubound}  and  \eqref{choice0}. 
 \medskip

 It remains to verify (iii). As before,  we restrict ourselves to  the more complicated case $a,b\in\F$. Let us denote 
 $$
 \lambda(\rho) := \langle k,\Omega'(\rho)\rangle = \langle k,\omega\rangle + \nu\langle k, M \rho\rangle +
 \langle k,(\Omega' -\Omega)(\rho)\rangle\,,
 $$
 and write the operator   $L(\rho)$ as 
 $$
 L(\rho) = \lambda(\rho) I +L^0(\rho)\,,\quad L^0(\yy) X = [X, iJ{(\nu K)}(\yy)]\,.
 $$ 
  In view of \eqref{basic3}, 
 \be\label{9.0}
 \|L^0\|_{C^j} \le C_j \nu^{1- \beta(j)\bb}\qquad\text{for}\; j\ge0\,.
 \ee
 Now it is easy to see that if
  $|\langle k,\omega| \rangle \ge C(\nu^{1-\beta(0)\bb} +\nu|k|)$ with a sufficiently big $C$, 
 then the first alternative in (iii) holds. 
 
 So it remains
 to consider the case when
 \be\label{9.1}
 |\langle k,\omega \rangle | \le C(\nu^{1-\beta(0)\bb} +\nu|k|)\,.
 \ee
 By Proposition \ref{D1D2} the l.h.s. is bigger than $\kappa|k|^{-n^2}$. Assuming that $\beta_0\ll1$, we derive from this and 
   \eqref{9.1}  that 
   \be\label{9.2}
 |k|  \geq C \nu^{-1/(1+n^2)}\,.
 \ee
 In view of \eqref{9.0}-\eqref{9.2}, again  if $\beta_0\ll1$, we have:
 \be\label{9.5}
 |\lambda(\rho) %- \langle k\cdot \omega\rangle
 | \le C\nu (\nu^{-\beta(0)\bb} +|k|)\le C_1 \nu |k|\,,
 \ee
 \be\label{9.6}
 |(\p_\rho)^j \lambda(\rho)| \le C_j |k| \delta_0,\qquad 2 \le j\le s_*\,,
 \ee
 \be\label{9.3}
 \|L\|_{C^j} \le C\nu (\nu^{-\beta(j)\bb} + |k|)  +C_j|k|\delta_0\,,\qquad j\ge0\,.
  \ee
  
  Denote det$\,L(\rho) = D(\rho)$. Then 
  $$
  D(\rho) = \prod_{a,b\in \F} \prod_{\sigma_1, \sigma_2=\pm} \Lambda(\rho;a,b,\sigma_1, \sigma_2)\,,
  $$
  where 
$   \Lambda(\rho;a,b,\sigma_1, \sigma_2) = 
  \lambda(\rho) +\sigma_1 \nu \Lambda_a(\rho) -\sigma_2 \nu \Lambda_b(\rho)\,.
  $
    Choosing $\zz$ as in 
  \eqref{zet} we get 
  $$
  |\Lambda|\le C\nu|k|\,,\;\;\; |\p_\zz\Lambda| \ge C^{-1} |k|\nu - |k|\delta_0\ge \tfrac12 C^{-1}|k|\nu\,,\;\;\;
   |\p_\zz^j  \Lambda| \le C_j|k|\delta_0\;\;\text{if}\;\;j\ge2
  $$
  (that is, these relations hold for all values of the arguments $\rho, a, b, \sigma_1,\sigma_2$). 
    Recall that  $2\,|\F|=m$; then $s_*=m^2$.  Chose in \eqref{altern1} $j=s_*=m^2$. 
  Then, in view of the relations above,    we get:
     $$
   | \p_\zz^{s_*} D(\rho)| \ge m^2! \, \big(C^{-1} |k|\nu\big)^{m^2} - C_1 (|k|\nu)^{m^2-1} (|k|\delta_0) 
   \ge \tfrac12  m^2! \, \big(C^{-1} |k|\nu\big)^{m^2}
   \,.
  $$
In the same time,  by \eqref{9.3} 
the r.h.s. of \eqref{altern1}  is bounded from above by 
$$
C_m\delta_0 (\nu^{(m^2-1)(1-\beta(m^2) \bb)} + \nu^{m^2-1} |k|^{m^2-1})\,.
$$
  In view of \eqref{choice0}, \eqref{choice2} this implies the relation 
   \eqref{altern1}   if we choose 
  $\bb<(\beta(m^2)(1+n^2))^{-1}$   (as always, we decrease $\nu_0$, if needed). 
% $$(|k|\nu)^{m^2}\nu^{-1-{\bb(2\beta(0)+2\bar c)  
 %}}\geq C_1C_m ((|k|\nu)^{m^2-1}+\nu^{(m^2-1)(1-\beta(j)\beta_{\#})})$$
% which is achieved as soon as $\beta_0\leq\frac{m^2}{\beta(j)(1+m^2)(m^2-1)}$. 

 \medskip
 
\noindent
{\it Hypothesis~A3}.   The  required inequality follows from Proposition \ref{prop-D3} since 
the  divisor, corresponding to \eqref{melnikov}  where $a,b\not\in\L_f$, 
cannot be resonant.
\medskip

 Finally, let us denote 
$$
\bb^0= \bb \max(1, \hat c,  2(\beta(0)+\bar c), \beta(s_*-1))\,.
$$
Our argument shows that  the assertions (ii), (iii) of the theorem hold with $\bb$ replaced by $\bb^0$.
The assertion (iv) with $\bb=:\bb^0$ follows from \eqref{estbis}. 
Now it remains to re-denote $\bb^0$ by $\bb$.
\end{proof}

\subsection{The main result}\label{s_10.2}

We have $c_0,\beta_0,\nu_0$ so small so that Theorem~\ref{p_KAM} applies. Now we shall make them even smaller.

\begin{theorem}\label{thm10.2} 
There exists a zero-measure Borel set $\Ca \subset[1,2]$ such that for any 
strongly admissible set  $\A$ and any  $m\notin \Cc$  there exist
 real numbers $c_0,\beta_0 >0$,  depending only on $\A$, $ m$ and $G$,  
such that, for any $0<c_*\le c_0$  and $0<\bb\le \beta_0$ the following hold.

There exists a $\nu_0$ such that
if $\nu\le\nu_0$, then there exist a closed set $Q'=Q'(c_*,\bb,\nu)\subset Q=Q(c_*,\bb,\nu)$, 
and a $\cC^{{s_*}}$-mapping $\Phi$ 
$$\Phi:\O_{ \ga_*}(1/4,\mu_*^2/2)\times Q\to \O_{ \ga_*}(1/2,\mu_*^2),\qquad {\mu_*}={\tfrac{c_*}{2\sqrt2}},\qquad \ga_*=(0,m_*+2),$$
real holomorphic and symplectic for each parameter $\r\in Q$, 
such that  
$$(h_{\textrm{up}}+ f)\circ \Phi(r,w,\r)= \langle \Omega'(\r), r\rangle +\frac 1 2\langle w, A'(\r)w\rangle+f'(r,w,\r)$$
with the following properties: 
%\begin{itemize}
%\item[$(i)$]

\noindent
(i)
 the frequency vector $\Om'$ satisfies
$$|\Omega'-\Omega|_{\cC^{{s_*-1}}(Q)}\le \nu^{1+ \aleph}\,,
$$
and the matrix 
$$A'(\r)=A'_\infty(\r)\oplus H'(\r)\in  \NF_{\infty}$$
 satisfies
$$ || \p_\r^j (H'(\r)-\nu K(\r) || \le  \nu^{1+\aleph} ,
$$
for $ |j| \le {{s_*}}$ and $\r\in Q$;

\noindent
(ii) for any  $x\in \O_{\ga_*}(1/4,\mu_*^2/2)$, $\r\in Q$ and $\ab{j}\le{s_*}-1$,
$$|| \p_\r^j (\Phi(x,\r)-x)||_{\ga_*}+ \aa{ \p_\r^j (d\Phi(x,r)-I)}_{\ga^*,\vark} \le
 \nu^{\frac12-\aleph(\kappa+2)} ;$$

\noindent
(iii) for $\r\in Q'$ and $\zeta=r=0$
$$d_r f'=d_\theta f'=
d_{\zeta} f'=d^2_{\zeta} f'=0;$$

\noindent
(iv)
if $\A$ is strongly admissible, then 
%$$\Leb(\bigcup_{\nu>0}Q'(c_*,\bb,\nu))= (1-c_*)^{\#\A}.$$
$$
\lim_{\nu\to0}
\Leb Q'(c_*,\bb,\nu)= (1-c_*)^{\#\A}.
$$
If  $\A$ is admissible but not strongly admissible, then %the same things holds with the difference that now
%$$\Leb(\bigcup_{\nu>0}Q'(c_*,\bb,\nu))\ge c_0^{\#\A}.$$
$$
\liminf_{\nu\to0}
\Leb Q'(c_*,\bb,\nu)\ge \tfrac12 c_0^{\#\A}.$$

%\end{itemize}

The exponent $\aleph$ is defined by $\aleph(\ka+2)=\min(\frac18,\alpha)$ where $\alpha$ and $\kappa$
are given in Corollary~\ref{cMain-bis}.
\end{theorem}

\begin{proof} 
By Proposition \ref{p_KAM} we know that the Hamiltonian $h_{\textrm{up}}$ of \eqref{unperturbedbis}
satisfies the Hypotheses~A1-A3 of Section \ref{ssUnperturbed}
with the choice of parameters \eqref{choice1}-\eqref{choice3}  --  $c',\delta_0 $ are here still to be  determined  --  
on any ball $\D\subset Q(c_*,\bb,\nu)  \subset [c_*,1]^\A$ with
\be\label{}
\Leb ([c_*,1]^\A \setminus  Q (c_*,\bb,\nu)
)\le C\nu^{\bb}\,,\ee
In order to apply Corollary~\ref{cMain-bis} to  the Hamiltonian $h_{up}+f$  it remains to verify the assumptions a), b) of that corollary,  and
  \eqref{epsi-bis}. 

Choose $\aleph$ so that $\aleph(\ka+2)=\min(\frac18,\alpha)$. (Here  $\ka$ and $\alpha$ are given in
Corollary~\ref{cMain-bis}.) If we take $\beta_0\le \aleph^2$, then
$$\chi,\ \xi \le \Cte \nu^{1-\aleph^2}\quad\textrm{and}\quad \eps\le\Cte (\nu^{1-\aleph^2})^{\frac32}$$
for any $\bb\le\beta_0$. By \eqref{choice1} and \eqref{choice2} we have
$$c'=\delta_0\ge \nu^{1+\aleph}.$$
Then a) and b) are fulfilled.

The smallness condition \eqref{epsi-bis} in Corollary~\ref{cMain-bis},  is now easily seen hold, by the first assumption on $\aleph$, if we 
take $\nu$ sufficiently small. (Notice that this bound on $\nu$ depends on $c_*$ through $\mu_*$.) We can therefore apply this corollary:
there exists a subset $\D'(\nu)\subset \D$, with 
the measure bound \eqref{measure-bis} becomes
$$
\Leb(\D\setminus \D'(\nu)) \le\frac1{\eps_0}\delta_0^{-\aleph\ka}\eps^{\alpha}\le \nu^{\aleph},$$
(by the second assumption on $\aleph$);  the bound in (ii) follows since $c'\ge \nu^{1+\aleph}$; the bound in (iii) holds
if $\nu_0$ is small enough. The diffeomorphism $\Phi$ is trivially extended from $\D$ to $Q$ since it equals the identity near the boundary of
$\D$.

  In order to prove (iv), assume first that $\A$ is  strongly admissible. Then for any $c_*$ the sets $Q_\nu=Q(c_*,\bb,\nu)$, form
   an increasing system of open sets in $[c_*,1]^\cA$    such that their union is of full measure. So for any $\epsilon>0$ 
  we can find $\nu_\epsilon>0$ such that $\Leb Q_\nu \ge (1-\epsilon)(1-c_*)^n$ ($n=\#\cA$) if $\nu\le\nu_\epsilon$. 
 Since
  $Q_{\nu_\epsilon}$ is open there is a finite  disjoint union 
  $\cup_{j=1}^N\D_j\subset Q_{\nu_\epsilon}$ of open balls (or  cubes) whose measure differ from that of $Q_{\nu_\epsilon}$
  by at most $\epsilon(1-c_*)^n$.  [Use for example the Vitali covering theorem.]

For any $j\ge1$ we construct a closed 
  set $\D'_j(\nu)$ as above. Then
  $$
  \D'_j(\nu)\subset \D_j \subset Q_{\nu_\epsilon} \subset Q_\nu
  $$
  for any $0<\nu\le\nu_\epsilon$, and $\meas (\D_j\setminus \D_j'(\nu))\le \nu^{\aleph}$. If now 
  $
  Q'_\nu =  \cup _{j=1}^N \D'_j(\nu)
  $
and  $\nu'_\epsilon\in[0,\nu_\epsilon]$ is sufficiently small, then $\Leb (Q_\nu\setminus Q'_\nu)\le 2\epsilon(1-c_*)^n$
for all  $0<\nu\le\nu'_\epsilon$. 
This implies the first assertion in (iv). To prove the second we simply replace in the argument above the cube 
$[0,1]^\A$ by the set $\D_0$ as in \eqref{hren2}. 
\end{proof}

\subsubsection{Proof of Theorem~\ref{t72} and \ref{t73}}

\begin{proof} Given $\bb$. For any
$c_*$ and $\nu$, let $Q'(c_*,\nu)\subset Q(c_*,\bb,\nu)$ be the set defined in Theorem~\ref{thm10.2}. Then, for any $c_*>0$,
$$\bigcup_{\nu\in\Q^*}Q'(c_*,\nu)$$
is of Lebesgue measure: $=(1-c_*)^{\#\A}$ when $\A$ is strongly admissible; $\ge c_0^{\#\A}$ when $\A$ is admissible.
It follows that the set
$$ \tilde\fJ=\{I=\nu\r: \r\in \bigcup_{\begin{subarray}{c}c_*,\nu \in\Q^*\\\nu^{\bb}\le c_*\end{subarray}}Q'(c_*,\nu)\}$$
at $I=0$ has: density $=1$ when $\A$ is strongly admissible;  positive density when $\A$ is admissible.

Chose an enumeration $\{(c_j,\nu_j)\}_j$ of $\Q^*\times\Q^*$ and let $ \tilde  \fJ_j=\nu_jQ'(c_j,\nu_j)$
so that $ \tilde\fJ=\bigcup_j \tilde\fJ_j$. 

Now we  fix $j$ and let $\nu=\nu_j$. We define for any $I\in \tilde  \fJ_j$,
 $$
U'_j(\theta_\A,I=\nu\r)=\Psi_\r\circ\Phi(r_\A=0,\theta_\A,\zeta_\L=0,\r).$$

We have, by Theorem~\ref{thm10.2},
$$ ||\Phi(x,\r)-x||_{\ga_*} \le \nu^{\frac12-\aleph(\ka+2)} ;$$
for any  $x\in \O_{\ga_*}(1/4,\mu_*^2/2)$, $\r\in Q(c_j,\bb,\nu_j)$, and, by Theorem~\ref{NFT},
\begin{equation*}
\begin{split}
\mid\mid \Psi_\r(r,\theta,\xi_\L,\eta_\L)-(\sqrt{\nu\r}\cos(\theta),&\sqrt{\nu\r}\sin(\theta),\sqrt{\nu\r}\xi_\L,\sqrt{\nu\r}\eta_\L    ) \mid\mid_{\ga_*}\le \\
&\le C(\sqrt\nu\ab{r}+\sqrt\nu\aa{(\xi_\L,\eta_\L)}_{\ga_*}+\nu^{\frac32} )\nu^{-\tilde c\bb}
\end{split}
\end{equation*}
for all $(r,\theta,\xi_\L,\eta_\L)\in \O_{\ga} (\frac 12, \mu_*^2)\cap\{\theta\ \textrm{real}\}$. Therefore
\begin{equation*}
\begin{split}
\mid\mid   U'_j(\theta_\A,\nu\r)-&(\sqrt{\nu\r}\cos(\theta),\sqrt{\nu\r}\sin(\theta),0,0  )\mid\mid _{\ga_*}\le\\
&C(\sqrt\nu\nu^{\frac12-\aleph(\ka+2)}+\nu^{\frac32} )\nu^{-\tilde c\bb}\le C\nu^{1-\aleph(\ka+2)-\tilde c\bb}
\le  CI^{1-\aleph(\ka+2)-\tilde c\bb-\bb}
\end{split}
\end{equation*}
which is $\le C I^{1-\aleph(\ka+3)}$ if $\bb$ is small enough.
 Thus $U_j'$ verifies \eqref{dist1}.

Also, by Theorem~\ref{thm10.2}, the frequency vector $\Om'_j$ satisfies
$$|\Omega_j'(\r)-\Omega(\r)|\le \nu^{1+\aleph}\le C  I^{1+\aleph-\bb}\le C  I^{1+\frac\aleph2}$$
for $\r\in Q(c_*,\bb,\nu)$, and, by Theorem~\ref{NFT},
$$\Om(\r)=\om_\A+\nu M\r.$$
Therefore the vector $ \Om'_{\A,j}(\nu\r)=\Omega_j'(\r)$
will satisfy \eqref{dist11}.

Part (i),  for $\r\in\tilde\fJ_j$ is  clear by construction.

If $\r$ is such that $\F=\F_\r$ is non-void, then the
eigenvalues $\{ \pm{\bf i}\Lambda_{a}(\r),  a\in \F\}$ of $J K(\r) $ verifies (see \eqref{hyperb})
$$
 | \Im \Lambda_a(\yy) | \ge C^{-1} \nu^{ \tilde c \bb},\qquad \forall  a\in\F.$$
Since, by Theorem~\ref{thm10.2},
$$ ||\frac1\nu J H'(\r)-JK(\r) || \le  \nu^{\aleph},$$
it follows (see for example Lemma A2 in \cite{E98} and Lemma~C.2 in \cite{EGK1}) 
that the eigenvalues of the matrix $\frac1\nu JH'(\r)$, hence those of $JH'(\r)$,
have  real parts bounded away from $0$ when $\tilde c \bb<\aleph$ and 
$\nu$ is small enough.This proves (iii).

If the $\tilde \fJ_j$'s were mutually disjoint, the mappings $U'_j$ would extend to a mapping $U'$ on $\tilde\fJ$.
But they are not. However 
there are closed subsets $\fJ_j$ of $\tilde\fJ_j$, mutually disjoint, such that the density of the set $\fJ=\bigcup_j \fJ_j$ at $I=0$
is the same as that of the set $\tilde\fJ$. Now we just restrict each $U'_j$ to $\fJ_j$, and these restrictions extend to a mapping 
$U'$ on $\fJ$.

 [To see the existence of the sets $\fJ_j$ we construct, by induction, subsets $\fJ_j'$ of $\tilde\fJ_j$, mutually disjoint, such that $\bigcup_j \fJ'_j=\tilde\fJ$. The set
$\fJ_j'$ are not closed, but each has a closed subset $\fJ_j$ such that $\Leb(\fJ'_j\setminus \fJ_j)<2^{-j}\Leb( \fJ'_j)$. Since each
$\tilde\fJ_j$ is separated from $I=0$, it follows that the density of $\fJ=\bigcup_j \fJ_j$ at $0$ is the same as that of $\tilde\fJ$.]
\end{proof}

\appendix
 
\section{Proofs of Lemmas \ref{lemP}   and \ref{XPanalytic}}

For any $\ga=(\ga_1, \ga_2)$ let us denote by $Z_\gamma$ the space of complex sequences $v=(v_s, s\in\Z^d)$
with the  finite norm $\|v\|_\gamma$, defined by the same relation as the norm in the space $Y_\gamma$. 
By $M_{\ga,0}$ we denote the space of complex 
 $\Z^d\times \Z^d$--matrices, given a norm, defined by the same formula as the norm in $\M_{\ga,0}$,
 but with $[a-b]$ replaced by $|a-b|$. 
 
For any  vector $v\in Z_\varrho$, $\varrho\ge0$,  we will denote by $\F(v)$ its  Fourier-transform:
$$
\F(v)= u(x)\; \Leftrightarrow \;
 u(x)=\sum v_a e^{{\bf i}  \langle a, x\rangle }\,.
 $$
 By Example~\ref{analyt} if $u(x)$ is a bounded real holomorphic 
function with the radius of analyticity $\varrho'>0$, 
  then $\F^{-1} u\in Z_\varrho$ for $\varrho<\varrho'$. 
Finally,  for a Banach space $X$ and $r>0$ we denote by $B_r(X)$ the open ball
$ \{x\in X\mid |x|_X< r\}$. 

Let $F$ be the Fourier-image of the nonlinearity $g$, regarded as the mapping $u(x)\mapsto g(x,u(x))$, i.e. 
$\ 
F(v) = \F^{-1} g(x, \F(v)(x)).
$

\begin{lemma}\label{l1}
For  sufficiently small $\mu_g>0$, $\ga_{g1}>0$ and for 
$\ga_g=(\ga_{g1}, \ga_{g2})$, where $\ga_{g2}\ge m_*+\varkappa$ 
we have:

i) $F$ defines a real holomorphic  mapping $B_{\mu_g}(Z_{\ga_g})
\to Z_{\ga_g}$, 

ii) $d F$ defines a real holomorphic mapping $B_{\mu_g}(Z_{\ga_g}) \to M^b_{\ga',0}\,$, where 
$\ga'= (\ga_{g1}, \ga_{g2}-m_*)$.
\end{lemma}

\begin{proof}
i) For sufficiently small $\varrho', \mu>0$ the 
 nonlinearity $g$ defines a real holomorphic function  $g:\T^d_{\varrho'}\times B_\mu (\C) 
 \to \C$ and the norm
 of this function  is bounded by some constant $M$. We may write it as
$\ 
g(x,u) = \sum_{r=3}^\infty g_r(x) u^r\,, 
$
where $g_r(x) = \frac1{r!} \frac{\p^r}{\p u^r}g(x,u)\!\mid_{u=0}$. So $g_r(x)$ is holomorphic in $x\in\T^d_{\varrho'}$ 
and by the Cauchy estimate $|g_r|\le M\mu^{-r}$ for  all $x\in\T^d_{\varrho'}$. Accordingly, 
$$
\|\F^{-1} g_r\|_{\ga_g}\le C_{\varrho} M\mu^{-r}\quad \text{if}\quad 
0\le\ga_{g1}\le \varrho\,,
$$
for any $\varrho<\varrho'$;  cf.  Example~\ref{analyt}.
We may write $F(v)$ as
\be\label{b1}
F(v) = \sum_{r=3}^\infty (\F^{-1} g_r)\star \underbrace {v\star\dots\star v}_{r}=: \sum_{r=3}^\infty F_r(v)
\,.
\ee
Since the space $Z_{\ga_g}$ is an algebra with respect to the convolution (see Lemma~1.1 in \cite{EK08}), 
the $r$-th term of the sum is a mapping from $Z_{\ga_g}$ to itself, whose norm is  bounded as follows:
\be\label{b2}
\| (\F^{-1} g_r)\star \underbrace {v\star\dots\star v}_{r}\|_{\ga_g}\le C_1
 C^{r+1} \mu^{-r} \|v\|^r_{\ga_g}\,.
\ee
This implies the assertion with  a suitable $\mu_g>0$.
\medskip

ii) The assertion i) and the Cauchy estimate imply that the operator-norm of $dF(v)$ is
bounded if $|v|_\ga <\mu_g$. To estimate $|d F(v)|_{\ga', 0}$, for 
 $r\ge3$ consider the term $F_r(v)$ in \eqref{b1}. This is the Fourier transform of the mapping 
$u(x)\mapsto g_r(x) u(x)^r$, and  its differential $dF_r(v)$ is a linear operator in $Z_{\ga_g}$ which is the Fourier-image 
of the operator of multiplication by the function $rg_r(x) u^{r-1}(x)$. So  the matrix $\big(\, dF_r(v)^b_a, \,a,b\in\Z^d\big)$
of the former operator is nothing but the matrix of the latter operator, written in the trigonometric 
 basis $\{e^{{\bf i}(a,x)}\}$. Therefore 
$$
(dF_r(v))_a^b =  (2\pi)^{-d} \int e^{-{\bf i}\langle b,  x\rangle } rg_r(x) u^{r-1} e^{ {\bf i}  \langle a, x\rangle }\, dx\,. 
$$
That is, $(dF_r(v))_a^b = G_r(b-a)$, where $G_r(a)$ is the Fourier transform  of the function $r g_r(x) u^{r-1}$. So
\begin{equation*}
\begin{split}
|d F_r(v)|_{\ga', 0} =& \sup_a C\sum_b |(|d F_r(v)^b_a|e^{\ga_{g1}|a-b|}  \langle a-b \rangle^{\ga_{g2}-m_*}\\
=& \sup_a C\sum_b |(|G_r(a-b)|e^{\ga_{g1}|a-b|}  \langle a-b \rangle^{\ga_{g2}}  \langle a-b \rangle^{-m_*}
\le C' |G_r(\cdot)|_{\ga_g}\\
\le&C \Big(\sum_c |G_r(c)|^2 e^{2\ga_{g1} |c|}   \langle c\rangle^{2(\ga_{g2}-m_*)}\Big)^{1/2}
\big( \sum_c \langle c\rangle^{-2m_*}\big)^{1/2}= C' |G_r|_{\ga_g}
\end{split}
\end{equation*}
 (we recall that $m_*>d/2$). 
Applying \eqref{b2} with $r$ convolutions instead of $r+1$, we see that 
$\ 
 |G_r(\cdot)|_{\ga_g} \le C_2 C^r \mu^{-r} \|v\|_{\ga_g}^{r-1}\,. 
$
So 
$$
|(dF_r(v))|_{\ga',0}  % C_2 C^r \mu^{-r} \|v\|_\ga^{r-1} \sup_a\sum_b \langle a-b\rangle ^{-m_*}
\le C_3C^r \mu^{-r} \|v\|_{\ga_g}^{r-1}\,.
$$
Since $d F(v) = \sum_{r\ge3} d F_r(v)$, then the assertion ii) follows, if we replace 
$\mu_g$ by a smaller positive number. 
\end{proof}

\noindent
\begin{proof}[ Proof of  Lemma~\ref{lemP}.]  Let us consider the functional $h_{\ge4}(\zeta)$ as  in \eqref{H1}, 
and write it as
$\ 
h_{\ge4}(\zeta) = {\mathbf G}\circ \Upsilon\circ D^{-}\zeta\,. 
$
Here $D^-$ is defined in \eqref{D-},  $\Upsilon$ is the  operator 
$$
\Upsilon: Y_\ga\to Z_\ga,\qquad \zeta\to v,\;\; v_a= 
{(\xi_a+\eta_{-a})}/{\sqrt2}
\;\;\forall\, a,
$$
and ${\mathbf G}(v) = \int g(x,(\F^{-1}v)(x))\,dx$. Lemma~\ref{l1} with $F$ replaced by ${\mathbf G}$ immediately implies that 
$p$ is a real holomorphic function on
$B_{\mu_g}(Y_\ga)$ 
 with a suitable $\mu_g>0$. 
 Next, since 
$$
\nabla h_{\ge4}(\zeta) =D^{-}\circ{}^t\Upsilon \circ \nabla {\mathbf G}(\Upsilon\circ D^{-}\zeta)\,,
$$
where $\nabla {\mathbf G}=F$ is the map in Lemma~\ref{l1}, then $\nabla h_{\ge4}$ defines a real holomorphic mapping 
$B_{\mu_g}(Y_\ga)\to  Y_{\ga}$, bounded uniformly in $\ga_*\le\ga\le\ga_g$. 

By the Cauchy estimate, for any $0<\mu_g'<\mu$ the Hessian of $h_{\ge4}$   defines an analytic mapping 
\be\label{Phess}
\nabla^2 h_{\ge4}:
B_{\mu_g'}(Y_\ga)\to  \B( Y_{\ga}, Y_\ga)\,,
\ee
and $\nabla^2 h_{\ge4}(\zeta)$  is the linear operator 
$$
\nabla^2 h_{\ge4}(\zeta) =D^{-}({}^t\Upsilon \ \nabla^2  {\mathbf G}(\Upsilon\circ D^{-}\zeta)\ \Upsilon)D^{-}\,.
$$
Note that  for any infinite matrix  $A$ the matrix ${}^t\Upsilon A \Upsilon$ is formed by $2\times2$--blocks and 
satisfies 
$$
|({}^t\Upsilon A \Upsilon)^b_a | \le \frac12 \sum_{a'=\pm a, \,b'=\pm b}
 |A^{b'}_{a'}|\,.
$$
Noting also that for $a' = \pm a$, $b' = \pm b$ we have $[a-b] \le|a'-b'|$, and that 
$
\min(r_1, r_2)^2 r_1^{-1} r_2^{-1}\le 1
$
if $r_1, r_2\ge1$, we find that the first term which enters the definition of $\nabla^2 h_{\ge4}|_{\ga', 2}$ estimates 
as follows:
\begin{equation*}\begin{split}
&\sup_{a\in\Z^d}  \sum_{b\in\Z^d} |\nabla^2 p|^b_a e^{\ga_1 [a-b]} \max(1, [a-b])^{\ga_2 - m_*} \min(\langle a\rangle, 
\langle b\rangle)^2\\
\le& \sup_{a\in\Z^d}  \frac12 \sum_{b\in\Z^d} \,\sum_{ a'=\pm a , b'=\pm b}
 |\nabla^2 {\mathbf G}|^{b'}_{a'} e^{\ga_1 |a'-b'|} \max(1, [a'-b'])^{\ga_2 - m_*} 
 \frac{ \min( \langle a'\rangle, 
\langle b'\rangle)^2}{ \langle a'\rangle \, \langle b'\rangle}\\
\le& \sup_{a'\in\Z^d}  2 \sum_{b'\in\Z^d} 
 |\nabla^2 {\mathbf G}|^{b'}_{a'} e^{\ga_1 |a'-b'|} \max(1, [a'-b'])^{\ga_2 - m_*} \,
% \frac{ \min( \langle a'\rangle,  \langle b'\rangle)^2}{ \langle a'\rangle \, \langle b'\rangle}
\le \,2  |\nabla^2 {\mathbf G}|_{\ga',0}
\end{split}
\end{equation*}
The second term which enters the definition of the norm  estimates similar, so 
\be\label{b3}
|\nabla^2 h_{\ge4}(\zeta)|_{\ga'},2 \le  2 |\nabla^2 {\mathbf G}(v)|_{\ga'} = 
 2 | d F(v)|_{\ga'}\,,
 \ee
$v=\Upsilon\zeta$. 
In view of \eqref{Phess} and 
 item ii) of Lemma~\ref{l1}, the mapping 
$$
\nabla^2  p:   
B_{\mu_g'}(Y_\ga)
 % \{\zeta\in Y_{\ga_g}\mid \|\zeta\|_{\ga_g}<\mu_g\}
%\O_{\mu_g}(Y_\ga)
 \to \M^{ b}_{\ga,2}  \,,
$$
is real holomorphic and is bounded in norm by a $\ga$-independent constant. Jointly with \eqref{b3}
and Lemma~\ref{l1}  this implies the assertion of Lemma~\ref{lemP}, if we replace $\mu_g$ by any smaller 
positive number. 
 \end{proof}

\noindent
\begin{proof}[ Proof of  Lemma~\ref{XPanalytic}.] 

The proof is similar to that of Lemma~\ref{lemP} but simpler, and we restrict ourselves to estimating the Hessian of $Q^r$.
Let us start with the Hessian of $P^r$. For any $\zeta\in\O(1,1,1)$ we have:
\be\label{z.2}
d^2P^r(\zeta)(\zeta', \zeta') =
2M \sum_a \sum_\vs A^\vs_a ( \zeta_{a_1}^{\vs_1} \dots \zeta_{a_{r-2}}^{\vs_{r-2}})   {\zeta'}_{a_{r-1}}^{\vs_{r-1}}
  {\zeta'}_{a_{r}}^{\vs_{r}}+\dots =: R(\zeta)(\zeta', \zeta')+\dots\,.
\ee
Here the dots $\dots$ stand for similar sums, where the pair $\zeta', \zeta'$ replaces $\zeta,\zeta$ on other 
$\binom{r}2$ positions.  For any $b_1,b_2\in\Z^d$ the element $(\nabla_1^2P^r(\zeta))_{b_1}^{b_2}$ of 
the Hessian 
$(\nabla^2P^r(\zeta))_{b_1}^{b_2}$, coming from the component $R$ of $d^2P^r$, corresponds to the 
quadratic form  $R(\zeta)\Big( 1_{b_1} (\xi,\eta), 1_{b_2} (\xi,\eta)\Big)$, where $1_b$ stands for the $\delta$-function
on the lattice $\Z^d$, equal one at $b$ at equal zero outside $b$.

Denote by $\tilde\zeta$ the vector
$\ 
\tilde\zeta_a = |\zeta_a| +  |\zeta_{-a}|,\; a\in\Z^d\,.
$
Then 
$|\zeta_{(\vs_j^0\ a_j)} | \le |\tilde\zeta_{a_j}|$, and we see from \eqref{z.2} that $| \nabla_1^2P^r(\zeta)_{b_1}^{b_2}|$
is bounded by 
\begin{equation*}
\begin{split}
%| \nabla_1^2P^r(\zeta)_{b_1}^{b_2}|& \le  
2^{r-1} M \sum_{\substack{%a_1,\dots, a_{r-2} \\
 a_1+\dots+ a_{r-2} = -\vs^0_{r-1} b_{r-1} -\vs^0_r b_r}} \tilde\zeta_{a_1}\dots
\tilde\zeta_{a_{r-2}}
= 2^{r-1} M (\tilde\zeta\star \dots\star \tilde\zeta)(-\vs^0_{r-1} b_{1} -\vs^0_rb_2)\,.
\end{split}
\end{equation*}
Since the space $Y_\ga$ is an algebra with respect to the convolution, then 
\be\label{z.3}
| \tilde\zeta\star \dots\star \tilde\zeta |_\ga \le C^{r-3}|\tilde\zeta|_\ga^{r-2}\,.
\ee

As in the proof of Lemma~\ref{lemP},
$\ 
| \nabla^2Q^r(\zeta)_{b_1}^{b_2}| \le \langle b_1\rangle^{-1} \langle b_2\rangle^{-1} 
| \nabla^2P^r(D^-\zeta)_{b_1}^{b_2}|\,.
$ 
Denoting by $\nabla_1^2Q^r$ the component of  $\nabla^2Q^r$, corresponding to  $\nabla_1^2P^r$, 
denoting $b_1' = -\vs^0_{r-1}b_1, \ b'_2 = \vs^0_rb_2$,
and using that $[b_1-b_2] \le |b'_1-b'_2|$, we find :
\begin{equation*}
\begin{split}
&\sup_{b_1}\  \sum_{b_2} | (\nabla_1^2Q^r)_{b_1}^{b_2}| e^{\ga_1[b_1-b_2]}
\max(1, [b_1-b_2])^{\ga_2-m_*}\min (\langle b_1\rangle , \langle b_2\rangle )^2 \\
&\le C^r M \sup_{b_1}\  \sum_{b_2} (\tilde\zeta\star \dots\star \tilde\zeta)(b'_1-b'_2)  e^{\ga_1[b_1-b_2]}
\max(1, [b_1-b_2])^{\ga_2-m_*}\frac{\min (\langle b_1\rangle , \langle b_2\rangle )^2 }
{\langle b_1\rangle  \langle b_2\rangle}\\
&\le C^r M \sup_{b'_1}\  \sum_{b'_2} (\tilde\zeta\star \dots\star \tilde\zeta)(b'_1-b'_2)  e^{\ga_1[b_1-b_2]}
 \langle b_1-b_2\rangle^{\ga_2-m_*}
  \le {C'}^r M |\tilde\zeta|_\ga^{r-2} \le C^r M
\end{split}
\end{equation*}
(since $|\zeta|_\ga\le1$). 
This implies the estimate for $\nabla_1^2 Q^r$, required by the lemma. Other components of $\nabla^2 Q^r$,
corresponding to the dots in \eqref{z.2}, may be estimated in the same way. 
\end{proof}

\section{Examples}
In this appendix we discuss some examples of  Hamiltonian operators $\H(\yy)={\bf i}JK(\yy)$ defined in \eqref{diag}, corresponding to 
various dimensions $d$ and   sets $\A$. In particular we are interested in examples which give rise to partially hyperbolic KAM solutions. 

\smallskip

\noindent{\bf Examples with $({\L_f}\times {\L_f})_+=\emptyset$.}\\
As we noticed in \eqref{Lf+=0}, if $({\L_f}\times {\L_f})_+=\emptyset$ then $\H$ is Hermitian, so the constructed 
 KAM-solutions  are linearly stable. This is always the case when $d=1$.\\
When $d=2$ and $\A=\{(k,0),(0,\ell)\}$ with the additional assumption that neither $k^2$ nor $\ell^2$ can be written as 
the sum of squares of two natural numbers,  we also have $({\L_f}\times {\L_f})_+=\emptyset$.\\
Similar examples can be  constructed in higher dimension, for instance for $d=3$ we can take 
$\A=\{(1,0,0),(0,2,0)\}$ or $\A=\{(1,0,0),(0,2,0),(0,0,3)\}$.\\
We note that in \cite{GY06b}  the authors perturb solutions \eqref{sol}, corresponding to 
 set $\A$ for which $({\L_f}\times {\L_f})_+=\emptyset$ 
and $({\L_f}\times {\L_f})_-=\emptyset$. This significantly simplifies 
 the analysis since in that case there is no matrix $K$ in the normal form
  \eqref{HNF} and  the unperturbed quadratic Hamiltonian is diagonal.

\smallskip

\noindent{\bf Examples with $({\L_f}\times {\L_f})_+\neq\emptyset$.}
In this case hyperbolic directions may appear as we show below.\\
The choice $\A=\{(j,k),(0,-k)\}$ leads to $((j,-k),(0,k))\in ({\L_f}\times {\L_f})_+$.\\
Note that this example can be  plunged in higher dimensions, e.g. the 3d-set 
$\A=\{(j,k,0),(0,-k,0)\}$ leads to a non trivial $({\L_f}\times {\L_f})_+$.

\smallskip

\noindent{\bf Examples with hyperbolic directions}\\
Here we  give  examples of normal forms with hyperbolic eigenvalues, first in  
  dimension two, then -- in higher dimensions.  That is, for the beam  equation \eqref{beam} we will find 
   admissible sets $\A$ such that the corresponding matrices ${\bf i}JK(\yy)$ in the normal form \eqref{HNF} have
 unstable directions. Then by  Theorem~\ref{t73}
   the  time-quasiperiodic solutions of \eqref{beam},  constructed in the theorem,  are linearly unstable.

We begin with dimension $d=2$. Let 
$$
\A=\{(0,1),(1,-1)\}\,.
$$
We easily compute using \eqref{L++}, \eqref{L+-} that 
$$\L_f=\big\{ (0,-1),(1,0),(-1,0),(1,1), (-1,1),(-1,-1)\big)\}\,,
$$
and
$$
( {\L_f}\times {\L_f})_+=\{\big( (0,-1),(1,1)\big);\big( (1,1),(0,-1)\big)\}, \qquad
( {\L_f}\times {\L_f})_-=\emptyset.
$$
So in this case the decomposition  \eqref{decomp} of the Hamiltonian operator $\H(\yy)={\bf i}JK(\yy)$ reads
$$
\H(\yy)=\H_1(\yy)\oplus\H_2(\yy)\oplus\H_3(\yy)\oplus\H_4(\yy)\oplus\H_5(\yy)\,,
$$
where $\H_1(\yy)\oplus\H_2(\yy)\oplus\H_3(\yy)\oplus\H_4(\yy)$ is a diagonal operator with purely imaginary eigenvalues and  $\H_5(\yy)$ is
an operator in $\C^4$ which may have  hyperbolic eigenvalues. That is, now $M=5$ and $M_0=4$.
\\
  Let us denote $\zeta_1=(\xi_1,\eta_1)$ (reps. $\zeta_2=(\xi_2,\eta_2)$) the $(\xi,\eta)$-variables corresponding to the mode $(0,-1)$ (reps. $(1,1)$). We also denote $\yy_1=\yy_{(1,0)}$, $\yy_2=\yy_{(1,-1)}$, $\lambda_1=\sqrt{1+m}$ and $\lambda_2=\sqrt{4+m}$. By construction $\H_5(\yy)$ is the restriction of the Hamiltonian $  \langle K(m,\yy)\zeta_f, \zeta_f\rangle$
to the modes $(\xi_1,\eta_1)$ and $(\xi_2,\eta_2)$. 
We   calculate using \eqref{K}
that 
\be\label{hr}\langle \H_5(\yy)(\zeta_1,\zeta_2),(\zeta_1,\zeta_2)\rangle=\beta(\yy) \xi_1\eta_1+\ga(\yy) \xi_2\eta_2+\alpha(\yy) (\eta_1\eta_2+\xi_1\xi_2)\,,
\ee
where
$$
\alpha(\yy)= \frac6{4\pi^{2}} \frac{\sqrt{\yy_1\yy_2}}{\lambda_1\lambda_2}\,,\quad
\beta(\yy)= \frac3{4\pi^{2}}\frac1{\lambda_1}\Big( \frac{\yy_1}{\lambda_1}-\frac{2\yy_2}{\lambda_2} \Big) \,,\quad
\ga(\yy)= \frac3{4\pi^{2}} \frac1{\lambda_2}\Big( \frac{\yy_2}{\lambda_2} -\frac{2\yy_1}{\lambda_1}\Big) \,.
$$
Thus the linear Hamiltonian 
system, governing the two modes, reads\footnote{Recall that the symplectic two-form is: $-{\bf i}\sum d\xi\wedge d\eta$.}
\ben \left\{\begin{array}{ll}
 \dot \xi_1 &=-{\bf i}(\beta \xi_1+\alpha \eta_2)\\
 \dot \eta_1 &={\bf i}(\beta \eta_1+\alpha \xi_2)\\
 \dot \xi_2 &=-{\bf i}(\ga \xi_2+\alpha \eta_1)\\
 \dot \eta_2 &={\bf i}(\ga \eta_2+\alpha \xi_1).
\end{array}\right.
\een
So the Hamiltonian operator $\H_5$ has the matrix ${\bf i}L$, where 
\ben 
 L= \left(\begin{array}{cccc}
 -\beta &0&0&-\alpha\\
  0&\beta &\alpha&0\\
0&-\alpha&    -\ga &0\\
   \alpha &0&0&\ga\\
\end{array}\right).
\een
We  calculate  the characteristic polynomial of $L$ and obtain after a factorisation that
$$\det (L-\lambda I)=\big(\lambda^2+(\ga-\beta)\lambda -\beta\ga+\alpha^2\big)\, \big(\lambda^2-(\ga-\beta)\lambda -\beta\ga+\alpha^2\big)\,.
$$
Both quadratic polynomials which are the factors in the r.h.s. have the same discriminant 
$\ \Delta= (\beta+\ga)^2-4\alpha^2.$ 
If $\yy_1\sim1$ and $0<\yy_2\ll1$, then $\Delta>0$. So all eigenvalues of $L$ are real, while the eigenvalues of $\H_5$
and $\H$ are pure imaginary (in agreement with Lemma~\ref{laK}). But if   $\yy_1=\yy_2=\yy$, then 
$$
\ga-\beta = \frac{3\yy}{4\pi^2}  \Big( \frac1{\lambda_2^2} - \frac1{\lambda_1^2}\big)\,,    \quad
\beta+\ga= \frac{3\yy}{4\pi^2}  \Big(\frac{1}{\lambda_1^2}+ \frac{1}{\lambda_2^2}-\frac{4}{\lambda_1\lambda_2} 
\Big),\quad \alpha=   \frac{6\yy}{4\pi^2}   \frac{1}{\lambda_1\lambda_2}\,,
$$
and
\begin{align*}\Delta=\frac{9 \rho}{(2\pi)^4}\Big(\frac1{\lambda_1^2}+\frac1{\lambda_2^2}\Big)\Big(\frac1{\lambda_1^2}+\frac1{\lambda_2^2}
-\frac{8}{\lambda_1\lambda_2}\Big)
\leq \frac{9  \rho}{(2\pi)^4}\Big(\frac1{\lambda_1^2}+\frac1{\lambda_2^2}\Big)\Big(\frac1{\lambda_1^2}-\frac7{\lambda_2^2}
\Big)\,.
\end{align*}
Thus,  $\Delta <0$ for all $m\in[1,2]$. Since the eigenvalues of the matrix $L=(1/{\bf i}) \H_5$ are $\pm(\ga-\beta)\pm \sqrt\Delta$, 
then all four of them have nontrivial imaginary parts for all values of the parameter $m\in[1,2]$, and accordingly
 the operator $\H$ has 4 
hyperbolic directions. 
By analyticity, 
for all $m\in[1,2]$ with a possible exception of finitely many points, the real parts of the eigenvalues also are non-zero. 
In this case the operator $\H$ has a quadruple of hyperbolic eigenvalues. 
\medskip

This example can be generalised to any dimension $d\geq 3$. Let us do it for $d=3$.
Let 
\be\label{AAA}
\A=\{(0,1,0),(1,-1,0)\}.
\ee
We verify that $\L_f$ contains 16 points,  that $(\L_f\times\L_f)_-=\emptyset$ and 
\begin{align*}(\L_f\times\L_f)_+=\{&((0,-1,0),(1,1,0)); ((1,1,0),(0,-1,0));\\
&((1,0,-1),(0,0,1)); ((0,0,1),(1,0,-1));\\
&((1,0,1),(0,0,-1)); ((0,0,-1),(1,0,1))\}\,.
\end{align*}
I.e. $(\L_f\times\L_f)_+$ contains three pairs of symmetric couples $(a,b),(b,a)$ which give  rise to three non trivial  
$2\times2$-blocks in the matrix $\H$. 
Now $M=13$, $M_0=10$ and the decomposition  \eqref{decomp} reads 
$$
\H(\yy)=\H_1(\yy)\oplus\cdots \oplus\H_{13}(\yy)\,.
$$
Here 
 $\H_1(\yy)\oplus\cdots \oplus\H_{10}(\yy)$ is the  diagonal part of $\H$  with purely imaginary eigenvalues, while the operators 
$\H_{11}(\yy)$, $\H_{12}(\yy)$, $\H_{13}(\yy)$ correspond to non-diagonal $4\times4$--matrices. 

Denoting 
 $\yy_1=\yy_{(0,1,0)}$ and $\yy_2=\yy_{(1,-1,0)}$ we find that   the restriction of the Hamiltonian $  \langle K(m,\yy)\zeta_f, \zeta_f\rangle$
to the modes $(\xi_1,\eta_1):=(\xi_{(0,-1,0)},\eta_{(0,-1,0)})$ and $(\xi_2,\eta_2):=(\xi_{(1,1,0)},\eta_{(1,1,0)})$ is governed by the Hamiltonian  \eqref{hr}, as in the 2d case. Similarly  the restrictions of the Hamiltonian $  \langle K(m,\yy)\zeta_f, \zeta_f\rangle$
to the pair of modes $(\xi_{(1,0,-1)},\eta_{(1,0,-1)})$ and $(\xi_{(0,0,1)},\eta_{(0,0,1)})$ and 
to the pair of modes $(\xi_{(1,0,1)},\eta_{(1,0,1)})$ and $(\xi_{(0,0,-1)},\eta_{(0,0,-1)})$ are given by the same 
Hamiltonian \eqref{hr}. So $\H_{11}(\yy)\equiv\H_{12}(\yy)\equiv\H_{13}(\yy)$ and for $\yy_1=\yy_2$ we have 3 hyperbolic directions, one in each block $Y^{f11}$, $Y^{f12}$ and  $Y^{f13}$ (see \eqref{deco}) with the same eigenvalues. 

We notice that  the eigenvalues are identically the same for all three blocks,  thus the relation \eqref{single} is violated.
This does not contradict Lemma~\ref{l_nond} since the set \eqref{AAA} is not\sa. Indeed, denoting 
$a=(0,1,0)$, $b=(1,-1,0)$ we see that $c:=a+b=(1,0,0)$. So three points 
$(0,-1,0), (0,0,\pm1)\in \{x\mid\, |x|=|a|\}$  all lie at the distance $\sqrt2$ from $c$.
Hence,  it is not true that $a\ann b$.

 \section{Admissible and strongly admissible random $R$-sets}

Given $d$ and $n$, let $B(R)$ be the (round) ball of radius $R$ in $\R^d$, and $\bB(R)=B(R)\cap \Z^d$. 
The family   $\Om=\Om(R)$ of $n$-sets $\{a_1,\dots, a_n\}$ in $\bB(R)$, $\Om = \bB\times \dots\times\bB$
($n$ times) has cardinality of order 
 $CR^{nd}$.

The family on  $n$-sets $\{a_1,\dots, a_n\}$ in $\Om$
such that $\ab{a_j}=\ab{a_k}$ for some $j\not=k$ has cardinality $\le C'R^{nd-1}$ 
(the constant $C'$ as well as all other constants in this section depend, without saying, on  $n,d$). 
Its complement in $\Om$ is
the set $\Om_{\text{adm}} =\Omad(R)$  of admissible $n$-sets  in $\bB(R)$. Hence
$$
\frac{\#\Omad(R)}{\#\Om(R)} = 1-O(R^{-1})\,,
%\approx \frac1R\to 0,
\qquad R\to\infty\,.
$$
We provide the set $\Om$ with the uniform probability measure $\PP$  and will call elements of $\Om$
{\it $n$-points random $R$-sets}. The calculation above shows that 
\be\label{om+}
\PP(\Omad)\to 1 \quad \text{as}
\quad R\to\infty\,.
\ee
That is, admissible $n$-points random $R$-sets with large $R$ are typical. 
\bigskip

To consider strongly admissible sets, let $S(R)$ be the sphere of radius $R$ in $\R^d$, i.e. the boundary of $B(R)$,
and let $\bS(R) = S(R)\cap \Z^d$ (this set is non-empty only if $R^2$ is an integer). 
 We have that, for any $\eps>0$ there exists $C_\eps>0$ such that
 \be\label{VC}
 \Gamma_{R,d}:= 
 |\bS(R)  | \le C_\eps R^{d-2+\eps} \qquad \forall\, R>0.
 \ee
Indeed, for $d=2$ this is a well-known result from number theory (see \cite{Har}, Theorem~338). For $d\ge 3$ it follows by induction and an easy integration argument. For example for $d=3$, then
$$\bS(R)=\{a\in\Z^3: \ab{a}^2=R^2\}=\bigcup_{a_3^2\le R^2}\{a=(a_1,a_2,a_3)\in\Z^3: a_1^2+a_2^2=R^2-a_3^2\}$$
so
$$
 \Gamma_{R,3}=\sum _{n^2\le R^2} \Gamma_{\sqrt{R^2-n^2},2}\le C_\eps \sum _{n^2\le R^2} (R^2-n^2)^{\eps/2} \le  
 C_\eps R^\eps \sum _{n^2\le R^2} \big(1-(\frac nR)^2\big)^{ \eps/2 }\,,$$
 which is
 $\ 
 \le C_\eps R^\eps(2R+1)\le C'_\eps R^{1+\eps}.
 $
\medskip

  For  vectors $a,b \in\Z^d$ we will write 
$\ 
a \ann b \quad \text{iff} \quad a \an a+b\,.
$ 
Consider again the ensemble $\Om=\Om(R) $ of  $n$-points random $R$-sets,  $\Om =\{\omega=(a_1,\dots,a_n) \}$,
and for $j=1,\dots,n$ define 
the random variable $\xi_j$  as $\xi_j(\om) = a_j$. Consider the event 
$$
\Om_{{}_{\ann}}  = \{ \xi^i\ann\xi^j \quad\text{for all}\quad i\ne j\}\,.
$$
Then $\Omsad = \Omad\cap \Om_{{}_{\ann}} $ is the collection of strongly admissible sets. Clearly 
\be\label{w22}
\PP(\Omega\setminus   \Om_{{}_{\ann}} ) \le n(n-1) (1- \PP\{ \xi^1\ann \xi^2\})\,.
\ee
So if  we prove   that 
\be\label{w3}
 1- \PP\{ \xi^1\ann \xi^2\} \le CR^{-\ka}\,,
\ee
then, in view of \eqref{om+}, we would show that 
\be\label{om++}
\PP(\Omsad)\to 1 \quad \text{as}
\quad R\to\infty\,.
\ee
\smallskip

 Below we restrict ourselves to the case $d=3$ since for higher dimension the argument is similar, but
 more cumbersome. We have that 
 \be\label{w4}
 1-\PP\{ \xi^1 \ann \xi^2\} = |\bB(R)|^{-2} C^{**}\,,\quad C^{**}=
  \#  \{ (a,b) \in \bB(R) \times \bB(R)\mid \no a\ann b\}\,, 
 \ee
  and, denoting $a+b=c$, that 
  \be\label{w5}
   % \#  \{ (a,b) \in \bB(R) \times \bB(R): \no a\ann b\} 
C^{**}    \le 
      \#  \{ (a,c) \in \bB(2R) \times \bB(2R)\mid \no a\an c\}\,.
  \ee
  Now we will estimate the r.h.s. of \eqref{w5}, re-denoting $2R$ back to $R$. That is, will 
  estimate the cardinality
  of the set
  $$
  X = \{ (a,b) \in \bB(R) \times \bB(R)\mid \no a\an b \}\,. 
  $$
 It is clear that $(a,b) \in X$, $a\ne0$, iff there exist points $a', \ap \in \bS(|a|)$ such that $b$ lies in the line 
 $\Pi_{\aaa}$, which is perpendicular to the triangle $(\aaa)$ and passes through its centre, so it also passes 
 through the origin. Let $v = v_{\aaa}$ be a primitive integer vector in the direction of  $\Pi_{\aaa}$. 
 For any $a\in\Z^d, a\ne0$, denote
 $$
 \Delta(a) = \big\{\, \{a',\ap\} \subset \bS(|a|)\setminus \{a\}\mid a'\ne \ap \big\}\,.
 $$
 Then
 $$
 |\Delta(a)| < \Gamma_{|a|,3}^2 \le C^2_\theta R^{2\theta},\qquad \theta = \theta_3\,, 
 $$
 see \eqref{VC}. For a fixed $a\in \bB(R)\setminus \{0\}$ consider the mapping 
 $$
 \Delta(a) \ni \{a', \ap\} \mapsto v=v_{\aaa}\,.
 $$
 It is  clear that each direction $v=v_{\aaa}$  gives rise to at most $2R |v|^{-1}$ points 
 $b$ such that $(a,b)\in X$. So, denoting 
 $$
 X_a = \{ b\in \bB(R) \mid (a,b) \in X\}\,,
 $$
 we have
 $$
 |X_a| \le 2R \sum   |v_{\aaa}|^{-1}\,,\quad \text{if}\; a\ne0\,,
 $$
 where the summation goes through all different vectors $v$, corresponding to various  $\{a', \ap\}\in\Delta(a)$. 
  As $|v|^{-1}$ is the bigger the smaller $|v|$ is, we 
 see that 
  the r.h.s. is $\,\le 2R\sum_{v\in\bB(R')\setminus\{0\}}|v|^{-1} $, 
 where $R'$ is any number 
 such that $|\bB(R')| \ge |\Delta(a)|$. Since $|\Delta(a)| \le \Gamma_{|a|,3}^2$, then choosing 
 $R'=R'_a=C\Gamma_{|a|,3}^{2/3}$ we get for any $a\in\bB(R)\setminus \{0\}$ that 
 \begin{equation*}
 \begin{split}
 |X_a| \le 2CR \sum_{ \bB(R'_a)\setminus\{0\}}|v|^{-1} \le C_1 R \int_{ B(R'_a)}|x|^{-1}\,dx  
 \le C_2 R(R'_a)^2 = C_3 R \,\Gamma_{|a|,3}^{4/3}\,.
 \end{split}
 \end{equation*}
 Since $0\an b$ for any $b$,  then  $X_0 = \{0\}$ and 
 $$
 |X| = \sum_{a\in\bB(R)} |X_a| \le 1+ CR \sum_{a\in \bB(R)\setminus\{0\}}  \Gamma_{|a|,3}^{4/3}\,.
 $$
 Evoking the estimate \eqref{VC} we finally get that  
  \begin{equation*}
 \begin{split}
 |X| \le C_1R \sum_{a\in \bB(R)\setminus\{0\}}  |a|^{{\frac43\theta_3}}
 \le C_2R \int_{B(R)} |x|^{{\frac43\theta_3}}\,dx\le C_3 R^{1+3+{\frac43\theta_3}} = C_3 R^{5+1/3+\eps'}\,,
 \end{split}
 \end{equation*}
 with any positive $\eps'$.
 Jointly with  \eqref{w4}, \eqref{w5} and the definition of the set $X$ this 
  implies the required relation \eqref{w3} with $\ka=2/3 - \eps'$, and \eqref{om++}  follows. 
  That is, $n$-points random $R$-sets with large $R$ are typical, for any $d$ and any $n$.

  \section{Two lemmas}

 \subsubsection{Transversality}\label{ssTransversality}
\
\begin{lemma}\label{lTransv1} Let $I$ be an open interval and let  $f:I \to\R$ 
be a $\cC^{j}$-function whose $j$:th derivative satisfies
$$\ab{f^{(j)}(x)}\ge \delta,\quad \forall x\in I.$$
Then, 
$$
\Leb \{x\in I: \ab{f(x)}<\eps\}\le C (\frac\eps{\de_0})^{\frac1j}.$$

$C$ is a constant that only depends on $\ab{f}_{\C^j(I)}$.
\end{lemma}

\begin{proof} It is enough to prove this for $\eps<1$. Let $I_1=I$, $\delta_1=\delta$ and
$g_k=f^{(j-k)}$, $k=1,\dots j$. Let $\delta_1,\delta_2,\dots\delta_{j+1}$ be a deceasing sequence of positive numbers.

Since $g_1'=f^{(j)}$ we have $\ab{g'_1(x)}\ge\delta_1$ for all $x\in I_1$ and, hence, the set
$$E_1= \{x\in I_1: \ab{g_1(x)}<\delta_2\}$$
has Lebesgue measure $\lsim\frac{\delta_2}{\delta_1}$. On $I_2=I_1\setminus E_1$ we have $\ab{g'_2(x)}\ge\delta_2$ for all $x\in I_2$ and, hence, the set
$$E_2= \{x\in I_2: \ab{g_2(x)}<\delta_3\}$$
has Lebesgue measure $\lsim\frac{\delta_3}{\delta_2}$.
Continue this $j$ steps. On $I_j=I_{j-1}\setminus E_{j-1}$ we have $\ab{g'_{j}(x)}\ge\delta_{j}$ for all $x\in I_{j}$ and, hence, the set
$$E_j= \{x\in I_j: \ab{g_j(x)}<\delta_{j+1}\}$$
has Lebesgue measure $\lsim\frac{\delta_{j+1}}{\delta_{j}}$.

Now the set $\{x\in I: \ab{f(x)}<\delta_{j+1}\}$ is contained in the union of the sets $E_k$ which has measure
$$\lsim \frac{\delta_{2}}{\delta_{1}}+\dots+\frac{\delta_{j+1}}{\delta_{j}}.$$
Take now $\delta_k=\eta^{k-1}\delta$. Then this measure is $\lsim \eta$  and $\delta_{j+1}=\eta^j\delta$.
Chose finally $\eta$ so that $\eta^j\delta=\eps$.
\end{proof}

\subsubsection{Extension}\label{ssExtension}

\begin{lemma}\label{lExtension} Let $X\subset Y$ be subsets of $\D_0$ such that
$$\dist(\D_0\setminus Y,X)\ge \eps,$$
then there exists a $\cC^\infty$-function $g:\D_0\to\R$, being $=1$ on $X$ and $=0$ outside
$Y$ and such that for all $j\ge 0$
$$| g |_{\cC^j(\D_0)}\le C(\frac C{\eps})^j.$$
$C$ is an absolute constant.

\end{lemma}

\begin{proof}
This is  a classical result obtained by convoluting the characteristic function of $X$
with a $\cC^\infty$-approximation of the Dirac-delta supported in a ball of radius $\le  \frac{\eps}2$.
\end{proof}

\end{document}